\PassOptionsToPackage{unicode}{hyperref}
\PassOptionsToPackage{hyphens}{url}
\PassOptionsToPackage{dvipsnames,svgnames,x11names}{xcolor}
\documentclass[
  12pt]{article}

\usepackage{amsmath,amssymb}
\usepackage{tikz}
\usepackage{algorithm}
\usepackage{setspace}
\usepackage{amsmath,amssymb,amsfonts,bbm,verbatim,multirow,booktabs}
\usepackage{epic,eepic,psfrag,epsfig}
\usepackage{graphicx}
\usepackage{algorithm}
\usepackage{algpseudocode}
\usepackage{amsthm,breakcites}
\usepackage{amscd}
\usepackage{epsfig}
\usepackage{mathtools}
\usepackage{lipsum} 
\usepackage{tabularx}
\usepackage{graphicx}
\usepackage{array}
\usepackage{multirow}
\usepackage{longtable}
\usepackage{booktabs}
\usepackage[title,toc,titletoc,page]{appendix}
\usepackage{rotating}
\usetikzlibrary{decorations.pathreplacing,calc, patterns}
\usepackage{fontawesome}
\usepackage{tabularx}
\usepackage{array}

\usepackage{iftex}
\ifPDFTeX
  \usepackage[T1]{fontenc}
  \usepackage[utf8]{inputenc}
  \usepackage{textcomp} 
\else 
  \usepackage{unicode-math}
  \defaultfontfeatures{Scale=MatchLowercase}
  \defaultfontfeatures[\rmfamily]{Ligatures=TeX,Scale=1}
\fi
\usepackage{lmodern}
\ifPDFTeX\else  
\fi
\IfFileExists{upquote.sty}{\usepackage{upquote}}{}
\IfFileExists{microtype.sty}{
  \usepackage[]{microtype}
  \UseMicrotypeSet[protrusion]{basicmath} 
}{}
\makeatletter
\@ifundefined{KOMAClassName}{
  \IfFileExists{parskip.sty}{%
    \usepackage{parskip}
  }{
    \setlength{\parindent}{0pt}
    \setlength{\parskip}{6pt plus 2pt minus 1pt}}
}{
  \KOMAoptions{parskip=half}}
\makeatother
\usepackage{xcolor}
\makeatletter
\ifx\paragraph\undefined\else
  \let\oldparagraph\paragraph
  \renewcommand{\paragraph}{
    \@ifstar
      \xxxParagraphStar
      \xxxParagraphNoStar
  }
  \newcommand{\xxxParagraphStar}[1]{\oldparagraph*{#1}\mbox{}}
  \newcommand{\xxxParagraphNoStar}[1]{\oldparagraph{#1}\mbox{}}
\fi
\ifx\subparagraph\undefined\else
  \let\oldsubparagraph\subparagraph
  \renewcommand{\subparagraph}{
    \@ifstar
      \xxxSubParagraphStar
      \xxxSubParagraphNoStar
  }
  \newcommand{\xxxSubParagraphStar}[1]{\oldsubparagraph*{#1}\mbox{}}
  \newcommand{\xxxSubParagraphNoStar}[1]{\oldsubparagraph{#1}\mbox{}}
\fi
\makeatother

\usepackage{longtable,booktabs,array}
\usepackage{calc} 
\usepackage{etoolbox}
\makeatletter
\patchcmd\longtable{\par}{\if@noskipsec\mbox{}\fi\par}{}{}
\makeatother
\IfFileExists{footnotehyper.sty}{\usepackage{footnotehyper}}{\usepackage{footnote}}
\makesavenoteenv{longtable}
\usepackage{graphicx}
\makeatletter
\def\maxwidth{\ifdim\Gin@nat@width>\linewidth\linewidth\else\Gin@nat@width\fi}
\def\maxheight{\ifdim\Gin@nat@height>\textheight\textheight\else\Gin@nat@height\fi}
\makeatother
\setkeys{Gin}{width=\maxwidth,height=\maxheight,keepaspectratio}
\makeatletter
\def\fps@figure{htbp}
\makeatother

\makeatletter
\@ifpackageloaded{caption}{}{\usepackage{caption}}
\AtBeginDocument{%
\ifdefined\contentsname
  \renewcommand*\contentsname{Table of contents}
\else
  \newcommand\contentsname{Table of contents}
\fi
\ifdefined\listfigurename
  \renewcommand*\listfigurename{List of Figures}
\else
  \newcommand\listfigurename{List of Figures}
\fi
\ifdefined\listtablename
  \renewcommand*\listtablename{List of Tables}
\else
  \newcommand\listtablename{List of Tables}
\fi
\ifdefined\figurename
  \renewcommand*\figurename{Figure}
\else
  \newcommand\figurename{Figure}
\fi
\ifdefined\tablename
  \renewcommand*\tablename{Table}
\else
  \newcommand\tablename{Table}
\fi
}
\@ifpackageloaded{float}{}{\usepackage{float}}
\floatstyle{ruled}
\@ifundefined{c@chapter}{\newfloat{codelisting}{h}{lop}}{\newfloat{codelisting}{h}{lop}[chapter]}
\floatname{codelisting}{Listing}

\makeatother
\makeatletter
\@ifpackageloaded{caption}{}{\usepackage{caption}}
\@ifpackageloaded{subcaption}{}{\usepackage{subcaption}}
\makeatother

\ifLuaTeX
  \usepackage{selnolig}  
\fi
\usepackage[]{natbib}
\usepackage{bookmark}

\IfFileExists{xurl.sty}{\usepackage{xurl}}{} 
\hypersetup{
  pdftitle={Title},
  pdfauthor={Author 1; Author 2},
  pdfkeywords={3 to 6 keywords, that do not appear in the title},
  colorlinks=true,
  linkcolor={blue},
  filecolor={Maroon},
  citecolor={Blue},
  urlcolor={Blue},
  pdfcreator={LaTeX via pandoc}}

\newcommand{\anon}{1}

\newtheorem{thm}{Theorem}

\newtheorem{lemma}[thm]{Lemma}

\newtheorem{prop}[thm]{Proposition}
\newtheorem{cor}[thm]{Corollary}

\newtheorem{assump}{Assumption}

\begin{document}

\def\spacingset#1{\renewcommand{\baselinestretch}%
{#1}\small\normalsize} \spacingset{1}


\if1\anon
{
  \title{\bf Estimation of Linear Functionals in Multilayer Panels under Staggered Adoption}
  \author{Alberto Bordino \hspace{.2cm}\\
    Department of Statistics, University of Warwick,\\
    Coventry, United Kingdom, CV4 7AL, \\
Alberto.Bordino@warwick.ac.uk \\  \\
     Thomas B.~Berrett \\
    Department of Statistics, University of Warwick,  \\
    Coventry, United Kingdom, CV4 7AL \\ \\
    Olga Klopp \\
    ESSEC Business School, \\
    3 Av. Bernard Hirsch, 95021 Cergy, France }
  \maketitle
} \fi

\if0\anon
{
  \bigskip
  \bigskip
  \bigskip
  \begin{center}
    {\LARGE\bf Estimation of Linear Functionals in
Multilayer Panels under Staggered Adoption}
\end{center}
  
} \fi

\bigskip
\begin{abstract}
We study the estimation of bilinear forms from noisy, partially observed multilayer data. The signal follows a Tucker2 model, with shared unit and time factors across tensor layers and slice-specific cores. The missingness pattern is structured and motivated by staggered adoption designs, which are common in causal inference and related applications. We first analyze the four-block missingness pattern, the basic building block for general staggered adoption, and propose a spectral algorithm that pools information across layers and targets the functional directly. We prove a non-asymptotic mean squared error bound that exhibits a phase transition in the number of layers, showing when pooling improves estimation, and match it with a local minimax lower bound up to constants {when ranks and logarithmic factors are treated as constant-order quantities}. We then extend the construction to general staggered adoption designs via an anchored four-block reduction, and derive analogous theoretical guarantees. Finally, we validate our theoretical findings using synthetic and real-world data on Castle Doctrine laws and COVID-19 policies.
\end{abstract}

\noindent%

\vfill

\newpage
\spacingset{1.8} 
\setlength{\parskip}{0pt}
\setlength{\parindent}{15pt}

\section{Introduction}

In modern statistics, data often depart from the idealized i.i.d.~setting and, in many applications, take the form of multiple incomplete panels. For instance, we may observe different responses for each individual over time, with data recorded at some time points but not others. Rather than recovering all missing values, practitioners may instead aim to estimate a simple functional of a target layer, such as a \textit{bilinear form}, using its observed entries and possibly incorporating information from additional layers. This naturally raises two questions: how should information from auxiliary panels be used to estimate the target functional, and when does adding further panels cease to improve accuracy?

\begin{figure}[!ht]
\centering
\includegraphics[width=0.8\linewidth]{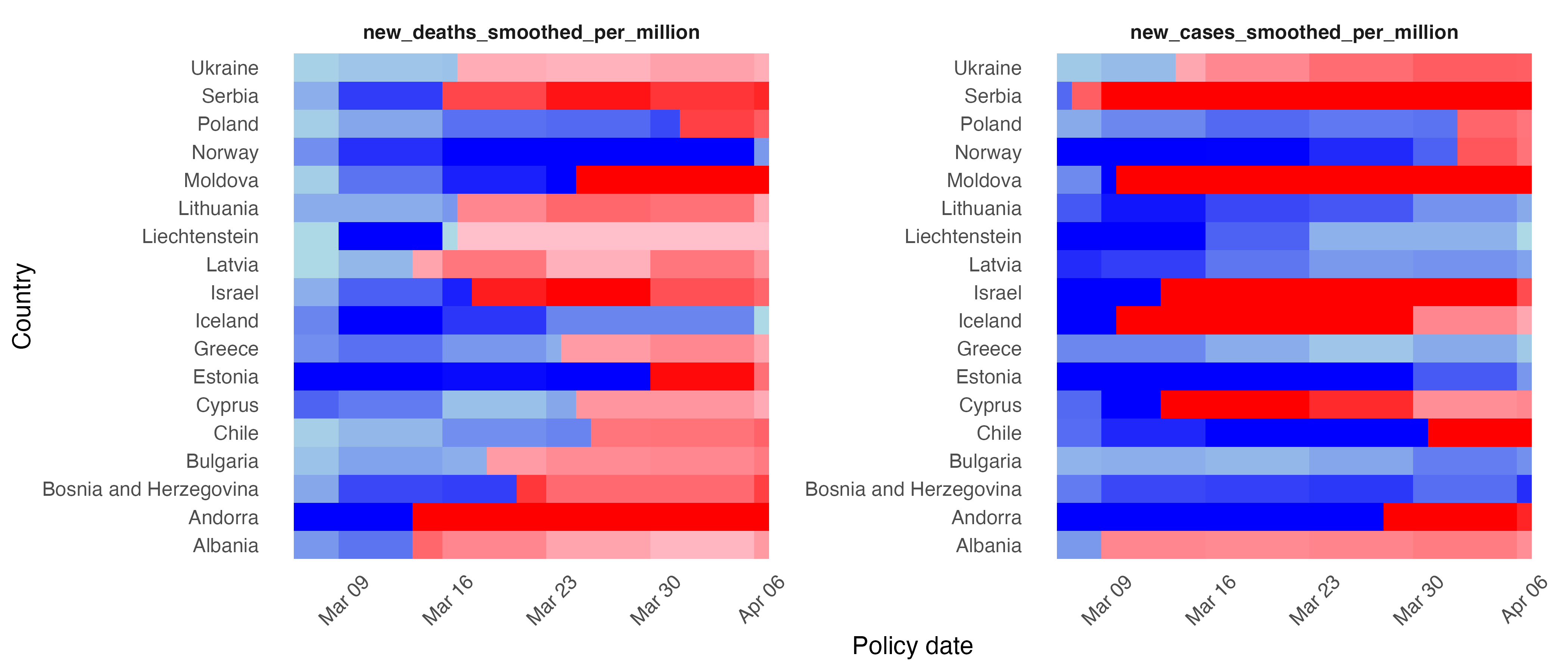}
\captionsetup{
font={scriptsize,stretch=0.8},
skip=10pt
}
\caption{Outcome tensor from the merged OxCGRT--OWID panel. Rows are countries, columns are dates from March 5 to April 5, 2020, and panels are outcome--policy pairs. Outcomes are measured at a 28-day delay, so date~$t$ shows the value at $t+28$. Blue denotes policy-off entries, red policy-on entries, and darker shades larger outcome values.}
\label{fig:covid-y-outcome-policy}
\label{fig:CovidY}
\end{figure}
\vspace{-15pt}

We study these questions under a \textit{Tucker2 model}~\citep[][Section~4]{Kolda2009Tensor} in which the panels share unit and time factors but have layer-specific cores, and the missingness follows a structured missing-not-at-random (MNAR) pattern. In particular, we focus on \textit{staggered adoption} designs~\citep{AtheyImbens2022DesignBasedDiD}, which often arise in causal inference and related applications, where units are observed over multiple periods and, once treated, remain treated thereafter. From the perspective of untreated potential outcomes, observations are therefore available for all units before adoption and missing for treated units after adoption. The simplest instance of this structure is the four-block setting illustrated in Fig.~\ref{fig:4BlockSlices} and analyzed in Section~\ref{sec:4block}. A concrete example of a multilayer panel with a  more general staggered missingness is provided  by the COVID-19 policy application studied in Section~\ref{sec:real-data-oxcgrt-owid}. Fig.~\ref{fig:CovidY} shows the two panels corresponding to stay-at-home policies and contact tracing, with red and blue indicating whether the policy is on or off, respectively, and color intensity representing the number of new deaths and new cases.

In applications of this kind, our contribution is a spectral method for estimating bilinear forms of a target layer under staggered missingness while borrowing information from related auxiliary layers. This class of functionals arises in many applications \citep[e.g.][Section~3]{callaway2021difference}, and includes averages, contrasts, individual components, and temporal trends. Our theory shows when such pooling helps and when its benefits saturate. More precisely, since additional layers improve estimation of the shared row and column spaces but do not reduce uncertainty in the layer-specific core of the target panel, the estimation error initially decreases as more auxiliary information becomes available, but eventually reaches a target-specific floor. This \textit{phase transition} is illustrated in Fig.~\ref{fig:phaseTransition}.

\textbf{Related work.} An alternative way to estimate a functional of a partially observed panel is to impute the missing entries and use the completed layer to construct plug-in estimators of these quantities. Low-rank matrix completion provides a natural framework for this task. Although such methodologies were developed mainly for missing completely at random observation patterns \citep[e.g.,][]{candes_recht_2009, Keshavan_Montanari_Oh, Negahban_Wainwright,Koltchinskii_Lounici_Tsybakov,klopp_general,
ChiLuChen2019NonconvexOptimization}, they can be repurposed in causal panels by treating unobserved untreated outcomes as missing entries of an approximately low-rank matrix. \citet{athey2021matrix} formalized this connection by relating low-rank matrix completion to unconfoundedness-based methods \citep[e.g.,][]{ImbensRubin2015} and synthetic-control methods \citep[e.g.,][]{Abadie2021}, and proposed estimating the missing counterfactual entries through nuclear-norm penalized least squares.

A subsequent literature has developed matrix-completion methods for MNAR settings with structured missingness. \citet{choi2024matrix} study staggered adoption designs and extend the nuclear-norm approach of \citet{athey2021matrix} by partitioning the missing entries into groups and applying convex relaxation within each group. They prove $\ell_\infty$ estimation error bounds that improve on the Frobenius-norm bound obtained in \citet{athey2021matrix}. Similarly, \citet{agarwal2026robustmatrixestimation} study a related problem with row and column side information, providing Frobenius-norm guarantees for an estimator based on sieve projection and nuclear-norm penalization. Alongside optimization-based approaches, \citet{yan2024entrywise} consider panels with staggered adoption and propose a spectral algorithm based on singular value decomposition, with non-asymptotic entrywise guarantees and Gaussian approximations. Related factor-based approaches, including \citet{BaiNg2021MatrixCompletionCounterfactuals} and \citet{CahanBaiNg2023FactorBasedImputation}, exploit tall and wide observed blocks to estimate latent factors and impute missing entries. Finally, \citet{Agarwal2023CausalMatrixCompletion} developed a completion method based on synthetic nearest neighbors for a broad class of MNAR patterns, with $\ell_\infty$ error bounds and asymptotic normality.

When the data are naturally organized as multiple related panels, as in our COVID-19 application, this matrix-completion perspective extends to tensors, allowing dependence across layers to be exploited. Recent work has begun to develop tensor-completion methods under structured missingness, but the theory remains less developed than in the matrix case. \citet{Auerbach2022TensorCompletionCausalInference} arrange multivariate longitudinal outcomes as a unit-by-time-by-outcome tensor and use nuclear-norm penalization to impute the missing entries and study COVID-19 mandates. \citet{Agarwal2025SyntheticInterventions} extend synthetic-control ideas to multiple treatments using a low-rank tensor factor model. \citet{Mandal2019WeightedTensorCompletion} and \citet{Gao2025CausalInferenceSequentialTreatmentsTensor} consider tensor formulations for longitudinal causal problems, where treatment histories are stacked along an additional tensor mode. In particular, \citet{Gao2025CausalInferenceSequentialTreatmentsTensor} estimate the latent tensor using an inverse-probability-weighted low-rank Tucker formulation, implemented by projected gradient descent. Their main guarantee is a non-asymptotic Frobenius-norm bound for tensor recovery (Theorem 1). In Remark 2, they relate their framework to \citet{athey2021matrix}, noting that a special case reduces to a staggered-adoption panel setting with two potential-outcome matrices.

The primary target in much of this literature, however, remains recovery of the missing matrix or tensor, either as a whole or entry by entry. The problem of estimating general bilinear forms is briefly mentioned by \citet[][Appendix~C]{xia2025inferencestaggeredadoptioncase}, but no theoretical guarantees are provided for this target. Instead, existing error bounds are typically stated for full recovery, for instance in Frobenius norm or entrywise $\ell_\infty$ norm, even when the ultimate object of interest is only a particular functional of the missing entries. Compared with these plug-in approaches, our method estimates the requested bilinear form without first materializing the entire missing block, thereby reducing computational cost and allowing a functional-specific analysis.

Moreover, pooling information across panel slices can improve statistical estimation. In the COVID-19 application in Section~\ref{sec:real-data-oxcgrt-owid}, our method is more accurate than approaches using only the target layer. This is likely due to the limited information available in the first layer, which contains just two fully observed rows (Fig.~\ref{fig:CovidY}), making estimation particularly challenging and highlighting the importance of exploiting shared structure across slices.

\textbf{Notation.} For $\mathcal X\in\mathbb R^{n_1\times n_2\times n_3}$ and index sequences $I^{(t)}=\{i^{(t)}_1,\ldots,i^{(t)}_{|I^{(t)}|}\}\subseteq[n_t]$, $t\in\{1,2,3\}$, let $\mathcal X_{I^{(1)},I^{(2)},I^{(3)}}\in\mathbb R^{|I^{(1)}|\times|I^{(2)}|\times|I^{(3)}|}$ denote the subtensor with $(\mathcal X_{I^{(1)},I^{(2)},I^{(3)}})_{k_1,k_2,k_3}=\mathcal X_{i^{(1)}_{k_1},i^{(2)}_{k_2},i^{(3)}_{k_3}}$. The symbol $\bullet$ denotes the full index set in the corresponding mode, e.g., $\mathcal X_{\bullet,I^{(2)},\bullet}:=\mathcal X_{[n_1],I^{(2)},[n_3]}$; singleton sets are written without braces, and the same convention applies to matrices and vectors. We write $\boldsymbol 0_d$, $\boldsymbol 1_d$, $\boldsymbol I_d$, and $\boldsymbol e_j^{(d)}$ for the zero vector, all-one vector, identity matrix, and $j$th canonical basis vector in $\mathbb R^d$, omitting $d$ when clear, and set $\boldsymbol O_{d_1\times d_2}:=\boldsymbol 0_{d_1}\boldsymbol 0_{d_2}^\top$ and $\boldsymbol 1_{d_1\times d_2}:=\boldsymbol 1_{d_1}\boldsymbol 1_{d_2}^\top$. For symmetric $\boldsymbol A,\boldsymbol B\in\mathbb R^{d\times d}$, $\boldsymbol A\succeq0$ means that $\boldsymbol A$ is positive semidefinite, and $\boldsymbol A\succeq\boldsymbol B$ means $\boldsymbol A-\boldsymbol B\succeq0$. We write $\operatorname{tr}(\boldsymbol A)$ for the trace and $\operatorname{diag}(\boldsymbol v)$ for the diagonal matrix with diagonal~$\boldsymbol v$. The minimum, maximum, and $j$th largest eigenvalues of a symmetric matrix $\boldsymbol A$ are denoted by $\lambda_{\min}(\boldsymbol A)$, $\lambda_{\max}(\boldsymbol A)$, and $\lambda_j(\boldsymbol A)$; for a general matrix $\boldsymbol B$, the analogous singular values are $\sigma_{\min}(\boldsymbol B)$, $\sigma_{\max}(\boldsymbol B)$, and $\sigma_j(\boldsymbol B)$. We define $P_{\boldsymbol\Omega}(\boldsymbol M):=\boldsymbol\Omega\odot\boldsymbol M$, where $\odot$ denotes the Hadamard product. Moreover, $\operatorname{SVD}_r(\boldsymbol A)=(\boldsymbol U,\boldsymbol\Sigma,\boldsymbol V)$ denotes the rank-$r$ truncated SVD, where $\boldsymbol U\in\mathbb R^{n_1\times r}$ and $\boldsymbol V\in\mathbb R^{n_2\times r}$ contain the top-$r$ left and right singular vectors and $\boldsymbol\Sigma\in\mathbb R^{r\times r}$ contains the corresponding singular values. If $\boldsymbol A=\boldsymbol U\operatorname{diag}(\sigma_1,\ldots,\sigma_r)\boldsymbol V^\top$ with $\sigma_i>0$, then $\boldsymbol A^\dagger:=\boldsymbol V\operatorname{diag}(\sigma_1^{-1},\ldots,\sigma_r^{-1})\boldsymbol U^\top$ is its Moore--Penrose pseudoinverse. We use $\|\cdot\|_p$ for vector $\ell_p$-norms, $\|\cdot\|_{\mathrm{op}}$ and $\|\cdot\|_F$ for matrix spectral and Frobenius norms, and $\langle\cdot,\cdot\rangle$ for the Euclidean inner product. Also, $\mathcal S^{d-1}:=\{\boldsymbol x\in\mathbb R^d:\|\boldsymbol x\|_2=1\}$ is the unit sphere, and $\mathcal O(x)$ denotes a quantity whose absolute value is bounded by $Cx$ for some constant~$C>0$.

\section{Bilinear form estimation with $4$-block~missingness}\label{sec:4block}
\subsection{Statistical setting and main result}\label{sec:4BlockUB}

In this section and the next, we study bilinear-form estimation in the four-block setting of Fig.~\ref{fig:4BlockSlices}, the simplest nontrivial building block for the staggered adoption design in Section~\ref{sec:bilinearStaggered}. For fixed dimensions $N,T,K\ge 1$, we define the deterministic missingness-pattern tensor $\Omega \in \{0,1\}^{N \times T \times K}$ where, for all $j \in [K]$,
\begin{align}\label{eq:Omega4block}
    \Omega_{\bullet, \bullet, j}
=
\begin{pmatrix}
\boldsymbol{1}_{N_{1j} \times T_{1j}} & \boldsymbol{1}_{N_{1j} \times T_{2j}}\\
\boldsymbol{1}_{N_{2j}\times T_{1j}} & \boldsymbol{O}_{N_{2j}\times T_{2j}}
\end{pmatrix} \in \mathbb{R}^{N \times T},
\end{align}
with $N=N_{1j} +N_{2j}$ and $T=T_{1j} + T_{2j}$. The tensor $\Omega$ will be fixed throughout this section. We observe $\mathcal{Y}\in\mathbb{R}^{N\times T\times K}$ with 
\begin{align}\label{eq:Observation4block}
    \mathcal{Y}_{\bullet, \bullet, j} := P_{\Omega_{\bullet, \bullet, j}}\!\big(\mathcal M_{\bullet, \bullet, j}+\mathcal E_{\bullet, \bullet, j}\big)
=
\begin{pmatrix}
\mathcal M_{\bullet, \bullet, j}^{(a)} + \mathcal E_{\bullet, \bullet, j}^{(a)} & \mathcal M_{\bullet, \bullet, j}^{(b)} + \mathcal E_{\bullet, \bullet, j}^{(b)} \\
\mathcal M_{\bullet, \bullet, j}^{(c)} + \mathcal E_{\bullet, \bullet, j}^{(c)} & \texttt{NA}
\end{pmatrix},
\end{align}
where $a,b,c$ refer to the observed blocks in~\eqref{eq:Omega4block}, and $\mathcal{E}\in\mathbb{R}^{N\times T\times K}$ is such that $\mathcal  E_{i, t, j} \overset{\mathrm{i.i.d.}}{\sim} {\cal N}(0, \sigma^2)$ for all $(i, t, j) \in  [N] \times [T] \times [K]$. Fig.~\ref{fig:4BlockSlices} illustrates the observed tensor $\mathcal{Y}$ for $K=3$. The bottom-right block need not be missing in every layer. It suffices that each row is either fully observed or missing from some slice-specific time onward. We nevertheless use the four-block design in~\eqref{eq:Omega4block} to simplify the exposition. 

\begin{figure}[htbp]
\vspace{-10pt}
\centering
\begin{tikzpicture}[
    scale=0.7,
    transform shape,
    x=0.75cm,
    y=0.75cm,
    grid/.style={black!70, line width=0.35pt},
    heavy/.style={black, line width=0.9pt},
    arrowlbl/.style={font=\small},
    blocklbl/.style={font=\small},
    matlbl/.style={font=\large},
    blueblock/.style={fill=blue!25},
    pinkblock/.style={fill=pink!45},
    greenblock/.style={fill=green!25},
    grayblock/.style={fill=gray!25},
]

\def\n{6}

\begin{scope}[shift={(0,0)}]

\def\xcut{2.2}
\def\ycut{3.8}

\fill[blueblock]  (0,\ycut) rectangle (\xcut,\n);
\fill[pinkblock]  (\xcut,\ycut) rectangle (\n,\n);
\fill[greenblock] (0,0) rectangle (\xcut,\ycut);
\fill[grayblock]  (\xcut,0) rectangle (\n,\ycut);

\draw[heavy] (0,0) rectangle (\n,\n);
\draw[grid] (\xcut,0) -- (\xcut,\n);
\draw[grid] (0,\ycut) -- (\n,\ycut);

\node[blocklbl] at ({\xcut/2},{(\ycut+\n)/2}) {$\mathcal Y_{\bullet,\bullet,1}^{(a)}$};
\node[blocklbl] at ({(\xcut+\n)/2},{(\ycut+\n)/2}) {$\mathcal Y_{\bullet,\bullet,1}^{(b)}$};
\node[blocklbl] at ({\xcut/2},{\ycut/2}) {$\mathcal Y_{\bullet,\bullet,1}^{(c)}$};
\node[blocklbl] at ({(\xcut+\n)/2},{\ycut/2}) {\texttt{NA}};

\draw[<->] (-0.45,\ycut) -- (-0.45,\n)
    node[midway,left,arrowlbl] {$N_{11}$};
\draw[<->] (-0.45,0) -- (-0.45,\ycut)
    node[midway,left,arrowlbl] {$N_{21}$};

\draw[<->] (-1.61,0) -- (-1.61,\n)
    node[midway,left,arrowlbl] {$N$};

\draw[<->] (0,\n+0.45) -- (\xcut,\n+0.45)
    node[midway,above,arrowlbl] {$T_{11}$};
\draw[<->] (\xcut,\n+0.45) -- (\n,\n+0.45)
    node[midway,above,arrowlbl] {$T_{21}$};

\draw[<->] (0,\n+1.15) -- (\n,\n+1.15)
    node[midway,above,arrowlbl] {$T$};

\node[matlbl] at (3,-0.75) {$\mathcal Y_{\bullet,\bullet,1}$};

\end{scope}

\begin{scope}[shift={(8.4,0)}]

\def\xcut{3.4}
\def\ycut{2.4}

\fill[blueblock]  (0,\ycut) rectangle (\xcut,\n);
\fill[pinkblock]  (\xcut,\ycut) rectangle (\n,\n);
\fill[greenblock] (0,0) rectangle (\xcut,\ycut);
\fill[grayblock]  (\xcut,0) rectangle (\n,\ycut);

\draw[heavy] (0,0) rectangle (\n,\n);
\draw[grid] (\xcut,0) -- (\xcut,\n);
\draw[grid] (0,\ycut) -- (\n,\ycut);

\node[blocklbl] at ({\xcut/2},{(\ycut+\n)/2}) {$\mathcal Y_{\bullet,\bullet,2}^{(a)}$};
\node[blocklbl] at ({(\xcut+\n)/2},{(\ycut+\n)/2}) {$\mathcal Y_{\bullet,\bullet,2}^{(b)}$};
\node[blocklbl] at ({\xcut/2},{\ycut/2}) {$\mathcal Y_{\bullet,\bullet,2}^{(c)}$};
\node[blocklbl] at ({(\xcut+\n)/2},{\ycut/2}) {\texttt{NA}};

\draw[<->] (-0.45,\ycut) -- (-0.45,\n)
    node[midway,left,arrowlbl] {$N_{12}$};
\draw[<->] (-0.45,0) -- (-0.45,\ycut)
    node[midway,left,arrowlbl] {$N_{22}$};

\draw[<->] (0,\n+0.45) -- (\xcut,\n+0.45)
    node[midway,above,arrowlbl] {$T_{12}$};
\draw[<->] (\xcut,\n+0.45) -- (\n,\n+0.45)
    node[midway,above,arrowlbl] {$T_{22}$};

\node[matlbl] at (3,-0.75) {$\mathcal Y_{\bullet,\bullet,2}$};

\end{scope}

\begin{scope}[shift={(16.8,0)}]

\def\xcut{1.6}
\def\ycut{4.6}

\fill[blueblock]  (0,\ycut) rectangle (\xcut,\n);
\fill[pinkblock]  (\xcut,\ycut) rectangle (\n,\n);
\fill[greenblock] (0,0) rectangle (\xcut,\ycut);
\fill[grayblock]  (\xcut,0) rectangle (\n,\ycut);

\draw[heavy] (0,0) rectangle (\n,\n);
\draw[grid] (\xcut,0) -- (\xcut,\n);
\draw[grid] (0,\ycut) -- (\n,\ycut);

\node[blocklbl] at ({\xcut/2},{(\ycut+\n)/2}) {$\mathcal Y_{\bullet,\bullet,3}^{(a)}$};
\node[blocklbl] at ({(\xcut+\n)/2},{(\ycut+\n)/2}) {$\mathcal Y_{\bullet,\bullet,3}^{(b)}$};
\node[blocklbl] at ({\xcut/2},{\ycut/2}) {$\mathcal Y_{\bullet,\bullet,3}^{(c)}$};
\node[blocklbl] at ({(\xcut+\n)/2},{\ycut/2}) {\texttt{NA}};

\draw[<->] (-0.45,\ycut) -- (-0.45,\n)
    node[midway,left,arrowlbl] {$N_{13}$};
\draw[<->] (-0.45,0) -- (-0.45,\ycut)
    node[midway,left,arrowlbl] {$N_{23}$};

\draw[<->] (0,\n+0.45) -- (\xcut,\n+0.45)
    node[midway,above,arrowlbl] {$T_{13}$};
\draw[<->] (\xcut,\n+0.45) -- (\n,\n+0.45)
    node[midway,above,arrowlbl] {$T_{23}$};

\node[matlbl] at (3,-0.75) {$\mathcal Y_{\bullet,\bullet,3}$};

\end{scope}

\end{tikzpicture}
\captionsetup{
  font={scriptsize,stretch=0.8},
  skip=0pt
}
\caption{Layer-specific four-block structure for $K=3$. In slice $\mathcal Y_{\bullet,\bullet,j}$, $N_{1j}$ rows are fully observed, while $N_{2j}$ rows are observed only in the first $T_{1j}$ columns. Here, $N$ and $T$ denote the total numbers of rows and columns; $a,b,c$ are observed blocks and $d$ is missing.}
\label{fig:4BlockSlices}
\vspace{-15pt}
\end{figure}
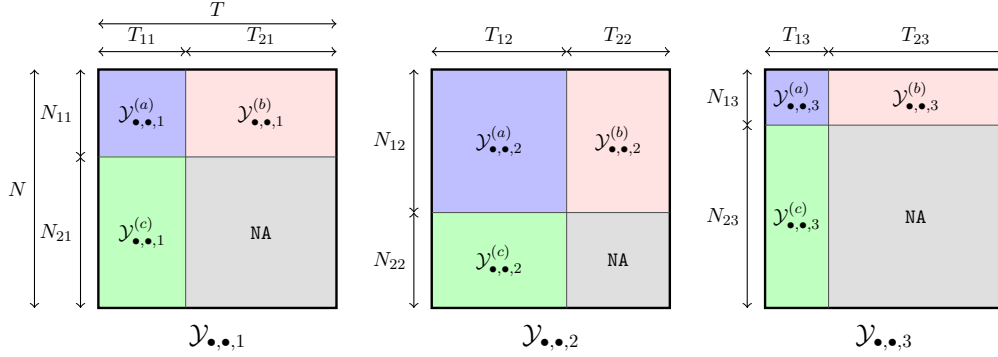

As for the signal tensor $\mathcal{M}\in\mathbb{R}^{N\times T\times K}$, for $r \geq 1$ with $r\le \min(N,T)$, we assume that $\mathcal{M}$ admits a Tucker2 decomposition of rank $(r,r,K)$, meaning that there exist factor matrices $\boldsymbol U\in\mathbb{R}^{N\times r}$ and $\boldsymbol V\in\mathbb{R}^{T\times r}$ satisfying $\boldsymbol U^\top \boldsymbol U=\boldsymbol V^\top \boldsymbol V=\boldsymbol I_r$, and a core tensor $\mathcal C \in\mathbb R^{r\times r\times K}$ such that $\mathcal{M} = \mathcal C \times_1 \boldsymbol U \times_2 \boldsymbol V \times_3 \boldsymbol I_{K}$. Equivalently, each layer-specific signal matrix has rank at most~$r$ and admits the factorization $\mathcal{M}_{\bullet,\bullet,j} = \boldsymbol U \, \mathcal C_{\bullet,\bullet,j}\, \boldsymbol V^\top \in~\mathbb{R}^{N \times T}$; a proof of the equivalence between these two formulations is given in Section~\ref{appendix:tensor}  in the Supplement. This condition permits heterogeneity across slices while borrowing strength through shared latent row and column spaces. Similar modeling assumptions, often referred to as \textit{common-subspace models}, have been studied 
in statistical settings under complete observation \citep{agterberg2026statistically,arroyo2021inference} and are motivated by biological applications, including neuroscience and RNA sequencing \citep{semedo2019cortical,ma2026optimal}. 

It will also be convenient to partition $\boldsymbol U=(\boldsymbol U_{1j}\,;\,\boldsymbol U_{2j})$ and~$\boldsymbol V=(\boldsymbol V_{1j}\,;\,\boldsymbol V_{2j})$ according to the $(N_{1j},N_{2j})$ and $(T_{1j},T_{2j})$ splits induced by slice $j$, where the semicolon denotes vertical stacking. Formally, for each $j\in[K]$ we define $\boldsymbol U_{1j} = \boldsymbol U_{[N_{1j}] \,, \, \bullet}\in\mathbb{R}^{N_{1j}\times r}$ and $\boldsymbol U_{2j} =  \boldsymbol U_{\{N_{1j} + 1, \ldots, N\} \, , \, \bullet}\in\mathbb{R}^{N_{2j}\times r}$; we define $\boldsymbol V_{1j}$ and $\boldsymbol V_{2j}$ analogously. Under this notation, the unobserved bottom-right block satisfies $\mathcal M_{\bullet, \bullet, j}^{(d)}=\boldsymbol U_{2j} \, \mathcal C_{\bullet, \bullet, j}\,  \boldsymbol V_{2j}^\top$.  

Our focus is on estimating general bilinear forms of the missing $d$-blocks. Formally, fix $k\in[K]$ and unit vectors $\boldsymbol x \in \mathcal S^{N_{2k}-1}, \, \boldsymbol y \in \mathcal S^{T_{2k}-1}$, and define
\begin{equation}\label{eq:mu4block}
\mu_{\boldsymbol x \boldsymbol y}^{(k)} := \boldsymbol x^\top \mathcal M_{\bullet, \bullet, k}^{(d)} \,\, \boldsymbol y .
\end{equation}
This target class includes both average and individual counterfactual components, depending on whether $\boldsymbol x, \,\boldsymbol y$ are constant or canonical basis vectors. Our proposed estimator, given in Algorithm~\ref{alg:bilinear4block} in Section~\ref{sec:methodology4Block}, predicts the missing $d$-block from the $b$-block by regressing the $c$-block on the $a$-block, exploiting the shared-subspace structure of the Tucker2 model. Its mean squared error is illustrated in Fig.~\ref{fig:phaseTransition}.

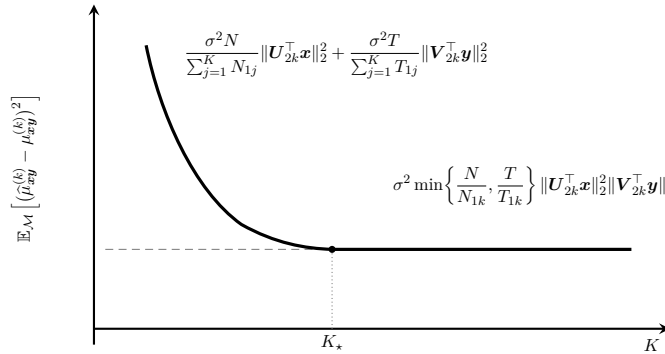
\begin{figure}
    \centering
\begin{tikzpicture}[
  >=stealth,
  every node/.style={font=\small},
  axis/.style={thick,->},
  curve/.style={very thick}, scale=0.6, transform shape]
  \def\xstar{5.25}
  \def\yfloor{1.75}

  \draw[axis] (0,0) -- (12.7,0) node[below left=2pt] {$K$};
  \draw[axis] (0,-0.35) -- (0,7.15);
  \node[rotate=90,anchor=south] at (-1.05,3.55)
    {$\displaystyle \mathbb{E}_{\mathcal M}\left[\bigl(\widehat\mu_{\boldsymbol x \boldsymbol y}^{(k)}-\mu_{\boldsymbol x \boldsymbol y}^{(k)}\bigr)^2\right]$};

  \draw[densely dashed,gray] (0.25,\yfloor) -- (\xstar,\yfloor);
  \draw[densely dotted,gray] (\xstar,0) -- (\xstar,\yfloor);

  \draw[curve]
    (1.15,6.25)
    .. controls (1.45,4.75) and (2.20,3.05) .. (3.25,2.30)
    .. controls (4.00,1.88) and (4.65,1.77) .. (\xstar,\yfloor)
    -- (11.85,\yfloor);
  \fill (\xstar,\yfloor) circle (2.2pt);
  \node[below] at (\xstar,0) {$K_\star$};

  \node[anchor=north west,align=left] at (1.85,6.70) {
    $\displaystyle
      \frac{\sigma^2 N}{\sum_{j=1}^{K}N_{1j}}
      \lVert \boldsymbol U_{2k}^{\top} \boldsymbol x\rVert_2^2
      +
      \frac{\sigma^2 T}{\sum_{j=1}^{K}T_{1j}}
      \lVert  \boldsymbol V_{2k}^{\top} \boldsymbol y\rVert_2^2$};
      
  \node[anchor=west,align=left] at (6.45,3.15) {
    $\displaystyle
      \sigma^2\min\!\left\{\frac{N}{N_{1k}},\frac{T}{T_{1k}}\right\}
      \lVert \boldsymbol U_{2k}^{\top}\boldsymbol x\rVert_2^2
      \lVert \boldsymbol V_{2k}^{\top}\boldsymbol y\rVert_2^2$
  };
\end{tikzpicture}
\captionsetup{
  font={scriptsize,stretch=0.8},
  skip=0pt
}
    \caption{Pictorial illustration of the MSE of the estimator in Algorithm~\ref{alg:bilinear4block} based on Theorem~\ref{thm:4BlockUBLinear}.}
    \label{fig:phaseTransition}
    \vspace{-15pt}
\end{figure}

In particular, under certain assumptions discussed in Section~\ref{sec:assumptions}, Theorem~\ref{thm:4BlockUBLinear} in Section~\ref{sec:upperBound} gives a precise upper bound on the estimation error that reveals the effect of pooling across layers. Treating ranks and logarithmic factors as constant-order quantities, the mean squared error is bounded up to constants by $\sigma^2 N \, \|\boldsymbol U_{2k}^\top \boldsymbol x\|_2^2/\sum_{j=1}^{K}N_{1j}+\sigma^2 T\,\|\boldsymbol V_{2k}^\top \boldsymbol y\|_2^2 / \sum_{j=1}^{K}T_{1j}$ in the small-$K$ regime, and by $\sigma^2\min\left(\frac{N}{N_{1k}},\frac{T}{T_{1k}}\right)\|\boldsymbol U_{2k}^\top \boldsymbol x\|_2^2\|\boldsymbol V_{2k}^\top \boldsymbol y\|_2^2$ in the large-$K$ regime. The~first two terms arise from estimating the shared factors $\boldsymbol U$ and $\boldsymbol V$, and therefore decrease with~$K$. The third is an irreducible layer-specific error that does not decrease with $K$, reflecting the difficulty of estimating $\mathcal C_{\bullet,\bullet,k}$. This phase transition in the rate as a function~of~$K$ is illustrated by the
simulation study in Fig.~\ref{fig:DecayWithK}, and complemented by a matching local minimax lower bound in Theorem~\ref{thm:CombinedLB} in Section~\ref{sec:MinimaxLowerBounds}.

\subsection{Proposed methodology for four-block missingness}\label{sec:methodology4Block}
We now introduce Algorithm~\ref{alg:bilinear4block} to estimate quantities of the form~\eqref{eq:mu4block}. To simplify the presentation of the method and its analysis, we introduce some notation, summarized in Table~\ref{tab:block-notation-main}. We recall that parentheses with a space denote horizontal concatenation; semicolons denote vertical concatenation.

\begin{table}[!ht]
    \centering
    \scriptsize
    \renewcommand{\arraystretch}{1.05}
    \setlength\extrarowheight{4pt}

    \resizebox{\textwidth}{!}{%
    \begin{tabular}{|
        >{\centering\arraybackslash}m{0.13\textwidth}|
        >{\raggedright\arraybackslash}m{0.35\textwidth}|
        !{\vrule width 1.2pt}
        >{\centering\arraybackslash}m{0.14\textwidth}|
        >{\raggedright\arraybackslash}m{0.33\textwidth}|}
        \hline
        \textbf{Notation} & \textbf{Definition}
        & \textbf{Notation} & \textbf{Definition} \\
        \hline

        $\boldsymbol{M}_{\mathrm{up}}^{(j)}$ &
        $(\mathcal{M}_{\bullet,\bullet,j}^{(a)}
        \ \mathcal{M}_{\bullet,\bullet,j}^{(b)})
        \in \mathbb{R}^{N_{1j}\times T}$
        &
        $\boldsymbol{M}_{\mathrm{left}}^{(j)}$ &
        $(\mathcal{M}_{\bullet,\bullet,j}^{(a)}
        \,;\,
        \mathcal{M}_{\bullet,\bullet,j}^{(c)})
        \in \mathbb{R}^{N\times T_{1j}}$ \\
        \hline

        $\boldsymbol{E}_{\mathrm{up}}^{(j)}$ &
        $(\mathcal{E}_{\bullet,\bullet,j}^{(a)}
        \ \mathcal{E}_{\bullet,\bullet,j}^{(b)})
        \in \mathbb{R}^{N_{1j}\times T}$
        &
        $\boldsymbol{E}_{\mathrm{left}}^{(j)}$ &
        $(\mathcal{E}_{\bullet,\bullet,j}^{(a)}
        \,;\,
        \mathcal{E}_{\bullet,\bullet,j}^{(c)})
        \in \mathbb{R}^{N\times T_{1j}}$ \\
        \hline

        $\boldsymbol{Y}_{\mathrm{up}}^{(j)}$ &
        $(\mathcal{Y}_{\bullet,\bullet,j}^{(a)}
        \ \mathcal{Y}_{\bullet,\bullet,j}^{(b)})
        \in \mathbb{R}^{N_{1j}\times T}$
        &
        $\boldsymbol{Y}_{\mathrm{left}}^{(j)}$ &
        $(\mathcal{Y}_{\bullet,\bullet,j}^{(a)}
        \,;\,
        \mathcal{Y}_{\bullet,\bullet,j}^{(c)})
        \in \mathbb{R}^{N\times T_{1j}}$ \\
        \hline

        $\boldsymbol{M}_{\mathrm{up}}^{\mathrm{p}}$ &
        $(\boldsymbol{M}_{\mathrm{up}}^{(1)}
        \,;\,\cdots\,;\,
        \boldsymbol{M}_{\mathrm{up}}^{(K)})
        \in \mathbb{R}^{N_{1,\mathrm{p}}\times T}$
        &
        $\boldsymbol{M}_{\mathrm{left}}^{\mathrm{p}}$ &
        $(\boldsymbol{M}_{\mathrm{left}}^{(1)}
        \ \cdots\
        \boldsymbol{M}_{\mathrm{left}}^{(K)})
        \in \mathbb{R}^{N\times T_{1,\mathrm{p}}}$ \\
        \hline

        $\boldsymbol{E}_{\mathrm{up}}^{\mathrm{p}}$ &
        $(\boldsymbol{E}_{\mathrm{up}}^{(1)}
        \,;\,\cdots\,;\,
        \boldsymbol{E}_{\mathrm{up}}^{(K)})
        \in \mathbb{R}^{N_{1,\mathrm{p}}\times T}$
        &
        $\boldsymbol{E}_{\mathrm{left}}^{\mathrm{p}}$ &
        $(\boldsymbol{E}_{\mathrm{left}}^{(1)}
        \ \cdots\
        \boldsymbol{E}_{\mathrm{left}}^{(K)})
        \in \mathbb{R}^{N\times T_{1,\mathrm{p}}}$ \\
        \hline

        $\boldsymbol{Y}_{\mathrm{up}}^{\mathrm{p}}$ &
        $(\boldsymbol{Y}_{\mathrm{up}}^{(1)}
        \,;\,\cdots\,;\,
        \boldsymbol{Y}_{\mathrm{up}}^{(K)})
        \in \mathbb{R}^{N_{1,\mathrm{p}}\times T}$
        &
        $\boldsymbol{Y}_{\mathrm{left}}^{\mathrm{p}}$ &
        $(\boldsymbol{Y}_{\mathrm{left}}^{(1)}
        \ \cdots\
        \boldsymbol{Y}_{\mathrm{left}}^{(K)})
        \in \mathbb{R}^{N\times T_{1,\mathrm{p}}}$ \\
        \hline

        $\boldsymbol{W}_{\mathrm{up}}$ &
        $(\boldsymbol{U}_{11}\mathcal{C}_{\bullet,\bullet,1}
        \,;\,\cdots\,;\,
        \boldsymbol{U}_{1K}\mathcal{C}_{\bullet,\bullet,K})
        \in \mathbb{R}^{N_{1,\mathrm{p}}\times r}$
        &
        $\boldsymbol{W}_{\mathrm{left}}$ &
        $(\boldsymbol{V}_{11}\mathcal{C}_{\bullet,\bullet,1}^{\top}
        \,;\,\cdots\,;\,
        \boldsymbol{V}_{1K}\mathcal{C}_{\bullet,\bullet,K}^{\top})
        \in \mathbb{R}^{T_{1,\mathrm{p}}\times r}$ \\
        \hline

        $(\boldsymbol{U}_{\mathrm{up}},
        \boldsymbol{\Sigma}_{\mathrm{up}},
        \boldsymbol{V}_{\mathrm{up}})$ &
        $\mathrm{SVD}_{r}
        (\boldsymbol{M}_{\mathrm{up}}^{\mathrm{p}})$
        &
        $(\boldsymbol{U}_{\mathrm{left}},
        \boldsymbol{\Sigma}_{\mathrm{left}},
        \boldsymbol{V}_{\mathrm{left}})$ &
        $\mathrm{SVD}_{r}
        (\boldsymbol{M}_{\mathrm{left}}^{\mathrm{p}})$ \\
        \hline

        $(\widehat{\boldsymbol{U}}_{\mathrm{up}},
        \widehat{\boldsymbol{\Sigma}}_{\mathrm{up}},
        \widehat{\boldsymbol{V}}_{\mathrm{up}})$ &
        $\mathrm{SVD}_{r}
        (\boldsymbol{Y}_{\mathrm{up}}^{\mathrm{p}})$
        &
        $(\widehat{\boldsymbol{U}}_{\mathrm{left}},
        \widehat{\boldsymbol{\Sigma}}_{\mathrm{left}},
        \widehat{\boldsymbol{V}}_{\mathrm{left}})$ &
        $\mathrm{SVD}_{r}
        (\boldsymbol{Y}_{\mathrm{left}}^{\mathrm{p}})$ \\
        \hline

        $N_{1,\mathrm{p}}$ &
        $\sum_{j=1}^{K}N_{1j}$
        &
        $T_{1,\mathrm{p}}$ &
        $\sum_{j=1}^{K}T_{1j}$ \\
        \hline

        $\rho_N$ &
        $N_{1,\mathrm{p}}/N$
        &
        $\rho_T$ &
        $T_{1,\mathrm{p}}/T$ \\
        \hline

        $p_N$ &
        $\max\{N_{1,\mathrm{p}},T\}$
        &
        $p_T$ &
        $\max\{N,T_{1,\mathrm{p}}\}$ \\
        \hline

        $\widehat{\boldsymbol{U}}_{1k}$ &
        $(\widehat{\boldsymbol{U}}_{\mathrm{left}})_{[N_{1k}],\bullet}
        \in \mathbb{R}^{N_{1k}\times r}$
        &
        $\widehat{\boldsymbol{U}}_{2k}$ &
        $(\widehat{\boldsymbol{U}}_{\mathrm{left}})
        _{\{N_{1k}+1,\ldots,N\},\bullet}
        \in \mathbb{R}^{N_{2k}\times r}$ \\
        \hline

        $\widehat{\boldsymbol{U}}_{\mathrm{up}}^{(k)}$ &
        $(\widehat{\boldsymbol{U}}_{\mathrm{up}})
        _{\{s_k+1,\ldots,s_k+N_{1k}\},\bullet}
        \in \mathbb{R}^{N_{1k}\times r}$
        &
        $\widehat{\boldsymbol{V}}_{2k}$ &
        $(\widehat{\boldsymbol{V}}_{\mathrm{up}})
        _{\{T_{1k}+1,\ldots,T\},\bullet}
        \in \mathbb{R}^{T_{2k}\times r}$ \\
        \hline

        $\zeta_N$ &
        $\log(N_{1,\mathrm{p}}+T)$
        &
        $\zeta_T$ &
        $\log(N+T_{1,\mathrm{p}})$ \\
        \hline

        $s_k$ &
        $\sum_{j=1}^{k-1}N_{1j}$
        &
        $\zeta$ &
        $\max(\zeta_N,\zeta_T)$ \\
        \hline
    \end{tabular}%
    }
\captionsetup{
  font={scriptsize,stretch=0.8},
  skip=10pt
}
    \caption{Notation used throughout the paper.}
    \vspace{-20pt}
    \label{tab:block-notation-main}
\end{table}

We now provide some intuition by considering the noiseless case. The key observation is that $\mathcal M_{\bullet,\bullet,k}^{(d)}
=
\boldsymbol U_{2k}(\boldsymbol U_{1k}^\top \boldsymbol U_{1k})^{-1} \boldsymbol U_{1k}^\top \, \mathcal M_{\bullet,\bullet,k}^{(b)}$, so the missing $d$-block can be recovered by mapping the observed $b$-block through the linear operator $\boldsymbol U_{2k}(\boldsymbol U_{1k}^\top \boldsymbol U_{1k})^{-1} \boldsymbol U_{1k}^\top$. For arbitrary unit vectors $\boldsymbol x$ and $\boldsymbol  y$, this yields
$
\mu_{\boldsymbol x \boldsymbol y}^{(k)}
=
\boldsymbol x^\top \mathcal M_{\bullet,\bullet,k}^{(d)} \, \boldsymbol y
=
\langle \boldsymbol U_{2k}^\top \, \boldsymbol x,\,
(\boldsymbol U_{1k}^\top \boldsymbol U_{1k})^{-1} \boldsymbol U_{1k}^\top \, \mathcal M_{\bullet,\bullet,k}^{(b)} \, \boldsymbol y
\rangle$.
This motivates a two-step procedure that exploits the shared-subspace assumption by first using the pooled left matrix $\boldsymbol M_\mathrm{left}^\mathrm{p}=\boldsymbol U\boldsymbol W_{\mathrm{left}}^\top$ to estimate the left singular subspace and associated least-squares map, and then using the pooled upper matrix $\boldsymbol M_\mathrm{up}^\mathrm{p}=\boldsymbol W_{\mathrm{up}}\boldsymbol V^\top$ to obtain a low-rank denoised estimate of the $b$-block, but only through its action on $\boldsymbol y$. 

More precisely, if $\boldsymbol W_{\mathrm{left}}$ has rank $r$, we have $\operatorname{SVD}_r(\boldsymbol M_\mathrm{left}^\mathrm{p}) = (\boldsymbol U_{\mathrm{left}},\boldsymbol \Sigma_{\mathrm{left}},\boldsymbol V_{\mathrm{left}})$ with $\boldsymbol U_{\mathrm{left}}=\boldsymbol U \boldsymbol Q_\mathrm{left}$ for some orthogonal $\boldsymbol Q_\mathrm{left}\in\mathbb{R}^{r\times r}$. Importantly, the operator $\boldsymbol U_{2k}(\boldsymbol U_{1k}^\top \boldsymbol U_{1k})^{-1}\boldsymbol U_{1k}^\top$ is rotationally invariant. On the other hand, if $\boldsymbol W_{\mathrm{up}}$ has rank $r$, for the pooled upper matrix we have $ \operatorname{SVD}_r(\boldsymbol M_\mathrm{up}^\mathrm{p}) = (\boldsymbol U_{\mathrm{up}},\boldsymbol \Sigma_{\mathrm{up}},\boldsymbol V_{\mathrm{up}})$, where $\boldsymbol V_{\mathrm{up}}=\boldsymbol V \boldsymbol Q_\mathrm{up}$ for some orthogonal $\boldsymbol Q_\mathrm{up}\in\mathbb{R}^{r\times r}$. Writing $\boldsymbol U_{\mathrm{up}}=\big(\boldsymbol U_{\mathrm{up}}^{(1)};\ \cdots;\ \boldsymbol U_{\mathrm{up}}^{(K)}\big)$ with $\boldsymbol U_{\mathrm{up}}^{(j)}\in\mathbb{R}^{N_{1j}\times r}$, we have $\boldsymbol U_{\mathrm{up}} \boldsymbol \Sigma_{\mathrm{up}} = \boldsymbol W_{\mathrm{up}} \boldsymbol Q_\mathrm{up}$ and $
\boldsymbol U_{\mathrm{up}}^{(k)}\Sigma_{\mathrm{up}} = \boldsymbol U_{1k}\mathcal C_{\bullet,\bullet,k}\boldsymbol  Q_\mathrm{up}$, hence $\boldsymbol U_{\mathrm{up}}^{(k)}\boldsymbol  \Sigma_{\mathrm{up}}(\boldsymbol V_{\mathrm{up}})_{\{T_{1k} + 1, \ldots, T\}, \bullet}^\top
=
\boldsymbol U_{1k}\mathcal C_{\bullet,\bullet,k}\boldsymbol Q_\mathrm{up} \boldsymbol Q_\mathrm{up}^\top \boldsymbol V_{2k}^\top
=
\boldsymbol U_{1k}\mathcal C_{\bullet,\bullet,k} \boldsymbol V_{2k}^\top
=~\mathcal{M}_{\bullet, \bullet, k}^{(b)}$. This also shows that this quantity is rotationally invariant, and further ensures that no cross-alignment between the two SVDs is required.

\begin{algorithm}[!ht] \captionsetup{font=scriptsize,skip=2pt} \caption{\textsc{BilinearTensor4Block} for estimating $\mu_{\boldsymbol x \boldsymbol y}^{(k)} =\boldsymbol x^\top\mathcal{M}_{\bullet,\bullet,k}^{(d)}\boldsymbol y$ in slice $k$ of a tensor with four-block missingness} \label{alg:bilinear4block}
\begingroup \footnotesize \linespread{0.92}\selectfont \algrenewcommand{\algorithmicindent}{1em} \setlength{\abovedisplayskip}{2pt} \setlength{\belowdisplayskip}{2pt} \setlength{\abovedisplayshortskip}{1pt} \setlength{\belowdisplayshortskip}{1pt}
 \begin{algorithmic}[1] \Require integer $k\in[K]$, rank $r$, unit vectors $\boldsymbol x\in\mathcal S^{N_{2k}-1}$, $\boldsymbol y\in\mathcal S^{T_{2k}-1}$, data $\mathcal Y$, block sizes $\{(N_{1j},N_{2j},T_{1j},T_{2j})\}_{j=1}^K$ satisfying $N=N_{1j}+N_{2j}$ and $T=T_{1j}+T_{2j}$ for all $j\in[K]$, parameter~$\tau > 0$. 
 \State Form pooled left matrix $\boldsymbol Y_{\mathrm{left}}^{\mathrm{p}}\gets (\boldsymbol Y^{(1)}_{\mathrm{left}}\ \ \cdots\ \ \boldsymbol Y^{(K)}_{\mathrm{left}})\in\mathbb R^{N\times T_{1, \mathrm{p}}}$. 
 \State Compute rank-$r$ truncated singular value decomposition $(\hat{\boldsymbol U}_{\mathrm{left}},\hat{\boldsymbol \Sigma}_{\mathrm{left}},\hat{\boldsymbol V}_{\mathrm{left}})\gets \mathrm{SVD}_r(\boldsymbol Y_{\mathrm{left}}^{\mathrm{p}})$. 
\State Set $\hat{\boldsymbol U}_{1k}\gets(\hat{\boldsymbol U}_{\mathrm{left}})_{[N_{1k}],\bullet}$ and $\hat{\boldsymbol U}_{2k}\gets(\hat{\boldsymbol U}_{\mathrm{left}})_{\{N_{1k}+1,\ldots,N\},\bullet}$. 
\State Compute $\hat{\boldsymbol H}_k\gets\hat{\boldsymbol U}_{1k}^\top\hat{\boldsymbol U}_{1k}\in\mathbb R^{r\times r}$, take the eigendecomposition $\hat{\boldsymbol H}_k=\boldsymbol  Q\operatorname{diag}(\lambda_1,\ldots,\lambda_r) \boldsymbol Q^\top$, and set \[ \hat{\boldsymbol H}_{k,\tau}^\mathrm{inv} \gets \boldsymbol Q\operatorname{diag}\left(\left\{\frac{1}{\max[\lambda_i,\tau]}\right\}_{i=1}^r\right) \, \boldsymbol Q^\top . \] \State Compute $\hat{\boldsymbol \alpha}_x^{(k)}\gets\hat{\boldsymbol U}_{2k}^\top \boldsymbol x\in\mathbb R^r$. \State Form pooled upper matrix $\boldsymbol Y_{\mathrm{up}}^{\mathrm{p}}\gets (\boldsymbol Y^{(1)}_{\mathrm{up}}\ ;\ \cdots\ ;\ \boldsymbol Y^{(K)}_{\mathrm{up}})\in\mathbb R^{ N_{1,\mathrm{p}}\times T}$. \State Compute rank-$r$ truncated singular value decomposition $(\hat{\boldsymbol U}_{\mathrm{up}},\hat{\boldsymbol \Sigma}_{\mathrm{up}},\hat{\boldsymbol  V}_{\mathrm{up}})\gets \mathrm{SVD}_r(\boldsymbol Y_{\mathrm{up}}^{\mathrm{p}})$. \State Compute $s_k \gets \sum_{j = 1}^{k-1}N_{1j}$, and extract $\hat{\boldsymbol U}_{\mathrm{up}}^{(k)}\gets(\hat{\boldsymbol U}_{\mathrm{up}})_{\{s_k+1,\ldots,s_k+N_{1k}\},\bullet}$, \,$\hat{\boldsymbol V}_{2k}\gets(\hat{\boldsymbol V}_{\mathrm{up}})_{\{T_{1k}+1,\ldots,T\},\bullet}$. \State Compute $\boldsymbol T_y\gets\hat{\boldsymbol V}_{2k}^\top \boldsymbol y\in\mathbb R^r$, $\boldsymbol W_y\gets\hat{\boldsymbol \Sigma}_{\mathrm{up}} \boldsymbol T_y\in\mathbb R^r$, and $\boldsymbol X_y\gets\hat{\boldsymbol U}_{\mathrm{up}}^{(k)}\boldsymbol W_y\in\mathbb R^{N_{1k}}$. \State Compute $\hat{\boldsymbol \beta}_y^{(k)}\gets \, \hat{\boldsymbol H}_{k,\tau}^\mathrm{inv} \hat{\boldsymbol U}_{1k}^\top \boldsymbol X_y\in\mathbb R^r$. \State \Return $\hat\mu_{\boldsymbol x \boldsymbol y}^{(k)}\gets\langle\hat{\boldsymbol \alpha}_x^{(k)},\hat{\boldsymbol \beta}_y^{(k)}\rangle$. \end{algorithmic} \endgroup 
\end{algorithm}
\vspace{-15pt}

In the noisy setting, Algorithm~\ref{alg:bilinear4block} follows the same principle, but applied to the observed tensor~$\mathcal Y$ rather than the signal tensor $\mathcal M$. Our method computes the rank-$r$ truncated SVDs of $\boldsymbol Y_\mathrm{left}^\mathrm{p}$ and $\boldsymbol Y_\mathrm{up}^\mathrm{p}$, formed by horizontally stacking the blue/green blocks and vertically stacking the blue/pink blocks in Fig.~\ref{fig:4BlockSlices}, respectively. This step, often referred to as \emph{Stack-SVD}, exploits the singular subspaces shared across panels in order to improve subspace estimation; see \cite{ma2026optimal,BaharavNicolIrizarryMa2025StackedSVD} for theoretical guarantees and comparisons with alternative aggregation schemes. 

We also notice that our denoise-then-regress procedure targets the linear functional directly rather than reconstructing the entire missing $d$-block. Specifically, instead of reconstructing the whole missing block, in Steps 9-10 we compute only $\hat{\boldsymbol \beta}_y^{(k)} \in \mathbb R^r$, and then return its inner product with $\hat{\boldsymbol \alpha}_x^{(k)} \in \mathbb R^r$ calculated in Step 5. We further elaborate on the computational complexity of our procedure in Appendix~\ref{appendix:compCost}.

{Finally, we use the clipped inverse $\hat{\boldsymbol H}_{k,\tau}^{\mathrm{inv}}$ to estimate $({\boldsymbol U}_{1k}^{\top}{\boldsymbol U}_{1k})^{-1}$ and safeguard against ill-conditioning. For sufficiently small $\tau$, clipping is inactive with high probability, while stabilizing the inverse on the complementary event. In the numerical experiments of Section~\ref{sec:simul}, the procedure appears stable for values of $\tau$ below $10^{-2}$, with little need for careful tuning; see Fig.~\ref{fig:StaggeredSynthetic}(a) for further details.}

\section{Theoretical analysis}\label{sec:theory}
\subsection{Upper bound}\label{sec:upperBound}

We now present the theoretical analysis of Algorithm~\ref{alg:bilinear4block}. Because the masks $\Omega_{\bullet,\bullet,j}$ may be adversarial and induce pathological MNAR patterns, assumptions are required to ensure that $\mu_{\boldsymbol x \boldsymbol y}^{(k)}$ is identifiable from the observed data. Building on prior work on MNAR matrix completion \citep{BaiNg2021MatrixCompletionCounterfactuals, yan2024entrywise, choi2024matrix, agarwal2026robustmatrixestimation}, we impose a condition controlling the spectra of the restricted Gram matrices $\boldsymbol U_{1j}^\top \boldsymbol U_{1j}$ and $\boldsymbol V_{1j}^\top \boldsymbol V_{1j}$.

\begin{assump}\label{assump:subblock-conditioning}
There exist $0 < c_\ell \le c_u$ such that for all $j\in[K]$ we have
\[
c_\ell\,\frac{N_{1j}}{N}\,I_r \preceq \boldsymbol U_{1j}^\top \boldsymbol U_{1j} \preceq c_u\,\frac{N_{1j}}{N}\,I_r,
\qquad
c_\ell\,\frac{T_{1j}}{T}\,I_r \preceq \boldsymbol V_{1j}^\top \boldsymbol V_{1j} \preceq c_u\,\frac{T_{1j}}{T}\,I_r.
\]
\end{assump}

Our theoretical result also requires the following three conditions, where $c_0>0$ and $c_{\mathrm{dim}}>0$ are fixed, sufficiently small absolute constants, while $c_{\mathrm{blk}}>0$ is a fixed absolute constant.  We use the additional notation in Table~\ref{tab:block-notation-main}.

\begin{assump}\label{assump:sampleSize}
We have $ r+\zeta
\le c_{\mathrm{dim}} \,
\min (
N-r,\,
T-r,\,
N_{1k}, T_{1k})$ and $c_{\mathrm{blk}} \min(\zeta_N, \zeta_T) \leq r$.
\end{assump}

\begin{assump}\label{assump:smallNoise}
Define $\theta:= \sigma \, \gamma_{\min}^{-1}
\max\left(
\sqrt{N},\,
\sqrt{T},\,
\sqrt{N / \rho_T},\,
\sqrt{N T / N_{1k}}
\right)$, and assume this signal-to-noise ratio satisfies $\theta \le c_0$.
\end{assump}

\begin{assump}\label{assump:Incoherence}
    We define the incoherence parameters $\nu_x:=\sqrt{N/r} \, \|\boldsymbol U_{2k}^\top \boldsymbol x\|_2 $ and $
\nu_y:=\sqrt{T/r} \, \|\boldsymbol V_{2k}^\top \boldsymbol y\|_2$, and assume they are of constant order.
\end{assump}

A detailed discussion of these assumptions is deferred to Section~\ref{sec:assumptions}. We collect all admissible signal tensors into the following class. For fixed $r,N,T,K,\gamma_{\min},\gamma_{\max}, \Omega$ and $0 < c_\ell \leq c_u$,  we define
\[
\begin{aligned}\label{eq:calF_class}
\mathcal F(c_\ell,c_u)
:= \bigl\{\, \mathcal M & \in \mathbb R^{N \times T \times K} :\;
\mathcal M =\mathcal C \times_1 \boldsymbol U \times_2 \boldsymbol V \times_3 \boldsymbol I_K, \\
& \mathcal C \in \mathbb R^{r \times r \times K},\quad
\boldsymbol U \in \mathbb R^{N \times r},\quad
\boldsymbol V \in \mathbb R^{T \times r},  \boldsymbol U^\top \boldsymbol U = \boldsymbol V^\top \boldsymbol V = \boldsymbol I_r, \\
& 0 < \gamma_{\min}
\leq \sigma_{\min}(\mathcal C_{\bullet, \bullet, j})
\leq \sigma_{\max}(\mathcal C_{\bullet, \bullet, j})
\leq \gamma_{\max} < \infty
\quad \text{for all } j \in [K], \\
& \text{Assumption \eqref{assump:subblock-conditioning} holds with constants } c_\ell,c_u
\text{ for the fixed design } \Omega
\,\bigr\}.
\end{aligned}
\]

We now state our main result. 
We write $\kappa := {\gamma_\mathrm{max}}/{\gamma_\mathrm{min}}$ for the condition number of~$\mathcal{C}$.
\begin{thm}\label{thm:4BlockUBLinear}
     Fix absolute constants $0 < c_\ell  \le c_u$,  a tensor $\mathcal{M} \in \mathcal{F}(c_\ell, c_u)$, an index $k \in [K]$, and unit vectors $\boldsymbol x \in \mathcal S^{N_{2k}-1}, \boldsymbol  y \in \mathcal S^{T_{2k}-1}$. Let $\mu_{\boldsymbol x \boldsymbol y}^{(k)}$ be as in~\eqref{eq:mu4block}, and define~$\hat \mu_{\boldsymbol x \boldsymbol y}^{(k)}$ to be the output of Algorithm~\ref{alg:bilinear4block} run with $0<\tau \leq \frac{c_\ell  N_{1k}}{{2 N}}$. Assume~\eqref{assump:sampleSize},~\eqref{assump:smallNoise}, and~\eqref{assump:Incoherence} with $\nu_x \neq 0, \nu_y \neq 0$. Let 
     \[
      \Upsilon_{xy} := \frac{\sigma^2 (r + \zeta_N)}{\rho_N} \, \|\boldsymbol U_{2k}^\top \boldsymbol x\|_2^2 +\frac{\sigma^2 (r + \zeta_T)}{\rho_T} \, \|\boldsymbol V_{2k}^\top \boldsymbol y\|_2^2 +   \frac{\sigma^2 N}{N_{1k}} \, \|\boldsymbol U_{2k}^\top \boldsymbol x\|_2^2 \, \|\boldsymbol V_{2k}^\top \boldsymbol y\|_2^2, 
     \]
     and further suppose that 
\begin{align}\label{eq:negligibilityForUB4Block}
   \frac{\gamma_{\max}^2}{\tau}\frac{N_{1k}}{N}
\left(p_N^{-10}+p_T^{-10}\right)
+
\frac{\sigma^2}{\tau}\left(N_{1k}+T\right)
\left(p_N^{-5}+p_T^{-5}\right) \leq c_0 \Upsilon_{xy}.
\end{align}
There exists $c_1 \equiv c_1(c_\ell, c_u, c_0,  c_{\mathrm{dim}}, c_{\mathrm{blk}}, \kappa, \nu_x, \nu_y)> 0$ such that $\mathbb E_\mathcal{M}[
\{
\hat \mu_{\boldsymbol x \boldsymbol y}^{(k)}-\mu_{\boldsymbol x \boldsymbol y}^{(k)}
\}^2]
\le~c_1\Upsilon_{xy}$.
\end{thm}

We note that the term $\sigma^2 \frac{N}{N_{1k}}\|\boldsymbol U_{2k}^\top \boldsymbol x\|_2^2\|\boldsymbol V_{2k}^\top \boldsymbol y\|_2^2$ is asymmetric because Algorithm~\ref{alg:bilinear4block} uses vertical regression to predict the missing $d$-block from the observed $b$-block; see Section~\ref{sec:methodology4Block}. However, applying the same construction to the transposed tensor gives~$\frac{T}{T_{1k}}$ in place of $\frac{N}{N_{1k}}$, so choosing the better orientation yields the symmetric quantity $\sigma^2\min\left(\frac{N}{N_{1k}},\frac{T}{T_{1k}}\right)\|\boldsymbol U_{2k}^\top \boldsymbol x\|_2^2\|\boldsymbol V_{2k}^\top \boldsymbol y\|_2^2$. 

The interpretation of this result is illustrated in Fig.~\ref{fig:phaseTransition}, and discussed immediately~after. A useful comparison is with~\citet{yan2024entrywise}, who study entrywise inference for causal panel data under staggered adoption in the special case $K=1$, $
\boldsymbol x=\boldsymbol e_i$, and $\boldsymbol y=\boldsymbol e_t$. Both methods build on the spectral approach of \citet{BaiNg2021MatrixCompletionCounterfactuals}, but differ in their inferential target and estimation procedure. \citet[][Algorithm~1]{yan2024entrywise} estimate missing entries of $\boldsymbol M_d$ by completing the full missing $d$-block, whereas we directly estimate a general bilinear form. More generally, however, our setting also allows $K\geq 1$ and extends this direct regress-then-denoise approach to tensor data with four-block missingness.

The theoretical comparison is also transparent in the matrix entrywise case. In this case, the upper bound in Theorem~\ref{thm:4BlockUBLinear} matches their Equation~4.7 up to constants, apart from the term proportional to $ \|\boldsymbol U_{2k}^\top \boldsymbol e_i\|_2^2 \|\boldsymbol V_{2k}^\top \boldsymbol e_t\|_2^2$. Since $\|\boldsymbol V_{2k}^\top \boldsymbol e_t\|_2^2 \leq 1$, this term is lower order and can be absorbed into the first one. In our tensor setting, however, it is important to keep this term explicit, as it captures the error component that does not decrease under pooling across layers and thereby characterizes the phase transition in the rate as $K$ grows.

At the proof level, both the leave-one-block-out method of \citet{yan2024entrywise}~and ours expand $\hat \mu_{\boldsymbol x \boldsymbol y}^{(k)}-\mu_{\boldsymbol x \boldsymbol y}^{(k)}$ into Gaussian terms and remainders. The key difference is how the remainders are controlled. Their approach could be adapted to our tensor setting, but yields negligible remainders only under signal-to-noise (SNR) conditions that worsen with~$K$. The proof of Theorem~\ref{thm:4BlockUBLinear} instead applies the Haar compression bounds from Section~\ref{appendix:tecnicalLemmas} of~the Supplement to a centered version of $\boldsymbol Y_\mathrm{left}^\mathrm{p}(\boldsymbol Y_\mathrm{left}^\mathrm{p})^\top$, following the opening paragraph of the proof of Lemma~\ref{lemma13_centered}, and obtains the same type of expansion under weaker assumptions. Moreover, beyond allowing general unit vectors,
Lemma~\ref{lemma13_centered} improves on \citet[][Equation~B.4]{yan2024entrywise} by replacing the second-order term $\frac{\sigma^2(N+T_{1,\mathrm{p}})}{\gamma_{\min}^{2}\rho_T}\|\boldsymbol U^\top \boldsymbol x\|_2$ with $\frac{\sigma^2 N}{\gamma_{\min}^{2}\rho_T}\|\boldsymbol U^\top \boldsymbol x\|_2$, while leaving the first-order term unchanged. This is an independent
contribution of interest, and ensures that the SNR requirement improves with~$K$.

\subsection{Discussion of the Assumptions}\label{sec:assumptions}
We now discuss the four assumptions in Theorem~\ref{thm:4BlockUBLinear}. Assumption~\eqref{assump:subblock-conditioning} requires $\boldsymbol U_{1j}^\top\boldsymbol U_{1j}$ and $\boldsymbol V_{1j}^\top\boldsymbol V_{1j}$ to be uniformly well-conditioned across $j\in[K]$, with eigenvalues proportional to $\frac{N_{1j}}{N}$ and $\frac{T_{1j}}{T}$. Propositions~\ref{prop:hard1}--\ref{prop:hard2} in the Supplement show that $c_\ell>0$ is necessary for \eqref{eq:mu4block} to be identifiable. This condition also implies $r\leq\min(N_{1j},T_{1j})$ for all $j\in[K]$, and hence $r\leq\min(N,T,N_{1,\mathrm{p}},T_{1,\mathrm{p}})$, which is the minimal dimensional requirement for the rank-$r$ SVDs in Algorithm~\ref{alg:bilinear4block}. Assumption~\eqref{assump:sampleSize} consists of mild dimension-regularity conditions, which are introduced to simplify the statement of the final result and make it more transparent, while~\eqref{assump:Incoherence} is a standard incoherence condition adapted to the directions $\boldsymbol x$ and $\boldsymbol y$ of interest.

Finally, \eqref{assump:smallNoise} is a SNR condition analogous to that in \cite{agterberg2026statistically} for estimating~$\boldsymbol U$ in a shared-subspace model with complete observations under the loss $\|\sin\Theta(\widehat{\boldsymbol U},\boldsymbol U)\|_F^2$. By analogy, their Theorem~1 suggests that estimating~\eqref{eq:mu4block} would be information-theoretically impossible without \eqref{assump:smallNoise}. Furthermore, compared to \citet[][Assumption 4.3]{yan2024entrywise}, \eqref{assump:smallNoise} becomes progressively less stringent as the number of layers~$K$ increases, reflecting the benefit of pooling information across layers. This improvement continues until the requirement saturates at a local term which cannot be further reduced by pooling, thereby highlighting that a sufficiently strong slice-specific signal is still needed to learn~$\mathcal C_{\bullet, \bullet, k}$.

\subsection{Local minimax lower bound}\label{sec:MinimaxLowerBounds}
We next complement the upper bound in Theorem~\ref{thm:4BlockUBLinear} by establishing a local minimax lower bound in neighborhoods of fixed tensors $\mathcal M_0$ that satisfy suitable conditions. For $\mathcal M\in\mathcal F(c_\ell, c_u)$, we write $\mu_{\boldsymbol x \boldsymbol y}^{(k)}(\mathcal M):=\boldsymbol x^\top \mathcal M_{\bullet,\bullet,k}^{(d)} \, \boldsymbol y$ and $Z_\Omega:=\{\,\mathcal M_{itj}+\mathcal E_{itj}:\Omega_{i,t,j}=~1,\ (i,t,j)\in[N]\times [T] \times [K]\,\}$ for the observed entries, and use $\mathbb P_\mathcal M$ and $\mathbb E_\mathcal M$ for probability and expectation under the law of~$Z_\Omega$. We also suppose for simplicity that $N_{1j}=N_1$ and $T_{1j}=T_1$ for all
$j\in[K]$. Under this assumption, $\boldsymbol U_{0,1j_1}=\boldsymbol U_{0,1j_2}$ and $\boldsymbol U_{0,2j_1}=\boldsymbol U_{0,2j_2}$ for all $j_1,j_2\in[K]$. We denote these common matrices by $\boldsymbol U_{0,1}$ and~$\boldsymbol U_{0,2}$, respectively, and adopt the analogous convention for $\boldsymbol V_{0,1}$ and $\boldsymbol V_{0,2}$.
\begin{thm}\label{thm:CombinedLB}
Fix $k\in[K]$ and unit vectors $\boldsymbol x \in \mathcal S^{N_{2k}-1}, \boldsymbol y \in \mathcal S^{T_{2k}-1}$, and set $N_{1j}=N_1$ and $T_{1j}=T_1$ for all
$j\in[K]$. Let $\mathcal M_0=\mathcal C_0\times_1 \boldsymbol U_0\times_2 \boldsymbol V_0\times_3 \boldsymbol I_K \in \mathcal F(c_\ell, c_u)$, and suppose that~\eqref{assump:subblock-conditioning}
holds with margin $0< \delta_{\mathrm{A1}} < (c_u - c_\ell)/2$, in the sense that
\begin{align*} &\left(c_\ell+\delta_{\mathrm{A1}}\right)\frac{N_1}{N}\boldsymbol I_r
    \preceq
    \boldsymbol U_{0,1}^{\top}\boldsymbol U_{0,1}
    \preceq
    \left(c_u-\delta_{\mathrm{A1}}\right)\frac{N_1}{N}\boldsymbol I_r, \\ &\left(c_\ell+\delta_{\mathrm{A1}}\right)\frac{T_1}{T}\boldsymbol I_r
    \preceq
    \boldsymbol V_{0,1}^{\top}\boldsymbol V_{0,1}
    \preceq
    \left(c_u-\delta_{\mathrm{A1}}\right)\frac{T_1}{T}\boldsymbol I_r.
\end{align*}
Assume that the singular values of the $k$-th core matrix are separated from the boundary of the
admissible interval, in the sense that $\delta_{\gamma}:=
    \min\{\sigma_{\min}((\mathcal C_{0})_{\bullet, \bullet, k})-\gamma_{\min},
    \gamma_{\max}-\sigma_{\max}((\mathcal C_{0})_{\bullet, \bullet, k})\}>0$. Also suppose that $c_u \max(\frac{N_1}{N}, \frac{T_1}{T}) \leq 1 - {\delta_\mathrm{sep}}$ with ${\delta_\mathrm{sep}} > 0$, and that ${\delta_\mathrm{sep}}^{-1} \max( \| \boldsymbol U_{0,2}^{\top}\, \boldsymbol x\|_2^2,  \| \boldsymbol V_{0,2}^{\top}\, \boldsymbol y\|_2^2) \leq 1/2$.
For $\varsigma > 0$ define $\mathcal{F}_\mathrm{loc}(\mathcal{M}_0, \varsigma) := \{\mathcal M\in\mathcal F(c_\ell,c_u):
\|\mathcal M-\mathcal M_0\|_F\le \varsigma\}$.
 There exists a constant $c \equiv c(c_u, \gamma_\mathrm{min}, \gamma_\mathrm{max})>0$ such~that
\begin{align*}
    \inf_{\phi}
    \sup_{\mathcal{F}_\mathrm{loc}(\mathcal{M}_0, \varsigma)}
    &\mathbb E_{\mathcal M}
    \left[
    \left\{\phi(Z_\Omega)-\mu_{\boldsymbol x \boldsymbol y}^{(k)}(\mathcal M)\right\}^2
    \right]  \\
    & \geq c \, \min\left( \frac{\sigma^{2} N}{K N_1}, \, \varepsilon_V^{2}\right) \, \|\boldsymbol U_{0,2}^{\top}\, x\|_2^2 + c \, \min\left( \frac{\sigma^{2} T}{K T_1}, \, \varepsilon_U^{2}\right) \, \| \boldsymbol V_{0,2}^{\top}\, \boldsymbol y\|_2^2 \\
    &  \qquad + c\,\min\left\{ \sigma^2 \min\left(\frac{N}{N_{1k}},
    \frac{T}{T_{1k}} \right), \varsigma^2, \delta_{\gamma}^2\right\} \,  \|\boldsymbol U_{0, 2k}^\top \, \boldsymbol x\|_2^{2} \, \|\boldsymbol V_{0, 2k}^\top \, \boldsymbol y\|_2^{2},
\end{align*}
where  $\varepsilon_U
:=
\min(8^{-1/2},\,
\sqrt{T_1/T},\,
\varsigma \gamma_{\max}^{-1}  K^{-1/2},\,
\sqrt{\delta_{\mathrm{A1}}/c_\ell})$, $\varepsilon_V
:=
\min(
8^{-1/2},
\sqrt{N_1/N},
\varsigma \gamma_{\max}^{-1}  K^{-1/2},$ $
\sqrt{\delta_{\mathrm{A1}}/c_\ell})$, and the infimum is over all Borel-measurable functions $\phi$ of $Z_\Omega$.
\end{thm}


This result shows the necessity of the elbow behavior in the rate as a function of~$K$. In particular, when $\sigma^2$ is sufficiently small, the right-hand side of the preceding display matches the upper bound in Theorem~\ref{thm:4BlockUBLinear} up to constants, rank factors, and logarithmic factors. Moreover, although the results above are local and stated around a fixed $\mathcal M_0$, they also imply global minimax lower bounds by taking $\varsigma$ sufficiently large so that it contains the entire parameter~space.

Finally, we observe that the assumption  ${\delta_\mathrm{sep}}^{-1} \max( \| \boldsymbol U_{0,2}^{\top}\, \boldsymbol x\|_2^2,  \| \boldsymbol V_{0,2}^{\top}\, \boldsymbol y\|_2^2) \leq 1/2$ is mild in the regime relevant for comparison with Theorem~\ref{thm:4BlockUBLinear}. Indeed, under~\eqref{assump:Incoherence} we have $\|\boldsymbol U_{0,2}^{\top}\boldsymbol x\|_2^2\asymp \frac{r}{N}$ and $\|\boldsymbol V_{0,2}^{\top}\boldsymbol y\|_2^2\asymp \frac{r}{T}$, hence, when $r$ and ${\delta_\mathrm{sep}}$ are of constant order, the condition holds for sufficiently large $N$ and $T$.

\section{Bilinear form estimation under staggered adoption}\label{sec:bilinearStaggered}

\subsection{Proposed methodology for staggered missingness}\label{sec:methodStaggered}
We now extend the four-block setting to general staggered-adoption designs. Our methodology reduces the staggered missingness problem to simpler four-block missingness patterns by constructing pooled upper and left matrices in the same spirit as before. The main additional challenge is to accommodate the more complex missingness structure induced by staggered adoption. In particular, for a fixed layer $j \in [K]$, staggered adoption means that missingness is irreversible, i.e.~for each unit $i\in[N]$, there is an adoption time $A_{ij}$ such that $\Omega_{i,t,j}=\mathbbm 1\{t<A_{ij}\}$. For completeness, we set $A_{ij} = \infty$ for never-adopters. 

It is useful to note that each layer has its own natural row ordering under which the corresponding missingness mask is a  staircase; these orderings need not agree across layers. Since our target will be a bilinear form in layer $k$, we use the adoption-time ordering of the target layer as the common row ordering for all slices. In other words, we permute the rows so that $A_{1k}\ge \cdots \ge A_{Nk}$. This entails no loss of generality, since applying a common row permutation to all layers preserves the Tucker2 structure. Under this convention, the mask $\Omega_{\bullet, \bullet, k}$ admits an equivalent staircase characterization: there exists an integer $o_k \geq 2$ and ordered non-empty contiguous partitions $[N]=R_{1k}\cup\cdots\cup R_{o_k,k}$ and $[T]=C_{1k}\cup\cdots\cup C_{o_k,k}$, with $|R_{ak}|=N_{ak}$ and $|C_{bk}|=T_{bk}$, such that $\Omega_{i,t,k}=\mathbbm 1\{(i,t)\in R_{ak}\times C_{bk}\text{ for some }a,b\text{ with }a+b\le o_k+1\}$. This block representation will be useful in what follows, as it allows us to describe the observed and missing regions of the target layer in terms of the staircase partitions $\{R_{ak}\}_{a=1}^{o_k}$ and $\{C_{bk}\}_{b=1}^{o_k}$. 

As is apparent from the staircase representation above, the assumption we make on~$\Omega$ is that the missingness pattern for the target slice $\Omega_{\bullet,\bullet,k}$ contains fully observed rows, corresponding to never-adopting units, as well as an initial time period during which no unit in that slice has adopted. This is the basic requirement that allows us to reuse the methodology developed for the four-block design. An example of a staggered-adoption pattern covered by our framework is illustrated in Fig.~\ref{fig:Staggered}. 
 
Having specified the structure of the missingness masks, we now introduce the signal and noise model. As in the previous sections, we assume that the signal tensor $\mathcal M$ admits a Tucker2 decomposition of rank $(r,r,K)$, so that $\mathcal M=\mathcal C\times_1 \boldsymbol U\times_2 \boldsymbol V\times_3 \boldsymbol I_K$, where $\boldsymbol U\in\mathbb R^{N\times r}$ and $\boldsymbol V\in\mathbb R^{T\times r}$ have orthonormal columns. The noise tensor $\mathcal E$ has independent Gaussian entries with mean zero and variance $\sigma^2$, and for each $j\in[K]$ we observe $\mathcal Y_{\bullet,\bullet,j}=P_{\Omega_{\bullet,\bullet,j}}(\mathcal M_{\bullet,\bullet,j}+\mathcal E_{\bullet,\bullet,j})$. Our goal here is to estimate bilinear forms over all missing entries in layer $k\in[K]$. For simplicity, we first focus on a specific missing block, since this is the key step needed for the general case. In this regard, choose indices $(a,b)$ such that $a+b>o_k+1$, so that the block $R_{ak}\times C_{bk}$ is unobserved. For unit vectors $\boldsymbol x\in\mathcal S^{N_{ak}-1}$ and $\boldsymbol y\in\mathcal S^{T_{bk}-1}$, our target estimand is the bilinear form
\begin{align}\label{eq:muStaggered}
   \mu_{\boldsymbol x \boldsymbol y}^{(k,a,b)} 
:=
\boldsymbol x^\top \mathcal M_{\bullet,\bullet,k}^{(a,b)} \, \boldsymbol y, 
\end{align}
where $\mathcal M_{\bullet,\bullet,k}^{(a,b)}:=\boldsymbol U_{ak} \, \mathcal C_{\bullet,\bullet,k}\boldsymbol V_{bk}^{\top}\in\mathbb R^{N_{ak}\times T_{bk}}$, with $\boldsymbol U_{ak}:=\boldsymbol U_{R_{ak},\bullet}\in\mathbb R^{N_{ak}\times r}, \, \boldsymbol V_{bk}:=\boldsymbol V_{C_{bk},\bullet}\in\mathbb R^{T_{bk}\times r}$ denoting the restrictions of the Tucker2 factors to specified row and column blocks.

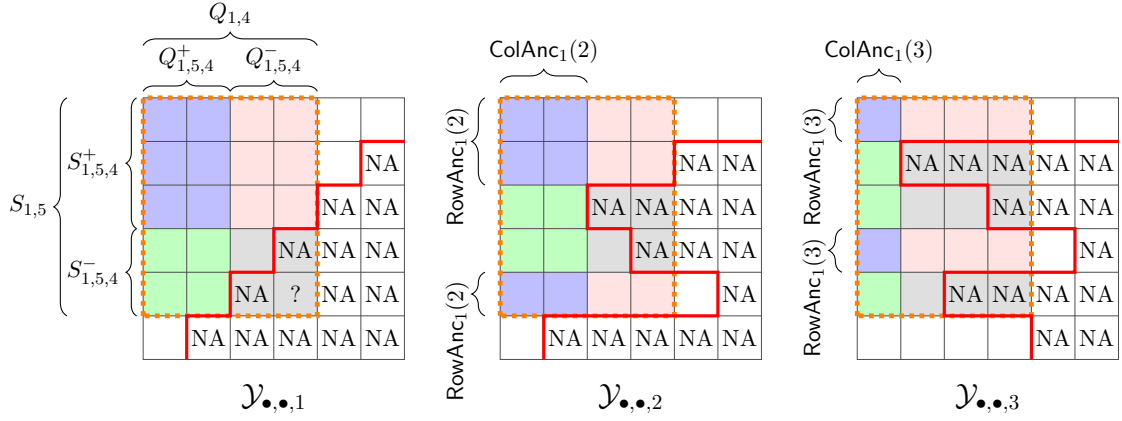
\begin{figure}[htbp]
\vspace{-10pt}
\centering
\begin{tikzpicture}[
    scale=0.8,
    transform shape,
    x=0.72cm,
    y=0.72cm,
    grid/.style={black!70, line width=0.35pt},
    heavy/.style={black, line width=0.9pt},
    stair/.style={red, line width=1.2pt},
    brace/.style={decorate, decoration={brace, amplitude=5pt}},
    lbl/.style={font=\small},
    matlbl/.style={font=\large},
    blueblock/.style={fill=blue!25},
    pinkblock/.style={fill=pink!45},
    greenblock/.style={fill=green!25},
    grayblock/.style={fill=gray!25},
    whiteblock/.style={fill=white},
    orangebox/.style={orange, dotted, line width=1.8pt},
]

\def\n{6}
\def\SoneTop{6}
\def\SoneBot{1}


\begin{scope}[shift={(0,0)}]

\fill[blueblock]  (0,3) rectangle (2,6);

\fill[pinkblock]  (2,4) rectangle (4,6);
\fill[pinkblock]  (3,3) rectangle (4,4);

\fill[greenblock] (0,1) rectangle (2,3);
\fill[grayblock]  (2,1) rectangle (4,3);

\fill[pinkblock] (2,3) rectangle (3,4);

\draw[grid] (0,0) rectangle (\n,\n);
\foreach \x in {1,...,5} {
    \draw[grid] (\x,0) -- (\x,\n);
}
\foreach \y in {1,...,5} {
    \draw[grid] (0,\y) -- (\n,\y);
}

\draw[orangebox] (0,1) rectangle (4,6);

\draw[stair]
    (1,0) -- (1,1)
    -- (2,1) -- (2,2)
    -- (3,2) -- (3,3)
    -- (4,3) -- (4,4)
    -- (5,4) -- (5,5)
    -- (6,5);

\node[lbl] at (1.5,0.5) {NA};

\node[lbl] at (2.5,0.5) {NA};
\node[lbl] at (2.5,1.5) {NA};

\node[lbl] at (3.5,0.5) {NA};
\node[lbl] at (3.5,1.5) {?};
\node[lbl] at (3.5,2.5) {NA};

\node[lbl] at (4.5,0.5) {NA};
\node[lbl] at (4.5,1.5) {NA};
\node[lbl] at (4.5,2.5) {NA};
\node[lbl] at (4.5,3.5) {NA};

\node[lbl] at (5.5,0.5) {NA};
\node[lbl] at (5.5,1.5) {NA};
\node[lbl] at (5.5,2.5) {NA};
\node[lbl] at (5.5,3.5) {NA};
\node[lbl] at (5.5,4.5) {NA};

\draw[brace, decoration={brace, amplitude=4pt}]
    (-0.15,3) -- (-0.15,6)
    node[midway,left=4pt,lbl] {$S^+_{1,5,4}$};

\draw[brace, decoration={brace, amplitude=4pt}]
    (-0.15,1) -- (-0.15,3)
    node[midway,left=4pt,lbl] {$S^-_{1,5,4}$};

\draw[brace, decoration={brace, amplitude=5pt}]
    (-1.75,\SoneBot) -- (-1.75,\SoneTop)
    node[midway,left=5pt,lbl] {$S_{1,5}$};

\draw[brace] (0,\n+1.2) -- (4,\n+1.2)
    node[midway,above=5pt,lbl] {$Q_{1,4}$};

\draw[brace, decoration={brace, amplitude=4pt}]
    (0,\n+0.15) -- (2,\n+0.15)
    node[midway,above=4pt,lbl] {$Q^+_{1,5,4}$};

\draw[brace, decoration={brace, amplitude=4pt}]
    (2,\n+0.15) -- (4,\n+0.15)
    node[midway,above=4pt,lbl] {$Q^-_{1,5,4}$};

\node[matlbl] at (3,-0.9) {$\mathcal Y_{\bullet, \bullet, 1}$};

\end{scope}


\begin{scope}[shift={(8.2,0)}]

\fill[blueblock] (0,4) rectangle (2,6);

\fill[pinkblock] (2,4) rectangle (4,6);

\fill[greenblock] (0,1) rectangle (2,4);

\fill[blueblock] (0,1) rectangle (2,2);

\fill[pinkblock] (2,1) rectangle (4,2);
\fill[grayblock]  (2,2) rectangle (4,4);

\draw[grid] (0,0) rectangle (\n,\n);
\foreach \x in {1,...,5} {
    \draw[grid] (\x,0) -- (\x,\n);
}
\foreach \y in {1,...,5} {
    \draw[grid] (0,\y) -- (\n,\y);
}

\draw[orangebox] (0,1) rectangle (4,6);

\draw[stair]
    (1,0) -- (1,1)
    -- (5,1) -- (5,2)
    -- (3,2) -- (3,3)
    -- (2,3) -- (2,4)
    -- (4,4) -- (4,5)
    -- (6,5);

\node[lbl] at (1.5,0.5) {NA};

\node[lbl] at (2.5,0.5) {NA};
\node[lbl] at (2.5,3.5) {NA};

\node[lbl] at (3.5,0.5) {NA};
\node[lbl] at (3.5,2.5) {NA};
\node[lbl] at (3.5,3.5) {NA};

\node[lbl] at (4.5,0.5) {NA};
\node[lbl] at (4.5,2.5) {NA};
\node[lbl] at (4.5,3.5) {NA};
\node[lbl] at (4.5,4.5) {NA};

\node[lbl] at (5.5,0.5) {NA};
\node[lbl] at (5.5,1.5) {NA};
\node[lbl] at (5.5,2.5) {NA};
\node[lbl] at (5.5,3.5) {NA};
\node[lbl] at (5.5,4.5) {NA};

\draw[brace, decoration={brace, amplitude=5pt}]
    (-0.35,4) -- (-0.35,6)
    node[left=13pt,lbl,rotate=90] {$\mathsf{RowAnc}_1(2)$};

\draw[brace, decoration={brace, amplitude=5pt}]
    (-0.35,1) -- (-0.35,2)
    node[left=13pt,lbl,rotate=90] {$\mathsf{RowAnc}_1(2)$};

\draw[brace] (0,\n+0.35) -- (2,\n+0.35)
    node[midway,above=5pt,lbl] {$\mathsf{ColAnc}_1(2)$};

\node[matlbl] at (3,-0.9) {$\mathcal Y_{\bullet, \bullet, 2}$};

\end{scope}


\begin{scope}[shift={(16.4,0)}]

\fill[greenblock] (0,1) rectangle (1,5);

\fill[blueblock] (0,5) rectangle (1,6);

\fill[pinkblock] (1,5) rectangle (4,6);

\fill[grayblock]  (1,3) rectangle (4,5);
\fill[grayblock]  (1,1) rectangle (4,2);

\fill[blueblock] (0,2) rectangle (1,3);

\fill[pinkblock] (1,2) rectangle (4,3);

\draw[grid] (0,0) rectangle (\n,\n);
\foreach \x in {1,...,5} {
    \draw[grid] (\x,0) -- (\x,\n);
}
\foreach \y in {1,...,5} {
    \draw[grid] (0,\y) -- (\n,\y);
}

\draw[orangebox] (0,1) rectangle (4,6);

\draw[stair]
    (4, 0) -- (4,1) 
    -- (2,1) -- (2, 2)
    -- (5,2) -- (5,3) --
    (3,3) -- (3,4) -- (1,4)
    -- (1,5) -- (6,5);

\node[lbl] at (4.5,0.5) {NA};
\node[lbl] at (5.5,0.5) {NA};

\node[lbl] at (2.5,1.5) {NA};
\node[lbl] at (3.5,1.5) {NA};
\node[lbl] at (4.5,1.5) {NA};
\node[lbl] at (5.5,1.5) {NA};

\node[lbl] at (5.5,2.5) {NA};

\node[lbl] at (3.5,3.5) {NA};
\node[lbl] at (4.5,3.5) {NA};
\node[lbl] at (5.5,3.5) {NA};

\node[lbl] at (1.5,4.5) {NA};
\node[lbl] at (2.5,4.5) {NA};
\node[lbl] at (3.5,4.5) {NA};
\node[lbl] at (4.5,4.5) {NA};
\node[lbl] at (5.5,4.5) {NA};

\draw[brace, decoration={brace, amplitude=5pt}]
    (-0.35,5) -- (-0.35,6)
    node[left=13pt,lbl,rotate=90] {$\mathsf{RowAnc}_1(3)$};

\draw[brace, decoration={brace, amplitude=5pt}]
    (-0.35,2) -- (-0.35,3)
    node[left=13pt,lbl,rotate=90] {$\mathsf{RowAnc}_1(3)$};

\draw[brace] (0,\n+0.35) -- (1,\n+0.35)
    node[midway,above=5pt,lbl] {$\mathsf{ColAnc}_1(3)$};

\node[matlbl] at (3,-0.9) {$\mathcal Y_{\bullet, \bullet,3}$};

\end{scope}

\end{tikzpicture}
\captionsetup{
  font={scriptsize,stretch=0.8},
  skip=10pt
}
\caption{Layer-specific staggered adoption for $K=3$, target layer $k=1$, and target block $(a,b)=(5,4)$. Panel 1 orders the target layer by its own rows, yielding a staircase missingness pattern; the orange dotted rectangle identifies the target rows $S_{1,5}$ and columns $Q_{1,4}$ used to form auxiliary matrices. Panels 2--3 show non-target layers under the same target-layer ordering, with general staggered missingness. These layers provide fully observed anchor rows $\mathsf{RowAnc}_1(j)$ and anchor columns $\mathsf{ColAnc}_1(j)$: blue/pink blocks form the upper pooled matrix, blue/green blocks form the left pooled matrix, and gray blocks are discarded.}
\label{fig:Staggered}
\end{figure}
\vspace{-15pt}

\noindent The following definitions are needed to present our algorithm. We set
\[
S_{k, a,b}^+:=\bigcup_{a^\prime=1}^{o_k+1-b}R_{a^\prime k},\,
S_{k, a,b}^-:=\bigcup_{a^\prime=o_k+2-b}^{a}R_{a^\prime k},\,
Q_{k, a,b}^+:=\bigcup_{b^\prime =1}^{o_k+1-a}C_{b^\prime k},\,
Q_{k, a,b}^-:=\bigcup_{b^\prime =o_k+2-a}^{b}C_{b^\prime k}.
\]
Since $a+b>o_k+1$, the lower limits
$o_k+2-b$ and $o_k+2-a$ are at most $a$ and $b$, respectively, so the sets $S^-_{k,a,b}$ and $Q^-_{k,a,b}$ are non-empty and contain $R_{ak}$ and $C_{bk}$. We also write $S_{k, a}:=S_{k, a,b}^+\cup S_{k, a,b}^-$, $Q_{k, b}:=Q_{k,a,b}^+\cup Q_{k, a,b}^-$, and observe that $S_{k,a}$ and $Q_{k,a,b}^+$ depend only on $a$, whereas $Q_{k,b}$ and $S_{k,a,b}^+$ depend only on $b$.

The four index sets in the display define a four-block structure in which the target missing block $(a,b)$ is contained in $S^-_{k,a,b}\times Q^-_{k,a,b}$. Fig.~\ref{fig:Staggered} illustrates this construction, with the target block marked by a question mark. Some entries of $S^-_{k,a,b}\times Q^-_{k,a,b}$ may be observed under the original staggered pattern, but we discard them to obtain a literal four-block construction. Furthermore, to leverage the shared subspaces $\boldsymbol U$ and $\boldsymbol V$ in the Tucker2 model, we define anchor sets that identify auxiliary fully observed rows and columns. For $j \neq k$, we write $  \mathsf{ColAnc}_k(j)\!:=
\!\{t\in Q_{k,b}:\Omega_{i,t, j} = 1\text{ for all } i\in S_{k, a}\}$ and $\mathsf{RowAnc}_k(j):=\{i\in S_{k,a}:\Omega_{i,t, j}=1\text{ for all }t\in Q_{k, b}\}$. These correspond to periods in~$Q_{k,b}$ that are observed for all units in $S_{k,a}$, and to units in $S_{k,a}$ that are observed throughout all periods in $Q_{k,b}$, respectively. For completeness, we also set $\mathsf{RowAnc}_k(k) := S_{k,a, b}^+$ and $\mathsf{ColAnc}_k(k) := Q_{k,a,b}^+$. These sets identify the auxiliary submatrices used to construct the upper and left pooled matrices, as outlined in the following algorithm.

\begin{algorithm}[!ht]
\captionsetup{font=scriptsize,skip=2pt} \caption{\textsc{BilinearTensorStaggered} for the estimation of 
$\mu_{\boldsymbol x \boldsymbol y}^{(k,a,b)}=\boldsymbol x^\top \mathcal M_{\bullet,\bullet,k}^{(a,b)} \,\boldsymbol y$ 
for a fixed missing block $(a,b)$ in slice~$k$ of a tensor with staggered adoption missingness}\label{alg:staggered-wrapper}
\begingroup \footnotesize \linespread{0.92}\selectfont \algrenewcommand{\algorithmicindent}{1em} \setlength{\abovedisplayskip}{2pt} \setlength{\belowdisplayskip}{2pt} \setlength{\abovedisplayshortskip}{1pt} \setlength{\belowdisplayshortskip}{1pt}
\begin{algorithmic}[1]
\Require index $k\in[K]$, missing block $(a,b)$ with $a+b>o_k+1$, rank $r$, unit vectors $\boldsymbol x\in\mathcal S^{N_{ak}-1}$ and $\boldsymbol y\in\mathcal S^{T_{bk}-1}$, data $\mathcal Y$, parameter $\tau > 0$.
\State Permute rows so that $\Omega_{\bullet, \bullet, k}$ is in staircase form.
\State Compute $S_{k, a, b}^+, S_{k, a}$, $Q_{k, a, b}^+, Q_{k, b}$, and $\mathsf{RowAnc}_k(j), \mathsf{ColAnc}_k(j)$ for all $j \in [K]$.
\State Form pooled left matrix $\boldsymbol Y_{\mathrm{left}}^{\mathrm{p}}\gets (\mathcal Y_{S_{k,a},\mathsf{ColAnc}_k(1),1}\ \ \cdots\ \ \mathcal Y_{S_{k,a},\mathsf{ColAnc}_k(K),K})\in\mathbb R^{|S_{k,a}|\times \sum_{j=1}^K |\mathsf{ColAnc}_k(j)|}$.
\State Compute rank-$r$ truncated singular value decomposition $(\hat{\boldsymbol U}_{\mathrm{left}},\hat{\boldsymbol \Sigma}_{\mathrm{left}},\hat{\boldsymbol V}_{\mathrm{left}})\gets \mathrm{SVD}_r(\boldsymbol Y_{\mathrm{left}}^{\mathrm{p}})$.
\State Set $\hat{\boldsymbol U}_{+k}\gets(\hat{\boldsymbol U}_{\mathrm{left}})_{S^+_{k,a,b},\bullet}$ and $\hat{\boldsymbol U}_{ak}\gets (\hat{\boldsymbol U}_{\mathrm{left}})_{R_{ak},\bullet}$.
\State Compute $\hat{\boldsymbol H}_k\gets\hat{\boldsymbol U}_{+k}^\top\hat{\boldsymbol U}_{+k}\in\mathbb R^{r\times r}$, take the eigendecomposition $\hat{\boldsymbol H}_k=\boldsymbol Q \,\operatorname{diag}(\lambda_1,\ldots,\lambda_r) \, \boldsymbol Q^\top$, and set
\[
\hat{\boldsymbol H}_{k,\tau}^\mathrm{inv}
\gets
\boldsymbol Q\operatorname{diag}\left(\left\{\frac{1}{\max[\lambda_i,\tau]}\right\}_{i=1}^r\right) \, \boldsymbol Q^\top .
\]
\State Compute $\hat{\boldsymbol \alpha}_x^{(k,a,b)}\gets\hat{\boldsymbol U}_{ak}^\top \, \boldsymbol x\in\mathbb R^r$.
\State Form pooled upper matrix $\boldsymbol Y_{\mathrm{up}}^{\mathrm{p}}\gets (\mathcal Y_{\mathsf{RowAnc}_k(1),Q_{k,b},1}\ ;\ \cdots\ ;\ \mathcal Y_{\mathsf{RowAnc}_k(K),Q_{k,b},K})\in~\mathbb R^{\sum_{j=1}^K |\mathsf{RowAnc}_k(j)|\times |Q_{k,b}|}$.
\State Compute rank-$r$ truncated singular value decomposition $(\hat{\boldsymbol U}_{\mathrm{up}},\hat{\boldsymbol \Sigma}_{\mathrm{up}},\hat{\boldsymbol V}_{\mathrm{up}})\gets \mathrm{SVD}_r(\boldsymbol Y_{\mathrm{up}}^{\mathrm{p}})$.
\State Let $s_k \gets \sum_{j=1}^{k-1}|\mathsf{RowAnc}_k(j)|$, and extract $\hat{\boldsymbol U}_{\mathrm{up}}^{(k)}\gets(\hat{\boldsymbol U}_{\mathrm{up}})_{\{s_k+1,\ldots,s_k+|S^+_{k,a,b}|\},\bullet}$, $\hat{\boldsymbol V}_{bk}\gets(\hat{\boldsymbol V}_{\mathrm{up}})_{C_{bk},\bullet}$.
\State Compute $\boldsymbol T_y\gets\hat{\boldsymbol V}_{{bk}}^\top \,\boldsymbol y\in\mathbb R^r$, $\boldsymbol W_y\gets\hat{\boldsymbol \Sigma}_{\mathrm{up}}\boldsymbol T_y\in\mathbb R^r$, and $\boldsymbol X_y\gets\hat{\boldsymbol U}_{\mathrm{up}}^{(k)}\boldsymbol W_y\in\mathbb R^{|S^+_{k,a,b}|}$.
\State Compute $\hat{\boldsymbol \beta}_y^{(k,a,b)}\gets \, \hat{\boldsymbol H}_{k,\tau}^\mathrm{inv} \, \hat{\boldsymbol U}_{+k}^\top \boldsymbol X_y\in\mathbb R^r$.
\State \Return $\hat\mu_{\boldsymbol x \boldsymbol y}^{(k,a,b)}\gets\langle\hat{\boldsymbol \alpha}_x^{(k,a,b)},\hat{\boldsymbol \beta}_y^{(k,a,b)}\rangle$.
\end{algorithmic}
\endgroup
\end{algorithm}
\vspace{-15pt}

Algorithm~\ref{alg:staggered-wrapper} extends Algorithm~\ref{alg:bilinear4block} to the staggered-adoption setting, and reduces to it when the missingness pattern has four-block form. For a fixed missing block $(a,b)$ in layer $k$, the algorithm restricts attention to $\mathcal Y_{S_{k,a}, Q_{k,b}, \bullet}$, and uses the observations lying in $\bigcup_{j=1}^K(S_{k,a}\times \mathsf{ColAnc}_k(j)\times\{j\})$ and $\bigcup_{j=1}^K(\mathsf{RowAnc}_k(j)\times Q_{k,b}\times\{j\})$ to construct an auxiliary four-block problem, discarding all remaining entries. The left pooled matrix is formed from the anchor-column blocks $\mathcal Y_{S_{k,a},\mathsf{ColAnc}_k(j),j}$, while the upper pooled matrix is formed from the anchor-row blocks $\mathcal Y_{\mathsf{RowAnc}_k(j),Q_{k,b},j}$. This construction is illustrated in Fig.~\ref{fig:Staggered}.

Our procedure generalizes~\citet[][Algorithm~2]{yan2024entrywise} to the tensor setting and targets the bilinear form directly, rather than reconstructing the entire missing block. Related denoising techniques, using anchor sets and combined with PCA, were employed in~\cite{liu2026representation} for a different statistical problem, where the goal is to recover the global left subspace from matrix data with blockwise missingness, with error measured with $\|\cdot\|_F$.

The auxiliary four-block construction also imposes a basic dimensional feasibility condition. In applications, the rank must satisfy $r\le \min(\sum_{j\in[K]} |\mathsf{RowAnc}_k(j)|, \, \sum_{j\in[K]} |\mathsf{ColAnc}_k(j)|, $ \, $ |S_{k,a}|,$ $|Q_{k,b}|)$, so that the two rank-$r$ truncated SVDs are well defined. This is only a minimal requirement for running the procedure. Even when this condition holds, additional assumptions are needed to guarantee that the resulting estimator is accurate; the theoretical analysis of Algorithm~\ref{alg:staggered-wrapper} is the object of the next section.

Finally, when aggregate quantities over multiple missing blocks of the same slice are required, such as those introduced in Section~\ref{sec:real-data-castle}, the most direct strategy is to apply Algorithm~\ref{alg:staggered-wrapper} separately to each missing block and then aggregate the resulting estimates. This blockwise implementation recomputes two rank-$r$ SVDs for every missing block, leading to the quadratic-cost procedure described in Algorithm~\ref{alg:qudraticStaggeredAggregate} in Appendix~\ref{app:simulEstimands}. We also propose a reduced-anchor variant that reuses computations by caching the pooled left SVD once for each active row block $a$, and the pooled upper SVD once for each active column block $b$. This yields the linear-SVD-cost procedure in Algorithm~\ref{alg:LinearStaggeredAggregate}, at the price of some loss in statistical efficiency. Appendix~\ref{app:simulEstimands} provides an extensive discussion of this computational--statistical tradeoff, and Fig.~\ref{fig:StaggeredSynthetic} in Section~\ref{sec:simulSynthetic} compares the two procedures in simulation.

\subsection{Theoretical analysis} \label{sec:analysisStaggered}
Algorithm~\ref{alg:staggered-wrapper} inherits the desirable properties of
Algorithm~\ref{alg:bilinear4block} under analogous assumptions, with
particular care needed in adapting Assumption~\eqref{assump:subblock-conditioning}
to the auxiliary four-block reduction. To state these assumptions and the resulting corollary, we suppress the dependence on $(k,a,b)$ for simplicity, and set $\mathfrak{n}:=|S|$ and $\mathfrak{t}:=|Q|$. For the target layer, define $\mathfrak{n}_{1k}:=|S^+|$ and $\mathfrak{t}_{1k}:=|Q^+|$. For each $j\neq k$, define $\mathfrak{n}_{1j}:=|\mathsf{RowAnc}_k(j)|$ and $\mathfrak{t}_{1j}:=|\mathsf{ColAnc}_k(j)|$, and set $\mathfrak{n}_{2j}:=\mathfrak{n}-\mathfrak{n}_{1j}$ and $\mathfrak{t}_{2j}:=\mathfrak{t}-\mathfrak{t}_{1j}$. Finally, let $\mathfrak{n}_{1,\mathrm{p}}:=\sum_{j=1}^K\mathfrak{n}_{1j}$, $\mathfrak{t}_{1,\mathrm{p}}:=\sum_{j=1}^K\mathfrak{t}_{1j}$, $\rho_{\mathfrak{n}}:=\mathfrak{n}_{1,\mathrm{p}}/\mathfrak{n}$, $\rho_{\mathfrak{t}}:=\mathfrak{t}_{1,\mathrm{p}}/\mathfrak{t}$, $p_{\mathfrak{n}}:=\max(\mathfrak{n}_{1,\mathrm{p}},\mathfrak{t})$, $p_{\mathfrak{t}}:=\max(\mathfrak{n},\mathfrak{t}_{1,\mathrm{p}})$, $\zeta_{\mathfrak{n}}:=\log(\mathfrak{n}_{1,\mathrm{p}}+\mathfrak{t})$, $\zeta_{\mathfrak{t}}:=\log(\mathfrak{n}+\mathfrak{t}_{1,\mathrm{p}})$, $\widetilde\gamma_{\min}:=c_\ell \, \gamma_{\min}\sqrt{\mathfrak{n}\mathfrak{t}/NT}$ and $\widetilde\gamma_{\max}:=c_u \,\gamma_{\max}\sqrt{\mathfrak{n}\mathfrak{t}/NT}$. We assume the following.
\begin{assump}\label{ass:subBlockAuxilliary}
There exist constants
$0<c_\ell\le c_u$ such that
\[
c_\ell \frac{\mathfrak{n}_{1k}}{N} \boldsymbol I_r
\preceq
\boldsymbol U_{S^+}^\top \boldsymbol U_{S^+}
\preceq
c_u \frac{\mathfrak{n}_{1k}}{N} \boldsymbol I_r,
\qquad
c_\ell \frac{\mathfrak{n}}{N} \boldsymbol I_r
\preceq
\boldsymbol U_S^\top \boldsymbol U_S
\preceq
c_u \frac{\mathfrak{n}}{N} \boldsymbol I_r,
\]
and, for every $j\neq k$,
\[
c_\ell \frac{\mathfrak{n}_{1j}}{N} \boldsymbol I_r
\preceq
\boldsymbol U_{\mathsf{RowAnc}_k(j)}^\top
\boldsymbol U_{\mathsf{RowAnc}_k(j)}
\preceq
c_u \frac{\mathfrak{n}_{1j}}{N}\boldsymbol I_r.
\]
We require that the column factors satisfy the analogous conditions with
$\boldsymbol V,Q,Q^+,\mathsf{ColAnc}_k(j), \mathfrak{t}_{1j}$ and $T$ in place of
$\boldsymbol U,S,S^+,\mathsf{RowAnc}_k(j), \mathfrak{n}_{1j}$ and $N$, respectively.
\end{assump}

\begin{assump}\label{ass:aux-extra}
We have $r+\max(\zeta_{\mathfrak{n}},\zeta_{\mathfrak{t}}) \leq c_{\mathrm{dim}} \, \min(\mathfrak{n}-r,\mathfrak{t}-r,\mathfrak{n}_{1k},\mathfrak{t}_{1k})$ and $c_{\mathrm{blk}} \min(\zeta_{\mathfrak{n}},\zeta_{\mathfrak{t}})\leq r$. Also, the noise level satisfies
\[
    \widetilde\theta
    :=
    \frac{\sigma}{\widetilde\gamma_{\min}}
    \max\left\{
        \sqrt{\mathfrak{n}},
        \sqrt{\mathfrak{t}},
        \sqrt{\frac{\mathfrak{n}}{\rho_{\mathfrak{t}}}},
        \sqrt{\frac{\mathfrak{n}\mathfrak{t}}{\mathfrak{n}_{1k}}}
    \right\}
    \leq c_0.
\]
Finally, we write $\widetilde{\boldsymbol U}:=\boldsymbol U_S(\boldsymbol U_S^{\top}\boldsymbol U_S)^{-1/2}$ and
$\widetilde{\boldsymbol V}:=\boldsymbol V_Q(\boldsymbol V_Q^{\top}\boldsymbol V_Q)^{-1/2}$, and assume that
$\widetilde\nu_x:=\sqrt{\mathfrak{n}/r}\,
\|\widetilde{\boldsymbol U}_{ R_{ak}}^{\top} \boldsymbol x\|_2$ and
$\widetilde\nu_y:=\sqrt{\mathfrak{t}/r}\,
\|\widetilde{\boldsymbol V}_{C_{bk}}^{\top} \boldsymbol y\|_2$ are of constant order.
\end{assump}

Assumption~\eqref{ass:aux-extra} is the adaptation of~\eqref{assump:sampleSize}, \eqref{assump:smallNoise} and~\eqref{assump:Incoherence} for staggered designs. Moreover, the first and third conditions in~\eqref{ass:subBlockAuxilliary} are the direct analogues of~\eqref{assump:subblock-conditioning}, and require the observed row and column blocks retained in the auxiliary four-block problem to contain all~$r$ latent directions in a well-conditioned way. The middle condition is the restricted analogue of the orthonormality condition $\boldsymbol U^\top \boldsymbol U=\boldsymbol V^\top \boldsymbol V=\boldsymbol I_r$, and requires that the Gram matrices associated to $\boldsymbol U_S$ and~$\boldsymbol V_Q$ remain full rank and well conditioned. 

\begin{cor}\label{corollary:StaggeredUB}
Consider a tensor $\mathcal{M}$ satisfying $\mathcal{M}=\mathcal{C}\times_1 \boldsymbol U\times_2 \boldsymbol V\times_3 \boldsymbol I_K$, where $\boldsymbol U\in\mathbb{R}^{N\times r}$ and $\boldsymbol V\in\mathbb{R}^{T\times r}$ have orthonormal columns, and the core tensor is such that $0 < \gamma_{\min} \leq \sigma_{\min}(\mathcal{C}_{\bullet, \bullet, j}) \leq \sigma_{\max}(\mathcal{C}_{\bullet, \bullet, j}) \leq \gamma_{\max} < \infty$ for all $j \in [K]$. Choose $k \in [K]$, indices $(a,b)$ such that $a+b>o_k+1$, and unit vectors $\boldsymbol x\in\mathcal{S}^{N_{ak}-1}$ and $\boldsymbol y\in\mathcal{S}^{T_{bk}-1}$. Let $\mu_{\boldsymbol x \boldsymbol y}^{(k, a, b)}$ be as in~\eqref{eq:muStaggered}, and define~$\hat{\mu}_{\boldsymbol x \boldsymbol y}^{(k, a, b)}$ to be the output of Algorithm~\ref{alg:staggered-wrapper} run with $0<\tau \leq \frac{c_\ell \mathfrak{n}_{1k}}{2 c_u  \mathfrak{n}}$. Fix also absolute constants $0 < c_\ell  \le c_u <\infty$,  and assume~\eqref{ass:subBlockAuxilliary}, \eqref{ass:aux-extra} with $\widetilde \nu_x \neq 0, \widetilde \nu_y \neq 0$. Let 
\[
    \widetilde{\Upsilon}_{xy}
    :=\;
    \frac{\sigma^2(r+\zeta_{\mathfrak{n}})}{\rho_{\mathfrak{n}}}
    \|\widetilde{\boldsymbol U}_{R_{ak}}^{\top} \boldsymbol x\|_2^2
    +
    \frac{\sigma^2(r+\zeta_{\mathfrak{t}})}{\rho_{\mathfrak{t}}}
    \|\widetilde{\boldsymbol V}_{C_{bk}}^{\top} \boldsymbol y\|_2^2 +
    \frac{\sigma^2 \mathfrak{n}}{\mathfrak{n}_{1k}}
    \|\widetilde{\boldsymbol U}_{R_{ak}}^{\top}\boldsymbol  x\|_2^2
    \|\widetilde{\boldsymbol V}_{C_{bk}}^{\top}\boldsymbol  y\|_2^2,
\]
and further assume that
\begin{align}\label{eq:negligibilityForUBStaggared}
\frac{\widetilde{\gamma}_{\max}^2}{\tau}
    \frac{\mathfrak{n}_{1k}}{\mathfrak{n}}
    \left(p_{\mathfrak{n}}^{-10}+p_{\mathfrak{t}}^{-10}\right)
    +
    \frac{\sigma^2}{\tau}
    (\mathfrak{n}_{1k}+\mathfrak{t})
    \left(p_{\mathfrak{n}}^{-5}+p_{\mathfrak{t}}^{-5}\right)
    \leq
    c_0\widetilde{\Upsilon}_{xy}.
\end{align}
There exists a constant
$c_1=
    c_1\left(
        c_\ell,c_u,c_0,c_{\mathrm{dim}},  c_{\mathrm{blk}},
        \kappa,
        \widetilde{\nu}_x,\widetilde{\nu}_y
    \right)<\infty$
such that $\mathbb{E}_\mathcal{M}
    [\{
            \widehat{\mu}_{xy}^{(k,a ,b)}
            -
            \mu_{xy}^{(k, a,b)}
    \}^2]
    \leq
    c_1 \, \widetilde{\Upsilon}_{xy}$.
\end{cor}
This result follows directly from  Theorem~\ref{thm:4BlockUBLinear}, with the original dimensions and block sizes replaced by their auxiliary counterparts. To see why, observe that $\boldsymbol U_S^\top \boldsymbol U_S$ and~$\boldsymbol V_Q^\top \boldsymbol V_Q$ are invertible under~\eqref{ass:subBlockAuxilliary}, hence $\widetilde{\boldsymbol U}:=\boldsymbol U_S(\boldsymbol U_S^\top \boldsymbol U_S)^{-1/2}$, $\widetilde{\boldsymbol V}:=\boldsymbol V_Q(\boldsymbol V_Q^\top \boldsymbol V_Q)^{-1/2}$, and $\widetilde{\mathcal{C}}_{\bullet,\bullet,j}:=(\boldsymbol U_S^\top \boldsymbol U_S)^{1/2} \mathcal C_{\bullet,\bullet,j}(\boldsymbol V_Q^\top \boldsymbol V_Q)^{1/2}$ give an orthonormal Tucker2 representation of the auxiliary signal $\mathcal M_{S,Q,j}=\widetilde{\boldsymbol U} \widetilde{\mathcal{C}}_{\bullet,\bullet,j}\widetilde{\boldsymbol V}^\top$. The auxiliary core tensor satisfies $0<\tilde \gamma_\mathrm{min} \leq \sigma_{\min}(\widetilde C_{\bullet,\bullet,j})
\le
\sigma_{\max}(\widetilde C_{\bullet,\bullet,j})
\le \tilde \gamma_\mathrm{max}$, while $\tilde {\boldsymbol U}$ and $\tilde{\boldsymbol V}$ satisfy~\eqref{assump:subblock-conditioning} with $c_\ell/c_u$ and $c_u/c_\ell$ in place of $c_\ell$ and $c_u$, respectively; see the proof of Corollary~\ref{corollary:StaggeredUB} for the precise details. This, together with~\eqref{ass:aux-extra}, implies that the pooled upper and left matrices obey the same conditions as in the four-block setting, with $N,T,N_{1k},T_{1k},\rho_N,\rho_T,p_N,p_T,\zeta_N,\zeta_T,\gamma_{\min},\gamma_{\max}, c_\ell, c_u$ replaced by $\mathfrak n,\mathfrak t,\mathfrak n_{1k},\mathfrak t_{1k},\rho_{\mathfrak n},\rho_{\mathfrak t},p_{\mathfrak n},p_{\mathfrak t},\zeta_{\mathfrak n},\zeta_{\mathfrak t},\widetilde\gamma_{\min},\widetilde\gamma_{\max}$, $c_\ell/c_u, c_u/c_\ell$.  This is precisely what is needed to apply Theorem~\ref{thm:4BlockUBLinear}, even though the auxiliary observation pattern need not consist of four contiguous blocks, and with these substitutions the stated result follows.

\section{Numerical results}\label{sec:simul}
\subsection{Synthetic data}\label{sec:simulSynthetic}

\if1\anon{
The code and datasets for reproducing the simulations are available on \href{https://github.com/abordino/LinearMultilayerStaggered}{GitHub}.
} \fi
In this subsection we empirically validate our theoretical claims on synthetic data. We verify that pooling improves performance for moderate values of $K$, while saturation occurs for large $K$, thereby confirming the phase transitions predicted by Theorem~\ref{thm:4BlockUBLinear}. We include robustness checks with respect to rank misspecification, SNR levels, vector inputs, and mask dimensions. For staggered adoption designs, we show that pooling across layers improves accuracy and demonstrate how to lower computational costs when computing an average counterfactual component over multiple~blocks.

In Fig.~\ref{fig:DecayWithK}(a), we fix $N=100, \, T=80, \, r=6$ and vary $K\in\{1,2,5,10,20,50,200\}$. The matrices $\boldsymbol U$ and~$\boldsymbol V$ are generated by drawing Gaussian random matrices and orthonormalizing their columns. For each slice $j \in [K]$, we generate
$\mathcal C_{\bullet, \bullet, j} = \boldsymbol O_{j} \operatorname{diag}(\sigma_1,\dots,\sigma_r) \, \tilde {\boldsymbol O_{j}}^\top$, where $\boldsymbol O_j, \tilde {\boldsymbol O}_j \in \mathbb R^{r \times r}$ are independent random orthonormal matrices, and the singular values are fixed at $(\sigma_1,\dots,\sigma_r) = (2,\,1.72,\,1.44,\,1.16,\,0.88,\,0.60)$. The observations are generated according to~\eqref{eq:Observation4block}, with noise variance $\sigma^2$ chosen so that $\mathrm{SNR}^{-1} := \sigma^2 N/\sigma_r^2$ and $\mathrm{SNR}=1$, and with block sizes $N_{1j}=70$ and $T_{1j}=60$ for all $j\in[K]$. We fix $k=1$ and consider unit vectors $\boldsymbol x \in \mathbb{R}^{N_{21}}$ and $\boldsymbol y \in \mathbb{R}^{T_{21}}$ independently drawn from standard Gaussian distributions and then normalized to have unit norm. We compare five procedures for estimating~\eqref{eq:mu4block}:
\begin{enumerate}
    \item \textsc{Estimated pooled}: Algorithm~\ref{alg:bilinear4block} with $\tau = 0.01$;
    \item  \textsc{Estimated no-pool}: Algorithm~\ref{alg:bilinear4block} with $\tau = 0.01$, applied to the target slice $\mathcal{Y}_{\bullet, \bullet, k}$;
    \item \textsc{Oracle pooled}: returns $
        \boldsymbol x^\top \boldsymbol U_{2k} \, \mathcal{C}_{\bullet, \bullet, k} \, (\boldsymbol W_{\rm up}^\top \boldsymbol W_{\rm up})^{-1}\boldsymbol W_{\rm up}^\top \, (\boldsymbol E_{\rm up}^{\rm p})_{\bullet,\{T_{1k}+1,\ldots,T\}} \, \boldsymbol  y +\boldsymbol x^\top  (\boldsymbol E_\mathrm{left}^\mathrm{p})_{\{N_{1k} +1, \ldots, N\}, \bullet} \,\, \boldsymbol W_\mathrm{left}(\boldsymbol W_\mathrm{left}^\top \boldsymbol W_\mathrm{left})^{-1}\mathcal C_{\bullet, \bullet, k} \, \boldsymbol V_{2k}^\top  \, \boldsymbol y +  \boldsymbol x^\top \mathcal M_{\bullet, \bullet, k}^{(d)}\boldsymbol y$;
    \item \textsc{Oracle no-pool}: layer-specific counterpart of \textsc{Oracle pooled}, which returns $\boldsymbol x^\top \boldsymbol U_{2k} ( \boldsymbol U_{1k}^\top  \boldsymbol U_{1k})^{-1} \, \boldsymbol U_{1k}^\top \mathcal E_{\bullet, \bullet, k}^{(b)} \, \boldsymbol y + \boldsymbol x^\top \mathcal E_{\bullet, \bullet, k}^{(c)} \boldsymbol V_{1k}  ( \boldsymbol V_{1k}^\top  \boldsymbol V_{1k})^{-1} \boldsymbol V_{2k}^\top \, \boldsymbol y+  \boldsymbol x^\top \mathcal M_{\bullet, \bullet, k}^{(d)}\boldsymbol y$; 
    \item \textsc{Oracle local}: $  \boldsymbol x^\top \boldsymbol U_{2k} (\boldsymbol U_{1k}^\top  \boldsymbol U_{1k})^{-1} \, \boldsymbol U_{1k}^\top (\boldsymbol E_{\rm up}^{\rm p})_{\{s_k + 1, \ldots, s_k + N_{1k}\}, \bullet}
            \,\boldsymbol V \, \boldsymbol V_{2k}^\top \,\boldsymbol  y + \boldsymbol x^\top \mathcal M_{\bullet, \bullet, k}^{(d)}\boldsymbol y$.
\end{enumerate}
The oracle quantities correspond to the Gaussian terms in the expansion of $\hat \mu_{\boldsymbol x \boldsymbol y}^{(k)} - \mu_{\boldsymbol x \boldsymbol y}^{(k)}$ given in Lemma~\ref{lemma1} in the Supplement, and capture the leading contribution to the stochastic error of our procedure. For each value of $K$ and for each of these estimators, we run $500$ replications and report the average squared error. The results show that the \textsc{no-pool} estimators do not decay with $K$, as expected. By contrast, the two \textsc{pooled} estimators exhibit an approximate $1/K$ decay and appear to approach the line corresponding to \textsc{Oracle local}. This line is much lower than the others because it is related to both $\|\boldsymbol U_{2k}^\top \, \boldsymbol x\|_2^2$ and $\|\boldsymbol V_{2k}^\top \, \boldsymbol y\|_2^2$, yielding an additional factor of roughly $r/T  = 6/80 = 0.075$. To better understand the relationship among these latter three methods, in Fig.~\ref{fig:DecayWithK}(b) we repeat the same simulation study for $K \in \{50, 150, 300, 500, 1000\}$. The results show that \textsc{Oracle pooled} continues to decay with $K$, whereas Algorithm~\ref{alg:bilinear4block} saturates at approximately the level of \textsc{Oracle local}. This is in accordance with Theorem~\ref{thm:4BlockUBLinear}. 

Fig.~\ref{fig:DecayWithK}(c) uses the same setup as Fig.~\ref{fig:DecayWithK}(a), with a rank-$6$ signal tensor, but draws the mask dimensions randomly over $N_{1j} \in \{30,\ldots,70\}$ and $T_{1j} \in \{30,\ldots,60\}$. The query vectors are chosen to be aligned with the leading eigenspaces of $\boldsymbol U_{2k}$ and $\boldsymbol V_{2k}$, thereby increasing $\|\boldsymbol U_{2k}^\top \boldsymbol x\|_2$ and $\|\boldsymbol V_{2k}^\top \boldsymbol y\|_2$ and deliberately violating the incoherence condition in~\eqref{assump:Incoherence}. In addition, all estimators are run with misspecified rank $r+5=11$. The results show that \textsc{Estimated pooled} still exhibits a similar phase transition as in Fig.~\ref{fig:DecayWithK}(a), despite the misspecification and incoherence violation. By contrast, \textsc{Estimated no-pool} is more sensitive to both sources of misspecification and has larger error than \textsc{Oracle no-pool}.

Finally, Fig.~\ref{fig:DecayWithK}(d) uses the same setup as Fig.~\ref{fig:DecayWithK}(b), but varies $\mathrm{SNR} \in \{1,10^{-2},10^{-4}\}$ to examine sensitivity under increasingly noisy regimes. The results indicate that the proposed methodology remains reasonably stable even at lower signal levels, with \textsc{Estimated pooled} appearing close to \textsc{Oracle local} across the considered SNR values. 

\begin{figure}[!ht]
    \centering
    \includegraphics[width=0.75\linewidth]{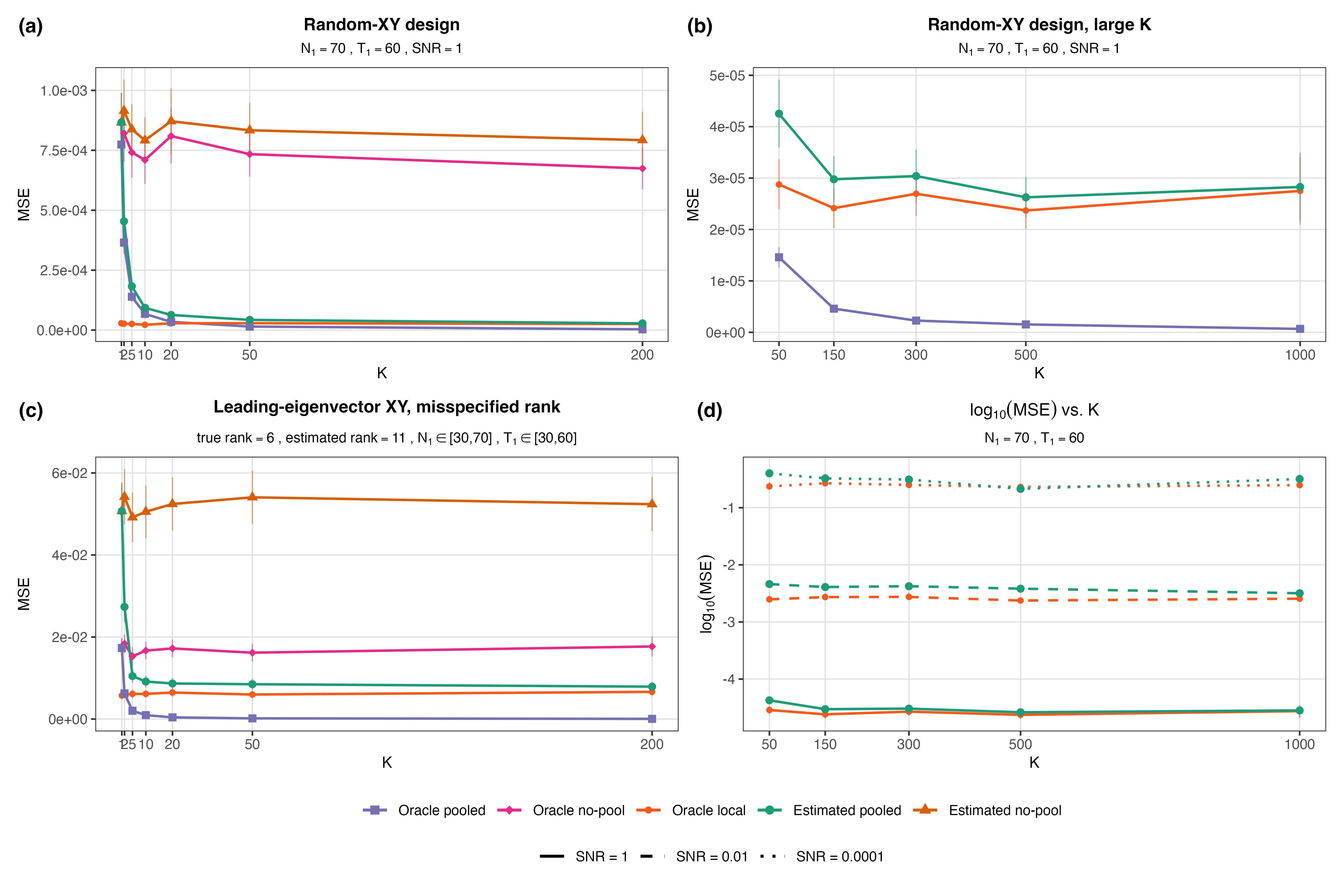}
   
    \captionsetup{
  font={scriptsize,stretch=0.8},
  skip=10pt
}
\caption{
(a) Mean squared error of the five estimators as a function of $K \in \{1,2,5,10,20,50,200\}$, with $N=100$, $T=80$, $r=6$, $\mathrm{SNR}=1$, and block sizes $N_{1j}=70$, $T_{1j}=60$ for all $j\in[K]$. Results are averaged over $500$ replications.
(b) Mean squared error of the pooled and local estimators for $K \in \{50,150,300,500,1000\}$ under the same simulation setting.  (c) Analogue of (a) with violation of incoherence, rank-misspecification and mask sizes $N_{1j} \in \{30,\ldots,70\}$ and $T_{1j} \in \{30,\ldots,60\}$. (d) Sensitivity analysis with different SNR values. Error bars show $\pm 1.96$ standard errors.}
    \label{fig:DecayWithK}
    \vspace{-15pt}
\end{figure}

We also empirically evaluate Algorithm~\ref{alg:staggered-wrapper} in a synthetic staggered-adoption design. We generate a rank-$5$ Tucker2 signal tensor with $N=150$, $T=200$, and $K=10$, where $\boldsymbol U\in\mathbb R^{N\times r}$ and $\boldsymbol V\in\mathbb R^{T\times r}$ are orthonormal and generated as before, and the entries of the core matrices $\mathcal C_{\bullet, \bullet, j}\in\mathbb R^{r\times r}$ are independent standard normal variables. We observe $\mathcal Y_{\bullet,\bullet,j}
    =
    P_{\Omega_{\bullet,\bullet,j}}
    (\mathcal M_{\bullet,\bullet,j}+\mathcal E_{\bullet,\bullet,j})$,
where the entries of $\mathcal E$ are independent $\mathcal N(0, 0.03^2)$. The missingness masks $\Omega_{\bullet, \bullet, j}$ are generated by an irreversible adoption process in which, for each unit $i$ and layer $j$, the adoption time $A_{ij}$ is sampled independently from a layer-specific grid of adoption times, with probability $0.20$ of never adopting. We then set $\Omega_{i,t,j}=\mathbbm 1\{t<A_{ij}\}$. We take $k=1$ as the target layer and use the adoption ordering in this slice to reorder the units across all layers. Under this ordering, $\Omega_{\bullet,\bullet,1}$ is a staircase and satisfies  $\Omega_{i,t,1}=\mathbbm 1\{(i,t)\in R_{a1}\times C_{b1}\text{ for some }a,b\text{ with }a+b\le 5\}$. We then target all six missing blocks and, for each of them, we generate $100$ independent query pairs $\boldsymbol x\in\mathcal S^{N_{a1}-1}$ and $\boldsymbol y\in\mathcal S^{T_{b1}-1}$ by drawing independent standard normal vectors and normalizing them to have unit norm. For each query we estimate~\eqref{eq:muStaggered} using both the pooled estimator outlined in Algorithm~\ref{alg:staggered-wrapper}, and its matrix-only counterpart, which runs the procedure on $\mathcal{Y}_{\bullet, \bullet, k}$ only and does not borrow information from the other layers. In particular, this algorithm uses $\mathcal Y_{S_{k,a}, Q_{k,a,b}^+,k}$ and $\mathcal Y_{S_{k,a,b}^+, Q_{k,b},k}$ in place of $\boldsymbol Y_\mathrm{left}^\mathrm{p}$ and $\boldsymbol Y_\mathrm{up}^\mathrm{p}$, respectively, and is essentially equivalent to a plug-in estimator based on the matrix completion estimator proposed in~\cite{yan2024entrywise}. Both procedures are run with $\tau \in \left\{10^{-0.5 + 0.025j} : j = 0, \ldots, 20\right\}$.

Fig.~\ref{fig:StaggeredSynthetic}(a) reports the mean absolute estimation error over the $100$ random bilinear queries for each selected block and each value of $\tau$. The pooled estimator has substantially lower estimation error than the matrix-only estimator. This is expected, since the pooled method exploits the shared row and column subspaces across layers, whereas the matrix-only method uses only the two anchor blocks available in $\mathcal{Y}_{\bullet, \bullet, k}$. In general, the fact that certain blocks have larger errors than others can be attributed to differences in the size of the missing blocks: smaller missing blocks may have a lower signal-to-noise ratio, which can in turn reduce estimation accuracy. Furthermore, we see that  both estimators are stable across values of~$\tau$, particularly when $\tau\ll 1$. This agrees with our theory: under Assumption~\eqref{assump:subblock-conditioning} with $c_\ell>0$, choosing~$\tau$ small enough ensures that clipping is inactive with high probability, while stabilizing the inverse on the complementary event. Thus, $\tau$ acts primarily as a safeguard rather than a tuning parameter, and careful tuning appears unnecessary. 

\begin{figure}[!ht]
    \centering
    \includegraphics[width=1\linewidth]{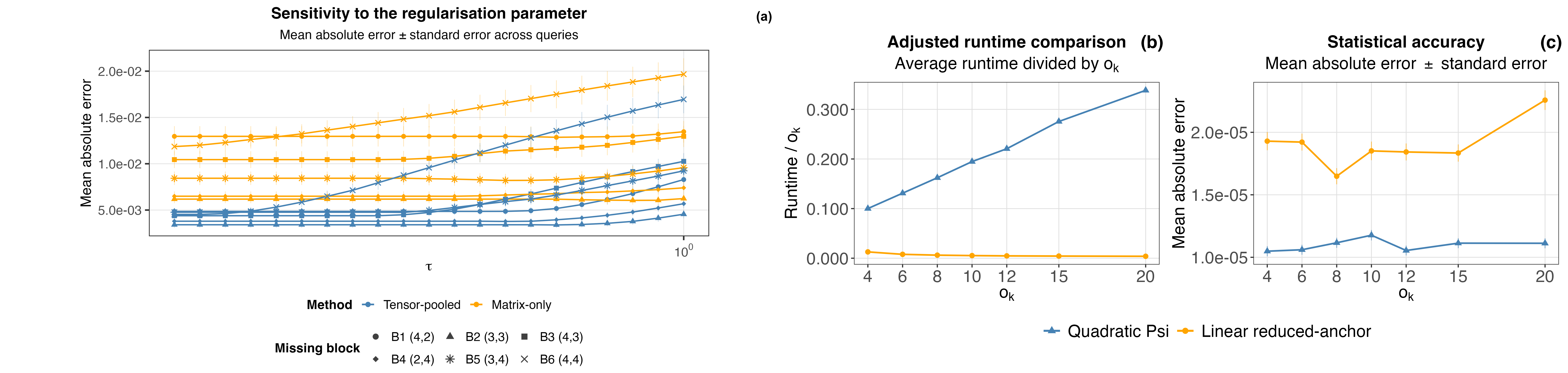}
    \captionsetup{
  font={scriptsize,stretch=0.8},
  skip=10pt
}
\caption{(a) Mean absolute estimation error over 100 random bilinear queries for each target block as a function of $\tau$ for Algorithm~\ref{alg:staggered-wrapper} and its matrix counterpart. Runtime (b) and accuracy (c) comparison for Algorithms~\ref{alg:qudraticStaggeredAggregate} and~\ref{alg:LinearStaggeredAggregate} when estimating the aggregate \textsc{Avg} functional in a synthetic staggered-adoption design. Here $o_k\in\{4,6,8,10,12,15,20\}$ denotes the number of row/time blocks induced by the staircase-adoption pattern in the target slice.}
    \label{fig:StaggeredSynthetic}
    \vspace{-15pt}
\end{figure}

Finally, we compare Algorithm~\ref{alg:qudraticStaggeredAggregate} and Algorithm~\ref{alg:LinearStaggeredAggregate} in terms of statistical efficiency and runtime when estimating the aggregate \textsc{Avg} functional defined as a weighted average of the bilinear forms in~\eqref{eq:muStaggered} across all missing blocks in the target slice; this quantity and the two procedures are presented  in detail in Appendix~\ref{app:simulEstimands}. We generate a rank-$5$ Tucker2 signal tensor with $N=150$, $T=200$, and $K=5$, using random orthonormal matrices $ \boldsymbol U\in\mathbb R^{N\times r}$ and $\boldsymbol V\in\mathbb R^{T\times r}$ and independent standard normal entries in each core matrix~$\mathcal C_{\bullet,\bullet,j}$. The error tensor $\mathcal E$ has independent $\mathcal N(0, 0.01^2)$ entries. The target layer is $k=1$ and has staircase missingness with $o_k\in\{4,6,8,10,12,15,20\}$; all non-target layers are fully observed. For each value of $o_k$, we run both procedures with $\tau=10^{-3}$ over $500$ replications and report the average runtime divided by~$o_k$ and the mean absolute deviation from the true value. 

The results in Figs.~\ref{fig:StaggeredSynthetic}(b,~c) are consistent with the discussion in Appendix~\ref{app:simulEstimands}. The runtime divided by~$o_k$ appears approximately linear for Algorithm~\ref{alg:qudraticStaggeredAggregate}, reflecting its quadratic dependence on the number of target blocks, while it is nearly constant for Algorithm~\ref{alg:LinearStaggeredAggregate}, reflecting the fact that the dominant SVD cost is linear in~$o_k$. Also, Algorithm~\ref{alg:qudraticStaggeredAggregate} has better statistical accuracy, as expected, since Algorithm~\ref{alg:LinearStaggeredAggregate} uses reduced anchor sets to improve runtime and therefore sacrifices some statistical efficiency.

\subsection{Real-data application: Castle Doctrine data}
\label{sec:real-data-castle}

In this and the next section, we consider real-data applications involving two signal tensors, $\mathcal M(0)$ and $\mathcal M(1)$, corresponding to the untreated and treated responses. Following the potential-outcomes framework \citep{rubin1974estimating}, the entries of the fully observed data tensor~$\cal Y$ satisfy
$\mathcal Y_{itj}
=
\Omega_{itj} \, \mathcal Y_{itj}(0)
+
(1-\Omega_{itj}) \,  \mathcal Y_{itj}(1)$, where $\mathcal Y(0)$ denotes the untreated potential outcome, which is observed on $\{(i,t,j) :\Omega_{itj}=1\}$ and missing on the complementary set, and $\mathcal Y(1)$ denotes the treated potential outcome, which is observed only over $\{(i,t,j) :\Omega_{itj}=0\}$. Here, we focus on estimating bilinear forms of $\mathcal M(0)$ and $\mathcal M(1)$ over the treated region. The latter problem is straightforward because $\mathcal Y(1)$ is observed on this region, so simple plug-in estimators can be used. By contrast, $\mathcal Y(0)$ is unobserved, hence bilinear functionals of $\mathcal M(0)$ require different approaches such as those introduced in the previous sections.

The empirical study uses the Castle Doctrine policy data from the \texttt{PolicyEval} repository, \href{https://github.com/guerramarcelino/PolicyEval/raw/main/Datasets/}{available on GitHub}, a standard staggered-adoption benchmark in the difference-in-differences literature \citep[e.g.][]{cheng2013castle}. After processing as described in Appendix~\ref{app:castle-processing}, we obtain tensors $\mathcal{Y}(0),\mathcal{Y}(1)\in\mathbb{R}^{50\times 11\times 4}$ for policy-off and policy-on potential outcomes. Here, $\mathcal{Y}(0)$ follows a staggered-adoption missingness pattern, while $\mathcal{Y}(1)$ is observed on the complementary policy-on region; see Fig.~\ref{fig:castle-crime-tensor-matrix}. Rows are U.S.~states, columns are years $2000$--$2010$, and slices are four logged crime rates. Importantly, because the tensor slices represent different outcomes rather than policy regimes, the staggered-adoption missingness pattern is shared across all layers of $\mathcal Y(0)$.

To assess the accuracy of our method, we conduct a masking experiment in which we artificially remove entries observed in the original data and use their known values to evaluate the resulting estimates. Starting from the original staggered observation pattern, we progressively introduce additional missingness in the target layer $k = 1$, corresponding to motor vehicle theft, while preserving the staggered structure by masking the last currently observed entry of selected rows. At each step, we estimate the average of the masked entries using the pooled estimator of Algorithm~\ref{alg:staggered-wrapper} and its matrix counterpart, as well as the nuclear-norm estimator of~\cite{choi2024matrix} and its debiased counterpart; we will refer to the latter two as CY. All rank-dependent procedures are run with $r\in\{1,2,3\}$, with $\tau=10^{-2}$ for the tensor and matrix methods. We repeat the procedure independently $10$ times, using a different masking path for each replicate, all converging to a terminal mask that retains $9$ fully observed rows and $4$ fully observed columns.

\begin{figure}[!ht]
    \centering
    \includegraphics[width=1\linewidth]{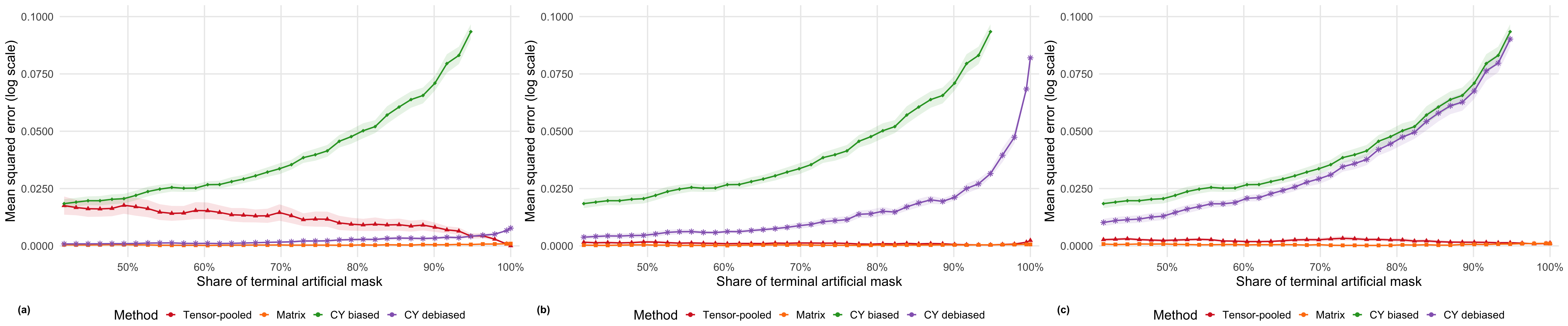}
    \captionsetup{
        font={scriptsize,stretch=0.8},
        skip=10pt
    }
    \caption{Castle masking experiment with additional missingness in the target layer $k=1$. The panels correspond to ranks $r=1,2,3$ used for the four methods, with $\tau=0.01$ for the tensor and matrix estimators. The terminal mask retains $9$ fully observed rows and $4$ fully observed columns. The curves report the average squared error over $10$ replicates, with one standard error. The percentage on the $x$-axis represents the percentage of the terminal mask reached.}
    \label{fig:Psi0ALL}
    \vspace{-15pt}
\end{figure}

Fig.~\ref{fig:Psi0ALL} shows the squared errors averaged over $10$ repetitions. In this setting, the two CY procedures perform worse than the others. One possible explanation is that these procedures accommodate more general missingness patterns, whereas the tensor and matrix estimators can be favored by structured staggered designs that are close to a four-block.

To better compare these methods, in Fig.~\ref{fig:Psi0} we repeat the experiment under a more severe masking scheme, in which the terminal mask retains only $3$ fully observed rows and $4$ fully observed columns. For $r=1$ and $r=2$, the matrix estimator performs better when additional masking is limited. As masking becomes more severe, however, its error increases relative to that of the pooled estimator, which achieves lower error in the high-missingness regime. This is consistent with the target layer initially containing enough information to estimate its low-rank structure on its own, while, as more information is removed, the auxiliary layers become increasingly useful; similar qualitative behavior arises in the synthetic masking experiment of Appendix~\ref{app:simulExtraMasking}. For $r=3$, the pooled estimator performs uniformly better than its matrix-only counterpart. Its performance is also comparable to that of the matrix-only estimators for $r=1$ and $r=2$, and may even be slightly better when the fraction of missing entries is large. 

\begin{figure}[!ht]
    \centering
    \includegraphics[width=1\linewidth]{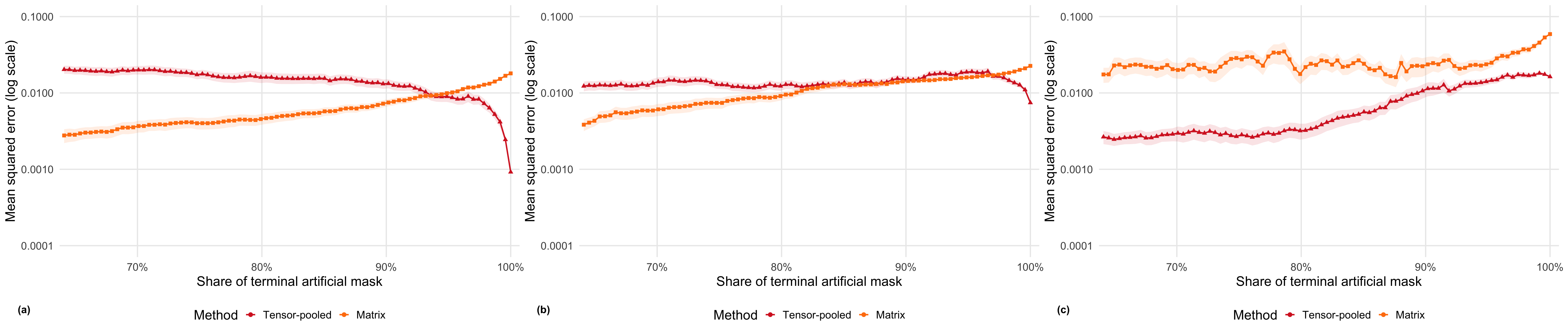}
    \captionsetup{
        font={scriptsize,stretch=0.8},
        skip=10pt
    }
    \caption{Castle masking experiment with additional missingness in the target layer $k=1$. The panels correspond to ranks $r=1,2,3$ used for the two methods; we again use $\tau=0.01$. The terminal mask retains $3$ fully observed rows and $4$ fully observed columns. The curves report the average squared error over $30$ replicates, with one standard error. The percentage on the $x$-axis represents the percentage of the terminal mask reached.}
    \label{fig:Psi0}
    \vspace{-15pt}
\end{figure}

Furthermore, to examine patterns in motor vehicle theft, we consider several linear aggregate quantities over the policy-on region. We order the rows of the first layer so that the adoption pattern has staircase form, with
$\Omega_{i,t,1}
=
\mathbbm 1\{(i,t)\in R_{a1}\times C_{b1}
\text{ for }a,b\text{ satisfying }a+b\leq o_1+1\}$
for some integer $o_1\geq 2$, and let
$\mathcal D_1:=\{(a,b):a+b>o_1+1\}$ denote the policy-on target blocks.
For $c\in\{0,1\}$ and
$h\in\{\textsc{Avg},\textsc{RowHet},\textsc{Local-}i_0,\textsc{Trend}\}$,
we consider weighted aggregates of the block-level bilinear forms $\Psi_c^{(h)}(1)
:=
\{W_h(1)\}^{-1}
\sum_{(a,b)\in\mathcal D_1}
c_{ab}^{(h)}
x_{a,h}^{\top}
\mathcal M_{\bullet,\bullet,1}^{(a,b)}(c) \,
y_{b,h}$. Formal definitions of
the query directions, weights, and normalizations are given in
Appendix~\ref{app:simulEstimands}. 
Informally, these are chosen so that $\Psi_c^{(\textsc{Avg})}(1)$ is the average potential outcome over all target
entries, $\Psi_c^{(\textsc{RowHet})}(1)$ is the corresponding signed contrast
across rows, $\Psi_c^{(\textsc{Local-}i_0)}(1)$ is the average potential
outcome for a specified row $i_0$ over its target periods, and
$\Psi_c^{(\textsc{Trend})}(1)$ captures the average slope of the row-averaged trajectory across blocks. We define the corresponding policy contrast
as $\Delta^{(h)}(1):=\Psi_1^{(h)}(1)-\Psi_0^{(h)}(1)$. For the $\textsc{Local-}i_0$ functional, we take $i_0$ to correspond to Florida, Montana, and Texas, while for $\textsc{RowHet}$ we set $\eta_i=+1$ for states that voted Republican in the 2000 presidential election and $\eta_i=-1$ for states that voted Democratic, so that this estimand measures a contrast between these two groups. We estimate $\Psi_0^{(h)}(1)$ using Algorithm~\ref{alg:qudraticStaggeredAggregate}, its matrix analogue, and the two CY plug-in estimators. Our procedure uses $r=3$ and $\tau=10^{-2}$, the matrix-only estimator uses $r=1$ and $\tau=10^{-2}$, while the debiased CY method uses $r=1$; these choices are motivated by the preceding masking experiment. Because ${\cal Y}(1)$ is fully observed over~$\mathcal D_1$, we estimate $\Psi_1^{(h)}(1)$ with a plug-in estimator. Finally, we estimate the policy contrast $\Delta^{(h)}(1)$ by subtracting the corresponding estimators.

\vspace{-5pt}
\begingroup
\setlength{\tabcolsep}{3pt}
\begin{table}[!ht]
\centering
\footnotesize
\captionsetup{
  font={scriptsize,stretch=0.8},
  skip=10pt
}
\caption{Estimates of $\Delta^{(h)}(1)$ using the pooled estimator of Algorithm~\ref{alg:qudraticStaggeredAggregate}, its matrix counterpart, and two plug-in approaches based on the nuclear-norm estimator of~\cite{choi2024matrix} and its debiased counterpart.}
\label{tab:robbery-delta}

\begin{tabular}{lcccccc}

\toprule
&
\textsc{Avg}
& \textsc{Local-Florida}
& \textsc{Local-Montana}
& \textsc{Local-Texas}
& \textsc{RowHet}
& \textsc{Trend}
\\

\midrule

$\widehat\Delta^{(h)}$
& 0.1080
& -0.2440
& 0.1430
& 0.1160
& 0.0992
& -0.0184
\\

$\widehat\Delta_{\mathrm{mat}}^{(h)}$
& 0.0746
& -0.1070
& 0.2170
& 0.1210
& 0.0764
& -0.0323
\\

$\widehat\Delta_{\mathrm{CY,nuc}}^{(h)}$
& 0.1900
& 0.0239
& 0.3020
& 0.2370
& 0.1790
& -0.0380
\\

$\widehat\Delta_{\mathrm{CY,deb}}^{(h)}$
& 0.0924
& -0.0832
& 0.2260
& 0.1380
& 0.0922
& -0.0336
\\

\bottomrule

\end{tabular}

\end{table}
\endgroup
\vspace{-15pt}

The estimates for $\Delta^{(h)}(1)$ in Table~\ref{tab:robbery-delta} are generally positive for the $\textsc{Avg}$, $\textsc{Local}$, and $\textsc{RowHet}$ functionals, although their magnitudes vary across methods and states. In particular, the $\textsc{Local-Montana}$ and $\textsc{Local-Texas}$ estimates are positive across all four estimators, while the $\textsc{Local-Florida}$ estimates are more heterogeneous. In contrast, the $\textsc{Trend}$ estimates are consistently negative but small, ranging from $-0.0380$ to $-0.0184$. Taken together, the results point to an increase in motor vehicle theft following adoption that appears to diminish over time; this is broadly consistent with \citet[][Section 4.1]{cheng2013castle}. The results also underscore the importance of considering functionals beyond simple averages, particularly when treatment responses may vary substantially across states.

\subsection{Real-data application: COVID-19 data}
\label{sec:real-data-oxcgrt-owid}

We construct a panel by merging policy-response indicators from the \href{https://github.com/OxCGRT/covid-policy-tracker}{Oxford COVID-19 Government Response Tracker} with epidemiological outcomes from the \href{https://github.com/owid/covid-19-data}{Our World in Data COVID-19 dataset}. We focus on stay-at-home requirements and contact tracing over the period from March 5, 2020, to April 5, 2020, and retain 18 countries, mostly European, where adoption of both policies was irreversible. This yields a tensor $\mathcal Y\in\mathbb R^{18\times32\times2}$, where the modes correspond to countries, days, and policy--outcome pairs. The first layer is associated with the 28-day delayed outcome \texttt{new\_deaths\_smoothed\_per\_million} and the stay-at-home policy, while the second is associated with the 28-day delayed outcome \texttt{new\_cases\_smoothed\_per\_million} and the contact-tracing policy. The 28-day delay allows time for policy effects to appear in reported outcomes. The pairing is also natural: stay-at-home requirements may affect downstream mortality, while contact tracing more directly targets contagion and reported cases. The outcome tensor is plotted in Fig.~\ref{fig:CovidY}. In each panel, blue entries indicate periods in which the policy is inactive, while red entries indicate periods in which it is active. From the perspective of the untreated potential outcomes, the blue entries are observed and the red entries are missing, with the resulting missingness presumably MNAR because policy adoption is more likely where disease burden is higher. 

We first conduct a masking experiment similar to that in Fig.~\ref{fig:Psi0ALL}, where the terminal mask retains $2$ rows and $3$ columns. The target layer is $k=1$, corresponding to new deaths. Fig.~\ref{fig:covid19} shows that, at ranks 1 and 2, the two CY methods perform much better than in the Castle data application, likely due to the less structured staggered missingness in this application. {Since the dataset contains only two fully observed rows, the matrix and CY methods cannot accommodate higher ranks, whereas Algorithm ~\ref{alg:staggered-wrapper} can. The plot on the right shows results for ranks $r=3$ and $r=4$; for these ranks, \eqref{ass:subBlockAuxilliary} fails for some target blocks, so the corresponding results are not covered by Corollary~\ref{corollary:StaggeredUB}. We nevertheless report them to examine the empirical behavior of the estimator beyond the assumptions required by our theory.} Our method, particularly for $r=4$, is substantially more stable as the percentage of missingness increases and achieves smaller errors than all competing methods, with the only exception being the matrix method with $r=1$ between $60\%$ and $90\%$ missingness.

\begin{figure}[!ht]
    \centering
    \includegraphics[width=1\linewidth]{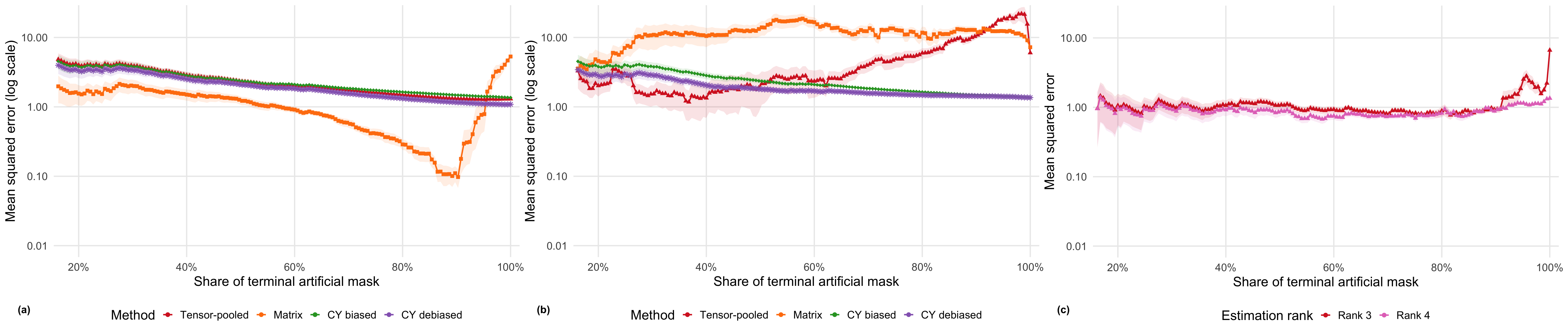}
    \captionsetup{
  font={scriptsize,stretch=0.8},
  skip=10pt
}
\caption{COVID-19 masking experiment with additional missingness in the target layer $k=1$. Results are shown for ranks $r=1,2$ (left and center) and $r=3,4$ (right), with $\tau=0.0001$. The terminal mask retains $2$ fully observed rows and $3$ fully observed columns. Curves show average squared error $\pm$ one standard error over $10$ replicates (left and center) and $30$ replicates (right). The $x$-axis gives the percentage of the terminal mask reached.}
    \label{fig:covid19}
    \vspace{-15pt}
\end{figure}

{Furthermore, to explore the relationship between stay-at-home policies and new deaths, Table~\ref{tab:covid} reports estimates of $\Delta^{(h)}(1)$ for $h\in\{\textsc{Avg},\textsc{Trend}\}$ across the feasible ranks for each method. Our pooled procedure is run with $r\in\{1,2,3,4\}$, while the remaining methods can accommodate only $r\in\{1,2\}$ in this application. As in the masking experiment, we also utilize $r=3,4$, which remain feasible for the SVD computations and can be competitive despite falling outside our current theory.}

{Across methods and ranks, the estimated trend contrast is generally close to zero,
although there is some sensitivity to the choice of rank. The average contrast shows
greater variation: our pooled estimator yields a negative estimate for $r\in \{2,3,4\}$,
whereas the comparison methods always yield positive estimates}.  {As a robustness check, since the cases and deaths layers may differ in scale and thereby affect estimation of the pooled subspaces, we also run our estimator after rescaling the cases layer to match the root mean square of the deaths layer over observed policy-off entries; the resulting estimates are similar to the original ones.} Under the maintained low-rank counterfactual model and the assumed interpretation of policy timing, {these negative estimates from our procedure are} consistent with a reduction in new deaths following stay-at-home requirements. The discrepancy in the estimated average contrast across methods may reflect the limited information in the target layer, where only Norway and Iceland are fully observed. By borrowing information on the latent unit and time factors from the second layer, our method may be better able to handle this setting, illustrating the potential benefit of exploiting information across layers.

\begingroup
\setlength{\tabcolsep}{3pt}
\begin{table}[!ht]
\centering
\footnotesize
\captionsetup{
font={scriptsize,stretch=0.8},
skip=10pt
}

\begin{minipage}{0.48\linewidth}
\centering
\begin{tabular}{cccccc}
\toprule
$r$
& $\widehat\Delta^{({\textsc{Avg}})}$ & $\widehat\Delta^{({\textsc{Avg}})}_{\textsc{std}}$
& $\widehat\Delta_{\mathrm{mat}}^{({\textsc{Avg}})}$
& $\widehat\Delta_{\mathrm{CY,nuc}}^{({\textsc{Avg}})}$
& $\widehat\Delta_{\mathrm{CY,deb}}^{({\textsc{Avg}})}$ \\
\midrule
1 & 1.525 &  1.302 & 0.428 & 1.485 & 1.338 \\
2 & -0.831 & -0.519 & 0.121 & 1.485 & 1.291 \\
3 & -1.833 &-1.993 & -- & -- & -- \\
4 & -1.196 & -0.750 & -- & -- & -- \\
\bottomrule
\end{tabular}
\end{minipage}
\hfill
\begin{minipage}{0.48\linewidth}
\centering
\begin{tabular}{ccccccc}
\toprule
$r$
& $\widehat\Delta^{({\textsc{Trend}})}$ & $\widehat\Delta^{({\textsc{Trend}})}_{\textsc{std}}$
& $\widehat\Delta_{\mathrm{mat}}^{({\textsc{Trend}})}$
& $\widehat\Delta_{\mathrm{CY,nuc}}^{({\textsc{Trend}})}$
& $\widehat\Delta_{\mathrm{CY,deb}}^{({\textsc{Trend}})}$ \\
\midrule
1 & -0.029 & -0.020 & -0.011 & -0.038 & -0.030 \\
2 & 0.154 & 0.028 & -0.049 & -0.038 & -0.031 \\
3 & 0.004 & 0.051 & --  & --  & --  \\
4 & -0.040 & -0.038 & --  & --  & --  \\
\bottomrule
\end{tabular}
\end{minipage}

\caption{Estimates of $\Delta^{(h)}(1)$ for
$h\in\{\textsc{Avg},\textsc{Trend}\}$ using the pooled estimator of
Algorithm~\ref{alg:qudraticStaggeredAggregate}, its matrix counterpart,
and the nuclear-norm and debiased estimators of~\cite{choi2024matrix},
across the feasible ranks for each method. We also report results from Algorithm~\ref{alg:qudraticStaggeredAggregate} after standardizing the cases layer relative to the deaths layer.}
\label{tab:covid}
\end{table}
\endgroup
\vspace{-30pt}

\section{Disclosure statement}\label{disclosure-statement}

The authors report there are no competing interests to declare.  \if1\anon
{The research of the second author was supported by European Research Council Starting Grant 101163546.}\fi

\phantomsection\label{supplementary-material}
\bigskip

\begin{center}

{\large\bf SUPPLEMENTARY MATERIAL}

\end{center}

Proofs and technical details are provided in the appendices. Code and dataset to replicate the paper’s figures and tables are included in the accompanying ZIP folder.

\begingroup
\small
\setlength{\bibsep}{0pt}
\setlength{\itemsep}{0pt}
\setlength{\parskip}{0pt}
\renewcommand{\baselinestretch}{1.15}\selectfont
\renewcommand{\newline}{\unskip\space}

\bibliography{bib.bib}

@article{Abadie2021,
  author = {Abadie, Alberto},
  title = {{U}sing {S}ynthetic {C}ontrols: {F}easibility, {D}ata {R}equirements, and {M}ethodological {A}spects},
  journal = {J. Econ. Lit.},
  year = {2021},
  volume = {59},
  number = {2},
  pages = {391--425}
}

@inproceedings{Agarwal2023CausalMatrixCompletion,
  author = {Anish Agarwal and Munther Dahleh and Devavrat Shah and Dennis Shen},
  title = {{C}ausal {M}atrix {C}ompletion},
  booktitle = {Proceedings of the 36th Annual Conference on Learning Theory},
  year = {2023},
  series = {Proc. Mach. Learn. Res.},
  volume = {195},
  pages = {3821--3826}
}

@misc{agarwal2026robustmatrixestimation,
  title         = {Robust {M}atrix {E}stimation with {S}ide {I}nformation},
  author        = {Agarwal, Anish and Choi, Jungjun and Yuan, Ming},
  year          = {2026},
  eprint        = {2603.24833},
  archivePrefix = {arXiv},
  primaryClass  = {stat.ME},
  doi           = {10.48550/arXiv.2603.24833},
  url           = {https://arxiv.org/abs/2603.24833}
}

@article{rubin1974estimating,
  title={Estimating Causal Effects of Treatments in Randomized and Nonrandomized Studies},
  author={Rubin, Donald B.},
  journal={J. Educ. Psychol.},
  volume={66},
  number={5},
  pages={688--701},
  year={1974},
  publisher={American Psychological Association}
}

@article{Agarwal2025SyntheticInterventions,
  author = {Anish Agarwal and Devavrat Shah and Dennis Shen},
  title = {{S}ynthetic {I}nterventions: {E}xtending {S}ynthetic {C}ontrols to {M}ultiple {T}reatments},
  journal = {Oper. Res.},
  year = {2025},
  volume = {74},
  number = {2},
  pages = {840--859}
}

@article{AtheyImbens2022DesignBasedDiD,
  title = {{D}esign-based analysis in {D}ifference-{I}n-{D}ifferences settings with staggered adoption},
  author = {Athey, Susan and Imbens, Guido W.},
  journal = {J. Econom.},
  volume = {226},
  number = {1},
  pages = {62--79},
  year = {2022}
}

@article{athey2021matrix,
  author = {Susan Athey and Mohsen Bayati and Nikolay Doudchenko and Guido Imbens and Khashayar Khosravi},
  title = {{M}atrix {C}ompletion {M}ethods for {C}ausal {P}anel {D}ata {M}odels},
  journal = {J. Amer. Statist. Assoc.},
  year = {2021},
  volume = {116},
  number = {536},
  pages = {1716--1730}
}

@misc{Auerbach2022TensorCompletionCausalInference,
  author = {Jonathan Auerbach and Martin Slawski and Shixue Zhang},
  title = {{T}ensor {C}ompletion for {C}ausal {I}nference with {M}ultivariate {L}ongitudinal {D}ata: {A} {R}eevaluation of {C}{O}{V}{I}{D}-19 {M}andates},
  year = {2022},
  eprint = {2203.04689},
  archivePrefix = {arXiv},
  url = {https://arxiv.org/abs/2203.04689}
}

@misc{BaharavNicolIrizarryMa2025StackedSVD,
  title = {{S}tacked {S}{V}{D} or {S}{V}{D} stacked? {A} {R}andom {M}atrix {T}heory perspective on data integration},
  author = {Baharav, Tavor Z. and Nicol, Phillip B. and Irizarry, Rafael A. and Ma, Rong},
  year = {2025},
  eprint        = {2507.22170},
  archivePrefix = {arXiv},
  url           = {https://arxiv.org/abs/2507.22170}
}

@article{BaiNg2021MatrixCompletionCounterfactuals,
  author = {Bai, Jushan and Ng, Serena},
  title = {{M}atrix {C}ompletion, {C}ounterfactuals, and {F}actor {A}nalysis of {M}issing {D}ata},
  journal = {J. Amer. Statist. Assoc.},
  year = {2021},
  volume = {116},
  number = {536},
  pages = {1746--1763}
}

@article{CahanBaiNg2023FactorBasedImputation,
  author = {Cahan, Ercument and Bai, Jushan and Ng, Serena},
  title = {{F}actor-{B}ased {I}mputation of {M}issing {V}alues and {C}ovariances in {P}anel {D}ata of {L}arge {D}imensions},
  journal = {J. Econom.},
  year = {2023},
  volume = {233},
  number = {1},
  pages = {113--131}
}

@article{candes_recht_2009,
  title = {{E}xact {M}atrix {C}ompletion via {C}onvex {O}ptimization},
  volume = {9},
  number = {6},
  journal = {Found. Comput. Math.},
  publisher = {Springer},
  author = {Candès, Emmanuel J. and Recht, Benjamin},
  year = {2009},
  pages = {717–772}
}

@article{spectralMethods,
  author = {Chen, Yuxin and Chi, Yuejie and Fan, Jianqing and Ma, Cong},
  title = {{S}pectral {M}ethods for {D}ata {S}cience: {A} {S}tatistical {P}erspective},
  year = {2021},
  publisher = {Now Publishers Inc.},
  address = {Hanover, MA, USA},
  volume = {14},
  number = {5},
  journal = {Found. Trends Mach. Learn.},
  pages = {566–806}
}

@article{cheng2013castle,
  title = {{D}oes {S}trengthening {S}elf-{D}efense {L}aw {D}eter {C}rime or {E}scalate {V}iolence? {E}vidence from {E}xpansions to {C}astle {D}octrine},
  author = {Cheng, Cheng and Hoekstra, Mark},
  journal = {J. Human Resources},
  volume = {48},
  number = {3},
  pages = {821--854},
  year = {2013}
}

@article{ChiLuChen2019NonconvexOptimization,
  author = {Chi, Yuejie and Lu, Yue M. and Chen, Yuxin},
  title = {{N}onconvex {O}ptimization {M}eets {L}ow-{R}ank {M}atrix {F}actorization: {A}n {O}verview},
  journal = {IEEE Trans. Signal Process.},
  year = {2019},
  volume = {67},
  number = {20},
  pages = {5239--5269}
}

@book{Chikuse2003,
  author = {Chikuse, Yasuko},
  title = {{S}tatistics on {S}pecial {M}anifolds},
  series = {Lecture Notes in Statistics},
  volume = {174},
  publisher = {Springer},
  address = {New York},
  year = {2003}
}

@article{choi2024matrix,
  author = {Jungjun Choi and Ming Yuan},
  title = {{M}atrix {C}ompletion {W}hen {M}issing {I}s {N}ot at {R}andom and {I}ts {A}pplications in {C}ausal {P}anel {D}ata {M}odels},
  journal = {J. Amer. Statist. Assoc.},
  year = {2026}, 
  note = {To appear}
}

@incollection{DavidsonSzarek2001,
  author = {Davidson, Kenneth R. and Szarek, Stanislaw J.},
  title = {{L}ocal {O}perator {T}heory, {R}andom {M}atrices and {B}anach {S}paces},
  booktitle = {Handbook of the Geometry of Banach Spaces},
  volume = {1},
  publisher = {Elsevier},
  year = {2001},
  pages = {317--366}
}

@misc{Gao2025CausalInferenceSequentialTreatmentsTensor,
  author = {Chenyin Gao and Han Chen and Anru R. Zhang and Shu Yang},
  title = {{C}ausal {I}nference on {S}equential {T}reatments via {T}ensor {C}ompletion},
  year = {2025},
  eprint = {2511.15866},
  archivePrefix = {arXiv},
  url = {https://arxiv.org/abs/2511.15866}
}

@article{GillVanTrees1995,
  author = {Richard D. Gill and Boris Y. Levit},
  title = {{A}pplications of the van {T}rees inequality: a {B}ayesian {C}ramér-{R}ao bound},
  volume = {1},
  journal = {Bernoulli},
  number = {1-2},
  publisher = {Bernoulli Society for Mathematical Statistics and Probability},
  pages = {59 -- 79},
  year = {1995}
}

@misc{GrossNesme2010NoteSamplingCoRR,
  title         = {{N}ote on sampling without replacing from a finite collection of matrices},
  author        = {Gross, David and Nesme, Vincent},
  year          = {2010},
  eprint        = {1001.2738},
  archivePrefix = {arXiv},
  url           = {https://arxiv.org/abs/1001.2738}
}

@book{ImbensRubin2015,
  author = {Imbens, Guido W. and Rubin, Donald B.},
  title = {{C}ausal {I}nference for {S}tatistics, {S}ocial, and {B}iomedical {S}ciences: {A}n {I}ntroduction},
  year = {2015},
  publisher = {Cambridge University Press},
  address = {New York}
}

@article{Keshavan_Montanari_Oh,
  AUTHOR = {Keshavan, R. H. and Montanari, A. and Oh,
              S.},
  TITLE = {{M}atrix completion from a few entries},
  JOURNAL = {IEEE Trans. Inform. Theory},
  FJOURNAL = {Institute of Electrical and Electronics Engineers.
              Transactions on Information Theory},
  VOLUME = {56},
  YEAR = {2010},
  NUMBER = {6},
  PAGES = {2980--2998},
  CODEN = {IETTAW}
}

@article{klopp_general,
  AUTHOR = {Klopp, O.},
  TITLE = {{N}oisy low-rank matrix completion with general sampling
              distribution},
  JOURNAL = {Bernoulli},
  FJOURNAL = {Bernoulli. Official Journal of the Bernoulli Society for
              Mathematical Statistics and Probability},
  VOLUME = {20},
  YEAR = {2014},
  NUMBER = {1},
  PAGES = {282--303}
}

@article{Kolda2009Tensor,
  author = {Tamara G. Kolda and Brett W. Bader},
  title = {{T}ensor {D}ecompositions and {A}pplications},
  journal = {SIAM Review},
  volume = {51},
  number = {3},
  pages = {455--500},
  year = {2009}
}

@article{Koltchinskii_Lounici_Tsybakov,
  AUTHOR = {Koltchinskii, Vladimir and Lounici, Karim and Tsybakov,
              Alexandre B.},
  TITLE = {{N}uclear-norm penalization and optimal rates for noisy low-rank
              matrix completion},
  JOURNAL = {Ann. Stat.},
  FJOURNAL = {Ann. Stat.},
  VOLUME = {39},
  YEAR = {2011},
  NUMBER = {5},
  PAGES = {2302--2329},
  CODEN = {ASTSC7}
}

@misc{Mandal2019WeightedTensorCompletion,
  author = {Debmalya Mandal and David Parkes},
  title = {{W}eighted {T}ensor {C}ompletion for {T}ime-{S}eries {C}ausal {I}nference},
  year = {2019},
  eprint = {1902.04646},
  archivePrefix = {arXiv},
  url = {https://arxiv.org/abs/1902.04646}
}

@article{Mezzadri2007,
  author = {Mezzadri, Francesco},
  title = {{H}ow to {G}enerate {R}andom {M}atrices from the {C}lassical {C}ompact {G}roups},
  journal = {Notices Amer. Math. Soc.},
  year = {2007},
  volume = {54},
  number = {5},
  pages = {592--604}
}

@article{Negahban_Wainwright,
  author = {Sahand N. Negahban and Martin J. Wainwright},
  title = {{R}estricted strong convexity and weighted matrix completion: {O}ptimal bounds with noise},
  journal = {J. Mach. Learn. Res.},
  year = {2012},
  volume = {13},
  pages = {1665--1697}
}

@article{Stewart1980,
  author = {Stewart, G. W.},
  title = {{T}he {E}fficient {G}eneration of {R}andom {O}rthogonal {M}atrices with an {A}pplication to {C}ondition {E}stimators},
  journal = {SIAM J. Numer. Anal.},
  year = {1980},
  volume = {17},
  number = {3},
  pages = {403--409}
}

@article{tropp2012userfriendly,
  title = {{U}ser-{F}riendly {T}ail {B}ounds for {S}ums of {R}andom {M}atrices},
  author = {Tropp, Joel A.},
  journal = {Found. Comput. Math.},
  volume = {12},
  number = {4},
  pages = {389--434},
  year = {2012},
  publisher = {Springer}
}

@book{tsybakov2009nonparametric,
  title = {{I}ntroduction to {N}onparametric {E}stimation},
  author = {Tsybakov, Alexandre B},
  year = {2009},
  publisher = {Springer}
}

@book{vershynin19hdp,
  author = {Vershynin, Roman},
  publisher = {Cambridge University Press},
  title = {{H}igh-{D}imensional {P}robability},
  year = {2019}
}

@book{wainwright2019high,
  title = {{H}igh-dimensional {S}tatistics: {A} {N}on-asymptotic {V}iewpoint},
  author = {Wainwright, Martin J},
  year = {2019},
  publisher = {Cambridge University Press}
}

@misc{xia2025inferencestaggeredadoptioncase,
  author = {Eric Xia and Yuling Yan and Martin J. Wainwright},
  title = {{I}nference under {S}taggered {A}doption: {C}ase {S}tudy of the {A}ffordable {C}are {A}ct},
  year = {2024},
  eprint = {2412.09482},
  archivePrefix = {arXiv},
  url = {https://arxiv.org/abs/2412.09482}
}

@misc{yan2024entrywise,
  author = {Yuling Yan and Martin J. Wainwright},
  title = {{E}ntrywise {I}nference for {M}issing {P}anel {D}ata: {A} {S}imple and {I}nstance-{O}ptimal {A}pproach},
  year = {2024},
  eprint = {2401.13665},
  archivePrefix = {arXiv},
  url = {https://arxiv.org/abs/2401.13665}
}

@misc{agterberg2026statistically,
  title         = {Statistically and {C}omputationally {O}ptimal {E}stimation and {I}nference of {C}ommon {S}ubspaces},
  author        = {Agterberg, Joshua},
  year          = {2026},
  eprint = {2606.06483},
  archivePrefix = {arXiv},
  url = {https://arxiv.org/abs/2606.06483}
}

@article{semedo2019cortical,
  title   = {Cortical {A}reas {I}nteract through a {C}ommunication {S}ubspace},
  author  = {Semedo, Jo{\~a}o D. and Zandvakili, Amin and Machens, Christian K. and Yu, Byron M. and Kohn, Adam},
  journal = {Neuron},
  volume  = {102},
  number  = {1},
  pages   = {249--259.e4},
  year    = {2019}
}

@misc{liu2026representation,
  title   = {Representation {L}earning with {B}lockwise {M}issingness and {S}ignal {H}eterogeneity},
  author  = {Liu, Ziqi and Tian, Ye and Tang, Weijing},
  year = {2026},
  eprint = {2602.11511},
  archivePrefix = {arXiv},
  url = {https://arxiv.org/abs/2602.11511}
}

@article{callaway2021difference,
  title   = {Difference-in-{D}ifferences with {M}ultiple {T}ime {P}eriods},
  author  = {Callaway, Brantly and Sant'Anna, Pedro H. C.},
  journal = {J. Econom.},
  volume  = {225},
  number  = {2},
  pages   = {200--230},
  year    = {2021}
}

@article{ma2026optimal,
  title   = {Optimal Estimation of Shared Singular Subspaces Across Multiple Noisy Matrices},
  author  = {Ma, Zhengchi and Ma, Rong},
  journal = {IEEE Trans. Inform. Theory},
  volume  = {72},
  number  = {5},
  pages   = {3277--3300},
  year    = {2026}
}

@article{arroyo2021inference,
  title   = {Inference for {M}ultiple {H}eterogeneous {N}etworks with a {C}ommon {I}nvariant {S}ubspace},
  author  = {Arroyo, Jes{\'u}s and Athreya, Avanti and Cape, Joshua and Chen, Guodong and Priebe, Carey E. and Vogelstein, Joshua T.},
  journal = {J. Mach. Learn. Res.},
  volume  = {22},
  number  = {142},
  pages   = {1--49},
  year    = {2021}
}
\endgroup

\clearpage

\section*{Supplementary Material for
``Estimation of Linear Functionals in
Multilayer Panels under Staggered Adoption''}

\addcontentsline{toc}{section}{Supplementary Material}

\setcounter{section}{0}
\setcounter{subsection}{0}
\setcounter{figure}{0}
\setcounter{table}{0}
\setcounter{equation}{0}
\setcounter{thm}{0}
\setcounter{algorithm}{0}

\renewcommand{\thesection}{S\arabic{section}}
\renewcommand{\thesubsection}{S\arabic{section}.\arabic{subsection}}
\renewcommand{\thesubsubsection}
  {S\arabic{section}.\arabic{subsection}.\arabic{subsubsection}}

\renewcommand{\thefigure}{S\arabic{figure}}
\renewcommand{\thetable}{S\arabic{table}}
\renewcommand{\theequation}{S\arabic{equation}}
\renewcommand{\thethm}{S\arabic{thm}}
\renewcommand{\thealgorithm}{S\arabic{algorithm}}

\renewcommand{\theHsection}{supp.\arabic{section}}
\renewcommand{\theHsubsection}{supp.\arabic{section}.\arabic{subsection}}
\renewcommand{\theHsubsubsection}
  {supp.\arabic{section}.\arabic{subsection}.\arabic{subsubsection}}

\renewcommand{\theHfigure}{supp.figure.\arabic{figure}}
\renewcommand{\theHtable}{supp.table.\arabic{table}}
\renewcommand{\theHequation}{supp.equation.\arabic{equation}}
\renewcommand{\theHthm}{supp.thm.\arabic{thm}}

\providecommand{\theHalgorithm}{}
\renewcommand{\theHalgorithm}{supp.algorithm.\arabic{algorithm}}

\spacingset{1.8}
\allowdisplaybreaks

{These Appendices are organized as follows. Appendix~\ref{sec:proofs} contains the proofs of all results stated in the main text. Appendix~\ref{sec:simulExtra} provides additional simulation studies and introduces some aggregate estimands of interest, together with a detailed discussion of procedures for estimating them. Appendix~\ref{appendix:compCost} discusses the computational cost of Algorithm~\ref{alg:bilinear4block}. Appendix~\ref{appendix:tecnicalLemmas} develops key matrix-denoising results and the technical tools required to analyze Algorithm~\ref{alg:bilinear4block} in the four-block setting. Appendix~\ref{appendix:additionalResults} contains additional results on identifiability and tensor structure that are only briefly mentioned in the main text. Finally, Appendix~\ref{appendix:AuxilliaryRes} collects standard probabilistic results and tail bounds, together with background material on the Stiefel manifold and Haar measure.}

\section{Proofs}\label{sec:proofs}
\subsection{Proofs for Section~\ref{sec:upperBound}}

\begin{proof}[Proof of Theorem~\ref{thm:4BlockUBLinear}]
The proof of this result relies on the matrix denoising theory developed in Appendix~\ref{appendix:tecnicalLemmas}. Start by writing $\hat \mu_{xy}^{(k)} - \mu_{xy}^{(k)}
    =
    Z_{xy}^{(1)} + Z_{xy}^{(2)} + Z_{xy}^{(3)} + Z_{xy}^{(4)} + \Delta_{xy} =:
    Z_{xy} + \Delta_{xy}$,
where 
\begin{align*}
    Z_{xy}^{(1)} 
    &:= x^\top 
    (E_{\rm left}^{\rm p})_{\mathcal I_k, \bullet}
    W_{\rm left}
    (W_{\rm left}^\top W_{\rm left})^{-1}
    \mathcal{C}_{\bullet, \bullet, k}
    (W_{\rm up}^\top W_{\rm up})^{-1}
    W_{\rm up}^\top
    (E_{\rm up}^{\rm p})_{\bullet, \mathcal J_k}
    y, \\
    Z_{xy}^{(2)} 
    &:= x^\top 
    U_{2k}
    \mathcal{C}_{\bullet, \bullet, k}
    (W_{\rm up}^\top W_{\rm up})^{-1}
    W_{\rm up}^\top
    (E_{\rm up}^{\rm p})_{\bullet, \mathcal J_k}
    y, \\
    Z_{xy}^{(3)} 
    &:= x^\top 
    U_{2k}
    (U_{1k}^\top U_{1k})^{-1}
    U_{1k}^\top
    (E_{\rm up}^{\rm p})_{\mathcal I_k^\mathrm{up}, \bullet}
    V
    V_{2k}^\top
    y, \\
    Z_{xy}^{(4)} 
    &:= x^\top 
    (E_\mathrm{left}^\mathrm{p})_{\mathcal I_k, \bullet}
    W_\mathrm{left}
    (W_\mathrm{left}^\top W_\mathrm{left})^{-1}
    \mathcal C_{\bullet, \bullet, k}
    V_{2k}^\top
    y.
\end{align*}
Under the assumptions of Theorem~\ref{thm:4BlockUBLinear}, Lemma~\ref{lemma:VarDominatesLemma1} gives $\mathbb{E} [Z_{xy}^2] \leq c_1 \, \Upsilon_{xy}$ for a sufficiently large constant $c_1 \equiv c_1(c_\ell, c_u, c_0,  c_{\mathrm{dim}}, c_{\mathrm{blk}}, \kappa, \nu_x, \nu_y) >~0$. Moreover, the same result ensures that there exists an event $\mathcal G_1$ with $\mathbb{P}(\mathcal G_1) \geq 1 - \mathcal{O}(p_N^{-10} + p_T^{-10})$ such that $\Delta_{xy}^2 \leq c_1 \, \Upsilon_{xy}$ under $\mathcal G_1$. 

Possibly enlarging the constant $c_1$ and allowing it to change from line to line, and writing $\hat \mu \equiv \hat \mu_{xy}^{(k)}, \, \mu \equiv~\mu_{xy}^{(k)}$ to simplify the notation, we then decompose the mean squared error as 
\begin{align*}
\mathbb{E}\left[
    \bigl\{\hat{\mu} - \mu\bigr\}^2
\right]
&=
\mathbb{E}\left[
    \bigl\{\hat{\mu} - \mu\bigr\}^2
    \, \mathbbm{1}_{\mathcal{G}_1}
\right]
+
\mathbb{E}\left[
    \bigl\{\hat{\mu} - \mu\bigr\}^2
    \, \mathbbm{1}_{\mathcal{G}_1^{\complement}}
\right] \\
&=
\mathbb{E}\left[
    \bigl\{Z_{xy}+\Delta_{xy}\bigr\}^2
    \mathbbm{1}_{\mathcal{G}_1}
\right]
+
\mathbb{E}\left[
    \bigl\{\hat{\mu} - \mu\bigr\}^2
    \mathbbm{1}_{\mathcal{G}_1^{\complement}}
\right] \\
&\le
2\mathbb{E}\left[
    Z_{xy}^2
\right]
+
2\mathbb{E}\left[
    \Delta_{xy}^2
    \mathbbm{1}_{\mathcal{G}_1}
\right]
+
2\,\mathbb{E}\left[
    \hat \mu^2
    \mathbbm{1}_{\mathcal{G}_1^{\complement}}
\right]
+
2\mu^2 \,
\mathbb{P}(\mathcal{G}_1^{\complement}) \\
&\le
2\mathbb{E} [Z_{xy}^2]
+
2 c_1 \Upsilon_{xy}
+
2\,\mathbb{E}\left[
    \hat \mu^2
    \mathbbm{1}_{\mathcal{G}_1^{\complement}}
\right]
+
2\mu^2 \,
\mathbb{P}(\mathcal{G}_1^{\complement}) \\
&\le
c_1 \, \Upsilon_{xy}
+
2\,\mathbb{E}\left[
    \hat \mu^2
    \mathbbm{1}_{\mathcal{G}_1^{\complement}}
\right]
+
c_1 \, \mu^2
\bigl(p_N^{-10} + p_T^{-10}\bigr),
\end{align*}
where the previous decomposition and the bound on the second moment of $Z_{xy}$ from Lemma~\ref{lemma:VarDominatesLemma1} are used to control the contribution on the good event $\mathcal G_1$, while the probability bound for $\mathcal G_1^\complement$ is used in the last inequality. We next derive separate bounds for each of the remaining two terms. For the third term, we have 
$|\mu| \leq \|\mathcal{C}_{\bullet, \bullet, k}\|_\mathrm{op}\|U_{2k}^\top x\|_2 \, \|V_{2k}^\top y\|_2 \leq \gamma_\mathrm{max}$, which yields \[
\mu^2 \, (p_N^{-10} + p_T^{-10}) \leq \gamma_\mathrm{max}^2 \, (p_N^{-10} + p_T^{-10}) \leq c_1  \tau^{-1} \,  \gamma_\mathrm{max}^2 \, (p_N^{-10} + p_T^{-10}) \,\frac{N_{1k}}{N}.
\]
The last inequality follows from $\tau \leq c_\ell N_{1k} / \, (2N)$. For the second term, we start by noticing that $\bigl\|\hat H_{k,\tau}^{\mathrm{inv}} \, \widehat U_{1k}^{\top}\bigr\|_{\mathrm{op}}^2
=
\bigl\|\hat H_{k,\tau}^{\mathrm{inv}} \hat H_k \hat H_{k,\tau}^{\mathrm{inv}}\bigr\|_{\mathrm{op}}
=
\max_{i\in[r]} \lambda_i/  (\lambda_i\vee\tau)^2
\le \tau^{-1}$. This allows showing 
\begin{align*}
\bigl|\hat\mu\bigr|
&=
\left|
x^{\top}\hat U_{2k}\hat H_{k,\tau}^{\mathrm{inv}}
\hat U_{1k}^{\top}
\hat U_{\mathrm{up}}^{(k)}\hat\Sigma_{\mathrm{up}}
\hat V_{2k}^{\top}y
\right| \le
\bigl\|\hat U_{2k}^{\top}x\bigr\|_2
\bigl\|\hat H_{k,\tau}^{\mathrm{inv}}\hat U_{1k}^{\top}\bigr\|_{\mathrm{op}}
\bigl\|\hat U_{\mathrm{up}}^{(k)}\hat\Sigma_{\mathrm{up}}\bigr\|_{\mathrm{op}}
\bigl\|\hat V_{2k}^{\top}y\bigr\|_2 \\
&\le
\tau^{-1/2} \,
\|(Y_{\mathrm{up}}^{\mathrm{p}})_{\mathcal I_k^{\mathrm{up}}, \bullet}
    \hat V_{\mathrm{up}}
\|_{\mathrm{op}} \le
\tau^{-1/2}
\|(Y_{\mathrm{up}}^{\mathrm{p}})_{\mathcal I_k^{\mathrm{up}}, \bullet}\|_{\mathrm{op}}
=
\tau^{-1/2}
\|Y_{\mathrm{up}}^{(k)}\|_{\mathrm{op}} \\ 
&\le
\tau^{-1/2}
    \|M_{\mathrm{up}}^{(k)}\|_{\mathrm{op}}
    +
    \tau^{-1/2} \|E_{\mathrm{up}}^{(k)}\|_{\mathrm{op}} =
\tau^{-1/2}
\|U_{1k}\mathcal C_{\bullet,\bullet,k}V^{\top}\|_{\mathrm{op}}
    +
    \tau^{-1/2} \|E_{\mathrm{up}}^{(k)}\|_{\mathrm{op}} \\
    &\le
\tau^{-1/2}
    c_u^{1/2}\gamma_{\mathrm{max}}\sqrt{N_{1k}/N}
    +
    \tau^{-1/2}\bigl\|E_{\mathrm{up}}^{(k)}\bigr\|_{\mathrm{op}},
\end{align*}
where the last inequality follows from  \eqref{assump:subblock-conditioning}. This, together with an application of the Cauchy--Schwarz inequality and Lemma~\ref{lemma:ExpOpNormE} with $p = 4$, gives
\begin{align*}
\mathbb{E}\left[
    \hat\mu^2
    \mathbbm{1}_{\mathcal G_1^{\complement}}
\right]
&\le
2\tau^{-1} c_u \gamma_{\mathrm{max}}^2
\frac{N_{1k}}{N}
\mathbb{P}\bigl(\mathcal G_1^{\complement}\bigr)
+
2\tau^{-1}
\mathbb{E}\left[
    \bigl\|E_{\mathrm{up}}^{(k)}\bigr\|_{\mathrm{op}}^2
    \mathbbm{1}_{\mathcal G_1^{\complement}}
\right] \\
&\le
c_1 \tau^{-1} \gamma_{\mathrm{max}}^2
\frac{N_{1k}}{N}
\bigl(p_N^{-10}+p_T^{-10}\bigr)
+
c_1 \tau^{-1}
\bigl(p_N^{-5}+p_T^{-5}\bigr)
\, \mathbb{E}^{1/2}\left[
    \bigl\|E_{\mathrm{up}}^{(k)}\bigr\|_{\mathrm{op}}^4
\right] \\
&\le
c_1\left\{
    \tau^{-1}\gamma_{\mathrm{max}}^2
    \frac{N_{1k}}{N}
    \bigl(p_N^{-10}+p_T^{-10}\bigr)
    +
    \tau^{-1}
    \bigl(p_N^{-5}+p_T^{-5}\bigr)
    \,\sigma^2\bigl(N_{1k}+T\bigr)
\right\}.
\end{align*}
Combining the previous bounds and using $\tau^{-1} \,  \gamma_\mathrm{max}^2 \, (p_N^{-10} + p_T^{-10})  \, \frac{N_{1k}}{N}\,\, + \,\, \tau^{-1}  (p_N^{-5} + p_T^{-5}) \, (\sigma^2 N_{1k} + \sigma^2 T) \leq c_0 \, \Upsilon_{xy}$ concludes the proof.
\end{proof}

\subsection{Proofs for Section~\ref{sec:MinimaxLowerBounds}}
The proof of Theorem~\ref{thm:CombinedLB} follows by combining the next two results. The first addresses the large-$K$ regime.

\begin{thm}\label{thm:localTermLB}
Fix $k\in[K]$ and unit vectors $ x \in \mathcal S^{N_{2k}-1},  y \in \mathcal S^{T_{2k}-1}$. Let $\mathcal M_0=\mathcal C_0\times_1  U_0\times_2  V_0\times_3  I_K
    \in \mathcal F(c_\ell,c_u)$.
Assume that the $k$-th core matrix is separated from the boundary of the
admissible singular-value interval, in the sense that $\delta_{\gamma}:=
    \min\{\sigma_{\min}((\mathcal C_{0})_{\bullet, \bullet, k})-\gamma_{\min},
    \gamma_{\max}-\sigma_{\max}((\mathcal C_{0})_{\bullet, \bullet, k})\}>0$.
For $\varsigma > 0$ define $\mathcal{F}_\mathrm{loc}(\mathcal{M}_0, \varsigma) := \{\mathcal M\in\mathcal F(c_\ell,c_u):
\|\mathcal M-\mathcal M_0\|_F\le \varsigma\}$. There exists $c \equiv c(c_u)>0$ such that
\begin{align*}
\inf_{\phi}
\sup_{\mathcal M\in\mathcal{F}_\mathrm{loc}(\mathcal{M}_0, \varsigma)}
\mathbb E_{\mathcal M} &
\left[
\{\phi(Z_\Omega)-\mu_{ x  y}^{(k)}(\mathcal M)\}^2
\right] \\
& \qquad \ge
c\,\min\left\{ \sigma^2 \min\left(\frac{N}{N_{1k}},
    \frac{T}{T_{1k}} \right), \varsigma^2, \delta_{\gamma}^2\right\} \,  \| U_{0, 2k}^\top \,  x\|_2^{2} \, \| V_{0, 2k}^\top \,  y\|_2^{2},
\end{align*}
where the infimum is over all Borel-measurable functions
$\phi$ of the observed entries $Z_\Omega$.
\end{thm}

\begin{proof}[Proof of Theorem~\ref{thm:localTermLB}]
We prove the result by reducing the problem to a one-dimensional parametric submodel. To this end, we assume that $U_{0,2k}^\top x \neq 0$ and $V_{0,2k}^\top y \neq 0$; if either of these conditions fails, the desired lower bound is trivially satisfied. Let
\[
    F_k:=(U_{0,2k}^\top \, x)( V_{0,2k}^\top \, y)^\top,
    \qquad
    G_k:=\frac{F_k}{\|F_k\|_F^2}.
\]
Write $C_{0,j}:=(\mathcal{C}_0)_{\bullet,\bullet,j}$ to simplify the notation. For $\theta\in\mathbb R$, define a perturbed core tensor $\mathcal C(\theta)$ by keeping
all slices except the $k$-th one fixed, and setting $C_k(\theta):=C_{0,k}+\theta \, G_k$. We then set $\mathcal M(\theta):=\mathcal C(\theta)\times_1 U_0\times_2 V_0\times_3 I_K$. We next verify the range of $\theta$ for which the path remains in the local parameter space. Since $U_0$ and $V_0$ have orthonormal columns, we have 
\[
    \|C_k(\theta)-C_{0,k}\|_{\mathrm{op}}
    \le
    \|C_k(\theta)-C_{0,k}\|_F
    =
    \frac{|\theta|}{\|F_k\|_F},\]
    \[\|\mathcal M(\theta)-\mathcal M_0\|_F
    =
    \|U_0\{C_k(\theta)-C_{0,k}\}V_0^\top\|_F
    =
    \frac{|\theta|}{\|F_k\|_F}.
\]
We thus deduce that
$\{\mathcal M(\theta):|\theta|\le h\}$ is contained in $\left\{
    \mathcal M\in\mathcal F(c_\ell,c_u):
    \|\mathcal M-\mathcal M_0\|_F\le \varsigma
    \right\}$ whenever $h
    \leq
    \|F_k\|_F \, \min (\varsigma, \delta_{\gamma})$. In particular, the singular values of $C_k(\theta)$ remain in
$[\gamma_{\min},\gamma_{\max}]$ by the operator-norm bound above and Weyl's inequality (Lemma~\ref{lemma:Weyl}). Assumption  \eqref{assump:subblock-conditioning} continues to hold along the path because $U_0$ and $V_0$ are fixed.

We next compute the induced change in the target functional. Since the bottom-right block of the $k$-th slice is
$U_{0,2k}C_k(\theta)V_{0,2k}^\top$, we have
\begin{align*}
    \mu(\theta):= \mu_{xy}^{(k)}(\mathcal M(\theta))
    &=
    x^\top \, 
        U_{0,2k}C_k(\theta)V_{0,2k}^\top \, y  =
    \mu_{xy}^{(k)}(\mathcal M_0)
    +\theta \, 
        x^\top U_{0,2k} \, G_k V_{0,2k}^\top \, y =
    \mu_{xy}^{(k)}(\mathcal M_0)+\theta .
\end{align*}
Hence estimating $\mu_{xy}^{(k)}(\mathcal M(\theta))$ along this path is equivalent to
estimating the scalar parameter~$\theta$, up to the known additive constant
$\mu_{xy}^{(k)}(\mathcal M_0)$. It is also useful to compute the Fisher information for $\theta$. Using the Gaussianity of the error and the four-block structure of $\Omega_{\bullet, \bullet, k}$, the Fisher
information is constant and equal~to 
\begin{align*}
    \mathfrak I_\theta & 
    =
    \sigma^{-2}
    \left\|
        P_{\Omega_{\bullet, \bullet, k}}(U_0G_kV_0^\top)
    \right\|_F^2 \\
    &=
    \sigma^{-2}
    \left\{
    \|U_{0,1k}G_kV_{0,1k}^\top\|_F^2
    +
    \|U_{0,1k}G_kV_{0,2k}^\top\|_F^2
    +
    \|U_{0,2k}G_kV_{0,1k}^\top\|_F^2
    \right\} \\
    &=
    \sigma^{-2}
    \left\{
    \|U_{0,1k}G_k\|_F^2
    +
    \|G_kV_{0,1k}^\top\|_F^2
    -
    \|U_{0,1k}G_kV_{0,1k}^\top\|_F^2
    \right\} \\
    & = \sigma^{-2} \, \|F_k\|_F^{-4}
    \left\{
    \|U_{0,1k}F_k\|_F^2
    +
    \|F_kV_{0,1k}^\top\|_F^2
    -
    \|U_{0,1k}F_kV_{0,1k}^\top\|_F^2
    \right\} \\
    & \leq \sigma^{-2} \, \|F_k\|_F^{-4}
    \left\{
    \|U_{0,1k}F_k\|_F^2
    +
    \|F_kV_{0,1k}^\top\|_F^2
    \right\} \\
    & \leq \sigma^{-2} \, c_u \, \|F_k\|_F^{-2} \, \left\{\frac{N_{1k}}{N}
    +
    \frac{T_{1k}}{T} \right\},
\end{align*}
where the last inequality follows from Assumption  \eqref{assump:subblock-conditioning}. Note that only the $k$-th slice contributes to $\mathfrak I_\theta$, since the submodel is fixed along all other slices. 

We now apply the van Trees inequality \citep[][Equation~3]{GillVanTrees1995} on the interval $[-h,h]$. Let~$p$ be an arbitrary 
absolutely continuous prior density on $[-h,h]$ satisfying
$p(-h)=p(h)=0$. For any estimator~$\phi(Z_\Omega)$, define $\widehat \theta
    :=
    \phi(Z_\Omega)-\mu_{xy}^{(k)}(\mathcal M_0)$. Since $\mu(\theta)=\mu_{xy}^{(k)}(\mathcal M_0)+\theta$, the risk for
estimating the functional is exactly the risk for estimating $\theta$ along the
submodel. Writing $J(p):=\int_{-h}^h\frac{\{p'(\theta)\}^2}{p(\theta)}\,d\theta$ and choosing $h
    =
    \|F_k\|_F \, \min (\varsigma, \delta_{\gamma})$, the van Trees inequality yields
\begin{align*}
    \sup_{|\theta|\le h}
    \mathbb E_\theta
    &\left[
        \{\phi(Z_\Omega)-\mu(\theta)\}^2
    \right]
    \ge
    \int_{-h}^h
    \mathbb E_\theta[(\widehat \theta-\theta)^2]\,p(\theta)\,d\theta
    \\
    &\ge \frac{1}{\int_{-h}^h\mathfrak I_\theta\,p(\theta)\,d\theta + \inf_{p:\,p(\pm h)=0}J(p)}
    = \frac{1}{\int_{-h}^h\mathfrak I_\theta\,p(\theta)\,d\theta+\pi^2/h^2} \\
    & \ge
    \left\{
    \sigma^{-2} c_u \, \|F_k\|_F^{-2}
    \left(\frac{N_{1k}}{N}
    +
    \frac{T_{1k}}{T} \right)
    + \pi^2 \|F_k\|_F^{-2}
    \mathrm{min}^{-2}(\varsigma, \delta_{\gamma})
    \right\}^{-1} \\
    & =
    \|F_k\|_F^{2}
    \left\{
    \sigma^{-2} c_u
    \left(\frac{N_{1k}}{N}
    +
    \frac{T_{1k}}{T} \right)
    + \pi^2
    \mathrm{min}^{-2}(\varsigma, \delta_{\gamma})
    \right\}^{-1} \\
    & \ge
    \|F_k\|_F^{2}
    \left\{
    2 c_u \sigma^{-2}
    \max\left(\frac{N_{1k}}{N},
    \frac{T_{1k}}{T}
     \right)
    + \pi^2 \mathrm{min}^{-2}(\varsigma, \delta_{\gamma})
    \right\}^{-1} \\
    & \ge
    \frac{1}{2} \,
    \min\left\{
    \frac{\sigma^2}{2 c_u}
    \min\left(\frac{N}{N_{1k}},
    \frac{T}{T_{1k}} \right),
    \pi^{-2} \mathrm{min}^{2}(\varsigma, \delta_{\gamma})
    \right\}
    \|F_k\|_F^{2} \\
    & =
    \min\left\{
    \frac{\sigma^2}{4 c_u}
    \min\left(\frac{N}{N_{1k}},
    \frac{T}{T_{1k}} \right),
    \frac{\varsigma^2}{2\pi^2},
    \frac{\delta_{\gamma}^2}{2\pi^2}
    \right\}
    \|F_k\|_F^{2} \\
    & \ge
    \min\left\{\frac{1}{4c_u},\frac{1}{2\pi^2}\right\}
    \min\left\{
    \sigma^2 \min\left(\frac{N}{N_{1k}},
    \frac{T}{T_{1k}} \right), \varsigma^2, \delta_{\gamma}^2
    \right\}
    \|F_k\|_F^{2} \\
    & =
    \min\left\{\frac{1}{4c_u},\frac{1}{2\pi^2}\right\}
    \min\left\{
    \sigma^2 \min\left(\frac{N}{N_{1k}},
    \frac{T}{T_{1k}} \right), \varsigma^2, \delta_{\gamma}^2
    \right\}
    \|U_{0, 2k}^\top \, x\|_2^2
    \| V_{0, 2k}^\top y\|_2^{2},
\end{align*}
where the first equality follows from the fact that $J(p)$ is minimized by $p(\theta) = h^{-1}\cos^2(\pi \, \theta/2h)$, while the successive inequality follows from our previous bound on $\mathfrak I_\theta$. Since the path is contained in the local parameter space, this lower bound also
applies to the local minimax risk over $\mathcal{F}_\mathrm{loc}(\mathcal{M}_0, \varsigma)$. This concludes the proof.
\end{proof}

\medskip\medskip
For the second result, which concerns the small-$K$ regime, we suppose for simplicity that $N_{1j}=N_1$ and $T_{1j}=T_1$ for all
$j\in[K]$. Under this assumption, $ U_{0,1j_1}= U_{0,1j_2}$ and $ U_{0,2j_1}= U_{0,2j_2}$ for all $j_1,j_2\in[K]$. As in the main body, we denote these common matrices by $ U_{0,1}$ and~$ U_{0,2}$, respectively, and adopt the analogous convention for $ V_{0,1}$ and $ V_{0,2}$. We define the projections
$ P_{ V_{0,2}}
    :=
     V_{0,2}
    \left( V_{0,2}^{\top} V_{0,2}\right)^{-1}
     V_{0,2}^{\top}, \,\,
      P_{ U_{0,2}}
    :=
     U_{0,2}
    \left( U_{0,2}^{\top} U_{0,2}\right)^{-1}
     U_{0,2}^{\top}, \,\,  P_{V_{0,2}}^{\perp} :=
     I_{T_{2}}- P_{ V_{0,2}}$, and $ P_{ U_{0,2}}^{\perp}
    :=
     I_{N_{2}}-P_{ U_{0,2}}$. The inverses are well defined when~\eqref{assump:subblock-conditioning} holds with $c_u \max(\frac{N_1}{N}, \frac{T_1}{T}) < 1$. We also introduce $\omega_V:=
    \| P_{ V_{0,2}}^{\perp} \,  y \|_2, \, 
    \omega_U
    :=~\| P_{ U_{0,2}}^{\perp} \,  x \|_2$, which measure the components of $ y$ and $ x$ orthogonal to the column spaces of $ V_{0,2}$ and~$ U_{0,2}$.

\begin{thm}\label{thm:globalTermLB}
Fix $k\in[K]$ and unit vectors $ x \in \mathcal S^{N_{2k}-1},  y \in \mathcal S^{T_{2k}-1}$. Set $N_{1j}=N_1$ and $T_{1j}=T_1$ for all
$j\in[K]$. Let $\mathcal M_0=\mathcal C_0\times_1  U_0\times_2  V_0\times_3  I_K \in \mathcal F(c_\ell, c_u)$, and suppose that~\eqref{assump:subblock-conditioning}
holds with margin $0< \delta_{\mathrm{A1}} < (c_u - c_\ell)/2$, in the sense that
\begin{align*}
    &\left(c_\ell+\delta_{\mathrm{A1}}\right)\frac{N_1}{N} I_r
    \preceq
     U_{0,1}^{\top} U_{0,1}
    \preceq
    \left(c_u-\delta_{\mathrm{A1}}\right)\frac{N_1}{N} I_r, \\ &\left(c_\ell+\delta_{\mathrm{A1}}\right)\frac{T_1}{T} I_r
    \preceq
     V_{0,1}^{\top} V_{0,1}
    \preceq
    \left(c_u-\delta_{\mathrm{A1}}\right)\frac{T_1}{T} I_r.
\end{align*}
Also assume that $c_u \max(\frac{N_1}{N}, \frac{T_1}{T}) < 1$. There exists $c \equiv c(c_u, \gamma_\mathrm{min}, \gamma_\mathrm{max})>0$ such~that
\begin{align*}
    \inf_{\phi}
    \sup_{\mathcal{F}_\mathrm{loc}(\mathcal{M}_0, \varsigma)}
    \mathbb E_{\mathcal M}
    \left[
    \left\{\phi(Z_\Omega)-\mu_{xy}^{(k)}(\mathcal M)\right\}^2
    \right] &\geq c \, \omega_V^2 \, \min\left( \frac{\sigma^{2} N}{K N_1}, \, \varepsilon_V^{2}\right) \, \| U_{0,2}^{\top}\, x\|_2^2 \\
    & \qquad\qquad\qquad+ c \, \omega_U^2 \, \min\left( \frac{\sigma^{2} T}{K T_1}, \, \varepsilon_U^{2}\right) \, \|  V_{0,2}^{\top}\,  y\|_2^2,
\end{align*}
where $\varepsilon_U
:=
\min(
\omega_U/2,\,
\sqrt{T_1/T},\,
\varsigma\gamma_{\max}^{-1}K^{-1/2},\,
\sqrt{\delta_{\mathrm{A1}}/c_\ell})$, $\varepsilon_V
:=
\min(
\omega_V/2,
\sqrt{N_1/N},
\varsigma\gamma_{\max}^{-1}K^{-1/2},$ $
\sqrt{\delta_{\mathrm{A1}}/c_\ell})$, and the infimum is over all Borel-measurable functions $\phi$ of $Z_\Omega$.
\end{thm}

\begin{proof}[Proof of Theorem~\ref{thm:globalTermLB}]
We prove the first term in the lower bound by reducing the problem to a one-dimensional parametric
submodel where only $V_0$ is perturbed. The second follows by the symmetric argument in which~$U_0$ is perturbed instead of $V_0$. To this end, we will assume that $U_{0,2}^{\top}\, x \neq 0$ and $\omega_V > 0$; if either of these conditions is not met, the desired lower bound is trivially satisfied. Let $C_{0,j}:=(\mathcal C_0)_{\bullet,\bullet,j}$ satisfying $\gamma_{\min}\le \sigma_r(C_{0,j})\le \sigma_1(C_{0,j})\le \gamma_{\max}$ for all $j\in[K]$. Let $G_V
    :=
    C_{0,k}^{\top} U_{0,2}^{\top} \, x \, y^\top  P_{V_{0,2}}^{\perp} \neq 0$ since $U_{0,2}^{\top}\, x \neq 0$, $\sigma_r(C_{0,k})\ge \gamma_{\min}>0$, and $ \omega_V > 0$. Since $G_V$ is rank one, write $G_V=dab^\top$, where
\[
    d:=\|G_V\|_F=\omega_V\|C_{0,k}^{\top}U_{0,2}^{\top}x\|_2,\qquad
    a:=\frac{C_{0,k}^{\top}U_{0,2}^{\top}x}{\|C_{0,k}^{\top}U_{0,2}^{\top}x\|_2},\qquad
    b:=\frac{P_{V_{0,2}}^\perp y}{\omega_V}.
\]
Then $\|a\|_2=\|b\|_2=1$ and $G_V b=d a$.
By definition of $b$ we also have
$b\in\operatorname{Im}(P_{V_{0,2}}^{\perp})$, hence $V_{0,2}^{\top}b=0$ and $P_{V_{0,2}}^{\perp} b = b$. Setting $w:= (\boldsymbol 0_{T_1}^\top  \, ;\, b^\top)^\top \in\mathbb R^{T}$, we obtain a vector supported only on the last $T_{2}$ entries that is orthogonal to $V_0$, meaning that $V_0^{\top}w=0$. 

We now introduce a one-dimensional submodel by perturbing $V_0$ only. For $\theta \in\mathbb R$, define $
    V(\theta) := (V_0+\theta w a^{\top})(I_r+\theta^2a a^{\top})^{-1/2}
    =
    (V_0+\theta w a^{\top})
    (I_r+ [\{1+\theta^2\}^{-1/2} - 1] \, a a^{\top})$. We have $V(\theta)^{\top}V(\theta)=I_r$ and 
    \begin{align*}
        V^\prime(\theta) &= (1+\theta^2)^{-3/2} (w - \theta \, V_0 a) \, a^\top.
    \end{align*}
We leave $U_0$ and $\mathcal{C}_0$ untouched, and set $\mathcal M(\theta):=\mathcal C_0\times_1 U_0\times_2 V(\theta)\times_3 I_K$. We next verify the range of $\theta$ for which the path
$\{\mathcal M(\theta):|\theta|\le h\}$ remains in the local parameter space. As \,$\mathcal C_0$ and $U_0$ are fixed, the only conditions to check are the Frobenius-norm bound, and assumption  \eqref{assump:subblock-conditioning} for $V(\theta)$. As for the former, since $U_0^{\top}U_0=I_r, V_0^\top w = 0, w^\top w = 1, a^\top a = 1$, we have
    \begin{align*}
    \|\mathcal M&(\theta)-\mathcal M_0\|_F^2
    =
    \sum_{j=1}^K
    \left\|
    U_0 \, C_{0,j} \, \{V(\theta)-V_0\}^{\top}
    \right\|_F^2 \leq \gamma_\mathrm{max}^2 \sum_{j=1}^K
    \left\|
    V(\theta)-V_0
    \right\|_F^2                                      \\
    &= \gamma_\mathrm{max}^2 \, K \, \left\|
    V(\theta)-V_0
    \right\|_F^2 = 2\,\gamma_\mathrm{max}^2\,K
\left\{
1-
\left(
1+
\theta^2
\right)^{-1/2}
\right\} \leq \gamma_\mathrm{max}^2\, K \, \theta^2 \leq \gamma_\mathrm{max}^2\, K \, h^2,
\end{align*}
where in the penultimate inequality we used the standard bound $2\left\{1-(1+x)^{-1/2}\right\}\le x$ for $x \ge 0$. This is upper bounded by $\varsigma^2$ for $h \leq \varsigma \, \gamma_\mathrm{max}^{-1}K^{-1/2}$. As for  \eqref{assump:subblock-conditioning}, start by observing that $(I_r+\theta^2a a^{\top})^{-1/2}$ is symmetric and has eigenvalues bounded between $(1+\theta^2)^{-1/2}$ and $1$. As a result, for $V_1(\theta) := V(\theta)_{[T_1], \bullet} = V_{0,1} \, (I_r+\theta^2a a^{\top})^{-1/2}$ we have 
\begin{align*}
    V_1(\theta)^\top V_1(\theta) &= \{I_r+\theta^2a a^{\top}\}^{-1/2} \, V_{0,1}^\top  V_{0,1} \, \{I_r+\theta^2a a^{\top}\}^{-1/2} \\
    & \preceq (c_u - \delta_\mathrm{A1}) \, \frac{T_1}{T} \, \{I_r+\theta^2a a^{\top}\}^{-1} \preceq (c_u - \delta_\mathrm{A1}) \, \frac{T_1}{T} I_r \preceq c_u \, \frac{T_1}{T} I_r.
\end{align*}
Similarly, 
\begin{align*}
    V_1(\theta)^\top V_1(\theta) & \succeq (c_\ell + \delta_\mathrm{A1}) \, \frac{T_1}{T} \, \{I_r+\theta^2a a^{\top}\}^{-1} \succeq \frac{c_\ell + \delta_\mathrm{A1}}{1+\theta^2} \,  \frac{T_1}{T} I_r,
\end{align*}    
which is lower bounded by $c_\ell T_1 T^{-1} I_r$ whenever $h^2 \leq \delta_\mathrm{A1} \, c_\ell ^{-1}$. We thus deduce that the path
$\{\mathcal M(\theta):|\theta|\le h\}$ is contained in $\left\{
    \mathcal M\in\mathcal F(c_\ell,c_u):
    \|\mathcal M-\mathcal M_0\|_F\le \varsigma
    \right\}$ whenever $h
    \leq \min (\varsigma \, \gamma_\mathrm{max}^{-1} K^{-1/2}, \delta_\mathrm{A1}^{1/2} \, c_\ell ^{-1/2})$. In order to simplify the computations below while remaining in the Frobenius neighborhood of $\mathcal M_0$, we will set $h
    = \varepsilon_V$.

We next compute the induced change in $\mu(\theta)
    :=
    \mu_{xy}^{(k)}(\mathcal M(\theta))
    =
    x^\top
    U_{0,2} \, C_{0,k} V_{2}(\theta)^{\top} \, y$. Using $G_V=C_{0,k}^{\top}U_{0,2}^{\top}x\,y^\top P_{V_{0,2}}^{\perp}=d ab^\top$,
$P_{V_{0,2}}^{\perp}b=b$, $(x^\top U_{0,2}C_{0,k}a)(b^\top y)=d$, and the expression for $V^\prime(\theta)$ derived above, we have
\begin{align*}
\mu'(\theta)
&=
x^\top U_{0,2} \, C_{0,k} V_2'(\theta)^\top y \\
&=
(1+\theta^2)^{-3/2}
x^\top U_{0,2}C_{0,k} \, a
\left(b-\theta V_{0,2}a\right)^\top y \\
&=
d
(1+\theta^2)^{-3/2}
\left\{
1-\theta d^{-1}
a^\top
(C_{0,k}^{\top}U_{0,2}^{\top}x\,y^\top V_{0,2}) \, a
\right\} \\
& \geq 
d
(1+\theta^2)^{-3/2}
\left\{
1-|\theta| \, d^{-1}
|a^\top
(C_{0,k}^{\top}U_{0,2}^{\top}x\,y^\top V_{0,2}) \, a|
\right\} \\
&\geq
d \, (1+h^2)^{-3/2}
(
1-hd^{-1}
\left\|C_{0,k}^{\top}U_{0,2}^{\top}x\,y^\top V_{0,2}\right\|_{\mathrm{op}}
)
 \\
&\geq
d \, (1+h^2)^{-3/2}
\left(1-h\omega_V^{-1}\right)
 \geq
2^{-5/2} \, d \\
&
= 2^{-5/2} \, \omega_V \,
\|C_{0,k}^{\top}U_{0,2}^{\top}x\|_2 \geq
2^{-5/2}\gamma_{\min}\,\omega_V
\|U_{0,2}^{\top}x\|_2
>0,
\end{align*}
where in the penultimate inequality we used
$h\leq \min(\sqrt{N_1/N}, \,\omega_V/2) \leq \min(1,\omega_V/2)$. Coming now to the Fisher information for this model, it is useful to compute $\|V_2^\prime(\theta)\|_F^2 \leq \|V^\prime(\theta)\|_F^2 = \left(1+\theta^2\right)^{-2} \leq 1$.    Similarly, we can show that $\|V_1^\prime(\theta)\|_F^2 = \theta^2(1+\theta^2)^{-3}\|V_{0,1}a\|_2^2\leq \theta^2 (c_u - \delta_{\mathrm{A1}}) \, \frac{T_1}{ T} (1 + \theta^2)^{-3}$. We thus have 
    \begin{align*}
        \mathfrak I_\theta & = \sigma^{-2} \sum_{j =1}^K \|P_{\Omega_{\bullet, \bullet, j}} ( U_0 \, C_{0,j} [V^\prime(\theta)]^\top)\|_F^2 \\
        &= \sigma^{-2} \, \sum_{j=1}^K \left\{\|V_1^\prime(\theta) C_{0,j}^{\top}U_0^\top\|_F^2 + \|V_2^\prime(\theta) C_{0,j}^{\top} U_{0,1}^\top\|_F^2 \right\} \\
        & \leq \gamma_\mathrm{max}^2 \, \sigma^{-2} \, K \left\{\|V_1^\prime(\theta) \|_F^2 + (c_u - \delta_{\mathrm{A1}}) \frac{N_1}{N} \, \|V_2^\prime(\theta) \|_F^2 \right\} \\
        &\leq \gamma_\mathrm{max}^2 \, \sigma^{-2} \, K \left\{\|V_1^\prime(\theta) \|_F^2 + (c_u - \delta_{\mathrm{A1}}) \frac{N_1}{N} \right\} \\
        & \leq \gamma_\mathrm{max}^2 \, \sigma^{-2} \, K (c_u - \delta_{\mathrm{A1}}) \left\{\theta^2 (1+\theta^2)^{-3} \frac{T_1}{T}  +  \frac{N_1}{N} \right\}  \\
        & \leq \gamma_\mathrm{max}^2 \, \sigma^{-2} \, K \, (c_u - \delta_{\mathrm{A1}}) \left\{h^2 +  \frac{N_1}{N} \right\} \leq 2\,\gamma_\mathrm{max}^2 \, \sigma^{-2} \, \, (c_u - \delta_{\mathrm{A1}}) \frac{K N_1}{N},
    \end{align*}
    where in the last inequality we used $h^2 \leq N_1/N$. 

   We now apply the van Trees inequality~\citep[][Equation~4]{GillVanTrees1995} on the interval $[-h,h]$ with $h = \varepsilon_V$. Let~$p$ be an arbitrary 
absolutely continuous prior density on $[-h,h]$ satisfying
$p(-h)=p(h)=0$. Writing $J(p):=\int_{-h}^h\frac{\{p'(\theta)\}^2}{p(\theta)}\,d\theta$, and setting $c_V:=2^{-6}\gamma_{\min}^2\min\left\{2^{-1}\gamma_{\max}^{-2} \, c_u^{-1},\pi^{-2}\right\}$, the van Trees inequality yields
\begin{align*}
    \sup_{|\theta|\le h}
    \mathbb E_\theta&
    \left[
        \{\phi(Z_\Omega)-\mu(\theta)\}^2
    \right]
     \ge
    \int_{-h}^h
    \mathbb E_\theta
    \left[
        \{\widehat \mu-\mu(\theta)\}^2
    \right] p(\theta)\,d\theta
    \\
    &\ge
    \frac{
    \left\{
    \int_{-h}^h\mu^\prime(\theta)\,p(\theta)\,d\theta
    \right\}^2
    }{
    \int_{-h}^h\mathfrak I_\theta\,p(\theta)\,d\theta
    +
    \inf_{p:\,p(\pm h)=0}J(p)
    } \\
    & =
    \frac{
    \left\{
    \int_{-h}^h\mu^\prime(\theta)p(\theta)\,d\theta
    \right\}^2
    }{
    \int_{-h}^h\mathfrak I_\theta\,p(\theta)\,d\theta
    +
    \pi^2/h^2
    }
    \\
    &\geq
    \frac{
    2^{-5}\gamma_{\min}^2\omega_V^2
    \|U_{0,2}^{\top}x\|_2^2
    }{
    2\gamma_{\max}^2\sigma^{-2}
    (c_u-\delta_{\mathrm{A1}})KN_1/N
    +
    \pi^2/h^2
    } \\
    & =
    \frac{
    2^{-5}\gamma_{\min}^2\omega_V^2
    \|U_{0,2}^{\top}x\|_2^2
    }{
    2\gamma_{\max}^2
    (c_u-\delta_{\mathrm{A1}})
    \sigma^{-2}KN_1/N
    +
    \pi^2h^{-2}
    }
    \\
    &\geq
    \frac{
    2^{-5}\gamma_{\min}^2\omega_V^2
    \|U_{0,2}^{\top}x\|_2^2
    }{
    2\max\left(
    2\gamma_{\max}^2
    (c_u-\delta_{\mathrm{A1}})
    \sigma^{-2}KN_1/N,
    \pi^2h^{-2}
    \right)
    } \\
    & =
    2^{-6}\gamma_{\min}^2\omega_V^2
    \|U_{0,2}^{\top}x\|_2^2 \,
    \min\left(
    \frac{\sigma^2N}
    {2\gamma_{\max}^2(c_u-\delta_{\mathrm{A1}})KN_1},
    \frac{h^2}{\pi^2}
    \right) \\
    & \geq
    2^{-6}\gamma_{\min}^2
    \min\left\{
    \frac{1}{2\gamma_{\max}^2(c_u-\delta_{\mathrm{A1}})},
    \frac{1}{\pi^2}
    \right\}
    \omega_V^2
    \min\left(
    \frac{\sigma^2N}{KN_1},
    h^2
    \right)
    \|U_{0,2}^{\top}x\|_2^2 \\
    & \geq
    c_V\,
    \omega_V^2
    \min\left(
    \frac{\sigma^2N}{KN_1},
    \varepsilon_V^2
    \right)
    \|U_{0,2}^{\top}x\|_2^2,
\end{align*}
where the first equality follows from the fact that $J(p)$ is minimized by $p(\theta) = h^{-1}\cos^2(\pi \, \theta/2h)$, while the successive inequalities follow from our previous bounds on $\mu^\prime(\theta)$ and $\mathfrak I_\theta$. Since the path is contained in the local parameter space, this lower bound also
applies to the local minimax risk over $\mathcal{F}_\mathrm{loc}(\mathcal{M}_0, \varsigma)$. 

Interchanging the roles of $U_0$ and $V_0$ and applying the same argument to $P_{U_{0,2}}^{\perp} x \,  y^\top \, V_{0,2} C_{0,k}^{\top}$, yields the second term in the lower bound. In particular, we will assume that $V_{0,2}^{\top}\, y \neq 0$ and $\omega_U > 0$; if either of these conditions is not met, the desired lower bound is trivially satisfied. We then define $G_U
:=
P_{U_{0,2}}^{\perp}x\,y^\top V_{0,2}C_{0,k}^{\top}$. Since $y^\top V_{0,2}C_{0,k}^{\top}\neq 0$ and
$\omega_U=\|P_{U_{0,2}}^\perp x\|_2 > 0$, this matrix is rank one.
Writing $G_U=d_U a_U b_U^\top$, where
\[
d_U
=
\omega_U\|C_{0,k}V_{0,2}^\top y\|_2,
\qquad
a_U=
\frac{P_{U_{0,2}}^\perp x}{\omega_U},
\qquad
b_U=
\frac{C_{0,k}V_{0,2}^\top y}
{\|C_{0,k}V_{0,2}^\top y\|_2},
\]
and perturbing $U_0$ along $(\boldsymbol 0_{N_1}^\top\,;\,a_U^\top)^\top$ gives an analogous one-dimensional path
$U(\theta)$ with $V_0$ and $\mathcal C_0$ fixed. The same calculations,
with $N_1/N$ and $T_1/T$ interchanged, yield
\[
\inf_{\phi}
\sup_{\mathcal M\in\mathcal F_{\mathrm{loc}}(\mathcal M_0,\varsigma)}
\mathbb E_{\mathcal M}
\left[
\{\phi(Z_\Omega)-\mu_{xy}^{(k)}(\mathcal M)\}^2
\right]
\ge
c_V\,
\omega_U^2
\min\left(
\frac{\sigma^2T}{KT_1},
\varepsilon_U^2
\right)
\|V_{0,2}^{\top}y\|_2^2.
\]
Combining this with the lower bound obtained from the $V_0$-perturbation,
and using that the maximum of two lower bounds is at least their average, gives
\begin{align*}
    \inf_{\phi}
\sup_{\mathcal M\in\mathcal F_{\mathrm{loc}}(\mathcal M_0,\varsigma)}
\mathbb E_{\mathcal M}
\left[
\{\phi(Z_\Omega)-\mu_{xy}^{(k)}(\mathcal M)\}^2
\right]
&\ge
\frac{c_V}{2}\,
\omega_V^2
\min\left(
\frac{\sigma^2N}{KN_1},
\varepsilon_V^2
\right)
\|U_{0,2}^{\top}x\|_2^2 \\
& \qquad +
\frac{c_V}{2}\,
\omega_U^2
\min\left(
\frac{\sigma^2T}{KT_1},
\varepsilon_U^2
\right)
\|V_{0,2}^{\top}y\|_2^2.
\end{align*}
Setting $c = c_V/2$ yields the desired bound.
\end{proof}

\medskip\medskip
For any $\mathcal M_0$ covered by both sets of assumptions, the combined results of Theorems~\ref{thm:localTermLB} and~\ref{thm:globalTermLB} immediately imply the lower bound given in Theorem~\ref{thm:CombinedLB} by taking the maximum of the respective right-hand sides.

In particular, the assumptions $c_u \max(\frac{N_1}{N}, \frac{T_1}{T}) \leq 1-{\delta_\mathrm{sep}}$ and ${\delta_\mathrm{sep}}^{-1} \max( \|  U_{0,2}^{\top}\,  x\|_2^2,  \|  V_{0,2}^{\top}\,  y\|_2^2) \leq 1/2$ allow simplifying the statement to avoid the presence of $\omega_U^2$ and $\omega_V^2$. To see why, start by noticing that $(1 - c_u \frac{N_1}{N}) (1 - \omega_U^2) \leq \sigma_\mathrm{min}^2( U_{0,2}) \,(1 - \omega_U^2) \leq \| U_{0, 2}^\top \,  x\|_2^2 \leq \sigma_\mathrm{max}^2( U_{0,2}) \,(1 - \omega_U^2) \leq 1 - \omega_U^2$. An analogous statement holds for $\omega_V$ and $\| V_{0, 2}^\top \,  y\|_2$. Under $c_u \max(\frac{N_1}{N}, \frac{T_1}{T}) \leq 1-{\delta_\mathrm{sep}}$, we thus get $\| U_{0, 2}^\top \,  x\|_2^2 \geq {\delta_\mathrm{sep}}\,( 1-\omega_U^2)$ and $\| V_{0, 2}^\top \,  y\|_2^2 \geq {\delta_\mathrm{sep}}\,(  1-\omega_V^2)$. Finally, ${\delta_\mathrm{sep}}^{-1} \max( \|   U_{0,2}^{\top}\,  x\|_2^2,  \|  V_{0,2}^{\top}\,  y\|_2^2) \leq 1/2$ implies  $\min(\omega_U^2, \omega_V^2)\geq 1/2$, allowing us to remove the dependence on $\omega_U$ and $\omega_V$ both from the prefactor multiplying the constant $c$ and from the quantities $\varepsilon_U, \varepsilon_V$ appearing inside the minima.

\subsection{Proofs for Section~\ref{sec:bilinearStaggered}}

\begin{proof}[Proof of Corollary~\ref{corollary:StaggeredUB}]
We verify that the auxiliary problem obtained by restricting to $S \times Q$
satisfies the hypotheses of Theorem~\ref{thm:4BlockUBLinear}. Although, for
$j \neq k$, the missingness masks $\Omega_{S,Q,j}$ are not necessarily in
four-block form, this is not an essential requirement. What is needed in order
to apply Theorem~\ref{thm:4BlockUBLinear}, and in particular
Lemma~\ref{lemma1} in Appendix~\ref{appendix:tecnicalLemmas}, is the relevant
subblock conditioning assumption for the rows and columns corresponding to fully
observed rows and columns, respectively. These are exactly the submatrices that
enter the definitions of the upper and left pooled matrices, and they are what
enable improved estimation of the shared subspaces.

We start by verifying the analogue of  \eqref{assump:subblock-conditioning}. Let $G_U:=U_S^\top U_S$ and $G_V:=V_Q^\top V_Q$. By Assumption \eqref{ass:subBlockAuxilliary} we have 
\[
c_\ell\frac{\mathfrak n}{N}I_r
\preceq
G_U
\preceq
c_u\frac{\mathfrak n}{N}I_r,
\qquad
c_\ell\frac{\mathfrak t}{T}I_r
\preceq
G_V
\preceq
c_u\frac{\mathfrak t}{T}I_r, 
\]
hence $G_U$ and $G_V$ are nonsingular. Letting
$\widetilde U:=U_SG_U^{-1/2},\ \widetilde V:=V_QG_V^{-1/2}$ and $\mathcal{\widetilde  C}_{\bullet, \bullet, j}:=G_U^{1/2} \mathcal C_{\bullet,\bullet,j}G_V^{1/2}$, 
we have that $\widetilde U^\top\widetilde U=I_r$, $\widetilde V^\top\widetilde V=I_r$ and $\mathcal M_{S,Q,j}
=
U_S \, \mathcal C_{\bullet,\bullet,j}V_Q^\top
=
\widetilde U \mathcal{\widetilde  C}_{\bullet, \bullet, j}\widetilde V^\top$. This shows that the restriction of the signal to $S\times Q$ admits an orthonormal
Tucker2 representation with row and column dimensions $\mathfrak n$ and
$\mathfrak t$, respectively.

We next check that the restricted Gram matrices are well conditioned. For the target layer, we have $\widetilde U_{S^+}^\top\widetilde U_{S^+} = G_U^{-1/2}U_{S^+}^\top U_{S^+}G_U^{-1/2}$, hence,  using \eqref{ass:subBlockAuxilliary} and the preceding bounds on $G_U$ gives 
\[
\frac{c_\ell}{c_u} \frac{\mathfrak n_{1k}}{\mathfrak n}I_r \preceq \widetilde U_{S^+}^\top\widetilde U_{S^+} \preceq \frac{c_u}{c_\ell} \frac{\mathfrak n_{1k}}{\mathfrak n}I_r.
\]
Similarly, for each $j\neq k$ we have $\widetilde U_{\mathsf{RowAnc}_k(j)}^\top \widetilde U_{\mathsf{RowAnc}_k(j)} = G_U^{-1/2} U_{\mathsf{RowAnc}_k(j)}^\top U_{\mathsf{RowAnc}_k(j)} G_U^{-1/2}$, thus \[
\frac{c_\ell}{c_u} \frac{\mathfrak n_{1j}}{\mathfrak n}I_r \preceq \widetilde U_{\mathsf{RowAnc}_k(j)}^\top \widetilde U_{\mathsf{RowAnc}_k(j)} \preceq \frac{c_u}{c_\ell} \frac{\mathfrak n_{1j}}{\mathfrak n}I_r.
\]
Since the same argument applies to the column factors, we can conclude that the auxiliary sub-block conditioning assumption holds with $c_\ell/c_u, c_u/c_\ell$ in place of $c_\ell, c_u$, respectively.

It remains to identify the signal strengths in the auxiliary model. Since $\widetilde{\mathcal{C}}_{\bullet, \bullet, j}=G_U^{1/2}\mathcal C_{\bullet,\bullet,j}G_V^{1/2}$, standard bounds on the singular values of a matrix product yield \[
\widetilde\gamma_{\min} := c_\ell \, \gamma_{\min} \sqrt{\frac{\mathfrak n\mathfrak t}{NT}} \leq \sigma_{\min}(\widetilde C_j) \leq \sigma_{\max}(\widetilde C_j) \leq c_u\gamma_{\max} \sqrt{\frac{\mathfrak n\mathfrak t}{NT}} =: \widetilde\gamma_{\max}, 
\]
thereby showing that the effective lower and upper signals are $\widetilde\gamma_{\min}$ and $\widetilde\gamma_{\max}$.

The noise distribution is unchanged by restriction since, on the observed auxiliary coordinates, the errors are still independent centered Gaussian variables with variance $\sigma^2$. Moreover, Assumption \eqref{ass:aux-extra} is the analogue of \eqref{assump:sampleSize},  \eqref{assump:smallNoise} and \eqref{assump:Incoherence} after replacing $N, T, N_{1k}, T_{1k}, \rho_N, \rho_T, p_N, p_T, \zeta_N, \zeta_T, \gamma_{\min}, \gamma_{\max}$ by $\mathfrak n, \mathfrak t, \mathfrak n_{1k}, \mathfrak t_{1k}, \rho_{\mathfrak n}, \rho_{\mathfrak t}, p_{\mathfrak n}, p_{\mathfrak t}, \zeta_{\mathfrak n}, \zeta_{\mathfrak t}, \widetilde\gamma_{\min}, \widetilde\gamma_{\max}$. Also, \eqref{eq:negligibilityForUBStaggared} is the analogue of \eqref{eq:negligibilityForUB4Block}, again with the same substitution of auxiliary dimensions and signal strengths. We can thus conclude that all hypotheses of Theorem~\ref{thm:4BlockUBLinear} hold for the auxiliary problem, hence applying this result gives $\mathbb E_\mathcal M [\{ \widehat\mu_{xy}^{(k,a,b)} - \mu_{xy}^{(k,a,b)} \}^2] \le c_1\widetilde\Upsilon_{xy}$. This concludes the proof.
\end{proof}

\section{Additional details on the numerical results}\label{sec:simulExtra}

\subsection{Additional simulation studies}\label{app:simulExtraMasking}

In this section we complement the empirical results on the masking experiments in Sections~\ref{sec:real-data-castle}-\ref{sec:real-data-oxcgrt-owid} by running similar masking experiments on synthetic data. In particular, this gives us an idea of what to expect in the favorable setting where the signal tensor follows a Tucker2 model and its true rank is known.

In this regard, we generate a tensor $\mathcal{Y}\in\mathbb{R}^{50\times 11\times 4}$ according to $\mathcal{Y}_{\bullet,\bullet,j}=\mathcal{M}_{\bullet,\bullet,j}+\mathcal{E}_{\bullet,\bullet,j}$, where $\mathcal{M}_{\bullet,\bullet,j}=U\mathcal{C}_{\bullet,\bullet,j}V^\top$ for all $j\in[4]$, so that the signal $\mathcal{M}$ has Tucker2 rank $(2,2,4)$. The unit factor $U\in\mathbb{R}^{50\times 2}$ is obtained by generating a matrix with independent standard Gaussian entries and orthonormalizing its columns.  The two columns of $V\in\mathbb{R}^{11\times 2}$ are obtained by orthonormalizing a constant and a centered linear function evaluated over an equally spaced grid on $[0,1]$. For each layer $j$, the entries of the core matrix $\mathcal{C}_{\bullet,\bullet,j}\in\mathbb{R}^{2\times 2}$ are independently drawn from a standard Gaussian distribution, after which the core is normalized to have unit Frobenius norm. The noise tensor has entries $\mathcal{E}_{itj}\sim\mathcal{N}(0,1)$.

We impose the same staggered-adoption pattern on all four layers. Of the $50$ units, $10$ are never treated, while the remaining $40$ units are equally divided across adoption times $4,\ldots,11$, with five units adopting at each time. As usual, an entry is observed before the corresponding adoption time and missing thereafter. We take $k=2$ as the target layer. Starting from this observation pattern, we generate additional missingness while preserving the staggered structure. More precisely, at each step, we randomly select an eligible row and mask its last currently observed entry, so that the observed entries in every row always form an initial segment of the time periods. The terminal mask is chosen to retain $6$ fully observed rows and $4$ fully observed columns, with all remaining rows observed only over the first $4$ periods. For each of these intermediate masks, our target is the average of the artificially masked entries in $\mathcal{Y}_{\bullet,\bullet,2}$, whose values before masking are treated as ground truth. We estimate this quantity using Algorithm~\ref{alg:staggered-wrapper}, aggregating the block-specific estimates with weights proportional to the number of artificially masked entries in each block. We compare it with its matrix-only counterpart, which is essentially equivalent to a plug-in approach based on the estimator of~\cite{yan2024entrywise}, as well as with two plug-in approaches based on the nuclear-norm estimator of~\cite{choi2024matrix} and its debiased counterpart.

All rank-dependent procedures are run with $r\in\{1,2,3\}$, corresponding to under-specification, correct specification, and over-specification of the true rank $r=2$. For Algorithm~\ref{alg:staggered-wrapper} and its matrix counterpart, we set $\tau=0.01$. Fig.~\ref{fig:SytheticALL} reports the average squared error over $30$ replicates, together with its standard error. Each replicate begins after $80$ additional entries have been masked and then adds one masked entry at a time until reaching the common terminal configuration. We repeat this procedure independently $30$ times, using a different masking path for each replicate.

\begin{figure}
\centering
\includegraphics[width=1\linewidth]{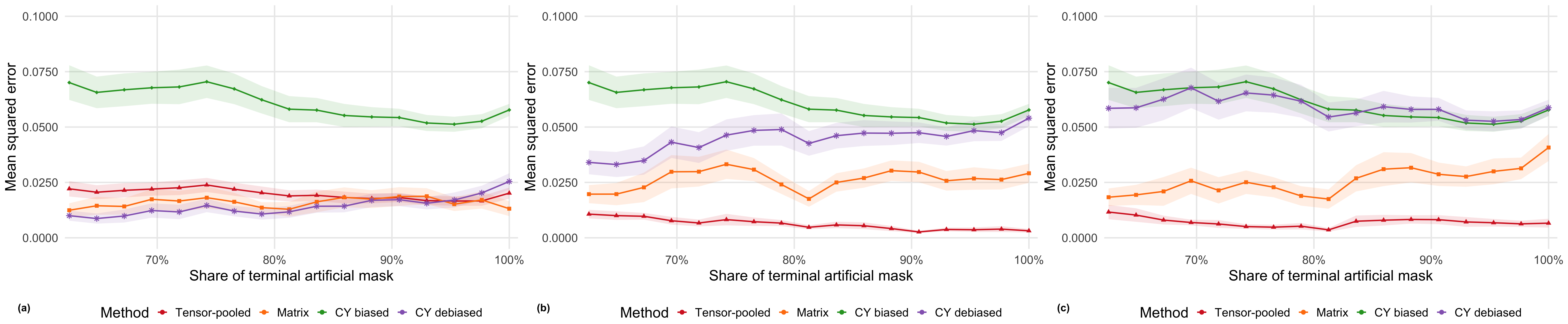}
  \captionsetup{
  font={footnotesize,stretch=0.9},
  skip=3pt
}
\caption{Synthetic masking experiment with additional missingness in the target layer $k =2$. The panels correspond to ranks $r=1,2,3$ used to run the four methods. The terminal mask retains $6$ fully observed rows and $4$ fully observed columns. The curves report the average squared error over $30$ replicates, with one standard error. The percentage on the $x$-axis represents the percentage of the terminal mask reached.}
\label{fig:SytheticALL}
\end{figure}
\vspace{-15pt}
The results show that the two CY procedures generally perform worse than the tensor and matrix estimators. A possible explanation is that the method of \cite{choi2024matrix} accommodates more general missingness patterns, whereas Algorithm~\ref{alg:staggered-wrapper} and its matrix counterpart are tailored to the staggered design and therefore perform particularly well when the missingness is close to a four-block structure. This should not be interpreted as a general ordering between the methods, as the COVID-19 application provides an example with less structured missingness where the CY procedures are highly competitive.

Furthermore, to better understand the relationship between the tensor and matrix estimators, we run a similar experiment in which the terminal mask is chosen to retain $3$ fully observed rows and $3$ fully observed columns. The results are shown in Fig.~\ref{fig:Synthetic}.

\begin{figure}
\centering
\includegraphics[width=1\linewidth]{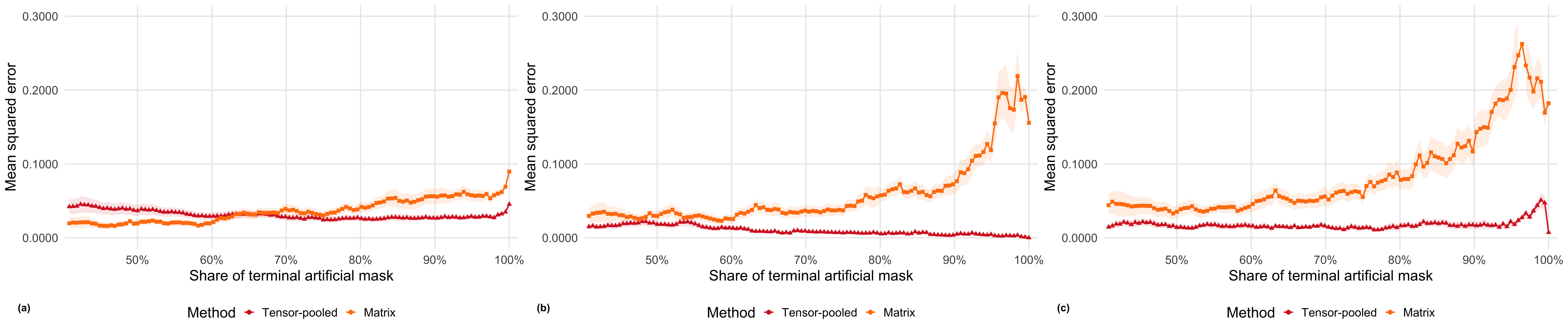}
  \captionsetup{
  font={footnotesize,stretch=0.9},
  skip=3pt
}
\caption{Synthetic masking experiment with additional missingness in the target layer $k=2$. The panels correspond to ranks $r=1,2,3$ used to run the two methods. The terminal mask retains $3$ fully observed rows and $3$ fully observed columns. The curves report the average squared error over $30$ replicates, with one standard error. The percentage on the $x$-axis represents the percentage of the terminal mask reached.}
\label{fig:Synthetic}
\vspace{-15pt}
\end{figure}

In the first panel, where the procedures are run with $r=1$, the matrix method performs better when the number of artificially masked entries is small. This may be because the target layer initially contains enough information on its own, while Algorithm~\ref{alg:bilinear4block} must estimate a common one-dimensional subspace across all layers. As the masking becomes more severe, however, the information in the target layer becomes scarcer and the matrix estimator deteriorates, while the pooled estimator can exploit additional information from the auxiliary layers.

The second and third panels show this advantage more clearly. When the rank is correctly specified at $r=2$, Algorithm~\ref{alg:staggered-wrapper} achieves smaller error. A similar pattern holds for $r=3$, showing that the gain from pooling persists under rank over-specification.

\subsection{Target estimands used in Sections~\ref{sec:real-data-castle} and~\ref{sec:real-data-oxcgrt-owid}}
\label{app:simulEstimands}

We provide more details on the aggregate quantities used in the
real-data experiments.  In these applications, we work with two signal tensors, $\mathcal M(0)$ and $\mathcal M(1)$, corresponding to the untreated and treated responses, respectively. The staggered-adoption mask $\Omega$ is constructed from the treatment variable, with $\Omega_{itj}=1$ if entry $(i,t,j) \in [N] \times [T] \times [K]$ lies in the untreated region, and $\Omega_{itj}=0$ otherwise. The fully observed tensor~$\cal Y$ satisfies $\mathcal Y_{itj}
=
\Omega_{itj} \, \mathcal Y_{itj}(0)
+
(1-\Omega_{itj}) \,  \mathcal Y_{itj}(1)$, where $\mathcal Y(0)$ denotes the untreated potential outcome, which is observed on $\{(i,t,j) :\Omega_{itj}=1\}$ and missing on the complementary set, and $\mathcal Y(1)$ denotes the treated potential outcome, which is observed only over $\{(i,t,j) :\Omega_{itj}=0\}$.

\medskip \medskip
We now introduce the four functionals used in the empirical applications:
\textsc{Avg}, \textsc{RowHet}, \textsc{Local-$i_0$}, and \textsc{Trend}. Fix a target slice $k$ with staircase adoption, and let $\mathcal D_k=\{(a,b):a+b>o_k+1\}$ for some integer $o_k\geq 2$ be the collection of policy-on target blocks. For all $(a,b)\in\mathcal D_k$ and $c \in \{0,1\}$, we also denote by 
$\mathcal M_{\bullet,\bullet,k}^{(a,b)}(c)$ the restriction of $\mathcal{M}(c)$ to rows $R_{ak}$ and columns $C_{bk}$. For $i\in R_{ak}$ and $t\in C_{bk}$, define the local-index maps
$\ell_a(i):=\operatorname{pos}_{R_{ak}}(i)$ and
$\ell_b(t):=\operatorname{pos}_{C_{bk}}(t)$, so that $\ell_a(i)$ is the
position of row $i$ within $R_{ak}$, and $\ell_b(t)$ is the position of
column $t$ within $C_{bk}$. For the
\textsc{RowHet} functional, choose a sign vector
$\eta=(\eta_1,\ldots,\eta_N)^\top\in\{\pm1\}^N$. For the \textsc{Local-$i_0$} functional, choose a row-block index $a_0$ such that $\{b:(a_0,b)\in\mathcal D_k\}\neq\varnothing$, and then fix a row index $i_0\in R_{a_0k}$. We also define $\mathcal D_k^{\mathrm{tr}}
:=
\{(a,b)\in\mathcal D_k:T_{bk}\ge2\}$, and assume $\mathcal D_k^{\mathrm{tr}}\neq\varnothing$ whenever the
\textsc{Trend} functional is considered. For
$h\in\{\textsc{Avg},\textsc{RowHet},\textsc{Local-}i_0,\textsc{Trend}\}$,
we write 
\[
\mathcal D_{k,h}:=
\begin{cases}
\mathcal D_k, & h\in\{\textsc{Avg},\textsc{RowHet}\},\\
\{(a_0,b):(a_0,b)\in\mathcal D_k\}, & h=\textsc{Local-}i_0,\\
\mathcal D_k^{\mathrm{tr}}, & h=\textsc{Trend}
\end{cases}
\]
for the active block set, and, for  fixed $(a,b)\in\mathcal D_{k,h}$, we consider the query
directions
\[
\begin{array}{lll}
\textsc{Avg}:
&
x_{a,h}=N_{ak}^{-1/2}\mathbf 1_{N_{ak}},
&
y_{b,h}=T_{bk}^{-1/2}\mathbf 1_{T_{bk}},
\\[5pt]
\textsc{RowHet}:
&
x_{a,h}=N_{ak}^{-1/2} \eta_{R_{ak}},
&
y_{b,h}=T_{bk}^{-1/2}\mathbf 1_{T_{bk}},
\\[5pt]
\textsc{Local-$i_0$}:
&
x_{a,h}=\boldsymbol e_{\ell_a(i_0)},
&
y_{b,h}=T_{bk}^{-1/2}\mathbf 1_{T_{bk}},
\\[5pt]
\textsc{Trend}:
&
x_{a,h}=N_{ak}^{-1/2}\mathbf 1_{N_{ak}},
&
y_{b,h}=\dfrac{z_b-\bar z_b\mathbf 1_{T_{bk}}}
{\|z_b-\bar z_b\mathbf 1_{T_{bk}}\|_2},
\end{array}
\]
with $z_b=(1,\ldots,T_{bk})^\top$ and $\bar z_b=(T_{bk}+1)/2$. 

Based on this, for $h\in\{\textsc{Avg},\textsc{RowHet},\textsc{Local-}i_0,\textsc{Trend}\}$  and  $(a,b)\in\mathcal D_{k,h}$, we define the block-level bilinear forms \[\mu_h^{(k,a,b)}(c)
:=
x_{a,h}^{\top}
\mathcal M_{\bullet,\bullet,k}^{(a,b)}(c) \,
y_{b,h}.
\]
As for the interpretation of these functionals, \textsc{Avg} averages all entries in the block, \textsc{RowHet} averages a signed row contrast over the block's columns, \textsc{Local-$i_0$} averages only row $i_0$ over the block's columns, while \textsc{Trend} averages over rows and contrasts later columns with earlier columns. In particular, this latter bilinear form also has a simple slope interpretation. For $t \in C_{bk}$ and $c \in \{0,1\}$, define the row-averaged trajectory $\bar m_t^{(a,b)}(c):=N_{ak}^{-1}\sum_{i\in R_{ak}}\mathcal M_{i,t,k}^{(a,b)}(c)$. If this trajectory is linear in local time, i.e.~$\bar m_t^{(a,b)}(c)=\alpha_{a,b}^c+\beta_{a,b}^c\,\ell_b(t)$, then
\[
x_{a, \textsc{Trend}}^\top \, \mathcal M_{\bullet,\bullet,k}^{(a,b)}(c) \, y_{b, \textsc{Trend}}
=
\sqrt{N_{ak}}\,\beta_{a,b}^c
\left\{\frac{T_{bk}(T_{bk}^2-1)}{12}\right\}^{1/2},
\]
thus showing that \textsc{Trend} recovers the slope of the row-averaged trajectory up to a known normalization.

We then aggregate these block-level summaries over all missing blocks in the target slice by
\[
\Psi_c^{(h)}(k)
:=
\{W_h(k)\}^{-1}
\sum_{(a,b)\in\mathcal D_{k,h}}
c_{ab}^{(h)}\mu_h^{(k,a,b)}(c),
\]
where the choices of weights and normalizing constants, together with the simplified form of each estimand,  are given below:
\begin{equation}
\label{eq:target-estimands-table}
\begin{array}{c|c|c|c}
h
&
c_{ab}^{(h)}
&
W_h(k)
&
\Psi_c^{(h)}(k)
\\ \hline
\textsc{Avg}
&
\sqrt{N_{ak}T_{bk}}
&
\displaystyle \sum_{(a,b)\in\mathcal D_k}N_{ak}T_{bk}
&
\displaystyle
\frac{
\sum_{(a,b)\in\mathcal D_k}
\sum_{i\in R_{ak}}
\sum_{t\in C_{bk}}
\mathcal M_{i,t,k}^{(a,b)}(c)
}{
\sum_{(a,b)\in\mathcal D_k}N_{ak}T_{bk}
}
\\[18pt]
\textsc{RowHet}
&
\sqrt{N_{ak}T_{bk}}
&
\displaystyle \sum_{(a,b)\in\mathcal D_k}N_{ak}T_{bk}
&
\displaystyle
\frac{
\sum_{(a,b)\in\mathcal D_k}
\sum_{i\in R_{ak}}
\sum_{t\in C_{bk}}
\eta_i\,\mathcal M_{i,t,k}^{(a,b)}(c)
}{
\sum_{(a,b)\in\mathcal D_k}N_{ak}T_{bk}
}
\\[18pt]
\textsc{Local-}i_0
&
\sqrt{T_{bk}} 
&
\displaystyle \sum_{b:\,(a_0,b)\in\mathcal D_k}T_{bk}
&
\displaystyle
\frac{
\sum_{b:\,(a_0,b)\in\mathcal D_k}
\sum_{t\in C_{bk}}
\mathcal M_{i_0,t,k}^{(a_0,b)}(c)
}{
\sum_{b:\,(a_0,b)\in\mathcal D_k}T_{bk}
}
\\[18pt]
\textsc{Trend}
&
\displaystyle
\frac{1}{\sqrt{
N_{ak} T_{bk}(T_{bk}^2-1)/12}}
&
\displaystyle |\mathcal D_k^{\mathrm{tr}}|
&
\displaystyle
\frac{1}{|\mathcal D_k^{\mathrm{tr}}|}
\sum_{(a,b)\in\mathcal D_k^{\mathrm{tr}}}
\beta_{a,b}^c .
\end{array}
\end{equation}
The final expression for $\Psi_c^{(\mathrm{Trend})}(k)$ uses the linearity condition on the row-averaged trajectory and the assumption $\mathcal D_k^{\mathrm{tr}} \neq \varnothing$. These four quantities have the following interpretations: $\Psi_c^{(\textsc{Avg})}(k)$ is the average potential outcome over the missing entries in slice $k$, $\Psi_c^{(\textsc{RowHet})}(k)$ is the corresponding signed row contrast, $\Psi_c^{(\textsc{Local}\text{-}i_0)}(k)$ is the average potential outcome for row $i_0$ over the missing target blocks containing that row, and $\Psi_c^{(\textsc{Trend})}(k)$ is the average within-block slope of the row-averaged trajectory. 

Based on these potential-outcome summaries, we also define the aggregate policy contrast for functional $h$ by $\Delta^{(h)}(k):=\Psi_1^{(h)}(k)-\Psi_0^{(h)}(k)$.

\medskip\medskip\medskip
Coming now to the estimation of these aggregate quantities, it is useful to recall that
$\mathcal Y_{\bullet, \bullet, k}(1)$ is observed over $\mathcal D_k$, since these are exactly the
policy-on target blocks in which $\mathcal Y_{\bullet, \bullet, k}(0)$ has missing entries. As a result,
functionals with $c=1$ are easier to target and can be estimated by simple
plug-in estimators. On the other hand, quantities such as
$\mu_h^{(k,a,b)}(0)$ require an alternative approach, and can be estimated using Algorithm~\ref{alg:staggered-wrapper}. This immediately leads to a naive estimator of $\Psi_0^{(h)}(k)$ that applies Algorithm~\ref{alg:staggered-wrapper}
separately to each missing target block, and then aggregates the resulting
block-level estimates using the weights $c_{ab}^{(h)}$ and normalizing constants
$W_h(k)$ defined in~\eqref{eq:target-estimands-table}. 

\begin{algorithm}[!ht]
\caption{\textsc{QuadraticStaggeredAggregate} for estimating
$\Psi_0^{(h)}(k)$}
\label{alg:qudraticStaggeredAggregate}
\begingroup \footnotesize \linespread{0.92}\selectfont \algrenewcommand{\algorithmicindent}{1em} \setlength{\abovedisplayskip}{2pt} \setlength{\belowdisplayskip}{2pt} \setlength{\abovedisplayshortskip}{1pt} \setlength{\belowdisplayshortskip}{1pt}
\begin{algorithmic}[1]
\Require target slice $k\in[K]$, functional
$h\in\{\textsc{Avg},\textsc{RowHet},\textsc{Local-}i_0,\textsc{Trend}\}$,
rank $r$, data $\mathcal Y$, staircase partitions
$\{R_{ak}\}_{a=1}^{o_k}$ and $\{C_{bk}\}_{b=1}^{o_k}$,
parameter $\tau>0$, and, when needed, sign vector $\eta$ and row index $i_0$.
\State Initialize $S_h\gets0$.
\For{$(a,b)\in\mathcal D_{k,h}$}
    \State Construct $x_{a,h}$ and $y_{b,h}$ according to the definitions above.
    \State Run Algorithm~\ref{alg:staggered-wrapper} with inputs
    $(k,a,b,r,x_{a,h},y_{b,h},\mathcal Y,\tau)$, and denote its output by
    $\widehat\mu_h^{(k,a,b)}(0)$.
    \State Update
    \[
    S_h\gets S_h+c_{ab}^{(h)}\widehat\mu_h^{(k,a,b)}(0).
    \]
\EndFor
\State \Return $\widehat\Psi_0^{(h)}(k)\gets \{W_h(k)\}^{-1}S_h$.
\end{algorithmic}
\endgroup
\end{algorithm}
\vspace{-15pt}
The blockwise plug-in
estimator in Algorithm~\ref{alg:qudraticStaggeredAggregate} applies Algorithm~\ref{alg:staggered-wrapper} separately to every target block
$(a,b)\in\mathcal D_{k,h} \subseteq \mathcal D_k$, hence it recomputes two rank-$r$ singular value
decompositions for each missing block. In the \textsc{Avg} and \textsc{RowHet} cases, we have
$|\mathcal D_{k,h}|=|\mathcal D_k|=o_k(o_k-1)/2$, so the cost is
quadratic in $o_k$.

This cost can be reduced by trading some statistical efficiency for computational savings through a reduced-anchor construction. In particular, for fixed $a$, we keep the
target-slice column anchor $\mathsf{ColAnc}_{k,a,b}(k)=Q^+_{k,a,b}$ but replace $\mathsf{ColAnc}_{k,a,b}(j)$ by $
\mathsf{ColAnc}_{k,a,b}(j)\cap Q_{k,1}$ for each $j\neq k$. Because $Q^+_{k,a,b}$ depends only on $a$, and because $Q_{k,1}\subseteq Q_{k,b}$
for all $b$, the resulting pooled left matrix depends only on $a$. Similarly,
for fixed $b$, we keep the target-slice row anchor $\mathsf{RowAnc}_{k,a,b}(k)=S^+_{k,a,b}$ but replace $\mathsf{RowAnc}_{k,a,b}(j)$ by $
\mathsf{RowAnc}_{k,a,b}(j)\cap S_{k,1}$ for each $j\neq k$. Because $S^+_{k,a,b}$ depends only on $b$, and because $S_{k,1}\subseteq S_{k,a}$
for all $a$, the resulting pooled upper matrix depends only on $b$.

With this reduced-anchor construction, the SVDs of the pooled left and upper matrices can be cached
and reused, as illustrated in the following algorithm. We will use the shorthand $Q^+_{k,a,b}\equiv Q^+_{k,a}$ and $S^+_{k,a,b}\equiv S^+_{k,b}$ to emphasize that these sets depend only on $a$ and $b$, respectively.

\begin{algorithm}[!ht]
\caption{\textsc{LinearStaggeredAggregate} for estimating $\Psi_0^{(h)}(k)$}
\label{alg:LinearStaggeredAggregate}
\begingroup \footnotesize \linespread{0.92}\selectfont \algrenewcommand{\algorithmicindent}{1em} \setlength{\abovedisplayskip}{2pt} \setlength{\belowdisplayskip}{2pt} \setlength{\abovedisplayshortskip}{1pt} \setlength{\belowdisplayshortskip}{1pt}
\begin{algorithmic}[1]
\Require target slice $k\in[K]$, functional
$h\in\{\textsc{Avg},\textsc{RowHet},\textsc{Local-}i_0,\textsc{Trend}\}$,
rank $r$, data $\mathcal Y$, staircase partitions
$\{R_{ak}\}_{a=1}^{o_k}$ and $\{C_{bk}\}_{b=1}^{o_k}$,
parameter $\tau>0$, and, when needed, sign vector $\eta$ and row index $i_0$.
\State Let $\mathcal A_h:=\{a:\exists b \text{ such that }(a,b)\in\mathcal D_{k,h}\}$ and
$\mathcal B_h:=\{b:\exists a \text{ such that }(a,b)\in\mathcal D_{k,h}\}$.

\For{$a\in\mathcal A_h$}
    \State Set $\mathsf{ColAnc}^{\mathrm{red}}_{k,a}(k):= Q^+_{k,a}$ and, for $j\neq k$, $\mathsf{ColAnc}^{\mathrm{red}}_{k,a}(j)
:=~\left\{
t\in Q_{k,1}:
\Omega_{i,t,j}=1 \ \text{for all } i\in S_{k,a}
\right\}$.
    \State Form reduced pooled left matrix
    $Y_{\mathrm{left},a}^{\mathrm{red}}
    \gets
    (\mathcal Y_{S_{k,a},\mathsf{ColAnc}^{\mathrm{red}}_{k,a}(1),1}
    \ \cdots\
    \mathcal Y_{S_{k,a},\mathsf{ColAnc}^{\mathrm{red}}_{k,a}(K),K})$.
    \State Compute rank-$r$ truncated singular value decomposition
    $(\widehat U_{\mathrm{left},a}^{\mathrm{red}},
    \widehat\Sigma_{\mathrm{left},a}^{\mathrm{red}},
    \widehat V_{\mathrm{left},a}^{\mathrm{red}})
    \gets
    \mathrm{SVD}_r(Y_{\mathrm{left},a}^{\mathrm{red}})$.
    \State Construct $x_{a,h}$, and cache
    $\widehat\alpha_{\mathrm{red}}^{(k, a, h)}
    \gets
    (\widehat U_{\mathrm{left},a}^{\mathrm{red}})_{R_{ak},\bullet}^{\top}
    \,\, x_{a,h}\in\mathbb R^r$.
\EndFor

\For{$b\in\mathcal B_h$}
    \State Set $\mathsf{RowAnc}^{\mathrm{red}}_{k,b}(k):=S_{k,b}^+$ and, for $j\neq k$, $\mathsf{RowAnc}^{\mathrm{red}}_{k,b}(j)
:=~\left\{
i\in S_{k,1}:
\Omega_{i,t,j}=1 \ \text{for all } t\in Q_{k,b}
\right\}$.
    \State Form reduced pooled upper matrix
    $Y_{\mathrm{up},b}^{\mathrm{red}}
    \gets
    (\mathcal Y_{\mathsf{RowAnc}^{\mathrm{red}}_{k,b}(1),Q_{k,b},1}
    \ ;\ \cdots\ ;\
    \mathcal Y_{\mathsf{RowAnc}^{\mathrm{red}}_{k,b}(K),Q_{k,b},K})$.
    \State Compute rank-$r$ truncated singular value decomposition
    $(\widehat U_{\mathrm{up},b}^{\mathrm{red}},
    \widehat\Sigma_{\mathrm{up},b}^{\mathrm{red}},
    \widehat V_{\mathrm{up},b}^{\mathrm{red}})
    \gets
    \mathrm{SVD}_r(Y_{\mathrm{up},b}^{\mathrm{red}})$.
    \State Set $s_{k,b}:=\sum_{j=1}^{k-1}|\mathsf{RowAnc}^{\mathrm{red}}_{k,b}(j)|$, \,\,
    $\widehat U_{\mathrm{up},b}^{(k)}
    \gets
    (\widehat U_{\mathrm{up},b}^{\mathrm{red}})_{
    \{s_{k,b}+1,\ldots,s_{k,b}+|S_{k,b}^+|\},\bullet}$, \,\, 
    $\widehat V_{bk}
    \gets~(\widehat V_{\mathrm{up},b}^{\mathrm{red}})_{C_{bk},\bullet}$.
    \State Construct $y_{b,h}$, and compute
    $T_{b,h}\gets \widehat V_{bk}^{\top} \,\, y_{b,h}$,
    $W_{b,h}\gets \widehat\Sigma_{\mathrm{up},b}^{\mathrm{red}}T_{b,h}$,
    $X_{b,h}\gets \widehat U_{\mathrm{up},b}^{(k)}W_{b,h}$.
\EndFor

\State Set $S_h\gets0$.

\For{$(a,b)\in\mathcal D_{k,h}$}
\State Set
$\widehat U_{+k}^{(a,b)}
\gets
(\widehat U_{\mathrm{left},a}^{\mathrm{red}})_{S_{k,b}^+,\bullet}$.
\State Compute
$\widehat H_{a,b}
\gets
(\widehat U_{+k}^{(a,b)})^\top
\widehat U_{+k}^{(a,b)}$ and $\widehat H_{a,b}
=
Q_{a,b}\operatorname{diag}(\lambda_{a,b,1},\ldots,\lambda_{a,b,r})\, Q_{a,b}^\top$. Then set
\[
\widehat H_{a,b,\tau}^{\mathrm{inv}}
\gets
Q_{a,b}
\operatorname{diag}\left(
\left\{
\frac{1}{\max[\lambda_{a,b,i},\tau]}
\right\}_{i=1}^r
\right)
Q_{a,b}^\top .
\]
    \State Compute
    $\widehat\beta_{\mathrm{red}}^{(k,a,b,h)}
    \gets
    \widehat H_{a,b,\tau}^{\mathrm{inv}}
    (\widehat U_{+k}^{(a,b)})^\top
    X_{b,h}\in\mathbb R^r$.
    \State Update
    $S_h\gets
    S_h+ c_{ab}^{(h)}
    \langle
    \widehat\alpha_{\mathrm{red}}^{(k,a, h)},
    \widehat\beta_{\mathrm{red}}^{(k,a,b,h)}
    \rangle$.
\EndFor

\State \Return
$\widehat\Psi_{0,\mathrm{lin}}^{(h)}(k)
\gets
\{W_h(k)\}^{-1}S_h$.
\end{algorithmic}
\endgroup
\end{algorithm}

Algorithm~\ref{alg:LinearStaggeredAggregate} avoids recomputing two spectral
decompositions for each missing block by exploiting the reduced-anchor
construction. Instead, it computes at most $o_k-1$ left decompositions and at most
$o_k-1$ upper decompositions, thus making the dominant spectral cost linear in
$o_k$ rather than quadratic. This reduction is obtained at the expense of statistical efficiency. While the
blockwise auxiliary construction uses all anchor rows and columns available for
each target block, the reduced-anchor construction uses a smaller common set of
anchors, restricting the non-target-slice column anchors to $Q_{k,1}$ and the
non-target-slice row anchors to $S_{k,1}$. The resulting pooled matrices may
therefore contain less information, which can  weaken the conditioning of the auxiliary Gram
matrices and make the estimated shared row and
column subspaces less accurate. Nevertheless, when $Q_{k,1}$ and $S_{k,1}$ are sufficiently large and
well conditioned, Algorithm~\ref{alg:LinearStaggeredAggregate} provides a computationally cheaper alternative to Algorithm~\ref{alg:qudraticStaggeredAggregate}. A numerical comparison between  Algorithm~\ref{alg:qudraticStaggeredAggregate} and Algorithm~\ref{alg:LinearStaggeredAggregate} is illustrated in Figs.~\ref{fig:StaggeredSynthetic}(b,~c).

\subsection{Castle Doctrine data}\label{app:castle-processing}
We provide more details on the data-processing setup of Section~\ref{sec:real-data-castle}. We use the Castle Doctrine data from the \texttt{PolicyEval} repository, 
\href{https://github.com/guerramarcelino/PolicyEval/raw/main/Datasets/}{available on GitHub}. The data contain state identifiers, calendar years, a Castle Doctrine treatment variable, and several state-level public-safety and socioeconomic variables. We use four logged crime-rate outcomes corresponding to the log motor-theft rate, log robbery rate, log aggravated-assault rate, and log murder rate.

We organize the observed outcomes into a fully observed tensor $\mathcal{Y}\in\mathbb{R}^{50\times 11\times 4}$, whose modes correspond to U.S.~states, calendar years, and crime outcomes. Thus, each entry $\mathcal{Y}_{itj}$ records outcome $j$ for state $i$ in year $t$. The four outcome slices are $\texttt{l\_motor}$, $\texttt{l\_robbery}$, $\texttt{l\_assault}$, and $\texttt{l\_homicide}$, respectively.

We construct the staggered-adoption mask $\Omega$ from the Castle Doctrine treatment variable. 
For all $j\in[K]$, we set $\Omega_{itj}=1$ if entry $(i,t,j)$ lies in the untreated region and zero otherwise. Because treatment status is common across crime outcomes, the same treatment pattern applies to each outcome slice; equivalently, $\Omega_{\bullet,\bullet,1}=\Omega_{\bullet,\bullet,2}=\Omega_{\bullet,\bullet,3}=\Omega_{\bullet,\bullet,4}$. Rows are ordered with never-adopting states at the top. Among adopting states, rows are arranged from later to earlier adopters, so that the treatment boundary moves smoothly across the panel. The resulting observation patterns are shown in Fig.~\ref{fig:castle-crime-tensor-matrix}. Blue cells indicate untreated observations, red cells indicate treated observations, and darker shades correspond to larger logged crime-rate values. In the potential-outcomes notation $\mathcal{Y}_{itj}= \Omega_{itj} \, \mathcal{Y}_{itj}(0)+(1-\Omega_{itj}) \, \mathcal{Y}_{itj}(1)$,  the blue cells are therefore the observed entries of $\mathcal{Y}(0)$, while the corresponding treated entries are treated as missing. Conversely, the red cells are the observed entries of $\mathcal{Y}(1)$, with the corresponding untreated entries treated as missing.

\begin{figure}[!ht]
    \centering
    \includegraphics[width=1\textwidth]{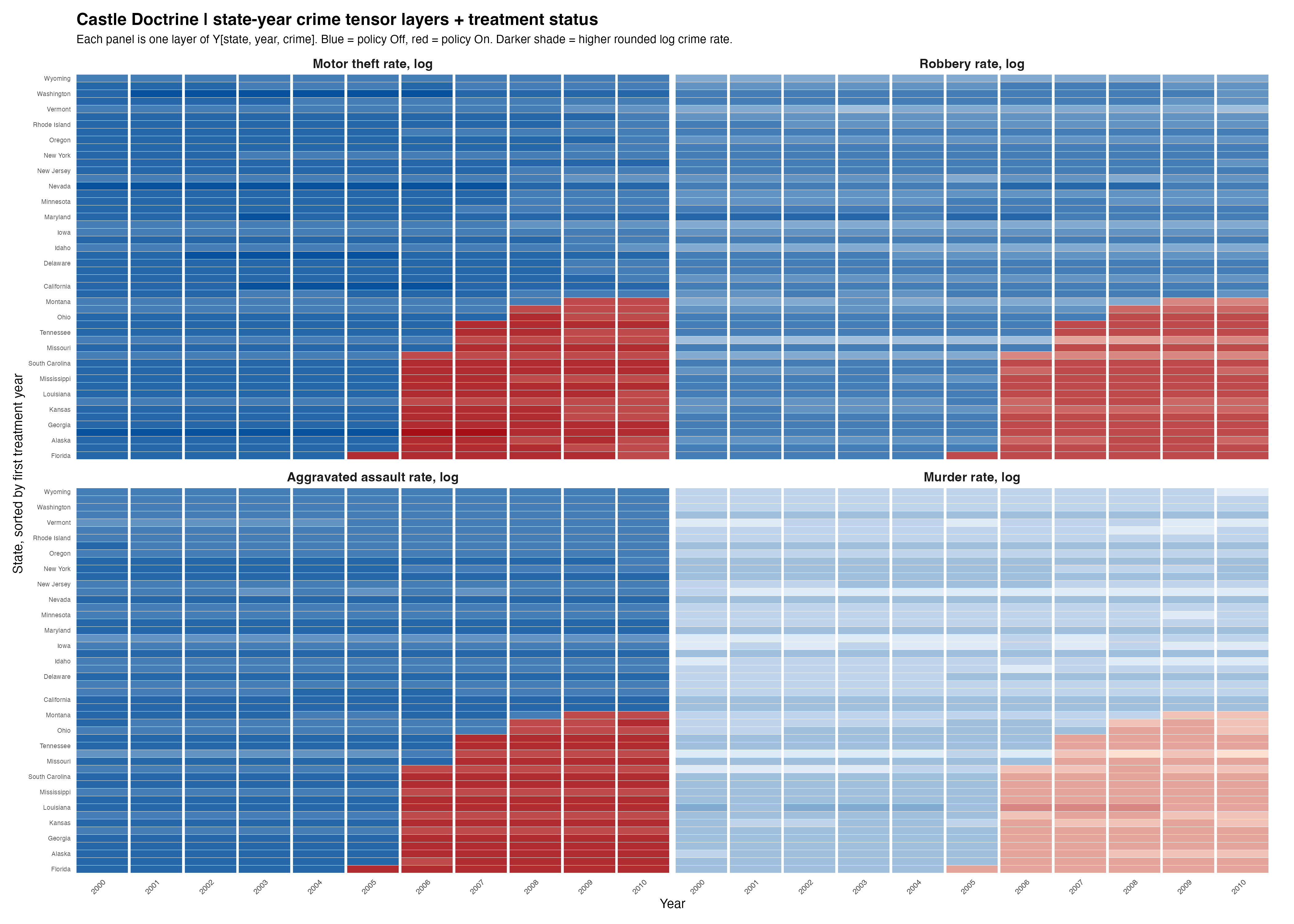}
    \captionsetup{
  font={footnotesize,stretch=0.9},
  skip=3pt
}
    \caption{Castle Doctrine state-year-crime tensor used in the real-data simulation. Panels show logged crime rates for motor theft, robbery, aggravated assault, and murder. Rows are U.S.~states, with never-adopters first and adopters ordered from later to earlier adoption years; columns are calendar years from 2000 to 2010. Blue cells are untreated, red cells are treated, and darker shades indicate larger logged crime-rate values.}
    \label{fig:castle-crime-tensor-matrix}
\end{figure}

 \section{Computational cost of Algorithm~\ref{alg:bilinear4block}}\label{appendix:compCost}
We further elaborate on the computational complexity of Algorithm~\ref{alg:bilinear4block}. The dominant cost is given by the rank-$r$ truncated singular value decompositions of $ Y_{\mathrm{left}}^{\mathrm{p}}\in\mathbb R^{N\times T_{1,{\mathrm{p}}}}$ and $ Y_{\mathrm{up}}^{\mathrm{p}}\in\mathbb R^{N_{1,{\mathrm{p}}}\times T}$, which are of the order $\mathcal O(NT_{1,{\mathrm{p}}}r + N_{1,{\mathrm{p}}}T \, r)$. After these decompositions are computed, the bilinear form is targeted directly. Indeed, Steps~9 and 10 compute $ T_y=\hat{ V}_{2k}^\top \,   y$, $ W_y=\hat{ \Sigma}_{\mathrm{up}}  T_y$, $ X_y=\hat{ U}_{\mathrm{up}}^{(k)} W_y$, and $\hat{ \beta}_y^{(k)}$ without ever materializing the full matrix. This avoids an $\mathcal O(N_{1k}T_{2k}\, r)$ block-construction cost, and computes the required action on~$ y$ in $\mathcal O(\,(T_{2k}+N_{1k}) \, r)$ time; including the computation of $\hat{ \alpha}_x^{(k)}=\hat{ U}_{2k}^\top  x$ in Step~5, the total cost per query is $\mathcal O( \, (N_{2k}+T_{2k}+N_{1k}) \, r)$. Furthermore, for the clipped inverse in Step~4, forming $\hat{ H}_k=\hat{ U}_{1k}^\top \hat{ U}_{1k}$ costs $\mathcal O(N_{1k}r^2)$, while its eigendecomposition and computing $\hat{ H}_{k,\tau}^\mathrm{inv}$ cost $\mathcal O(r^3)$. Taken together, the runtime analysis of these steps also indicate that our algorithm is well suited to caching. In particular, for fixed~$k$, once the pooled singular value decompositions and the slice-specific regression factorization have been computed, each additional query costs only $\mathcal O(\,(N_{2k}+T_{2k}+N_{1k}) \, r)$. When $k$ varies, the pooled singular value decompositions can still be reused, and the only additional slice-specific computation is the $\mathcal O(N_{1k}r^2)$ regression factorization, giving a total cost of $\mathcal O(N_{1k}r^2+(N_{2k}+T_{2k}+N_{1k}) \, r)$.

\medskip\medskip\medskip

\section{Matrix denoising in the pooled four-block setting}\label{appendix:tecnicalLemmas}
This appendix collects matrix denoising results specialized to the tensor four-block framework introduced in Section~\ref{sec:4block}. Throughout, we allow the constant $c_1$ to vary from line to line, while still depending only on $c_\ell, c_u, c_0,c_{\mathrm{dim}},  c_{\mathrm{blk}}, \kappa$. We use the notation introduced in Table~\ref{tab:block-notation-main}.  We also denote by $\mathbb O(d):= \{ Q \in \mathbb R^{d \times d} : Q^\top Q = QQ^\top = I_d \}$ the set of orthogonal matrices of dimension $d$, and write $\mathcal I_k :=\{N_{1k}+1,\dots, N\}$, $\mathcal J_k := \{T_{1k}+1,\ldots,T\}$ and $\mathcal I_k^\mathrm{up}:=\{s_k+1,\dots,s_k+N_{1k}\}$, where we recall $s_k = \sum_{j  =1}^{k-1} N_{1j}$.

The main results of this section are Lemmas~\ref{lemma1} and \ref{lemma:VarDominatesLemma1}. Their proofs rely on the intermediate results presented below, which provide first-order expansions of some relevant quantities appearing in the definition of $\hat \mu_{xy}^{(k)}$ in terms of the matrices $W_\mathrm{left}$, $W_\mathrm{up}$,  $E_\mathrm{left}^\mathrm{p}$, $E_\mathrm{up}^\mathrm{p}$. The first step in applying these results to $Y_{\mathrm{left}}^{\mathrm{p}}$ and $Y_{\mathrm{up}}^{\mathrm{p}}$ is to characterize the spectrum of the corresponding signal matrices $M_{\mathrm{left}}^{\mathrm{p}}$ and $M_{\mathrm{up}}^{\mathrm{p}}$.  

\begin{lemma}\label{lemma:spectrumMp}
    Grant Assumption   \eqref{assump:subblock-conditioning} with fixed constants $c_\ell, c_u$ satisfying $0 < c_\ell \leq c_u < \infty$, and  suppose $0<\gamma_\mathrm{min} \leq \sigma_\mathrm{min}(\mathcal{C}_{\bullet, \bullet, j}) \leq \sigma_\mathrm{max}(\mathcal{C}_{\bullet, \bullet, j}) \leq \gamma_\mathrm{max} < \infty$ for all $j \in [K]$. We have 
    \[
    \begin{cases}
        c_\ell^{1/2} \gamma_\mathrm{min} \,\rho_T^{1/2} \leq \sigma_r(M_{\mathrm{left}}^{\mathrm{p}}) \leq \sigma_1(M_{\mathrm{left}}^{\mathrm{p}}) \leq   c_u^{1/2} \gamma_\mathrm{max} \,\rho_T^{1/2}, \\ 
        \\
                c_\ell^{1/2} \gamma_\mathrm{min} \,\rho_N^{1/2} \leq \sigma_r(M_{\mathrm{up}}^{\mathrm{p}}) \leq \sigma_1(M_{\mathrm{up}}^{\mathrm{p}}) \leq   c_u^{1/2} \gamma_\mathrm{max} \,\rho_N^{1/2}.
    \end{cases}
    \]
\end{lemma}
\begin{proof}
Assumption  \eqref{assump:subblock-conditioning} implies
\[
c_\ell \sum_{j=1}^K \frac{T_{1j}}{T}\,
\mathcal C_{\bullet,\bullet,j}\mathcal C_{\bullet,\bullet,j}^\top
\;\preceq\;
W_{\mathrm{left}}^\top W_{\mathrm{left}} = \sum_{k=1}^K
\mathcal C_{\bullet,\bullet,k}\,V_{1k}^\top V_{1k}\,\mathcal C_{\bullet,\bullet,k}^\top
\;\preceq\;
c_u \sum_{j=1}^K \frac{T_{1j}}{T}\,
\mathcal C_{\bullet,\bullet,j}\mathcal C_{\bullet,\bullet,j}^\top,
\]
\[
c_\ell \sum_{j=1}^K \frac{N_{1j}}{N}\,
\mathcal C_{\bullet,\bullet,j}^\top \mathcal C_{\bullet,\bullet,j}
\;\preceq\;
W_{\mathrm{up}}^\top W_{\mathrm{up}} = \sum_{k=1}^K
\mathcal C_{\bullet,\bullet,k}^\top\,U_{1k}^\top U_{1k}\,\mathcal C_{\bullet,\bullet,k}
\;\preceq\;
c_u \sum_{j=1}^K \frac{N_{1j}}{N}\,
\mathcal C_{\bullet,\bullet,j}^\top \mathcal C_{\bullet,\bullet,j}.
\]
Since $M_{\mathrm{left}}^{\mathrm p}=U W_{\mathrm{left}}^\top, \,
M_{\mathrm{up}}^{\mathrm p}=W_{\mathrm{up}}V^\top$ and $U, V$ have orthonormal columns, we also have $\sigma_j(M_{\mathrm{left}}^{\mathrm{p}})
=
\sigma_j\!\left(W_{\mathrm{left}}\right)$ and $
\sigma_j\!\left(M_{\mathrm{up}}^{\mathrm{p}}\right)
=
\sigma_j\!\left(W_{\mathrm{up}}\right)$ for all $j \in [r]$. We thus get
\[
\sqrt{
c_\ell\,
\lambda_r\!\left(
\sum_{j=1}^K \frac{T_{1j}}{T}\,
\mathcal C_{\bullet,\bullet,j}\mathcal C_{\bullet,\bullet,j}^\top
\right)}
\;\le\;
\sigma_r\!\left(M_{\mathrm{left}}^{\mathrm{p}}\right)
\;\le\; \sigma_1\!\left(M_{\mathrm{left}}^{\mathrm{p}}\right) \leq 
\sqrt{
c_u\,
\lambda_1\!\left(
\sum_{j=1}^K \frac{T_{1j}}{T}\,
\mathcal C_{\bullet,\bullet,j}\mathcal C_{\bullet,\bullet,j}^\top
\right)},
\]
\[
\sqrt{
c_\ell\,
\lambda_r\!\left(
\sum_{j=1}^K \frac{N_{1j}}{N}\,
\mathcal C_{\bullet,\bullet,j}^\top \mathcal C_{\bullet,\bullet,j}
\right)}
\;\le\;
\sigma_r\!\left(M_{\mathrm{up}}^{\mathrm{p}}\right)
\;\le\; \sigma_1\!\left(M_{\mathrm{up}}^{\mathrm{p}}\right)\leq 
\sqrt{
c_u\,
\lambda_1\!\left(
\sum_{j=1}^K \frac{N_{1j}}{N}\,
\mathcal C_{\bullet,\bullet,j}^\top \mathcal C_{\bullet,\bullet,j}
\right)}.
\]
Combining this with $0<\gamma_\mathrm{min} \leq \sigma_\mathrm{min}(\mathcal{C}_{\bullet, \bullet, j}) \leq \sigma_\mathrm{max}(\mathcal{C}_{\bullet, \bullet, j}) \leq \gamma_\mathrm{max} < \infty$ concludes the proof.
\end{proof}

\medskip \medskip\medskip \medskip
The following lemma provides an upper bound on the estimation error of $\widehat U_{\mathrm{left}}$ relative to $U$ measured by the operator norm of the projected error $\Pi_N^\top(\widehat U_{\mathrm{left}}H_U-U)$. The proof relies on tools from Haar compression and properties of the Stiefel manifold, as outlined in Appendix~\ref{appendix:AuxilliaryRes}.

\begin{lemma}
\label{lemma13_centered}
Grant Assumptions  \eqref{assump:subblock-conditioning}
with fixed constants $0<c_\ell\le c_u$, \eqref{assump:sampleSize} and \eqref{assump:smallNoise}. Suppose further that $0<\gamma_{\min}\le \sigma_{\min}(\mathcal C_{\bullet,\bullet,j})
\le \sigma_{\max}(\mathcal C_{\bullet,\bullet,j})\le \gamma_{\max}<\infty$  for all
$j\in[K]$, and let $\kappa:=\gamma_{\max}/\gamma_{\min}$. Write
$Y_{\mathrm{left}}^{\mathrm{p}}
=M_{\mathrm{left}}^{\mathrm{p}}+E_{\mathrm{left}}^{\mathrm{p}}$, with 
$M_{\mathrm{left}}^{\mathrm{p}}=UW_{\mathrm{left}}^\top$, and set
$\Lambda:=W_{\mathrm{left}}^\top W_{\mathrm{left}}$. Let
$(\hat U_{\mathrm{left}},\hat\Sigma_{\mathrm{left}},\hat V_{\mathrm{left}})
:=\operatorname{SVD}_r(Y_{\mathrm{left}}^{\mathrm{p}})$ and
$H_U:=\operatorname{sgn}(\hat U_{\mathrm{left}}^\top U)$, and define the centered empirical eigenvalue matrix
$\hat\Lambda_c:=H_U^\top(\hat\Sigma_{\rm left}^2-\sigma^2 T_{1,\mathrm{p}} I_r)H_U$. Also fix $1\le p\le N$ and $\Pi_N\in\mathbb R^{N\times p}$ with
$\Pi_N^\top\Pi_N=I_p$. There exists a
constant $c_1\equiv c_1(c_\ell,c_u,c_0,c_{\mathrm{dim}}, c_{\mathrm{blk}},\kappa)>~0$ such
that, with probability at least $1-\mathcal O(p_T^{-10})$,
the following statements hold:

\begin{enumerate}
\item[(i)]
The centered empirical eigenvalue matrix is well-conditioned, i.e.~
\begin{align}
\lambda_r(\hat\Lambda_c)\ge \frac34\,\lambda_r(\Lambda),
\qquad
\|\hat\Lambda_c^{-1}\|_{\mathrm{op}}
\le \frac43\,\lambda_r(\Lambda)^{-1}.
\label{eq:lemma13_centered_lambdalower}
\end{align}

\item[(ii)]
We have 
\begin{align}
\begin{aligned}
\|\Pi_N^\top(\hat U_{\mathrm{left}}H_U-U)\|_{\mathrm{op}}
&\le
c_1
\frac{\sigma\sqrt{p+r+\zeta_T}}
{\gamma_{\min} \, \rho_T^{1/2}} +
c_1 \, \frac{\sigma^2 N}{\gamma_\mathrm{min}^2 \, \rho_T}\|\Pi_N^\top U\|_{\mathrm{op}} .
\end{aligned}
\label{eq:lemma13_centered_projected}
\end{align}
\end{enumerate}
\end{lemma}

\begin{proof}
For readability, only in this proof we write
$Y=Y_{\mathrm{left}}^{\mathrm{p}}$,
$M=M_{\mathrm{left}}^{\mathrm{p}}=UW_{\mathrm{left}}^\top$, and
$E=E_{\mathrm{left}}^{\mathrm{p}}$. Define
$\hat S:=YY^\top-\sigma^2T_{1,{\mathrm{p}}}I_N$,
$S_0:=MM^\top=U\Lambda U^\top$, and
$\Xi:=\hat S-S_0$. Let $U_\perp\in\mathbb R^{N\times(N-r)}$ be such that
$[U\ \ U_\perp]\in\mathbb O(N)$.

Since $YY^\top$ and $\hat S$ differ by a scalar multiple of the identity,
they have the same eigenvectors. Hence $\hat U_{\mathrm{left}}$ is also the
top-$r$ eigenspace of $\hat S$, and
\begin{align}
\hat S\hat U_{\mathrm{left}}H_U
=
\hat U_{\mathrm{left}}H_U\hat\Lambda_c .
\label{eq:left-centered-eig}
\end{align}
Set $G_1:=U^\top E\in\mathbb R^{r\times T_{1,{\mathrm{p}}}}$ and
$G_2:=U_\perp^\top E\in\mathbb R^{(N-r)\times T_{1,{\mathrm{p}}}}$.
Since $E$ has independent $\mathcal N(0,\sigma^2)$ entries, rotational invariance ensures that $G_1$ and $G_2$ are independent Gaussian matrices with i.i.d.
$\mathcal N(0,\sigma^2)$ entries. Expanding $\Xi$ gives
$\Xi=UW_{\mathrm{left}}^\top E^\top+EW_{\mathrm{left}}U^\top
+(EE^\top-\sigma^2T_{1,{\mathrm{p}}}I_N)$, and therefore
\begin{align}
\begin{aligned}
U^\top\Xi U
&=
W_{\mathrm{left}}^\top G_1^\top
+
G_1W_{\mathrm{left}}
+
(G_1G_1^\top-\sigma^2T_{1,{\mathrm{p}}}I_r),  \\
U_\perp^\top\Xi U
&=
G_2W_{\mathrm{left}}+G_2G_1^\top
=
G_2K,
\qquad
K:=W_{\mathrm{left}}+G_1^\top,  \\
U_\perp^\top\Xi U_\perp
&=
G_2G_2^\top-\sigma^2T_{1,{\mathrm{p}}}I_{N-r}.
\end{aligned}
\label{eq:left-centered-Xi-blocks}
\end{align}
Write the aligned empirical eigenspace as
$\hat U_{\mathrm{left}}H_U=UC+U_\perp S$, where
$C:=U^\top\hat U_{\mathrm{left}}H_U$ and
$S:=U_\perp^\top\hat U_{\mathrm{left}}H_U$. Since $H_U$ is the Procrustes
alignment, $C$ is symmetric and satisfies $C\succeq0$,
$C^\top C+S^\top S=I_r$. We can thus write
\begin{align}
\Pi_N^\top(\hat U_{\mathrm{left}}H_U-U)
=
\Pi_N^\top U_\perp S+\Pi_N^\top U(C-I_r),
\label{eq:left-centered-split-clean}
\end{align}
which shows that it is enough to control the two terms on the right-hand side. 

More precisely, the preceding decomposition shows that sharp control of some projection of
$\hat U_{\rm left}H_U-U$ reduces mainly to controlling the off-subspace
component $S=U_\perp^\top\hat U_{\rm left}H_U$. Indeed,
$C=(I_r-S^\top S)^{1/2}$ and $\|C-I_r\|_{\rm op}\le \|S\|_{\rm op}^2$, so the
term involving $C-I_r$ is second order. A direct application of Wedin's theorem~\citep[][Section~2.4]{spectralMethods}
would control only the global subspace error
$\|S\|_{\rm op}
=
\|U_\perp^\top\hat U_{\rm left}\|_{\rm op}
=
\|\sin\Theta(\hat U_{\rm left},U)\|_{\rm op}
\lesssim
\|E_{\rm left}^{\rm p}\|_{\rm op}/\sigma_r(M_{\rm left}^{\rm p})$,
which is governed by an ambient noise norm and hence scales with the full row
dimension $N$. When combined with the triangle inequality
$\|\Pi_N^\top(\hat U_{\rm left}H_U-U)\|_{\rm op}
\le
\|\Pi_N^\top U_\perp\|_{\rm op}\|S\|_{\rm op}
+
\|\Pi_N^\top U\|_{\rm op}\|S\|_{\rm op}^2$,
this would not exploit the fact that $\Pi_N$ has only $p$ columns. Instead, we apply the Haar--Stiefel compression bounds from Appendix~\ref{appendix:AuxilliaryRes}, which allow us to use the fixed projection $\Pi_N^\top U_\perp$ to reduce the random off-subspace component by a factor of order $\sqrt{(p+r+\zeta_T)/N}$, as shown in~\eqref{eq:left-centered-random-range-step}. This is the key idea that
turns an ambient subspace perturbation estimate into the projected bound needed
here.

\medskip
\noindent\textbf{$\bullet$ Conditioning of $\hat\Lambda_c$.}
The eigenvalues of $\hat\Lambda_c$ are the top $r$ eigenvalues of $\hat S$.
By the Courant--Fischer formula restricted to $\operatorname{col}(U)$, we have 
$\lambda_r(\hat\Lambda_c)=\lambda_r(\hat S)
\ge \lambda_r(U^\top\hat S U)
=\lambda_r(\Lambda+U^\top\Xi U)$. Weyl's inequality (Lemma~\ref{lemma:Weyl}) then gives
\begin{align}
\lambda_r(\hat\Lambda_c)
\ge
\lambda_r(\Lambda)-\|U^\top\Xi U\|_{\mathrm{op}} .
\label{eq:left-centered-lambdahat-lower}
\end{align}
Define the events
\[
\|G_1W_{\mathrm{left}}\|_{\mathrm{op}}
\le c_1\sigma\|W_{\mathrm{left}}\|_{\mathrm{op}}\sqrt{r+\zeta_T},
\]
\[
\|G_1G_1^\top-\sigma^2T_{1,{\mathrm{p}}}I_r\|_{\mathrm{op}}
\le
c_1\sigma^2
\left\{\sqrt{T_{1,{\mathrm{p}}}(r+\zeta_T)}+r+\zeta_T\right\}.
\]
By Lemma~\ref{lemma:bilinearGaussian} in Appendix~\ref{appendix:AuxilliaryRes} and a standard Wishart concentration bound
\citep[e.g.][Theorem 4.6.1]{vershynin19hdp}, the probability that at least one of the two displayed
events fails is at most
$\mathcal O(p_T^{-10})$. 
If these bounds hold, the first display in \eqref{eq:left-centered-Xi-blocks} implies $    \|U^\top\Xi U\|_{\mathrm{op}}
\le
c_1\sigma\|W_{\mathrm{left}}\|_{\mathrm{op}}\sqrt{r+\zeta_T}
+
c_1\sigma^2
\bigl\{\sqrt{T_{1,{\mathrm{p}}}(r+\zeta_T)}+r+\zeta_T\bigr\}$. By Lemma~\ref{lemma:spectrumMp} we have
$\lambda_r(\Lambda)\ge c_\ell \, \gamma_{\min}^2 \rho_T$ and
$\|W_{\mathrm{left}}\|_{\mathrm{op}}
=\sigma_1(M_{\mathrm{left}}^{\mathrm{p}})
\le c_u^{1/2} \, \gamma_{\max} \rho_T^{1/2}$, which yield
\begin{align}\label{eq:UXiU}
    \frac{\|U^\top\Xi U\|_{\mathrm{op}}}{\lambda_r(\Lambda)}
&\le
c_1
\kappa\frac{\sigma}{\gamma_{\min}}
\sqrt{\frac{T(r+\zeta_T)}{T_{1,{\mathrm{p}}}}}
+ c_1 \, 
\frac{\sigma^2T}{\gamma_{\min}^2}
\left(
\sqrt{\frac{r+\zeta_T}{T_{1,{\mathrm{p}}}}}
+ 
\frac{r+\zeta_T}{T_{1,{\mathrm{p}}}}
\right) \nonumber \\
& \leq c_1
\frac{\sigma}{\gamma_{\min}}
\sqrt{\frac{NT}{T_{1,{\mathrm{p}}}}}
+ c_1 \, 
\frac{\sigma^2T}{\gamma_{\min}^2}
\left(
\sqrt{\frac{N}{T_{1,{\mathrm{p}}}}}
+ 
\frac{N}{T_{1,{\mathrm{p}}}}
\right) \leq c_1
\frac{\sigma}{\gamma_{\min}}
\sqrt{\frac{NT}{T_{1,{\mathrm{p}}}}} \le c_1 \theta \leq \frac{1}{4},
\end{align}
where the last inequality follows from Assumptions \eqref{assump:sampleSize}
and \eqref{assump:smallNoise}, provided that the absolute constants
$c_0>0$ and $c_{\mathrm{dim}}>0$ are sufficiently small. Combining this with \eqref{eq:left-centered-lambdahat-lower} proves (i).

\medskip
\noindent\textbf{$\bullet$ Reduction to the range of $G_2$.}
Using \eqref{eq:left-centered-eig} and
$\hat U_{\mathrm{left}}H_U=UC+U_\perp S$, and then left-multiplying by~$U_\perp^\top$, we obtain
\begin{align}
G_2KC+(G_2G_2^\top-\sigma^2T_{1,{\mathrm{p}}}I_{N-r})S
=
S\hat\Lambda_c .
\label{eq:left-centered-master}
\end{align}
Let $P_2:=\operatorname{Proj}_{\operatorname{col}(G_2)}$ and
$P_2^\perp:=I_{N-r}-P_2$. Since $P_2^\perp G_2K=0$, multiplying
\eqref{eq:left-centered-master} by $P_2^\perp$ gives
$P_2^\perp S(\hat\Lambda_c+\sigma^2T_{1,{\mathrm{p}}}I_r)=0$. On the event where $\hat\Lambda_c\succ0$, we also have that 
$\hat\Lambda_c+\sigma^2T_{1,{\mathrm{p}}}I_r$ is invertible, hence
$P_2^\perp S(\hat\Lambda_c+\sigma^2T_{1,{\mathrm{p}}}I_r)=0$ implies
$P_2^\perp S=0$, and therefore $S=P_2S$. Defining
$D_2:=(G_2G_2^\top-\sigma^2T_{1,{\mathrm{p}}}I_{N-r})P_2$, we may rewrite
\eqref{eq:left-centered-master} as
\begin{align}
G_2KC+D_2S=S\hat\Lambda_c .
\label{eq:left-centered-master-restricted}
\end{align}
The inclusion of the projection $P_2$ in the definition of $D_2$ is crucial:
although $G_2G_2^\top-\sigma^2T_{1,{\mathrm{p}}}I_{N-r}$ may be large on
$\operatorname{col}(G_2)^\perp$, the identity $S=P_2S$ ensures that only its
restriction to $\operatorname{col}(G_2)$ is relevant.

\medskip
\noindent\textbf{$\bullet$ Restricted centered-Wishart bound for $D_2$.}
Let $q:=\operatorname{rank}(G_2)=\min(N-r, \, T_{1,{\mathrm{p}}})$ almost
surely. Since $D_2$ acts on $\operatorname{col}(G_2)$, we have
$\|D_2\|_{\mathrm{op}}
=\max_{1\le i\le q}|\sigma_i(G_2)^2-\sigma^2T_{1,{\mathrm{p}}}|$. By the standard two-sided singular value bound for Gaussian matrices
\citep[e.g.][Theorem~4.6.1]{vershynin19hdp} applied to
$G_2^\top/\sigma$, and using
 \eqref{assump:sampleSize} to absorb the terms involving $\zeta_T$ into the right-hand side, we have $\|D_2\|_{\mathrm{op}}
\le
c_1\sigma^2
(\sqrt{NT_{1,{\mathrm{p}}}}+N+\zeta_T)$ with probability at least
$1-\mathcal O(p_T^{-10})$. This also implies
\begin{align}\label{eq:WishartBoundD2}
\|D_2\|_{\mathrm{op}}\|\hat\Lambda_c^{-1}\|_{\mathrm{op}}
\le
c_1
\frac{\sigma^2T}{\gamma_{\min}^2}
\left(
\sqrt{\frac{N}{T_{1,{\mathrm{p}}}}}
+
\frac{N+\zeta_T}{T_{1,{\mathrm{p}}}}
\right)
\le c_1 \, \theta^2 \leq  \frac14,
\end{align}
where the last inequality follows from  \eqref{assump:smallNoise}.

\medskip
\noindent\textbf{$\bullet$ Control of $G_2K$.}
We now control $G_2K\hat\Lambda_c^{-1}$ both globally and after projection by $\Pi_N^\top U_\perp$. Conditional on $G_1$, the matrix
$K=W_{\mathrm{left}}+G_1^\top$ is deterministic, independent of $G_2$, and
$\operatorname{rank}(K)\le r$. Hence Lemma~\ref{lemma:bilinearGaussian} gives
\begin{align}
\|\Pi_N^\top U_\perp G_2K\|_{\mathrm{op}}
\le c_1\sigma\|K\|_{\mathrm{op}}\sqrt{p+r+\zeta_T} \nonumber \\
\|G_2K\|_{\mathrm{op}}
\le c_1\sigma\|K\|_{\mathrm{op}}\sqrt{N+r+\zeta_T}
\label{eq:left-centered-G2K}
\end{align}
with probability at least
$1-\mathcal O(p_T^{-10})$. The same bounds hold unconditionally with the same probability. Moreover, on
the event
$\|G_1\|_{\mathrm{op}}\le
c_1\sigma(\sqrt{T_{1,{\mathrm{p}}}}+\sqrt{r+\zeta_T})$, which has
probability at least
$1-\mathcal O(p_T^{-10})$ again by Lemma~\ref{lemma:bilinearGaussian}, we have
$\|K\|_{\mathrm{op}}
\le
\|W_{\mathrm{left}}\|_{\mathrm{op}}+\|G_1\|_{\mathrm{op}}
\le
c_1 \,  (
\gamma_{\max} \, \rho_T^{1/2}
+\sigma\sqrt{T_{1,{\mathrm{p}}}}
+\sigma\sqrt{r+\zeta_T})$, where the second inequality follows from Lemma~\ref{lemma:spectrumMp}. Combining this bound with \eqref{eq:left-centered-G2K} and
(i) gives, on an event of probability at
least $1-\mathcal O(p_T^{-10})$, we have
\begin{align*}
\begin{aligned}
\|\Pi_N^\top U_\perp G_2K\|_{\mathrm{op}}
& \|\hat\Lambda_c^{-1}\|_{\mathrm{op}} \\
& \le c_1
\frac{\sigma\sqrt{p+r+\zeta_T}}
{\gamma_{\min} \, \rho_T^{1/2}} +
c_1
\frac{\sigma^2T}{\gamma_{\min}^2}
\left(
\sqrt{\frac{p+r+\zeta_T}{T_{1,{\mathrm{p}}}}}
+
\frac{\sqrt{(p+r+\zeta_T)(r+\zeta_T)}}{T_{1,{\mathrm{p}}}}
\right) \\
& \leq c_1
\frac{\sigma\sqrt{p+r+\zeta_T}}
{\gamma_{\min} \, \rho_T^{1/2}},
\end{aligned}
\end{align*}
\begin{align}
\begin{aligned}
\|G_2K\|_{\mathrm{op}}\|\hat\Lambda_c^{-1}\|_{\mathrm{op}}
&\le
c_1
\frac{\sigma\sqrt{N+r+\zeta_T}}
{\gamma_{\min} \, \rho_T^{1/2}}  +
c_1
\frac{\sigma^2T}{\gamma_{\min}^2}
\left(
\sqrt{\frac{N+r+\zeta_T}{T_{1,{\mathrm{p}}}}}
+
\frac{\sqrt{(N+r+\zeta_T)(r+\zeta_T)}}{T_{1,{\mathrm{p}}}}
\right) \\
&\le
c_1
\frac{\sigma\sqrt{N}}
{\gamma_{\min} \, \rho_T^{1/2}},
\end{aligned}
\label{eq:left-centered-G2K-final}
\end{align}
where the final inequalities in both displays follow from \eqref{assump:sampleSize},  \eqref{assump:smallNoise}.

\medskip
\noindent\textbf{$\bullet$ Haar-measure step.}
We now aim to control $\|\Pi_N^\top U_\perp S\|_{\mathrm{op}}$ using Lemma~\ref{lem:haar-frame-compression}. This is achieved through a Haar-measure argument; see
Appendix~\ref{appendix:AuxilliaryRes} for the statement of the lemma and the
relevant background material on the Stiefel manifold. The first step is to rewrite~\eqref{eq:left-centered-master-restricted} in coordinates adapted to
$\operatorname{col}(G_2)$. In this regard, let $q=\operatorname{rank}(G_2)$. Choose $(V_2,\Sigma_2)$ measurably from the eigendecomposition of $G_2^\top G_2 = V_2 \Sigma_2^2 V_2^\top$, with the positive eigenvalues sorted in decreasing order, and define
$Q:=G_2V_2\Sigma_2^{-1}\in\mathrm{St}(N-r,q)$, so that $G_2=Q\Sigma_2V_2^\top$. Since~$P_2$ is the
orthogonal projector onto $\operatorname{col}(G_2)$, we have $P_2=QQ^\top$.
Thus $S=P_2S$ implies $S=QR$, where $R:=Q^\top S\in\mathbb R^{q\times r}$.
Substituting $G_2=Q\Sigma_2V_2^\top$ and $S=QR$ into
\eqref{eq:left-centered-master-restricted}, and using
$D_2=Q(\Sigma_2^2-\sigma^2T_{1,{\mathrm{p}}}I_q)Q^\top$, gives $\Sigma_2V_2^\top KC
+
(\Sigma_2^2-\sigma^2T_{1,{\mathrm{p}}}I_q)R
=
R\hat\Lambda_c$ after
multiplying by $Q^\top$.

Now, let $Q_\perp$ be
chosen measurably so that $[Q\ Q_\perp]\in\mathbb O(N-r)$, with the last
block omitted if $q=N-r$, and set
$O:=[\,U\;\;U_\perp Q\;\;U_\perp Q_\perp\,]$. The following calculations allow us to show that $\hat{S}$ has a simple block form with respect to the basis given by $O$. Using
\eqref{eq:left-centered-Xi-blocks}, $\widehat S = U\Lambda U^\top+\Xi$ and $G_2=Q\Sigma_2V_2^\top$, we get
\[
U^\top \widehat S (U_\perp Q)
= (Q^\top U_\perp^\top \widehat S U)^\top
= (Q^\top G_2K)^\top
= (\Sigma_2V_2^\top K)^\top
= K^\top V_2\Sigma_2,
\] 
\[
(U_\perp Q)^\top \widehat S U
= Q^\top U_\perp^\top \widehat S U
= Q^\top G_2K
= \Sigma_2V_2^\top K,
\]
\[
(U_\perp Q)^\top \widehat S (U_\perp Q)
= Q^\top (G_2G_2^\top-\sigma^2 T_{1,\mathrm{p}} I_{N-r})Q
= \Sigma_2^2-\sigma^2 T_{1,\mathrm{p}} I_q.
\]
Moreover, since $\operatorname{col}(G_2)=\operatorname{col}(Q)$, we have $Q_\perp^\top G_2=0$. Hence
$U^\top \widehat S (U_\perp Q_\perp)=0$,
$(U_\perp Q)^\top \widehat S (U_\perp Q_\perp)=0$,
$(U_\perp Q_\perp)^\top \widehat S U=0$,
$(U_\perp Q_\perp)^\top \widehat S (U_\perp Q)=~0$,
and
$(U_\perp Q_\perp)^\top \widehat S (U_\perp Q_\perp)
= Q_\perp^\top (G_2G_2^\top-\sigma^2 T_{1,\mathrm{p}} I_{N-r})Q_\perp
= -\sigma^2 T_{1,\mathrm{p}} I_{N-r-q}$.
We therefore get
\[
O^\top\hat S O
=
\begin{pmatrix}
\Lambda+U^\top\Xi U & K^\top V_2\Sigma_2 & 0\\
\Sigma_2V_2^\top K & \Sigma_2^2-\sigma^2T_{1,{\mathrm{p}}}I_q & 0\\
0 & 0 & -\sigma^2T_{1,{\mathrm{p}}}I_{N-r-q}
\end{pmatrix}.
\]
This shows that, on the event that $\hat{S}$ has at least $r$ positive eigenvalues, the top-$r$ eigenspace of $\hat{S}$ is contained in the column space of $(U,U_\perp Q)$. Formally, using $S=QR$,
we have $\hat U_{\mathrm{left}}H_U
=
UC+U_\perp QR
=
O(C^\top,R^\top,0)^\top$, and since $\hat U_{\mathrm{left}}H_U$ is the aligned top-$r$ eigenspace of
$\hat S$, we can write
\[
(C^\top,R^\top,0)^\top\hat\Lambda_c
=
O^\top \hat U_{\mathrm{left}}H_U\hat\Lambda_c
=
O^\top\hat S\hat U_{\mathrm{left}}H_U
=
O^\top\hat S O(C^\top,R^\top,0)^\top .
\]
Comparing the first two block rows gives
\[
\begin{pmatrix}
\Lambda+U^\top\Xi U & K^\top V_2\Sigma_2\\
\Sigma_2V_2^\top K & \Sigma_2^2-\sigma^2T_{1,\mathrm{p}}I_q
\end{pmatrix}
\begin{pmatrix}
C\\
R
\end{pmatrix}
=
\begin{pmatrix}
C\\
R
\end{pmatrix}
\hat\Lambda_c.
\]
On the event $\lambda_r(\hat\Lambda_c)>0$, the $r$ largest eigenvalues of
$\hat S$ are positive and therefore cannot arise from the negative block
$-\sigma^2T_{1,\mathrm p}I_{N-r-q}$, hence they correspond precisely to the $r$
largest eigenvalues of the reduced block above. We thus get that $(C^\top,R^\top)^\top$ is the top-$r$ eigenspace of the reduced block, and $\hat\Lambda_c$ is the associated eigenvalue matrix. 

Now, this reduced block depends on $G_2$ only through $(\Sigma_2,V_2)$, and not
through $Q$. Moreover, since $K=W_{\mathrm{left}}+G_1^\top$ and
$U^\top\Xi U$ is a function of $G_1$, the reduced block is measurable with
respect to $\mathcal F:=\sigma(G_1,\Sigma_2,V_2)$, hence, after fixing deterministic measurable choices of eigenspaces and
of the Procrustes alignment, $(C,R,\hat\Lambda_c)$ is
$\mathcal F$-measurable. Furthermore, writing
$R=H_R\Omega_RJ_R^\top$ for the compact singular value decomposition of $R$
and $\ell=\operatorname{rank}(R)\le r$, also $H_R$ is
$\mathcal F$-measurable. For completeness, observe that we may assume $\ell \geq 1$; if $\ell = 0$ the desired bound is immediate.

On the other hand, since $G_2$ has i.i.d.~Gaussian entries, its law is
left-orthogonally invariant, in the sense that for every deterministic
$O_0\in\mathbb O(N-r)$ we have $O_0G_2\stackrel{d}{=}G_2$. Moreover,
$(O_0G_2)^\top(O_0G_2)=G_2^\top G_2$, so left multiplication changes only
the left singular subspace, from $Q$ to $O_0Q$, while leaving
$(\Sigma_2,V_2)$ unchanged. It follows that the conditional law of $Q$
given $(\Sigma_2,V_2)$ is left-invariant on $\mathrm{St}(N-r,q)$, and by uniqueness of the left-orthogonally invariant
probability measure on the Stiefel manifold~\citep[e.g.][Theorem~1.2.2 and Section~1.3.1]{Chikuse2003}, it is Haar. Furthermore, since this conditional law does not depend on the value of $(\Sigma_2, V_2)$, then $Q$ is
independent of $\sigma(\Sigma_2,V_2)$, and, since $G_1$ is independent of~$G_2$, we also have $Q\perp\!\!\!\perp\mathcal F$.

We have therefore shown that, conditional on $\mathcal F$, the matrix
$H_R$ is fixed, while $Q$ is Haar-distributed on
$\mathrm{St}(N-r,q)$ and independent of $\mathcal F$. Hence, by
Lemma~\ref{lem:haar-frame-compression}, $QH_R$ is Haar-distributed on
$\mathrm{St}(N-r,\ell)$ conditionally on $\mathcal F$.  We also note here that the first condition in Assumption~\eqref{assump:sampleSize}
implies $r\leq c_{\mathrm{dim}}(N-r)$, and therefore
$N-r\geq N/(1+c_{\mathrm{dim}})$.  Moreover, since $\ell\le r$, $t^2\asymp\zeta_T$, and
$r+\zeta_T\le c_{\mathrm{dim}}(N-r)$, choosing
$c_{\mathrm{dim}}>0$ sufficiently small ensures $\sqrt{\ell}+t\le {\sqrt{N-r}}/{2}$
as required for the simplified bound in
Lemma~\ref{lem:haar-frame-compression}. We then obtain $\|\Pi_N^\top U_\perp QH_R\|_{\mathrm{op}}
\le
c_1
\sqrt{(p+r+\zeta_T)/(N-r)}
\le
c_1
\sqrt{(p+r+\zeta_T)/ N}$
with conditional probability at least $1-\mathcal O(p_T^{-10})$. The same bound also holds unconditionally. Since $S=QR=QH_R\Omega_RJ_R^\top$ and
$\|\Omega_R\|_{\mathrm{op}}=\|R\|_{\mathrm{op}}=\|QR\|_{\mathrm{op}}
=\|S\|_{\mathrm{op}}$, we obtain
\begin{equation}
\|\Pi_N^\top U_\perp S\|_{\mathrm{op}}
\le \|\Pi_N^\top U_\perp QH_R\|_{\mathrm{op}}\,
\|\Omega_R\|_{\mathrm{op}}\,
\|J_R^\top\|_{\mathrm{op}} \le
c_1
\sqrt{\frac{p+r+\zeta_T}{N}}
\|S\|_{\mathrm{op}} .
\label{eq:left-centered-random-range-step}
\end{equation}
The prefactor $\sqrt{(p+r+\zeta_T)/N}$ is precisely the projection factor that we aimed to obtain, as discussed at the beginning of the proof.

\medskip
\noindent\textbf{$\bullet$ Conclusion.}
It remains to bound $\|S\|_{\mathrm{op}}$. From
\eqref{eq:left-centered-master-restricted} we get
$S=G_2KC\hat\Lambda_c^{-1}+D_2S\hat\Lambda_c^{-1}$, and, since
$\|C\|_{\mathrm{op}}\le1$, the triangle inequality gives $\|S\|_{\mathrm{op}}
\le
\|G_2K\|_{\mathrm{op}}\|\hat\Lambda_c^{-1}\|_{\mathrm{op}}
+
\|D_2\|_{\mathrm{op}}\|\hat\Lambda_c^{-1}\|_{\mathrm{op}}
\|S\|_{\mathrm{op}}$. The preceding Wishart bound for $D_2$ in~\eqref{eq:WishartBoundD2} allows reordering this inequality, and gives
\begin{equation}
\|S\|_{\mathrm{op}}
\le
c_1\|G_2K\|_{\mathrm{op}}\|\hat\Lambda_c^{-1}\|_{\mathrm{op}} \leq c_1 \frac{\sigma \sqrt{N}}{\gamma_\mathrm{min} \rho_T^{1/2}}
\label{eq:left-centered-S-bound}
\end{equation}
by~\eqref{eq:left-centered-G2K-final}. This, together with \eqref{eq:left-centered-random-range-step}, yields $\|\Pi_N^\top U_\perp S\|_{\mathrm{op}}
\le c_1 \sigma \gamma_\mathrm{min}^{-1} \sqrt{(p+r+\zeta_T) /\rho_T}  $. Finally, since $C=(I_r-S^\top S)^{1/2}$ and
$\|C-I_r\|_{\mathrm{op}}\le \|S\|_{\mathrm{op}}^2$, we have $\|\Pi_N^\top U(C-I_r)\|_{\mathrm{op}}
\le
\|\Pi_N^\top U\|_{\mathrm{op}}\|S\|_{\mathrm{op}}^2 \leq c_1 \frac{\sigma^2 N}{ \gamma_\mathrm{min}^2 \rho_T} \, \|\Pi_N^\top U\|_{\mathrm{op}}$. This proves
\eqref{eq:lemma13_centered_projected}, and concludes the proof.
\end{proof}

\medskip \medskip \medskip
Many useful corollaries can be derived from~\eqref{eq:lemma13_centered_projected}, yielding bounds that hold with probability at least $1 - \mathcal O(p_T^{-10})$. In particular, 
for any fixed $x\in \mathcal S^{N-1}$, by choosing $\Pi_N=x$ we have
\begin{align}\label{eq:boundU_vectorX}
  \|(\hat U_{\mathrm{left}}H_U-U)^\top x\|_2
\leq c_1
\frac{\sigma\sqrt{r+\zeta_T}}
{\gamma_{\min} \, \rho_T^{1/2}} +
c_1 \, \frac{\sigma^2 N}{\gamma_\mathrm{min}^2 \, \rho_T}\|U^\top x\|_2
\end{align}
with high probability,  which gives an estimation bound directly for the action of the error
$\hat U_{\mathrm{left}}H_U-U$ along any fixed unit direction. Furthermore, for fixed $k\in[K]$ the same bounds hold blockwise, in the sense that, for every fixed $1\le p\le N_{1k}$ and
$\Pi_N\in\mathbb R^{N_{1k}\times p}$ with $\Pi_N^\top\Pi_N=I_p$, we have
\begin{align}
\begin{aligned}
\|\Pi_N^\top(\hat U_{1k}H_U-U_{1k})\|_{\mathrm{op}}
&\le
c_1
\frac{\sigma\sqrt{p+r+\zeta_T}}
{\gamma_{\min} \, \rho_T^{1/2}} +
c_1 \, \frac{\sigma^2 N}{\gamma_\mathrm{min}^2 \, \rho_T}\|\Pi_N^\top U_{1k}\|_{\mathrm{op}}.
\end{aligned}
\label{eq:lemma13_centered_block1}
\end{align}
As a result, choosing $\Pi_N = I_{N_{1k}}$ and using  \eqref{assump:subblock-conditioning},  \eqref{assump:sampleSize} give $ \|\hat U_{1k} \,H_U-U_{1k} \|_\mathrm{op}
\ \le\ c_1 \, \sigma\, \gamma_\mathrm{min}^{-1} \sqrt{N_{1k}/\rho_T}$
and, as a byproduct, $\hat U_{1k}^\top \hat U_{1k}$ is invertible, satisfying $\frac{c_\ell}{2}\frac{N_{1k}}{N}I_r
 \preceq
\hat U_{1k}^\top \hat U_{1k}
 \preceq 
2c_u\frac{N_{1k}}{N}I_r$. This follows from Weyl's inequality (Lemma~\ref{lemma:Weyl}), along similar lines to the proof of~\eqref{eq:lemma13_centered_lambdalower}. This is particularly useful in our setting, as it ensures that the matrix $\hat H_{k,\tau}^\mathrm{inv}$ used in Algorithm~\ref{alg:bilinear4block} coincides with $(\hat U_{1k}^\top \hat U_{1k})^{-1}$ with high probability whenever the algorithm is run with $\tau \leq \frac{c_\ell \, N_{1k}}{2 \, N}$. Indeed, in this case we have 
\begin{align}\label{eq:HDaggerCoincidesWithInverse}
   \left\{\frac{c_\ell}{2}\frac{N_{1k}}{N}I_r
 \preceq
\hat H_k
 \preceq 
2c_u\frac{N_{1k}}{N}I_r\right\}  &= \left\{\frac{c_\ell}{2}\frac{N_{1k}}{N}I_r
 \preceq
\hat H_k
\right\} \bigcap \left\{ \hat H_k
 \preceq 
2c_u\frac{N_{1k}}{N}I_r\right\}  \nonumber \\
& \subseteq \left\{\frac{c_\ell}{2}\frac{N_{1k}}{N}I_r
 \preceq
\hat H_k
\right\} \subseteq \left\{\tau \, I_r
 \preceq
\hat H_k
\right\},
\end{align}
which further implies that $\mathbb P(\hat H_k \succeq \tau \, I_r) \geq 1 - \mathcal{O}(p_T^{-10})$. On this event, all eigenvalues of $\hat H_k$ are at least $\tau$, hence
$\lambda_i\vee\tau=\lambda_i$ for every $i \in [r]$, and $\hat H_{k,\tau}^{\mathrm{inv}}
=
\hat H_k^{-1}
=
(\hat U_{1k}^\top \hat U_{1k})^{-1}$.

\medskip
\medskip
\medskip

\begin{lemma}
\label{lemma:centered_inverse_comparison}
Let $\Lambda, \hat \Lambda_c, C$ be as in Lemma~\ref{lemma13_centered}, and suppose that the
assumptions of Lemma~\ref{lemma13_centered} hold. There exists a
constant $c_1\equiv
c_1(c_\ell,c_u,c_0,c_{\mathrm{dim}}, c_{\mathrm{blk}},\kappa)>0$ such that, with probability at least
$1-\mathcal O(p_T^{-10})$, we have
\[
\bigl\|
C\widehat\Lambda_c^{-1}-\Lambda^{-1}
\bigr\|_{\mathrm{op}}
\leq
c_1
\sigma\gamma_{\min}^{-3}
\sqrt N\,\rho_T^{-3/2}.
\]
\end{lemma}

\begin{proof}
All the bounds below will hold on the events from Lemma~\ref{lemma13_centered}, which have probability at least
$1-\mathcal O(p_T^{-10})$. Using $U_\perp^\top\Xi U=G_2K$, $C\widehat\Lambda_c-\Lambda
=
\Lambda(C-I_r)+U^\top\Xi U\,C+U^\top\Xi U_\perp S$, $\Lambda-\widehat\Lambda_c
=
\Lambda-C\widehat\Lambda_c+(C-I_r)\widehat\Lambda_c$, and $ C\widehat\Lambda_c^{-1}-\Lambda^{-1}
=
(C-I_r)\widehat\Lambda_c^{-1}
+
\widehat\Lambda_c^{-1}
(\Lambda-\widehat\Lambda_c)\Lambda^{-1}$ from Lemma~\ref{lemma13_centered}, we can write
\begin{align}\label{eq:CLambdaInv}
&\|C\widehat\Lambda_c^{-1}-\Lambda^{-1}\|_{\mathrm{op}}
\leq
\|C-I_r\|_{\mathrm{op}}
\|\widehat\Lambda_c^{-1}\|_{\mathrm{op}} +
\|\widehat\Lambda_c^{-1}\|_{\mathrm{op}}
\|\Lambda-\widehat\Lambda_c\|_{\mathrm{op}}
\|\Lambda^{-1}\|_{\mathrm{op}} \nonumber \\
&\leq
\|C-I_r\|_{\mathrm{op}}
\|\widehat\Lambda_c^{-1}\|_{\mathrm{op}} +
\|\widehat\Lambda_c^{-1}\|_{\mathrm{op}}
\|\Lambda-C\widehat\Lambda_c\|_{\mathrm{op}}
\|\Lambda^{-1}\|_{\mathrm{op}} +
\|\widehat\Lambda_c^{-1}\|_{\mathrm{op}}
\|C-I_r\|_{\mathrm{op}}
\|\widehat\Lambda_c\|_{\mathrm{op}}
\|\Lambda^{-1}\|_{\mathrm{op}} \nonumber \\
&\leq
\|C-I_r\|_{\mathrm{op}}
\|\widehat\Lambda_c^{-1}\|_{\mathrm{op}} +
\|\widehat\Lambda_c^{-1}\|_{\mathrm{op}}
\|\Lambda\|_{\mathrm{op}}
\|C-I_r\|_{\mathrm{op}}
\|\Lambda^{-1}\|_{\mathrm{op}} \nonumber \\
& \qquad +
\|\widehat\Lambda_c^{-1}\|_{\mathrm{op}}
\|U^\top\Xi U\|_{\mathrm{op}}
\|C\|_{\mathrm{op}}
\|\Lambda^{-1}\|_{\mathrm{op}} +
\|\widehat\Lambda_c^{-1}\|_{\mathrm{op}}
\|U^\top\Xi U_\perp S\|_{\mathrm{op}}
\|\Lambda^{-1}\|_{\mathrm{op}} \nonumber \\
& \qquad +
\|\widehat\Lambda_c^{-1}\|_{\mathrm{op}}
\|C-I_r\|_{\mathrm{op}}
\|\widehat\Lambda_c\|_{\mathrm{op}}
\|\Lambda^{-1}\|_{\mathrm{op}}.
\end{align}
We now bound the five terms on the right-hand side individually. For the first one, combining $\|C-I_r\|_{\mathrm{op}}
\leq
\|S\|_{\mathrm{op}}^2$, $\|\Lambda^{-1}\|_{\mathrm{op}}
\lesssim
\gamma_{\min}^{-2}\rho_T^{-1}$, \eqref{eq:left-centered-S-bound} and~\eqref{eq:lemma13_centered_lambdalower} gives
\begin{align*}
\|C-I_r\|_{\mathrm{op}}
\|\widehat\Lambda_c^{-1}\|_{\mathrm{op}}
&\leq c_1
\frac{\sigma^2N}
{\gamma_{\min}^4\rho_T^2} = c_1
\frac{\sigma\sqrt N}
{\gamma_{\min}\rho_T^{1/2}}
\frac{\sigma\sqrt N}
{\gamma_{\min}^3\rho_T^{3/2}} \leq c_1
\frac{\sigma\sqrt N}
{\gamma_{\min}^3\rho_T^{3/2}},
\end{align*}
where the last inequality follows from \eqref{assump:smallNoise}. The second term can be controlled by combining the previous bound with $\|\Lambda\|_{\mathrm{op}} \|\Lambda^{-1}\|_{\mathrm{op}} \leq c_1 \gamma_\mathrm{max}^2 \rho_T \, \gamma_\mathrm{min}^{-2} \rho_T^{-1} = \kappa^2 c_1$ and incorporating the constant $\kappa^2$ into $c_1$. A similar argument applies to the fifth term, using the bound $\|\widehat{\Lambda}_c\|_{\mathrm{op}} \lesssim \gamma_{\max}^2\rho_T$.

For the third term in~\eqref{eq:CLambdaInv}, we can use $\|C\|_{\mathrm{op}}\leq 1$, $\|\Lambda^{-1}\|_{\mathrm{op}}
\lesssim
\gamma_{\min}^{-2}\rho_T^{-1}$, \eqref{eq:lemma13_centered_lambdalower} and the second line in~\eqref{eq:UXiU} to write  
\begin{align*}
&\|\widehat\Lambda_c^{-1}\|_{\mathrm{op}}
\|U^\top\Xi U\|_{\mathrm{op}}
\|C\|_{\mathrm{op}}
\|\Lambda^{-1}\|_{\mathrm{op}} \le c_1 \frac{\|U^\top\Xi U\|_{\mathrm{op}}}{\lambda_r(\Lambda)} \frac{1}{\lambda_r(\Lambda)} \leq c_1 
\frac{\sigma \sqrt{N}}{\gamma_{\min} \rho_T^{1/2}} \frac{1}{\gamma_\mathrm{min}^2 \rho_T}
 = c_1
\frac{\sigma\sqrt N}
{\gamma_{\min}^3\rho_T^{3/2}}.
\end{align*}

For the fourth piece, \eqref{eq:left-centered-Xi-blocks} gives $U_\perp^\top\Xi U=G_2K$ hence $U^\top\Xi U_\perp S
=
K^\top G_2^\top S$. As a result, combining this with~\eqref{eq:left-centered-S-bound} and~\eqref{eq:left-centered-G2K-final} therefore implies
\begin{align*}
\|\widehat\Lambda_c^{-1}\|_{\mathrm{op}}
\|U^\top\Xi U_\perp S\|_{\mathrm{op}}
\|\Lambda^{-1}\|_{\mathrm{op}} & \leq \|\widehat\Lambda_c^{-1}\|_{\mathrm{op}}
\|G_2 K\|_{\mathrm{op}} \|S\|_{\mathrm{op}}
\|\Lambda^{-1}\|_{\mathrm{op}} \leq c_1 \frac{\sigma\sqrt N}
{\gamma_{\min} \rho_T^{1/2}} \|S\|_{\mathrm{op}}
\|\Lambda^{-1}\|_{\mathrm{op}} \\
&\leq c_1 \frac{\sigma^2 N}
{\gamma_{\min}^2 \rho_T} \frac{1}{\gamma_\mathrm{min}^2 \rho_T}
= c_1 \frac{\sigma \sqrt{N}}{\gamma_{\min}\rho_T^{1/2}}
\frac{\sigma\sqrt N}
{\gamma_{\min}^3\rho_T^{3/2}} \leq  c_1
\frac{\sigma\sqrt N}
{\gamma_{\min}^3\rho_T^{3/2}},
\end{align*}
where the last inequality follows from \eqref{assump:smallNoise}. Combining the five bounds concludes the proof.
\end{proof}

\medskip\medskip\medskip
We next provide a representation of $\widehat U_{\mathrm{left}}H_U-U$ in terms of a Gaussian term and a remainder.

\begin{cor}
\label{cor:centered-left-perturbation}
Use the assumptions and notation of Lemma~\ref{lemma13_centered}. Set $
\Psi_U:= \widehat U_{\rm left}H_U-U
-
E_{\rm left}^{\rm p}W_\mathrm{left}(W_\mathrm{left}^\top W_\mathrm{left})^{-1}$, and for fixed $k\in[K]$ define $\Psi_{U,2k}:=(\Psi_U)_{\mathcal I_{k},\bullet}$, where $\mathcal I_{k} = \{N_{1k}+1, \ldots, N\}$. For every fixed
$x\in\mathcal S^{N_{2k}-1}$, there exists a constant
$c_1\equiv c_1(c_\ell,c_u,c_0,c_{\mathrm{dim}}, c_{\mathrm{blk}},\kappa)>0$ such that, with
probability at least $1-\mathcal O(p_T^{-10})$, we have
\begin{align}
\|(\Psi_{U,2k})^\top x\|_2
\le\;&
c_1
\frac{\sigma^2\sqrt{(N+T_{1,\mathrm{p}})(r+\zeta_T)}}{
\gamma_{\min}^2\rho_T}
+
c_1\left(
\frac{\sigma^2N}{\gamma_{\min}^2\rho_T}
+
\frac{\sigma\sqrt{r+\zeta_T}}{
\gamma_{\min}\rho_T^{1/2}}
\right)
\|U_{2k}^\top x\|_2 .
\label{cor:bound_Psi}
\end{align}
\end{cor}

\begin{proof}
Let $\bar x\in\mathcal S^{N-1}$
be the zero extension of $x$ to the coordinates $\mathcal I_{k}$, i.e.~$\bar x_i = x_{i-N_{1k}} \, \mathbbm 1\{i\in\mathcal I_{k}\}$ for all $i \in [N]$. Then
$\|(\Psi_{U,2k})^\top x\|_2=\|\Psi_U^\top\bar x\|_2$ and
$\|U^\top\bar x\|_2=\|U_{2k}^\top x\|_2$.

We now recall some important facts that were already stated and justified in the proof of Lemma~\ref{lemma13_centered}. For readability write $E:=E_{\rm left}^{\rm p}$,
$Y:=Y_{\rm left}^{\rm p}$, $W := W_\mathrm{left}$, $M:=M_{\rm left}^{\rm p}=UW^\top$, and
$\Lambda:=W^\top W$. Let $U_\perp\in\mathbb R^{N\times(N-r)}$ be such that
$[U\ U_\perp]\in\mathbb O(N)$, and define
$G_1:=U^\top E$ and $G_2:=U_\perp^\top E$. Since $E$ has independent
$\mathcal N(0,\sigma^2)$ entries, rotational invariance gives that $G_1$ and
$G_2$ are independent Gaussian matrices with independent
$\mathcal N(0,\sigma^2)$ entries. Set $\widehat S:=YY^\top-\sigma^2T_{1,\mathrm{p}}I_N$ and
$\widehat\Lambda_c:=H_U^\top(\widehat\Sigma_{\rm left}^2
-\sigma^2T_{1,\mathrm{p}}I_r)H_U$. Since subtracting
$\sigma^2T_{1,\mathrm{p}}I_N$ does not change eigenvectors,
$\widehat U_{\rm left}H_U$ satisfies
$\widehat S\widehat U_{\rm left}H_U
=\widehat U_{\rm left}H_U\widehat\Lambda_c$. Write
$\widehat U_{\rm left}H_U=UC+U_\perp S$, where
$C:=U^\top\widehat U_{\rm left}H_U$ and
$S:=U_\perp^\top\widehat U_{\rm left}H_U$. Because $H_U$ is the Procrustes
alignment, $C$ is symmetric positive semidefinite and
$C=(I_r-S^\top S)^{1/2}$. Expanding $\widehat S-U\Lambda U^\top$ gives
$\Xi:=\widehat S-U\Lambda U^\top
=UW^\top E^\top+EWU^\top+(EE^\top-\sigma^2T_{1,\mathrm{p}}I_N)$. Therefore
$U_\perp^\top\Xi U=G_2W+G_2G_1^\top=G_2(W+G_1^\top)$, and projecting onto
$U_\perp$ gives
\begin{align}\label{eq:S_D_2}
S
=
G_2(W+G_1^\top)C\widehat\Lambda_c^{-1}
+
D_2S\widehat\Lambda_c^{-1},
\qquad
D_2:=(G_2G_2^\top-\sigma^2T_{1,\mathrm{p}}I_{N-r})P_2,
\end{align}
where $P_2$ denotes the orthogonal projector onto the column space of $G_2$. In particular, this follows from $S=P_2S$, which holds under $\widehat\Lambda_c\succ0$. Also, since $EW=UG_1W+U_\perp G_2W$, the definition of $\Psi_U$ gives $\Psi_U
=
U_\perp\{S-G_2W\Lambda^{-1}\}
+
U\{C-I_r-G_1W\Lambda^{-1}\}$,
consequently we have
\begin{align}\label{eq:PSIU_proof}
   \|\Psi_U^\top\bar x\|_2
\le
\|\{S-G_2W\Lambda^{-1}\}^\top U_\perp^\top\bar x\|_2
+
\|(C-I_r)U^\top\bar x\|_2
+
\|\{G_1W\Lambda^{-1}\}^\top U^\top\bar x\|_2 . 
\end{align}

We now bound these three terms separately. First, subtracting $G_2W\Lambda^{-1}$ from
the first equation in~\eqref{eq:S_D_2}
yields
\begin{align}\label{eq:S-G_2WLambda}
S-G_2W\Lambda^{-1}
=
G_2W(C\widehat\Lambda_c^{-1}-\Lambda^{-1})
+
G_2G_1^\top C\widehat\Lambda_c^{-1}
+
D_2S\widehat\Lambda_c^{-1}.
\end{align}
We recall that on the high-probability event from 
Lemma~\ref{lemma13_centered} we have $\lambda_r(\widehat\Lambda_c)\ge 3\lambda_r(\Lambda)/4$,
$\|\widehat\Lambda_c^{-1}\|_{\rm op}\le c_1\lambda_r(\Lambda)^{-1}$,
$\|S\|_{\rm op}\le c_1\sigma \gamma_{\min}^{-1} \sqrt{N/\rho_T}$, and
$\|C-I_r\|_{\rm op}\le \|S\|_{\rm op}^2$. Moreover, arguing as in the paragraph after~\eqref{eq:left-centered-lambdahat-lower} and simplifying some bounds using \eqref{assump:sampleSize}, the same event also gives
$\|G_2W\|_{\rm op}\le c_1\sigma\|W\|_{\rm op}\sqrt{N}$,
$\|G_1W\|_{\rm op}\le c_1\sigma\|W\|_{\rm op}\sqrt{r+\zeta_T}$, and
$\|G_2G_1^\top\|_{\rm op}\le c_1\sigma^2 
\sqrt{NT_{1,\mathrm{p}}}$. Now, using $\lambda_r(\Lambda)\ge c_1\gamma_{\min}^2\rho_T$, the middle term of~\eqref{eq:S-G_2WLambda} above satisfies
\[
\|G_2G_1^\top C\widehat\Lambda_c^{-1}\|_{\rm op}
\le
c_1
\frac{\sigma^2\sqrt{N T_{1,\mathrm{p}}}}{
\gamma_{\min}^2\rho_T}.
\]
We now deal with the other two terms. First, combining $\|C\widehat\Lambda_c^{-1}-\Lambda^{-1}\|_{\rm op}
\le c_1 \sigma \gamma_\mathrm{min}^{-3} \, \sqrt{N} \rho_T^{-3/2}$ from Lemma~\ref{lemma:centered_inverse_comparison} with the above
bound for $\|G_2W\|_{\rm op}$ gives
$\|G_2W(C\widehat\Lambda_c^{-1}-\Lambda^{-1})\|_{\rm op}
\le c_1\sigma^2 \gamma_{\min}^{-2} N/ \rho_T$. Similarly, the restricted
Wishart bound~\eqref{eq:WishartBoundD2} for $D_2$ and the bound on $\|S\|_{\rm op}$ in~\eqref{eq:left-centered-S-bound} give
$\|D_2S\widehat\Lambda_c^{-1}\|_{\rm op}
\le c_1\sigma^2 \gamma_{\min}^{-2} N/\rho_T$, thereby implying
\[
\|S-G_2W\Lambda^{-1}\|_{\rm op}
\le
c_1
\frac{\sigma^2\sqrt{T_{1,\mathrm{p}}N }}{
\gamma_{\min}^2\rho_T}
+
c_1
\frac{\sigma^2 N}{\gamma_{\min}^2\rho_T} .
\]

It remains to convert this operator bound into a Euclidean norm bound. Let
$G_2=Q\Sigma_2V_2^\top$ be its singular value decomposition. Arguing as in the proof of Lemma~\ref{lemma13_centered}, conditional on
$\sigma(G_1,\Sigma_2,V_2)$, the factor~$Q$ is Haar-distributed on the
appropriate Stiefel manifold. Since $S=QR$ for an
$\sigma(G_1,\Sigma_2,V_2)$-measurable matrix~$R$ and
$G_2W\Lambda^{-1}=Q\Sigma_2V_2^\top W\Lambda^{-1}$, we can write
$S-G_2W\Lambda^{-1}=Q\widetilde R$, where $\widetilde R$ is $\sigma(G_1,\Sigma_2,V_2)$-measurable and has rank at most $r$. Now, let $\ell:=\operatorname{rank}(\widetilde R)\le r$, and consider the compact SVD of $\widetilde R=L_{\widetilde R}D_{\widetilde R}M_{\widetilde R}^{\top}$ with $
L_{\widetilde R}^{\top}L_{\widetilde R}=~I_\ell$, where the factors may be chosen $\sigma(G_1,\Sigma_2,V_2)$-measurable. Conditional
on $\sigma(G_1,\Sigma_2,V_2)$, the matrix $L_{\widetilde R}$ is fixed and
$Q$ is Haar-distributed on $\mathrm{St}(N-r,q)$. As a result, 
Lemma~\ref{lem:haar-frame-compression}, applied with $d=N-r$,
$H=L_{\widetilde R}$, $A=\bar x^\top U_\perp$, and
$t^2\asymp \zeta_T$, gives $\|L_{\widetilde R}^{\top}Q^{\top}U_\perp^{\top}\bar x\|_2
\le
c_1\|\bar x^\top U_\perp\|_2
\sqrt{(r+\zeta_T)/(N-r)} \leq c_1\|\bar x^\top U_\perp\|_2
\sqrt{(r+\zeta_T)/ N}$, where the last inequality follows from \eqref{assump:sampleSize}. Finally, using $\|\bar x^\top U_\perp\|_2\le \|\bar x\|_2 = 1$,
we obtain 
\begin{align*}
    \|\widetilde R^\top Q^\top U_\perp^\top\bar x\|_2 &=\|M_{\widetilde R}D_{\widetilde R}L_{\widetilde R}^\top Q^\top U_\perp^\top\bar x\|_2\le \|D_{\widetilde R}\|_{\rm op}\|L_{\widetilde R}^\top Q^\top U_\perp^\top\bar x\|_2\\
    &=\|\widetilde R\|_{\rm op}\|L_{\widetilde R}^\top Q^\top U_\perp^\top\bar x\|_2\le 
c_1\sqrt{\frac{r+\zeta_T}{N}}\|\widetilde R\|_{\rm op}
\end{align*}
with probability at
least $1-\mathcal O(p_T^{-10})$, which therefore implies
\[
\|\{S-G_2W\Lambda^{-1}\}^\top U_\perp^\top\bar x\|_2
\le
c_1
\frac{\sigma^2\sqrt{(N+T_{1,\mathrm{p}})(r+\zeta_T)}}{
\gamma_{\min}^2\rho_T}.
\]

This concludes the analysis for the first and more problematic term in~\eqref{eq:PSIU_proof}. For the second one, using $\|C-I_r\|_{\rm op}\le \|S\|_{\rm op}^2$ and $\|S\|_{\rm op}\le c_1\sigma \, \gamma_{\min}^{-1} \sqrt{N/\rho_T}$ from the proof of Lemma~\ref{lemma13_centered}, we get
$\|(C-I_r)U^\top\bar x\|_2
\le c_1\sigma^2 \, \gamma_{\min}^{-2} N\rho_T^{-1} \,
\|U^\top\bar x\|_2 = c_1\sigma^2 \, \gamma_{\min}^{-2} N \rho_T^{-1} \,
\|U^\top_{2k} x\|_2$. For the third term, Lemma~\ref{lemma:bilinearGaussian} gives $\|G_1W\|_{\rm op}\le
c_1\sigma\|W\|_{\rm op}\sqrt{r+\zeta_T}$ with probability at least
$1-\mathcal O(p_T^{-10})$. Also, Lemma~\ref{lemma:spectrumMp} gives
$\|W\|_{\rm op}\le c_u^{1/2}\gamma_{\max}\rho_T^{1/2}$ and
$\|\Lambda^{-1}\|_{\rm op}=\lambda_r(\Lambda)^{-1}
=\sigma_r(W)^{-2}\le c_\ell^{-1}\gamma_{\min}^{-2}\rho_T^{-1}$. As a result, we have
$\|G_1W\Lambda^{-1}\|_{\rm op}\le
\|G_1W\|_{\rm op}\|\Lambda^{-1}\|_{\rm op}
\le
c_1\sigma\sqrt{r+\zeta_T}\,
\gamma_{\max}\rho_T^{1/2}
(\gamma_{\min}^2\rho_T)^{-1}
\le
c_1\sigma \gamma_{\min}^{-1} \sqrt{(r+\zeta_T)/\rho_T}$, which further implies $\|\{G_1W\Lambda^{-1}\}^\top U^\top\bar x\|_2
\le c_1\sigma \gamma_{\min}^{-1}\sqrt{(r+\zeta_T)/\rho_T} \, \|U_{2k}^\top x\|_2$.

Combining the three bounds concludes the proof.
\end{proof}

\medskip\medskip
As a sanity check, combining~\eqref{cor:bound_Psi} and $
\widehat U_{\rm left}H_U-U
= \Psi_U +
E_{\rm left}^{\rm p}W_\mathrm{left}(W_\mathrm{left}^\top W_\mathrm{left})^{-1}$ with Lemmas~\ref{lemma:spectrumMp} and~\ref{lemma:bilinearGaussian} and  Assumptions \eqref{assump:sampleSize},  \eqref{assump:smallNoise}, allows proving a bound for $\|(\widehat U_{\rm left}H_U-U)^\top x\|_2$ that agrees with~\eqref{eq:boundU_vectorX}. 

Moreover, by applying Lemma~\ref{lemma13_centered} and Corollary~\ref{cor:centered-left-perturbation} to $Y_\mathrm{up}^\top$ we get the following corollary. 

\begin{cor}
\label{lemma13_Transpose}
Grant Assumptions  \eqref{assump:subblock-conditioning}
with fixed constants $0<c_\ell\le c_u$, \eqref{assump:sampleSize} and \eqref{assump:smallNoise}. Suppose further that $0<\gamma_{\min}\le \sigma_{\min}(\mathcal C_{\bullet,\bullet,j})
\le \sigma_{\max}(\mathcal C_{\bullet,\bullet,j})\le \gamma_{\max}<\infty$  for all
$j\in[K]$, and let $\kappa:=\gamma_{\max}/\gamma_{\min}$. Write
$Y_{\mathrm{up}}^{\mathrm{p}}
=M_{\mathrm{up}}^{\mathrm{p}}+E_{\mathrm{up}}^{\mathrm{p}}$, with 
$M_{\mathrm{up}}^{\mathrm{p}}=W_{\mathrm{up}}V^\top$. Let
$(\hat U_{\mathrm{up}},\hat\Sigma_{\mathrm{up}},\hat V_{\mathrm{up}})
:=\operatorname{SVD}_r(Y_{\mathrm{up}}^{\mathrm{p}})$ and
$H_V:=\operatorname{sgn}(\hat V_{\mathrm{up}}^\top V)$, and define $
\Psi_V:= \widehat V_{\rm up}H_V-V
-
(E_{\rm up}^{\rm p})^\top W_\mathrm{up}(W_\mathrm{up}^\top W_\mathrm{up})^{-1}$. For fixed $k \in [K]$ we also set $\Psi_{V,2k}:=(\Psi_V)_{\mathcal J_{k},\bullet}$, where $\mathcal J_k = \{T_{1k} + 1, \ldots, T\}$. Fix also  $\bar y \in \mathcal S^{T-1}$ and  $y \in \mathcal S^{T_{2k}-1}$. There exists a
constant $c_1\equiv c_1(c_\ell,c_u,c_0,c_{\mathrm{dim}}, c_{\mathrm{blk}},\kappa)>~0$ such
that, with probability at least $1-\mathcal O(p_N^{-10})$,
the following statements hold:
\begin{enumerate}
    \item[(i)] We have 
\begin{align}\label{eq:DeltaVtranspose}
        \|(\hat V_\mathrm{up} H_V - V)^\top \, \bar y\|_2 \leq c_1 \frac{\sigma \sqrt{r + \zeta_N}}{\gamma_\mathrm{min} \, \rho_N^{1/2}} + c_1\frac{\sigma^2 T}{\gamma_\mathrm{min}^2 \, \rho_N} \|V^\top \bar y\|_2.
    \end{align} 
        \item[(ii)] We can bound the operator norm as \begin{align}\label{eq:DeltaVtransposeOP} \|\hat V_\mathrm{up} H_V - V\|_\mathrm{op} \leq c_1 \frac{\sigma \sqrt{T}}{\gamma_\mathrm{min}\,\rho_N^{1/2}}.
    \end{align} 
\item[(iii)] We have 
    \begin{align}\label{eq:PsiVranspose}
    \|\Psi_{V,2k}^\top \, y\|_2 \leq c_1 \frac{\sigma^2\sqrt{(T+N_{1,\mathrm{p}})(r+\zeta_N)}}{
\gamma_{\min}^2\rho_N}
+ c_1
\left(
\frac{\sigma^2T}{\gamma_{\min}^2\rho_N}
+ 
\frac{\sigma\sqrt{r+\zeta_N}}{
\gamma_{\min}\rho_N^{1/2}}
\right)
\|V_{2k}^\top y\|_2.
\end{align}
\end{enumerate}
\end{cor}

\begin{proof}
    The proof follows by applying Lemmas~\ref{lemma13_centered} and~\ref{cor:centered-left-perturbation} to $(Y_\mathrm{up}^\mathrm{p})^\top = V W_\mathrm{up}^\top + (E_\mathrm{up}^\mathrm{p})^\top$. In particular, \eqref{eq:DeltaVtranspose} follows from~\eqref{eq:boundU_vectorX}, \eqref{eq:DeltaVtransposeOP} from \eqref{eq:lemma13_centered_projected} with $\Pi_T = I_T$ and \eqref{assump:smallNoise}, and \eqref{eq:PsiVranspose} from~\eqref{cor:bound_Psi}.
\end{proof}

\medskip\medskip\medskip
Returning to theoretical guarantees for quantities obtained from the SVD of $Y_{\mathrm{left}}^\mathrm{p}$, the following result controls the estimation error for $(U_{1k}^\top U_{1k})^{-1}U_{1k}^\top$, a key object for predicting the $c$-block from the $a$-block.
\begin{lemma}
\label{lemma15_centered}
Adopt the assumptions and notation of Lemma~\ref{lemma13_centered}. Fix
$k\in[K]$, write $H_k:=U_{1k}^\top U_{1k}$ and
$\widehat H_k:=(\hat U_{1k}H_U)^\top\hat U_{1k}H_U$, and set $D_k:=\widehat H_k^{-1}(\hat U_{1k}H_U)^\top-H_k^{-1}U_{1k}^\top$. 
There exists $c_1\equiv c_1(c_\ell,c_u,c_0,c_{\mathrm{dim}}, c_{\mathrm{blk}},\kappa)>0$
such that, with probability at least
$1-\mathcal O(p_T^{-10})$, we have
\begin{align}\label{eq:lemma15_centered_D_rewrite}
\|D_k\|_{\rm op}& \leq
c_1 \frac{\sigma \sqrt{N}}{\gamma_\mathrm{min} \rho_T^{1/2}} \sqrt{\frac{N}{N_{1k}}}.
\end{align}
\end{lemma}

\begin{proof}
We showed in~\eqref{eq:lemma13_centered_block1} and the discussion thereafter that
$\hat U_{1k}H_U$ has full column rank on the high-probability event of
Lemma~\ref{lemma13_centered}. On this event, letting $\Delta_k := \hat U_{1k} H_U - U_{1k}$ and using $D_k=\widehat H_k^{-1}(\hat U_{1k}H_U)^\top-H_k^{-1}U_{1k}^\top
=(\hat U_{1k}H_U)^\dagger-U_{1k}^\dagger$,  Lemma~\ref{lem:pinv_exact_identity}
gives \[
D_k
=
-\,U_{1k}^\dagger\Delta_k(\hat U_{1k}H_U)^\dagger
+
U_{1k}^\dagger(U_{1k}^\dagger)^\top\Delta_k^\top
\{I_{N_{1k}}-(\hat U_{1k}H_U)(\hat U_{1k}H_U)^\dagger\}.\]
As a result, we have
\begin{align}
    \|D_k\|_{\rm op}
& \le
\|U_{1k}^\dagger\|_{\rm op}\|\Delta_k\|_{\rm op}
\|(\hat U_{1k}H_U)^\dagger\|_{\rm op}
+
\|U_{1k}^\dagger\|_{\rm op}^2\|\Delta_k\|_{\rm op}
\nonumber \\
&\le
(\sqrt{2}+1)c_\ell^{-1}\frac{N}{N_{1k}}\|\Delta_k\|_{\rm op} \leq c_1 \frac{\sigma \sqrt{N}}{\gamma_\mathrm{min} \rho_T^{1/2}} \sqrt{\frac{N}{N_{1k}}},
\label{eq:lemma15_basic_rewrite}
\end{align}
where the penultimate inequality follows from  \eqref{assump:subblock-conditioning} and $\frac{c_\ell}{2}\frac{N_{1k}}{N}I_r
 \preceq
\hat U_{1k}^\top \hat U_{1k}
 \preceq 
2c_u\frac{N_{1k}}{N}I_r$, while the last one follows from the discussion right after~\eqref{eq:lemma13_centered_block1}.
This completes the proof.
\end{proof}

\begin{lemma}
\label{lemmaL_action_centered}
Grant the assumptions of Lemma~\ref{lemma13_centered}. Fix $k\in[K]$, and recall $H_k=U_{1k}^\top U_{1k}, \,\hat H_k= H_U^\top \hat U_{1k}^\top\hat U_{1k}H_U$. On the high probability event where $\hat H_k$ is invertible, define $L:=
\hat U_{2k}H_U
\hat H_k^{-1}
H_U^\top \hat U_{1k}^\top
-
U_{2k}H_k^{-1} U_{1k}^\top$. Writing 
\begin{align}
 \Delta_L := L \, U_{1k}\mathcal C_{\bullet,\bullet,k}
-
(E_{\rm left}^{\rm p})_{\{N_{1k}+1,\ldots,N\},\bullet}
W_{\rm left}(W_{\rm left}^\top W_{\rm left})^{-1}
\mathcal C_{\bullet,\bullet,k},
\label{eq:lemma16_centered_decomp_rewrite}
\end{align}
for every fixed $x\in\mathcal S^{N_{2k}-1}$ there exists $c_1\equiv c_1(c_\ell,c_u,c_0,c_{\mathrm{dim}}, c_{\mathrm{blk}},\kappa)>0$
such that, with probability at least
$1-\mathcal O(p_T^{-10})$, we have 
\begin{align}
\|x^\top \Delta_L \|_2
&\le
c_1
\frac{\sigma^2\sqrt{(N+T_{1,\mathrm{p}})(r+\zeta_T)}}{
\gamma_{\min} \rho_T} \, + \,  c_1 \frac{\sigma \sqrt{N}}{\rho_T^{1/2}}\, \|U_{2k}^\top x\|_2.
\label{eq:lemma16_centered_bound_rewrite}
\end{align}
\end{lemma}

\begin{proof}
Using $H_k^{-1}U_{1k}^\top U_{1k}=I_r$, by definition of $L$, we have
\begin{align*}
L U_{1k}\mathcal C_{\bullet,\bullet,k}
&=
\hat U_{2k}H_U
\hat H_k^{-1}
H_U^\top \hat U_{1k}^\top
U_{1k}\mathcal C_{\bullet,\bullet,k}
-
U_{2k}H_k^{-1} U_{1k}^\top
U_{1k}\mathcal C_{\bullet,\bullet,k}  \\
&=
\bigl(\hat U_{2k}H_U-U_{2k}\bigr)
\mathcal C_{\bullet,\bullet,k}  
+
\hat U_{2k}H_U
\left(
\hat H_k^{-1}
H_U^\top \hat U_{1k}^\top
-
H_k^{-1}U_{1k}^\top
\right)
U_{1k}\mathcal C_{\bullet,\bullet,k},
\end{align*}
Also, the definition of $\Psi_{U,2k}$ in Corollary~\ref{cor:centered-left-perturbation} gives \[\hat U_{2k}H_U-U_{2k}
=
(E_{\rm left}^{\rm p})_{\{N_{1k}+1,\ldots,N\},\bullet}
W_{\rm left}(W_{\rm left}^\top W_{\rm left})^{-1}
 +\Psi_{U,2k}.
 \]
 Substituting this identity into the above display and using the definition of $\Delta_L$ in~\eqref{eq:lemma16_centered_decomp_rewrite} give $\Delta_L
=
\Psi_{U,2k}\mathcal C_{\bullet,\bullet,k} 
+
\hat U_{2k}H_U
(
\hat H_k^{-1}
H_U^\top \hat U_{1k}^\top
-
H_k^{-1}U_{1k}^\top
)
U_{1k}\mathcal C_{\bullet,\bullet,k}$.

We will use this to control the Euclidean norm of $x^\top \Delta_L$ by bounding the norm of each of the two terms separately. For the first one, combining~\eqref{cor:bound_Psi} with $\|\mathcal C_{\bullet,\bullet,k}\|_{\rm op}
\le \gamma_{\max}
=
\kappa\gamma_{\min}$ gives
\begin{align*}
\|x^\top \Psi_{U,2k}\mathcal C_{\bullet,\bullet,k}\|_2
&\le \|\mathcal C_{\bullet,\bullet,k} \|_\mathrm{op} \, \|\Psi_{U,2k}^\top x\|_2 \\
& \leq  
c_1
\frac{\sigma^2\sqrt{(N+T_{1,\mathrm{p}})(r+\zeta_T)}}{
\gamma_{\min} \rho_T}
+
c_1\left(
\frac{\sigma^2N}{\gamma_{\min} \, \rho_T}
+
\frac{\sigma\sqrt{r+\zeta_T}}{
\rho_T^{1/2}}
\right)
\|U_{2k}^\top x\|_2.
\end{align*}
Second, recalling $D_k=\widehat H_k^{-1}(\hat U_{1k}H_U)^\top-H_k^{-1}U_{1k}^\top$ from Lemma~\ref{lemma15_centered}, we have
\begin{align*}
\|
x^\top \hat U_{2k}H_U
\, D_k \, 
U_{1k}\mathcal C_{\bullet,\bullet,k}\|_2 & \le
\|H_U^\top\hat U_{2k}^\top x\|_2
\left\|
 \, D_k \, 
U_{1k}\mathcal C_{\bullet,\bullet,k}
\right\|_\mathrm{op} \\
& \leq \left(\|(\hat U_{2k} H_U - U_{2k})^\top x\|_2+\|U_{2k}^\top x\|_2 \right)
\left\|
 \, D_k \, 
U_{1k}\mathcal C_{\bullet,\bullet,k}
\right\|_\mathrm{op} \\
& \leq c_1 \left(\frac{\sigma \sqrt{r + \zeta_T}}{\gamma_\mathrm{min} \, \rho_T^{1/2}} +\frac{\sigma^2 N}{\gamma_\mathrm{min}^2 \, \rho_T}\|U_{2k}^\top x\|_2+\|U_{2k}^\top x\|_2 \right)
\left\|
 \, D_k \, 
U_{1k}\mathcal C_{\bullet,\bullet,k}
\right\|_\mathrm{op} \\
& \leq c_1 \left(\frac{\sigma \sqrt{r + \zeta_T}}{\gamma_\mathrm{min} \, \rho_T^{1/2}} +\|U_{2k}^\top x\|_2 \right)
\left\|
 \, D_k \, 
U_{1k}\mathcal C_{\bullet,\bullet,k}
\right\|_\mathrm{op} \\ 
& \leq c_1 \left(\frac{\sigma \sqrt{r + \zeta_T}}{\gamma_\mathrm{min} \, \rho_T^{1/2}} +\|U_{2k}^\top x\|_2 \right)
\left\|
 \, D_k \right\|_\mathrm{op} \left\| 
U_{1k} \right\|_\mathrm{op} \left\| \mathcal C_{\bullet,\bullet,k}
\right\|_\mathrm{op} \\ 
& \leq c_1 \frac{\sigma^2 \sqrt{N (r + \zeta_T)}}{\gamma_\mathrm{min} \, \rho_T} + c_1 \frac{\sigma \sqrt{N}}{\rho_T^{1/2}}\, \|U_{2k}^\top x\|_2, 
\end{align*}
where the third inequality follows from~\eqref{eq:boundU_vectorX}, and the last one from~\eqref{eq:lemma15_centered_D_rewrite},   \eqref{assump:subblock-conditioning} and $\|\mathcal C_{\bullet,\bullet,k}\|_{\rm op}
\le \gamma_{\max}
=
\kappa\gamma_{\min}$. Combining the two displays and using \eqref{assump:sampleSize},  \eqref{assump:smallNoise} concludes the proof.
\end{proof}

\medskip \medskip
\begin{lemma}
\label{lem:up-cen}
Grant Assumptions  \eqref{assump:subblock-conditioning}
with fixed constants $0<c_\ell\le c_u$, \eqref{assump:sampleSize} and \eqref{assump:smallNoise}. Suppose that
$0<\gamma_{\min}\le \sigma_{\min}(\mathcal C_{\bullet,\bullet,j})
\le \sigma_{\max}(\mathcal C_{\bullet,\bullet,j})\le \gamma_{\max}<\infty$
for all $j\in[K]$. Write $Y_{\rm up}^{\rm p}=M_{\rm up}^{\rm p}
+E_{\rm up}^{\rm p}$, where $M_{\rm up}^{\rm p}= W_\mathrm{up} V^\top$. Also recall
$(\widehat U_{\mathrm{up}},\widehat\Sigma_{\mathrm{up}},
\widehat V_{\mathrm{up}})
=\operatorname{SVD}_r(Y_\mathrm{up}^\mathrm{p})$,
$H_V=\operatorname{sgn}(\widehat V_{\mathrm{up}}^\top V)$. Writing
\begin{equation}
\Phi_{\mathrm{up}} := \widehat U_{\mathrm{up}}\widehat\Sigma_{\mathrm{up}}
\widehat V_{\mathrm{up}}^\top-M_\mathrm{up}^\mathrm{p}
-
 E_\mathrm{up}^\mathrm{p}VV^\top 
-
 W_{\rm up}(W_{\rm up}^\top W_{\rm up})^{-1}W_{\rm up}^\top E_{\rm up}^{\rm p}, 
\label{eq:up-decomp}
\end{equation}
for every fixed $g\in\mathbb R^{N_{1,{\mathrm{p}}}}$ and
$y\in\mathcal S^{T-1}$ there exists a constant
$c_1\equiv c_1(c_\ell,c_u,c_0,c_{\mathrm{dim}}, c_{\mathrm{blk}},\kappa)>~0$ such that 
\begin{align}
|g^\top \Phi_{\mathrm{up}} \, y|
& \le c_1  \frac{\sigma^2 \sqrt{T(r + \zeta_N)}}{\gamma_\mathrm{min} \, \rho_N^{1/2}} \, \|g\|_2 \, +  c_1  \, \frac{\sigma \sqrt{r + \zeta_N}}{\gamma_\mathrm{min} \, \rho_N^{1/2}}  \|W_{\mathrm{up}}^\top g\|_2 + c_1 \frac{\sigma^2 T}{\gamma_\mathrm{min}\, \rho_N^{1/2}} \, \|g\|_2 \, \|V^\top y\|_2 \nonumber \\
& \qquad \qquad + c_1 \frac{\sigma \sqrt{T}}{\gamma_\mathrm{min}\, \rho_N^{1/2}} \|W_{\mathrm{up}}^\top g\|_2
 \, \|V^\top y\|_2
\label{eq:Rup-bound}
\end{align}
 with
probability at least $1-\mathcal O(p_N^{-10})$.
\end{lemma}

\begin{proof}
Write
$\widehat M^\mathrm{p}_\mathrm{up}:=
\widehat U_{\mathrm{up}}\widehat\Sigma_{\mathrm{up}}
\widehat V_{\mathrm{up}}^\top$, $\delta_V :=\widehat V_{\mathrm{up}}H_V-~V$, and $\Psi_V := \delta_V - (E_\mathrm{up}^\mathrm{p})^\top W_\mathrm{up} (W_\mathrm{up}^\top W_\mathrm{up})^{-1}$. Combining this with 
$\widehat U_{\mathrm{up}}\widehat\Sigma_{\mathrm{up}}H_V
=Y_\mathrm{up}^\mathrm{p}(V+\delta_V)$ gives
\begin{align*}
\widehat M_\mathrm{up}^\mathrm{p}&-M_\mathrm{up}^\mathrm{p} = Y_{\rm up}^{\rm p}(V+\delta_V)(V+\delta_V)^\top - W_{\rm up}V^\top \\
& = (W_{\rm up} V^\top+ E_{\rm up}^{\rm p})(V+\delta_V)(V+\delta_V)^\top - W_{\rm up}V^\top
\\
& =    E_\mathrm{up}^\mathrm{p}VV^\top
+
E_\mathrm{up}^\mathrm{p}V\delta_V^\top
+
Y_\mathrm{up}^\mathrm{p}\delta_VV^\top
+
Y_\mathrm{up}^\mathrm{p}\delta_V\delta_V^\top
+
W_{\mathrm{up}}\delta_V^\top \\
& = E_\mathrm{up}^\mathrm{p}VV^\top
+
W_{\mathrm{up}}(W_{\mathrm{up}}^\top W_{\mathrm{up}})^{-1}W_{\mathrm{up}}^\top E_\mathrm{up}^\mathrm{p}
+
E_\mathrm{up}^\mathrm{p}V\delta_V^\top
+
Y_\mathrm{up}^\mathrm{p}\delta_VV^\top
+
Y_\mathrm{up}^\mathrm{p}\delta_V\delta_V^\top
+
W_{\mathrm{up}}\Psi_V^\top.
\end{align*}
This, together with the definition of $\Phi_{\mathrm{up}}$ in~\eqref{eq:up-decomp}, gives  $g^\top \Phi_{\mathrm{up}}y
=
g^\top E_\mathrm{up}^\mathrm{p}V\delta_V^\top y 
+
g^\top Y_\mathrm{up}^\mathrm{p}\delta_VV^\top  y
+
g^\top Y_\mathrm{up}^\mathrm{p}\delta_V\delta_V^\top 
y + g^\top W_{\mathrm{up}}\Psi_V^\top y$.

We now bound each of the four terms in $g^\top \Phi_\mathrm{up} \, y$ individually. By~\eqref{eq:DeltaVtranspose} and~\eqref{eq:DeltaVtransposeOP} in Corollary~\ref{lemma13_Transpose}, with probability at least $1 - \mathcal{O}(p_N^{-10})$ we have
\begin{align}\label{eq:deltaVbound}
\|\delta_V^\top y\|_2 \leq c_1 \frac{\sigma \sqrt{r + \zeta_N}}{\gamma_\mathrm{min} \, \rho_N^{1/2}} + c_1\frac{\sigma^2 T}{\gamma_\mathrm{min}^2 \, \rho_N} \|V^\top y\|_2, \qquad\qquad \|\delta_V\|_\mathrm{op} \leq c_1 \frac{\sigma \sqrt{T}}{\gamma_\mathrm{min}\, \rho_N^{1/2}}.
\end{align}
Combining the first bound with Lemma~\ref{lemma:bilinearGaussian} and \eqref{assump:sampleSize} gives 
\begin{align*}
    |g^\top E_\mathrm{up}^\mathrm{p}V\delta_V^\top y|
& \le
\|(E_\mathrm{up}^\mathrm{p})^\top g\|_2 \|\delta_V^\top y\|_2
\le
c_1\sigma\sqrt{T+\zeta_N} \, \|g\|_2 \,\, \|\delta_V^\top y\|_2 \\
& \leq c_1 \frac{\sigma^2 \sqrt{T(r + \zeta_N)}}{\gamma_\mathrm{min} \, \rho_N^{1/2}} \, \|g\|_2 + c_1 \frac{\sigma^3 T^{3/2}}{\gamma_\mathrm{min}^2 \, \rho_N} \, \|g\|_2 \,\|V^\top y\|_2.
\end{align*}
Next, using $\|(Y_\mathrm{up}^\mathrm{p})^\top g\|_2
\le
\|(M_\mathrm{up}^\mathrm{p})^\top g\|_2 + \|(E_\mathrm{up}^\mathrm{p})^\top g\|_2
\le
\|W_{\mathrm{up}}^\top g\|_2
+
c_1\sigma\sqrt{T} \, \|g\|_2$, we also get
\begin{align*}
    |g^\top Y_\mathrm{up}^\mathrm{p} \delta_VV^\top y|
&\le  \|(Y_\mathrm{up}^\mathrm{p})^\top g\|_2 \, \|\delta_V\|_\mathrm{op} \, \|V^\top y\|_2 \leq c_1 \frac{\sigma \sqrt{T}}{\gamma_\mathrm{min}\, \rho_N^{1/2}} \|(Y_\mathrm{up}^\mathrm{p})^\top g\|_2 \, \|V^\top y\|_2 \\
& \leq  c_1 \frac{\sigma \sqrt{T}}{\gamma_\mathrm{min}\, \rho_N^{1/2}} \left(\|W_{\mathrm{up}}^\top g\|_2
+\sigma\sqrt{T} \, \|g\|_2 \right) \, \|V^\top y\|_2 \\
& \leq c_1 \frac{\sigma \sqrt{T}}{\gamma_\mathrm{min}\, \rho_N^{1/2}} \|W_{\mathrm{up}}^\top g\|_2
 \, \|V^\top y\|_2 + c_1 \frac{\sigma^2 T}{\gamma_\mathrm{min}\, \rho_N^{1/2}} \, \|g\|_2 \, \|V^\top y\|_2.
\end{align*}
Similarly, for the third term we have 
\begin{align*}
    |g^\top Y_\mathrm{up}^\mathrm{p}\delta_V\delta_V^\top y|
&\le  \|(Y_\mathrm{up}^\mathrm{p})^\top g\|_2 \, \|\delta_V\|_\mathrm{op} \, \|\delta_V^\top y\|_2 \leq c_1 \frac{\sigma \sqrt{T}}{\gamma_\mathrm{min}\, \rho_N^{1/2}} \|(Y_\mathrm{up}^\mathrm{p})^\top g\|_2 \, \|\delta_V^\top y\|_2 \\
& \leq  c_1 \frac{\sigma \sqrt{T}}{\gamma_\mathrm{min}\, \rho_N^{1/2}} \left(\|W_{\mathrm{up}}^\top g\|_2
+\sigma\sqrt{T} \, \|g\|_2 \right) \, \|\delta_V^\top y\|_2 \\
& \leq  c_1 \frac{\sigma \sqrt{T}}{\gamma_\mathrm{min}\, \rho_N^{1/2}} \left(\|W_{\mathrm{up}}^\top g\|_2
+\sigma\sqrt{T} \, \|g\|_2 \right) \, \left(  \frac{\sigma \sqrt{r + \zeta_N}}{\gamma_\mathrm{min} \, \rho_N^{1/2}} + \frac{\sigma^2 T}{\gamma_\mathrm{min}^2 \, \rho_N} \|V^\top y\|_2 \right) \\
& \leq c_1 \frac{\sigma^2 \sqrt{T(r + \zeta_N)}}{\gamma_\mathrm{min}^2\, \rho_N} \|W_{\mathrm{up}}^\top g\|_2 + c_1 \frac{\sigma^3 T \sqrt{r + \zeta_N}}{\gamma_\mathrm{min}^2\, \rho_N} \|g\|_2  \\
& \qquad \qquad + c_1 \frac{\sigma^3 T^{3/2}}{\gamma_\mathrm{min}^3\, \rho_N^{3/2}} \|W_{\mathrm{up}}^\top g\|_2 \, \|V^\top y\|_2 + c_1 \frac{\sigma^4 T^2}{\gamma_\mathrm{min}^3 \, \rho_N^{3/2}} \|g\|_2 \, \|V^\top y\|_2.
\end{align*}
It remains to control the term involving $\Psi_V$. Using the definition of $\Psi_V$ and Lemmas~\ref{lemma:spectrumMp} and \ref{lemma:bilinearGaussian} we can write 
\begin{align*}
|g^\top W_{\mathrm{up}}\Psi_V^\top y| &\leq |g^\top W_{\mathrm{up}}\delta_V^\top y| + |g^\top W_{\mathrm{up}}(W_{\mathrm{up}}^\top W_{\mathrm{up}})^{-1} W_{\mathrm{up}}^\top E_\mathrm{up}^\mathrm{p} \, y| \\
& \le \|W_{\mathrm{up}}^\top g\|_2 \, \left(\|\delta_V^\top y\|_2 + \|(W_{\mathrm{up}}^\top W_{\mathrm{up}})^{-1}\|_\mathrm{op} \, \|W_{\mathrm{up}}^\top E_\mathrm{up}^\mathrm{p} \, y \|_2 \right) \\
& \leq c_1  \|W_{\mathrm{up}}^\top g\|_2 \, \left(\frac{\sigma \sqrt{r + \zeta_N}}{\gamma_\mathrm{min} \, \rho_N^{1/2}} + c_1\frac{\sigma^2 T}{\gamma_\mathrm{min}^2 \, \rho_N} \|V^\top y\|_2 + \frac{1}{\gamma_\mathrm{min}^2 \rho_N} \, \sigma \|W_{\mathrm{up}}\|_\mathrm{op} \sqrt{r + \zeta_N} \right) \\
& \leq c_1  \, \frac{\sigma \sqrt{r + \zeta_N}}{\gamma_\mathrm{min} \, \rho_N^{1/2}}  \|W_{\mathrm{up}}^\top g\|_2 + c_1\frac{\sigma^2 T}{\gamma_\mathrm{min}^2 \, \rho_N} \,  \|W_{\mathrm{up}}^\top g\|_2\, \|V^\top y\|_2.
\end{align*}
Combining the four preceding bounds and simplifying them further using \eqref{assump:sampleSize},  \eqref{assump:smallNoise} proves \eqref{eq:Rup-bound}.
\end{proof}

\medskip \medskip \medskip
Lemma~\ref{lem:up-cen} provides an upper bound on the approximation error of the entire pooled upper matrix $M_\mathrm{up}^\mathrm{p} = W_\mathrm{up} V^\top$. By restricting $\Phi_\mathrm{up}$ in~\eqref{eq:up-decomp} to the subsets $\mathcal I_k^\mathrm{up} = \{s_k + 1, \ldots, s_k + N_{1k}\}$ and $\mathcal J_k = \{T_{1k} + 1, \ldots, T\}$, computations similar to those in the previous proof allow us to quantify the approximation error of $\mathcal{M}_{\bullet, \bullet, k}^{(b)}$. 

\begin{cor}
\label{cor:R1-sharp}
Suppose the assumptions of Lemma~\ref{lem:up-cen} are satisfied, and use the notation introduced there. Also define $\hat M_b^{(k)} := (\widehat U_{\mathrm{up}}\widehat\Sigma_{\mathrm{up}}
\widehat V_{\mathrm{up}}^\top)_{\mathcal I_k^\mathrm{up}, \mathcal J_k}$ and recall $\mathcal{M}_{\bullet, \bullet, k}^{(b)} = (M_\mathrm{up}^\mathrm{p})_{\mathcal I_k^\mathrm{up}, \mathcal J_k} = U_{1k} \mathcal C_{\bullet, \bullet, k} \, V_{2k}^\top$.
Writing
\begin{align}\label{eq:decomopistionPhi_k}
    \Phi_{k}: = \widehat M_{b}^{(k)}-\mathcal{M}_{\bullet, \bullet, k}^{(b)}
-
(E_{\rm up}^{\rm p})_{\mathcal{I}_k^\mathrm{up}, \bullet }
            \,VV_{2k}^\top -  U_{1k} \,  \mathcal{C}_{\bullet, \bullet, k} \, (W_{\rm up}^\top W_{\rm up})^{-1}W_{\rm up}^\top
            (E_{\rm up}^{\rm p})_{\bullet, \mathcal{J}_k},
\end{align}
for fixed $x\in \mathcal S^{N_{2k}-1}$,
$y\in \mathcal S^{T_{2k}-1}$ there exists a constant
$c_1\equiv c_1(c_\ell,c_u,c_0,c_{\mathrm{dim}}, c_{\mathrm{blk}},\kappa)>~0$ such that
\begin{align}
|x^\top U_{2k}(U_{1k}^\top U_{1k})^{-1}U_{1k}^\top \, \Phi_{k} \, y|
& \le c_1  \, \frac{\sigma \sqrt{r + \zeta_N}}{ \rho_N^{1/2}}  \|U_{2k}^\top x\|_2  + c_1 \frac{\sigma \sqrt{T}}{\rho_N^{1/2}}
 \, \|U_{2k}^\top x\|_2 \, \|V_{2k}^\top y\|_2,
\label{eq:R1-sharp-bound-corrected}
\end{align}
\begin{align}\label{eq:prokection_1Phi_k}
\left\|(U_{1k}^\top U_{1k})^{-1/2} \, U_{1k}^\top \, \Phi_k \,y\right\|_2
& \le c_1 \frac{\sigma \sqrt{r + \zeta_N}}{\rho_N^{1/2} } \sqrt{\frac{N_{1k}}{N}} + c_1\frac{\sigma \sqrt{T}}{\rho_N^{1/2}} \sqrt{\frac{N_{1k}}{N}} \|V_{2k}^\top y\|_2,
\end{align}
\begin{align}\label{eq:prokection_2Phi_k}
    \|(I_{N_{1k}}-U_{1k} \{U_{1k}^\top U_{1k}\}^{-1} \, U_{1k}^\top) \, \Phi_k \, y\|_2 & \leq c_1 \frac{\sigma^3 T \sqrt{r + \zeta_N}}{\gamma_\mathrm{min}^2 \rho_N} + c_1 \frac{\sigma^2 \sqrt{N_{1k}(r + \zeta_N)}}{\gamma_\mathrm{min} \, \rho_N^{1/2}}\nonumber \\
    & \qquad + c_1\frac{\sigma^2 \sqrt{T(N_{1k} + T)}}{\gamma_\mathrm{min} \, \rho_N^{1/2}} \, \|V_{2k}^\top y \|_2
\end{align}
with probability at least $1-O(p_N^{-10})$.
\end{cor}

\begin{proof}
Restricting~\eqref{eq:up-decomp} to the subsets 
$\mathcal I_k^\mathrm{up}$ and $\mathcal J_k$ and using the definition of $\Phi_k$ in~\eqref{eq:decomopistionPhi_k} immediately yield $\Phi_k = (\Phi_\mathrm{up})_{\mathcal I_k^\mathrm{up}, \mathcal J_k}$. In order to prove~\eqref{eq:R1-sharp-bound-corrected}, set $B_k:=U_{2k}(U_{1k}^\top U_{1k})^{-1}U_{1k}^\top$ and define $g \in \mathbb R^{N_{1, \mathrm{p}}} $ and $\bar y \in \mathcal S^{T-1}$ to be the vectors with entries $g_i = (B_k^\top x)_{i - s_k} \mathbbm 1\{i \in \mathcal{I}_k^\mathrm{up}\}$ and $\bar y_t = y_{t - T_{1k}} \, \mathbbm 1 \{ t \in \mathcal J_k\}$, respectively. This ensures that $V^\top \bar y = V_{2k}^\top \, y$ and that $|g^\top \Phi_\mathrm{up} \bar y|$ is equal to the left-hand side of~\eqref{eq:R1-sharp-bound-corrected}. Furthermore, from  \eqref{assump:subblock-conditioning} we have  $\|g\|_2 = \|B_k^\top x\|_2 = \|U_{1k}(U_{1k}^\top U_{1k})^{-1} U_{2k}^\top x\|_2 \leq \|U_{1k}(U_{1k}^\top U_{1k})^{-1}\|_{\mathrm{op}} \|U_{2k}^\top x\|_2 \leq c_\ell^{-1/2}\sqrt{N/N_{1k}}\,\|U_{2k}^\top x\|_2$, and $\|W_\mathrm{up}^\top g\|_2 =  \| \mathcal C_{\bullet, \bullet, k}^\top U_{1k}^\top B_k^\top\, x\|_2 = \| \mathcal C_{\bullet, \bullet, k}^\top U_{2k}^\top \, x\|_2  \leq \gamma_\mathrm{max} \| U_{2k}^\top \, x\|_2$. Combining these with~\eqref{eq:Rup-bound} and further simplifying the resulting bound using \eqref{assump:sampleSize},  \eqref{assump:smallNoise}  proves \eqref{eq:R1-sharp-bound-corrected}.

It remains to prove the last two bounds. Using the expression for $\Phi_\mathrm{up}$ from the proof of Lemma~\ref{lem:up-cen} we get 
\begin{align}\label{eq:decompPhi_ky}
    \Phi_{k} \, y
=
 (E_\mathrm{up}^\mathrm{p})_{\mathcal I^\mathrm{up}_k, \bullet}V\delta_V^\top \bar y
+
 (Y_\mathrm{up}^\mathrm{p})_{\mathcal I^\mathrm{up}_k, \bullet} \delta_VV^\top \, \bar y 
+
(Y_\mathrm{up}^\mathrm{p})_{\mathcal I^\mathrm{up}_k, \bullet}\delta_V\delta_V^\top \bar y
+ U_{1k}\mathcal{C}_{\bullet, \bullet, k} \, \Psi_V^\top \, \bar y.
\end{align} We next bound the norms of the four terms separately under the action of $(U_{1k}^\top U_{1k})^{-1/2}U_{1k}^\top$. As for the first one, using Lemma~\ref{lemma:bilinearGaussian} and the bounds for $\|\delta_V^\top \bar y\|_2$ and $\|\delta_V\|_\mathrm{op}$ in~\eqref{eq:deltaVbound}, we get
\begin{align*}
    \|(U_{1k}^\top U_{1k})^{-1/2}U_{1k}^\top (E_\mathrm{up}^\mathrm{p})_{\mathcal I^\mathrm{up}_k, \bullet}V\delta_V^\top \bar y\|_2 &  \leq \|(U_{1k}^\top U_{1k})^{-1/2}U_{1k}^\top (E_\mathrm{up}^\mathrm{p})_{\mathcal I^\mathrm{up}_k, \bullet}V\|_\mathrm{op} \, \|\delta_V^\top \bar y\|_2 \\
    & \leq c_1 \sigma \sqrt{r + \zeta_N} \, \|\delta_V^\top \bar y\|_2 \\
    & \leq c_1 \sigma \sqrt{r + \zeta_N} \left( \frac{\sigma \sqrt{r + \zeta_N}}{\gamma_\mathrm{min} \, \rho_N^{1/2}} + \frac{\sigma^2 T}{\gamma_\mathrm{min}^2 \, \rho_N} \, \|V_{2k}^\top y\|_2 \right).
\end{align*}
Similarly, for the second and third terms in~\eqref{eq:decompPhi_ky} we have 
\begin{align*}
    \|(U_{1k}^\top U_{1k})^{-1/2}U_{1k}^\top \, & (Y_\mathrm{up}^\mathrm{p})_{\mathcal I^\mathrm{up}_k, \bullet} \, \delta_V \, (V^\top \bar y+\delta_V^\top \bar y)\|_2 \\
    & \leq \|(U_{1k}^\top U_{1k})^{-1/2}U_{1k}^\top \, \{(M_\mathrm{up}^\mathrm{p})_{\mathcal I^\mathrm{up}_k, \bullet} + (E_\mathrm{up}^\mathrm{p})_{\mathcal I^\mathrm{up}_k, \bullet} \} \, \delta_V \, (V^\top \bar y+\delta_V^\top \bar y)\|_2 \\
    & \leq \|(U_{1k}^\top U_{1k})^{-1/2}U_{1k}^\top \, \{(M_\mathrm{up}^\mathrm{p})_{\mathcal I^\mathrm{up}_k, \bullet} + (E_\mathrm{up}^\mathrm{p})_{\mathcal I^\mathrm{up}_k, \bullet} \} \|_\mathrm{op} \|\delta_V\|_\mathrm{op} \|V^\top \bar y+\delta_V^\top \bar y\|_2 \\
    & = \|(U_{1k}^\top U_{1k})^{1/2} \mathcal{C}_{\bullet, \bullet, k} \, V^\top + (U_{1k}^\top U_{1k})^{-1/2} U_{1k}^\top E_\mathrm{up}^{(k)} \|_\mathrm{op} \|\delta_V\|_\mathrm{op} \|V^\top \bar y+\delta_V^\top \bar y\|_2 \\
    & \leq c_1 \left(\kappa \,  \gamma_\mathrm{min}\sqrt{\frac{N_{1k}}{N}} + \sigma \sqrt{T} \right) \, \frac{\sigma \sqrt{T}}{\gamma_\mathrm{min} \, \rho_N^{1/2}} \, \left(\frac{\sigma \sqrt{r + \zeta_N}}{\gamma_\mathrm{min} \rho_N^{1/2}} + \|V_{2k}^\top y\|_2 \right).
\end{align*}
Finally, using~\eqref{eq:PsiVranspose} in Corollary~\ref{lemma13_Transpose} to bound $\|\Psi_V^\top \bar y\|_2 = \|\Psi_{V,2k}^\top \, y\|_2$, with probability at least $1 - \mathcal O(p_N^{-10})$ we have
\begin{align*}
    \|&(U_{1k}^\top U_{1k})^{-1/2}U_{1k}^\top \, U_{1k}\mathcal{C}_{\bullet, \bullet, k} \, \Psi_V^\top \, \bar y\|_2  \leq c_1 \gamma_\mathrm{max} \, \sqrt{\frac{N_{1k}}{N}} \, \|\Psi_V^\top \bar y\|_2 = c_1 \kappa \gamma_\mathrm{min} \, \sqrt{\frac{N_{1k}}{N}} \, \|\Psi_{V,2k}^\top \, y\|_2 \\
    & \leq c_1 \gamma_\mathrm{min} \, \sqrt{\frac{N_{1k}}{N}} \left\{\frac{\sigma^2\sqrt{(T+N_{1,\mathrm{p}})(r+\zeta_N)}}{
\gamma_{\min}^2\rho_N}
+
\left(
\frac{\sigma^2T}{\gamma_{\min}^2\rho_N}
+
\frac{\sigma\sqrt{r+\zeta_N}}{
\gamma_{\min}\rho_N^{1/2}}
\right)
\|V_{2k}^\top y\|_2 \right\}.
\end{align*}
Combining the last three displays and further simplifying the bound using \eqref{assump:sampleSize},  \eqref{assump:smallNoise} proves~\eqref{eq:prokection_1Phi_k}. 

In order to prove~\eqref{eq:prokection_2Phi_k}, we will make use of the fact that $I_{N_{1k}}-U_{1k} \{U_{1k}^\top U_{1k}\}^{-1} \, U_{1k}^\top$ is the orthogonal projector onto the orthogonal complement of $\operatorname{col}(U_{1k})$. This also implies that $\|I_{N_{1k}}-U_{1k} \{U_{1k}^\top U_{1k}\}^{-1} \, U_{1k}^\top\|_\mathrm{op} = 1$ and $\operatorname{rank}(I_{N_{1k}}-U_{1k} \{U_{1k}^\top U_{1k}\}^{-1} \, U_{1k}^\top) = N_{1k} - r$. This implies that the signal contribution from the second and third terms in~\eqref{eq:decompPhi_ky} vanishes, and we are left with the error matrix only. More precisely, we have 
\begin{align*}
    \|(I_{N_{1k}}-U_{1k}& \{U_{1k}^\top U_{1k}\}^{-1} \, U_{1k}^\top) \, (Y_\mathrm{up}^\mathrm{p})_{\mathcal I^\mathrm{up}_k, \bullet} \, \delta_V \, (V^\top \bar y+\delta_V^\top \bar y)\|_2 \\
    & = \|(I_{N_{1k}}-U_{1k} \{U_{1k}^\top U_{1k}\}^{-1} \, U_{1k}^\top) \, (E_\mathrm{up}^\mathrm{p})_{\mathcal I^\mathrm{up}_k, \bullet} \, \delta_V \, (V^\top \bar y+\delta_V^\top \bar y)\|_2 \\
    & = \|(I_{N_{1k}}-U_{1k} \{U_{1k}^\top U_{1k}\}^{-1} \, U_{1k}^\top) \, E_\mathrm{up}^{(k)} \, \delta_V \, (V^\top \bar y+\delta_V^\top \bar y)\|_2 \\
    & \leq \|(I_{N_{1k}}-U_{1k} \{U_{1k}^\top U_{1k}\}^{-1} \, U_{1k}^\top) \, E_\mathrm{up}^{(k)} \|_\mathrm{op}\, \|\delta_V\|_\mathrm{op} \, \|V^\top \bar y+\delta_V^\top \bar y\|_2 \\
    & \leq c_1 \sigma \sqrt{T + (N_{1k} - r) + \zeta_N} \, \, \|\delta_V\|_\mathrm{op} \, \|V^\top \bar y+\delta_V^\top \bar y\|_2 \\
     & \leq c_1 \sigma \sqrt{N_{1k} + T} \,  \frac{\sigma \sqrt{T}}{\gamma_\mathrm{min} \, \rho_N^{1/2}} \, \left(\frac{\sigma \sqrt{r + \zeta_N}}{\gamma_\mathrm{min} \rho_N^{1/2}} + \|V_{2k}^\top y\|_2 \right).
\end{align*}
The fourth term in~\eqref{eq:decompPhi_ky} completely vanishes under the action of $I_{N_{1k}}-U_{1k} \{U_{1k}^\top U_{1k}\}^{-1} \, U_{1k}^\top$, while the first one gives
\begin{align*}
    \|&(I_{N_{1k}}-U_{1k} \{U_{1k}^\top U_{1k}\}^{-1} \, U_{1k}^\top) \, (E_\mathrm{up}^\mathrm{p})_{\mathcal I^\mathrm{up}_k, \bullet}V\delta_V^\top \bar y\|_2 \\
    & \leq  \|(I_{N_{1k}}-U_{1k}\{U_{1k}^\top U_{1k}\}^{-1} \, U_{1k}^\top) \, (E_\mathrm{up}^\mathrm{p})_{\mathcal I^\mathrm{up}_k, \bullet}V \|_\mathrm{op} \, \|\delta_V^\top \bar y\|_2 \\
    & \leq c_1 \sigma \sqrt{r + (N_{1k} - r) + \zeta_N} \, \, \|\delta_V^\top \bar y\|_2 \leq c_1 \sigma \sqrt{N_{1k}} \, \left( \frac{\sigma \sqrt{r + \zeta_N}}{\gamma_\mathrm{min} \, \rho_N^{1/2}} + \frac{\sigma^2 T}{\gamma_\mathrm{min}^2 \, \rho_N} \, \|V_{2k}^\top y\|_2 \right).
\end{align*}
Combining the last three displays and further simplifying the bound using \eqref{assump:sampleSize},  \eqref{assump:smallNoise} proves~\eqref{eq:prokection_2Phi_k}. This concludes the proof.
\end{proof}

\medskip\medskip
\noindent We now present the main results of this section, which give a first-order expansion of $\hat \mu_{xy}^{(k)} - \mu_{xy}^{(k)}$ as the sum of Gaussian terms and remainder terms that can be suitably bounded from above. In the special case $K=1$ with $x=\boldsymbol e_i$ and $y=\boldsymbol e_t$, analogous expansions were proved in \cite{yan2024entrywise}. Our results therefore generalize these earlier expansions to arbitrary $K \geq 1$, $x \in \mathcal S^{N_{2k}-1}$, and $y \in \mathcal S^{T_{2k}-1}$, while also providing remainder bounds obtained via different techniques. In particular, the approach in \cite{yan2024entrywise} relies on a leave-one-block-out argument, which could also be adapted to the present setting. However, this approach becomes suboptimal when $K$ grows: in particular, the resulting remainder term is negligible only under a signal-to-noise ratio condition that deteriorates with $K$. We therefore instead rely on the preceding lemmas, which yield analogous results under weaker conditions.
\begin{lemma}\label{lemma1}
    Grant Assumption   \eqref{assump:subblock-conditioning} with fixed constants $c_\ell, c_u$ satisfying $0 < c_\ell \leq c_u < \infty$, \eqref{assump:sampleSize} and \eqref{assump:smallNoise}. Suppose further that $0<\gamma_{\min}\le \sigma_{\min}(\mathcal C_{\bullet,\bullet,j})
\le \sigma_{\max}(\mathcal C_{\bullet,\bullet,j})\le \gamma_{\max}<\infty$  for all
$j\in[K]$, and let $\kappa:=\gamma_{\max}/\gamma_{\min}$. Fix $k \in [K]$, unit vectors $x \in \mathcal S^{N_{2k}-1}$, $y \in \mathcal S^{T_{2k}-1}$, and let $\hat \mu_{xy}^{(k)}$ be the output of Algorithm~\ref{alg:bilinear4block} run with $\tau \leq \frac{c_\ell \, N_{1k}}{2 \, N}$. Also write the decomposition $\hat \mu_{xy}^{(k)} -  \mu_{xy}^{(k)} = Z_{xy}^{(1)}+ Z_{xy}^{(2)} + Z_{xy}^{(3)}+ Z_{xy}^{(4)}+ \Delta_{xy} =: Z_{xy}+ \Delta_{xy}$, where
    \begin{align}\label{eq:Z_xy}
        Z_{xy}^{(1)} 
        & := x^\top\,
        (E_{\rm left}^{\rm p})_{\mathcal I_k, \bullet}\,W_{\rm left}\,(W_{\rm left}^\top W_{\rm left})^{-1} \mathcal{C}_{\bullet, \bullet, k} \, (W_{\rm up}^\top W_{\rm up})^{-1}W_{\rm up}^\top
            (E_{\rm up}^{\rm p})_{\bullet, \mathcal J_k} \,  \, y \nonumber \\
        Z_{xy}^{(2)} 
        & := x^\top\,
        U_{2k} \,  \mathcal{C}_{\bullet, \bullet, k} \, (W_{\rm up}^\top W_{\rm up})^{-1}W_{\rm up}^\top \, (E_{\rm up}^{\rm p})_{\bullet, \mathcal J_k} \,  \, y \nonumber \\
         Z_{xy}^{(3)} 
         & := x^\top\,
         U_{2k} (U_{1k}^\top  U_{1k})^{-1} \, U_{1k}^\top (E_{\rm up}^{\rm p})_{\mathcal I_k^\mathrm{up}, \bullet}
            \,V \, V_{2k}^\top \, \, y \nonumber \\
        Z_{xy}^{(4)} 
        & := x^\top\,
        (E_\mathrm{left}^\mathrm{p})_{\mathcal I_k, \bullet} \,\, W_\mathrm{left}(W_\mathrm{left}^\top W_\mathrm{left})^{-1}\mathcal C_{\bullet, \bullet, k} \, V_{2k}^\top  \, \, y.
    \end{align}
   There exists an event $\mathcal G_1$ with $\mathbb{P}(\mathcal G_1) \geq 1 - \mathcal{O}(p_N^{-10} + p_T^{-10})$ such that, under $\mathcal{G}_1$, the remainder satisfies \begin{align}\label{eq:BoundOnDeltaXY}
|\Delta_{xy}  | &\le
c_1 
\frac{\sigma^2\sqrt{(r+\zeta_T)(r+\zeta_N)}}{
\gamma_{\min}\rho_N^{1/2} \rho_T^{1/2}}                     
+
c_1
\frac{\sigma \sqrt{r+\zeta_T}}{
\rho_T^{1/2}}
\|V_{2k}^\top y\|_2   +
c_1
\frac{\sigma \sqrt{r+\zeta_N}}{
\rho_N^{1/2}}
\|U_{2k}^\top x\|_2                                 
   \nonumber \\
    & \qquad \qquad  
+
c_1
\frac{\sigma\sqrt{N}}{
\rho_T^{1/2}}
\|U_{2k}^\top x\|_2
\|V_{2k}^\top y\|_2 +
c_1
\frac{\sigma\sqrt{T}}{
\rho_N^{1/2}}
\|U_{2k}^\top x\|_2
\|V_{2k}^\top y\|_2
   \end{align}
   for a sufficiently large constant $c_1 \equiv c_1(c_\ell, c_u, c_0,  c_{\mathrm{dim}}, c_{\mathrm{blk}}, \kappa) > 0$. 
\end{lemma}

\begin{proof}
    Let $\mathcal G_1$ be the intersection of the high-probability events in Lemmas~\ref{lemma13_centered},~\ref{lemma15_centered},~\ref{lemmaL_action_centered},~\ref{lem:up-cen}, \ref{lemma:bilinearGaussian} and Corollaries~\ref{cor:centered-left-perturbation}, \ref{lemma13_Transpose}, \ref{cor:R1-sharp}, applied with the specific deterministic choices of projection matrices and vectors used below. By a union bound, we have $\mathbb{P}(\mathcal{G}_1) \geq 1 - \mathcal{O}(p_N^{-10} + p_T^{-10})$. In particular, arguing as in~\eqref{eq:HDaggerCoincidesWithInverse}, we know that, under $\mathcal{G}_1$, the matrix $\hat H_{k, \tau}^\mathrm{inv}$ used in Algorithm~\ref{alg:bilinear4block} coincides with $(\hat U_{1k}^\top \hat U_{1k})^{-1}$ whenever the algorithm is run with $\tau \leq \frac{c_\ell \, N_{1k}}{2 \, N}$. Recalling the notation $H_k =U_{1k}^\top U_{1k}, \,\hat H_k= H_U^\top \hat U_{1k}^\top\hat U_{1k}H_U, L=
\hat U_{2k}H_U
\hat H_k^{-1}
H_U^\top \hat U_{1k}^\top
-
U_{2k}H_k^{-1} U_{1k}^\top$, and using the expression~\eqref{eq:lemma16_centered_decomp_rewrite} and~\eqref{eq:decomopistionPhi_k} for $\Delta_L$ and $\Phi_k$, respectively, we can write
\begin{align}\label{eq:fullDecomposition}
   &\hat \mu_{xy}^{(k)} - \mu_{xy}^{(k)}
   = x^\top
    \hat U_{2k}(\hat U_{1k}^\top \hat U_{1k})^{-1}
    \hat U_{1k}^\top
    \hat U_\mathrm{up}^{(k)} \hat \Sigma_\mathrm{up} \hat V_{2k}^\top y
    -
    x^\top U_{2k}\mathcal C_{\bullet,\bullet,k}V_{2k}^\top y \nonumber \\
   &= x^\top
    \hat U_{2k}(\hat U_{1k}^\top \hat U_{1k})^{-1}
    \hat U_{1k}^\top \widehat M_b^{(k)} y
    -
    x^\top
    U_{2k}(U_{1k}^\top U_{1k})^{-1}
    U_{1k}^\top \mathcal M_{\bullet,\bullet,k}^{(b)} y \nonumber \\
   &= x^\top
    U_{2k}(U_{1k}^\top U_{1k})^{-1}
    U_{1k}^\top
    \bigl(\widehat M_b^{(k)}-\mathcal M_{\bullet,\bullet,k}^{(b)}\bigr)y
    +x^\top L\mathcal M_{\bullet,\bullet,k}^{(b)}y 
    +x^\top L
    \bigl(\widehat M_b^{(k)}-\mathcal M_{\bullet,\bullet,k}^{(b)}\bigr)y \nonumber \\[0.5em]
   &= x^\top
    U_{2k}(U_{1k}^\top U_{1k})^{-1}
    U_{1k}^\top
    \Bigl[
        (E_{\rm up}^{\rm p})_{\mathcal I_k^{\rm up},\bullet}VV_{2k}^\top
        +U_{1k}\mathcal C_{\bullet,\bullet,k}
        (W_{\rm up}^\top W_{\rm up})^{-1}W_{\rm up}^\top
        (E_{\rm up}^{\rm p})_{\bullet,\mathcal J_k}
        +\Phi_k
    \Bigr]y \nonumber \\
   &\qquad
    +x^\top L U_{1k}\mathcal C_{\bullet,\bullet,k}V_{2k}^\top y 
    +x^\top L
    \Bigl[
        (E_{\rm up}^{\rm p})_{\mathcal I_k^{\rm up},\bullet}VV_{2k}^\top
        +U_{1k}\mathcal C_{\bullet,\bullet,k}
        (W_{\rm up}^\top W_{\rm up})^{-1}W_{\rm up}^\top
        (E_{\rm up}^{\rm p})_{\bullet,\mathcal J_k}
        +\Phi_k
    \Bigr]y \nonumber \\[0.5em]
   &= x^\top
    U_{2k}(U_{1k}^\top U_{1k})^{-1}U_{1k}^\top
    (E_{\rm up}^{\rm p})_{\mathcal I_k^{\rm up},\bullet}
    VV_{2k}^\top y 
    +x^\top U_{2k}\mathcal C_{\bullet,\bullet,k}
    (W_{\rm up}^\top W_{\rm up})^{-1}W_{\rm up}^\top
    (E_{\rm up}^{\rm p})_{\bullet,\mathcal J_k}y \nonumber \\
   &\qquad
    +x^\top
    U_{2k}(U_{1k}^\top U_{1k})^{-1}U_{1k}^\top \Phi_k y 
    +x^\top L U_{1k}\mathcal C_{\bullet,\bullet,k}V_{2k}^\top y
    +x^\top L
    (E_{\rm up}^{\rm p})_{\mathcal I_k^{\rm up},\bullet}VV_{2k}^\top y \nonumber \\
   &\qquad
    +x^\top L U_{1k}\mathcal C_{\bullet,\bullet,k}
    (W_{\rm up}^\top W_{\rm up})^{-1}W_{\rm up}^\top
    (E_{\rm up}^{\rm p})_{\bullet,\mathcal J_k}y
    +x^\top L\Phi_k y \nonumber \\[0.5em]
   &= Z_{xy}^{(3)}+Z_{xy}^{(2)}
    +x^\top
    U_{2k}(U_{1k}^\top U_{1k})^{-1}U_{1k}^\top \Phi_k y \nonumber \\
   &\qquad
    +
        x^\top(E_{\rm left}^{\rm p})_{\mathcal I_k,\bullet}
        W_{\rm left}(W_{\rm left}^\top W_{\rm left})^{-1}
        \mathcal C_{\bullet,\bullet,k}V_{2k}^\top y
        +x^\top\Delta_L V_{2k}^\top y
      \nonumber \\
   &\qquad
    +x^\top L
    (E_{\rm up}^{\rm p})_{\mathcal I_k^{\rm up},\bullet}VV_{2k}^\top y +
        x^\top(E_{\rm left}^{\rm p})_{\mathcal I_k,\bullet}
        W_{\rm left}(W_{\rm left}^\top W_{\rm left})^{-1}
        \mathcal C_{\bullet,\bullet,k}
        (W_{\rm up}^\top W_{\rm up})^{-1}W_{\rm up}^\top
        (E_{\rm up}^{\rm p})_{\bullet,\mathcal J_k}y \nonumber \\
   & \qquad+x^\top\Delta_L
        (W_{\rm up}^\top W_{\rm up})^{-1}W_{\rm up}^\top
        (E_{\rm up}^{\rm p})_{\bullet,\mathcal J_k}y
    +x^\top L\Phi_k y \nonumber \\[0.5em]
   &= Z_{xy}^{(1)}+Z_{xy}^{(2)}+Z_{xy}^{(3)}+Z_{xy}^{(4)} 
    +x^\top
    U_{2k}(U_{1k}^\top U_{1k})^{-1}U_{1k}^\top\Phi_k y
    +x^\top\Delta_L V_{2k}^\top y \nonumber \\
   &\qquad 
    +x^\top\Delta_L
    (W_{\rm up}^\top W_{\rm up})^{-1}W_{\rm up}^\top
    (E_{\rm up}^{\rm p})_{\bullet,\mathcal J_k}y
    +x^\top L
    (E_{\rm up}^{\rm p})_{\mathcal I_k^{\rm up},\bullet}VV_{2k}^\top y +x^\top L\Phi_k y .
\end{align}
We will now bound each of the remainder terms individually. The first term is controlled directly by~\eqref{eq:R1-sharp-bound-corrected}, while for the second term it is enough to write $|x^\top \Delta_L V_{2k}^\top y| \leq \|x^\top \Delta_L\|_2 \|V_{2k}^\top y\|_2$, and bound the first factor using~\eqref{eq:lemma16_centered_bound_rewrite}. For the third one, start by observing that Lemma~\ref{lemma:bilinearGaussian} gives $\|(W_{\rm up}^\top W_{\rm up})^{-1}W_{\rm up}^\top
    (E_{\rm up}^{\rm p})_{\bullet,\mathcal J_k}y\|_2 \leq c_1 \sigma \, \gamma_\mathrm{min}^{-1}\sqrt{r +\zeta_N} \, \rho_N^{-1/2}$. Applying again~\eqref{eq:lemma16_centered_bound_rewrite} from Lemma~\ref{lemmaL_action_centered} then gives
    \begin{align*}
        |x^\top\Delta_L
    &(W_{\rm up}^\top W_{\rm up})^{-1}W_{\rm up}^\top
    (E_{\rm up}^{\rm p})_{\bullet,\mathcal J_k}y | \leq \|x^\top\Delta_L\|_2 \,\|
    (W_{\rm up}^\top W_{\rm up})^{-1}W_{\rm up}^\top
    (E_{\rm up}^{\rm p})_{\bullet,\mathcal J_k}y\|_2 \\
    &\leq c_1
\frac{\sigma^2\sqrt{(N+T_{1,\mathrm{p}})(r+\zeta_T)}}{
\gamma_{\min} \rho_T} \, \frac{\sigma \sqrt{r+ \zeta_N}}{\gamma_\mathrm{min} \, \rho_N^{1/2}} \, + \,  c_1 \frac{\sigma \sqrt{N}}{\rho_T^{1/2}}\, \|U_{2k}^\top x\|_2 \, \frac{\sigma \sqrt{r+ \zeta_N}}{\gamma_\mathrm{min} \, \rho_N^{1/2}} \\
& = c_1
\frac{\sigma^3 \sqrt{(N+T_{1,\mathrm{p}})(r+\zeta_N)(r + \zeta_T)}}{
\gamma_{\min}^2 \rho_N^{1/2} \rho_T} + c_1 \frac{\sigma \sqrt{N}}{\rho_T^{1/2}}\, \, \frac{\sigma \sqrt{r+ \zeta_N}}{\gamma_\mathrm{min} \, \rho_N^{1/2}} \|U_{2k}^\top x\|_2 \\
& \leq c_1
\frac{\sigma^2 \sqrt{(r +\zeta_N)(r+\zeta_T)}}{
\gamma_{\min} \rho_N^{1/2} \rho_T^{1/2}} + c_1 \, \frac{\sigma \sqrt{r+ \zeta_N}}{\rho_N^{1/2}} \|U_{2k}^\top x\|_2,
    \end{align*}
where the last inequality follows from \eqref{assump:smallNoise}. 
    
We next control the fourth term in~\eqref{eq:fullDecomposition}. By the definition of $L$, for any vector $w\in\mathbb R^{N_{1k}}$ we have $Lw
=
(\hat U_{2k}H_U-U_{2k})
H_k^{-1}U_{1k}^\top w
+
\hat U_{2k}H_U D_k w$, 
where $D_k=
\widehat H_k^{-1}(\hat U_{1k}H_U)^\top
-
H_k^{-1}U_{1k}^\top$. Applying this identity with $w=(E_{\rm up}^{\rm p})_{\mathcal I_k^{\rm up},\bullet}
VV_{2k}^\top y$ and using
$\hat U_{2k}H_U-U_{2k}
=
(E_{\rm left}^{\rm p})_{\{N_{1k}+1,\ldots,N\},\bullet}
W_{\rm left}(W_{\rm left}^\top W_{\rm left})^{-1}
+
\Psi_{U,2k}$ with~$\Psi_{U,2k}$ defined in Corollary~\ref{cor:centered-left-perturbation}, 
gives 
\begin{align}\label{eq:fourthTerm}
x^\top L
(E_{\rm up}^{\rm p})_{\mathcal I_k^{\rm up},\bullet}
VV_{2k}^\top y                                  & =
x^\top
(E_{\rm left}^{\rm p})_{\{N_{1k}+1,\ldots,N\},\bullet}
W_{\rm left}(W_{\rm left}^\top W_{\rm left})^{-1}
H_k^{-1}U_{1k}^\top
(E_{\rm up}^{\rm p})_{\mathcal I_k^{\rm up},\bullet}
VV_{2k}^\top y                                                 \nonumber  \\
&\quad
+
x^\top
\Psi_{U,2k}
H_k^{-1}U_{1k}^\top
(E_{\rm up}^{\rm p})_{\mathcal I_k^{\rm up},\bullet}
VV_{2k}^\top y                                  
+
x^\top
\hat U_{2k}H_U D_k
(E_{\rm up}^{\rm p})_{\mathcal I_k^{\rm up},\bullet}
VV_{2k}^\top y .
\end{align}
  We bound these three pieces separately. Since $(E_{\rm left}^{\rm p})_{\{N_{1k}+1, \ldots, N\},\bullet}$ and $(E_{\rm up}^{\rm p})_{\mathcal I_k^{\rm up},\bullet}$ are independent, conditionally on the upper-pooled noise
$(E_{\rm up}^{\rm p})_{\mathcal I_k^{\rm up},\bullet}$, the first term
is a Gaussian random variable with conditional variance $\sigma^2
\|
W_{\rm left}(W_{\rm left}^\top W_{\rm left})^{-1}
H_k^{-1}U_{1k}^\top
(E_{\rm up}^{\rm p})_{\mathcal I_k^{\rm up},\bullet}
VV_{2k}^\top y
\|_2^2$ and mean zero. Combining this with Lemma~\ref{lemma:bilinearGaussian} and a standard Gaussian tail bound gives
\begin{align*}
|x^\top &
(E_{\rm left}^{\rm p})_{\{N_{1k}+1,\ldots,N\},\bullet}
W_{\rm left}(W_{\rm left}^\top W_{\rm left})^{-1}
H_k^{-1}U_{1k}^\top
(E_{\rm up}^{\rm p})_{\mathcal I_k^{\rm up},\bullet}
VV_{2k}^\top y| \\
& \leq c_1 \sigma \sqrt{\zeta_T}\,
\left\|
W_{\rm left}(W_{\rm left}^\top W_{\rm left})^{-1}
(U_{1k}^\top U_{1k})^{-1} U_{1k}^\top
(E_{\rm up}^{\rm p})_{\mathcal I_k^\mathrm{up}, \bullet}
\,V \, V_{2k}^\top
y
\right\|_2 \\
& \leq c_1 \sigma \sqrt{\zeta_T }\,
\left\|
W_{\rm left}(W_{\rm left}^\top W_{\rm left})^{-1}
(U_{1k}^\top U_{1k})^{-1} U_{1k}^\top
(E_{\rm up}^{\rm p})_{\mathcal I_k^\mathrm{up}, \bullet}
\,V \right\|_\mathrm{op} \, \left\| V_{2k}^\top
y
\right\|_2 \\
& \leq
c_1 \sigma^2 \sqrt{\zeta_T\, (r + \zeta_N)}
\left\|W_{\rm left}(W_{\rm left}^\top W_{\rm left})^{-1} \right\|_\mathrm{op} \,
\left\| (U_{1k}^\top U_{1k})^{-1} U_{1k}^\top \right\|_\mathrm{op}
\, \left\| V_{2k}^\top
y
\right\|_2 \\
& \leq  c_1 \frac{\sigma^2}{\gamma_\mathrm{min}}
\sqrt{\frac{N \,\zeta_T\, (r + \zeta_N)}{N_{1k} \, \rho_T}}
\, \left\| V_{2k}^\top
y
\right\|_2 \leq c_1 \frac{\sigma \sqrt{r + \zeta_T}}{\rho_T^{1/2}} \, \left\| V_{2k}^\top
y
\right\|_2
\end{align*}
           with probability at least $1 - \mathcal{O}(p_N^{-10} + p_T^{-10})$,  where in the penultimate inequality we used Lemma~\ref{lemma:spectrumMp} to get $\|W_{\rm left}\,(W_{\rm left}^\top W_{\rm left})^{-1}\|_\mathrm{op} \leq \sigma_r^{-1}(M_\mathrm{left}^\mathrm{p}) \leq c_\ell^{-1/2} \gamma_\mathrm{min}^{-1} \, \rho_T^{-1/2}$, and  \eqref{assump:subblock-conditioning} to get $\left\| (U_{1k}^\top U_{1k})^{-1} U_{1k}^\top \right\|_\mathrm{op} \leq c_\ell^{-1/2} \sqrt{N/N_{1k}}$.  Similarly, for the second piece in~\eqref{eq:fourthTerm}, Lemma~\ref{lemma:bilinearGaussian},   \eqref{assump:subblock-conditioning},  \eqref{assump:sampleSize},  \eqref{assump:smallNoise} and~\eqref{cor:bound_Psi} give 
\begin{align*}
&\bigg|
x^\top
\Psi_{U,2k}
H_k^{-1}U_{1k}^\top
(E_{\rm up}^{\rm p})_{\mathcal I_k^{\rm up},\bullet}
VV_{2k}^\top y
\bigg|  \\                  & \le
\|\Psi_{U,2k}^\top x\|_2
\left\|
H_k^{-1}U_{1k}^\top
(E_{\rm up}^{\rm p})_{\mathcal I_k^{\rm up},\bullet}
VV_{2k}^\top y
\right\|_2 \leq  c_1 \, \sigma
\sqrt{\frac{N(r+\zeta_N)}{N_{1k}}}
\,
\|V_{2k}^\top y\|_2 \, \|\Psi_{U,2k}^\top x\|_2 \\
& \leq c_1 \, \sigma
\sqrt{\frac{N(r+\zeta_N)}{N_{1k}}}
\,
\|V_{2k}^\top y\|_2 \, \left\{\frac{\sigma^2\sqrt{(N+T_{1,\mathrm{p}})(r+\zeta_T)}}{
\gamma_{\min}^2\rho_T}
+
\left(
\frac{\sigma^2N}{\gamma_{\min}^2\rho_T}
+
\frac{\sigma\sqrt{r+\zeta_T}}{\gamma_{\min}\rho_T^{1/2}}
\right)
\|U_{2k}^\top x\|_2\right\} \\
& \leq c_1
\frac{\sigma^3
\sqrt{N(r+\zeta_N)(N+T_{1,\mathrm{p}})(r+\zeta_T)}}{
\gamma_{\min}^2\rho_T\sqrt{N_{1k}}}
\,
\|V_{2k}^\top y\|_2                                    \\
&\quad
+
c_1\sigma
\sqrt{\frac{N(r+\zeta_N)}{N_{1k}}}
\left(
\frac{\sigma^2N}{\gamma_{\min}^2\rho_T}
+
\frac{\sigma\sqrt{r+\zeta_T}}{\gamma_{\min}\rho_T^{1/2}}
\right)
\|U_{2k}^\top x\|_2
\|V_{2k}^\top y\|_2 \\
&  \leq c_1 \frac{\sigma \sqrt{r + \zeta_T}}{\rho_T^{1/2}} \, \left\| V_{2k}^\top
y
\right\|_2 + c_1 \frac{\sigma \sqrt{N}}{\rho_T^{1/2}}  \, \left\| U_{2k}^\top
x
\right\|_2  \, \left\| V_{2k}^\top
y
\right\|_2.
\end{align*}
For the third piece,~\eqref{eq:boundU_vectorX}, \eqref{assump:sampleSize},  \eqref{assump:smallNoise}, Lemmas~\ref{lemma15_centered} and \ref{lemma:bilinearGaussian} yield
\begin{align*}
\bigg|
x^\top
\hat U_{2k}H_U D_k
&(E_{\rm up}^{\rm p})_{\mathcal I_k^{\rm up},\bullet}
VV_{2k}^\top y
\bigg|                                                \le
\|H_U^\top\hat U_{2k}^\top x\|_2
\,
\|D_k\|_{\rm op}
\,
\left\|
(E_{\rm up}^{\rm p})_{\mathcal I_k^{\rm up},\bullet}
VV_{2k}^\top y
\right\|_2 \\
& \leq c_1
\left(
\frac{\sigma\sqrt{r+\zeta_T}}{\gamma_{\min}\rho_T^{1/2}}
+
\|U_{2k}^\top x\|_2
\right) \,
\frac{\sigma\sqrt N}{\gamma_{\min}\rho_T^{1/2}}
\sqrt{\frac{N}{N_{1k}}}
\,
\sigma\sqrt{N_{1k}}\,
\|V_{2k}^\top y\|_2 \\
& \leq c_1
\frac{\sigma^3 N\sqrt{r+\zeta_T}}{
\gamma_{\min}^2\rho_T}
\,
\|V_{2k}^\top y\|_2         +
c_1
\frac{\sigma^2N}{
\gamma_{\min}\rho_T^{1/2}}
\,
\|U_{2k}^\top x\|_2
\|V_{2k}^\top y\|_2 \\
& \leq c_1\frac{\sigma \sqrt{r + \zeta_T}}{\rho_T^{1/2}} \, \left\| V_{2k}^\top
y
\right\|_2 + c_1 \frac{\sigma \sqrt{N}}{\rho_T^{1/2}}  \, \left\| U_{2k}^\top
x
\right\|_2  \, \left\| V_{2k}^\top
y
\right\|_2.
\end{align*}
Combining the last three displays leads to 
\begin{align*}
\bigg|
x^\top L
&(E_{\rm up}^{\rm p})_{\mathcal I_k^{\rm up},\bullet}
VV_{2k}^\top y
\bigg|                                                  \leq c_1\frac{\sigma \sqrt{r + \zeta_T}}{\rho_T^{1/2}} \, \left\| V_{2k}^\top
y
\right\|_2 + c_1\frac{\sigma \sqrt{N}}{\rho_T^{1/2}}  \, \left\| U_{2k}^\top
x
\right\|_2  \, \left\| V_{2k}^\top
y
\right\|_2.
\end{align*}

Finally, we control the fifth term in~\eqref{eq:fullDecomposition}. Letting $P_{1k}:=U_{1k}(U_{1k}^\top U_{1k})^{-1}U_{1k}^\top$, we can write \begin{align}\label{eq:fifthTerm}
    x^\top L\Phi_k y
=
x^\top L P_{1k}\Phi_k y
+
x^\top L(I_{N_{1k}}-P_{1k})\Phi_k y.
\end{align} 
We will now bound the first projected component. Set $z:=
\mathcal C_{\bullet,\bullet,k}^{-1}
(U_{1k}^\top U_{1k})^{-1}U_{1k}^\top\Phi_k y$, so that $P_{1k}\Phi_k y
=
U_{1k}(U_{1k}^\top U_{1k})^{-1}U_{1k}^\top\Phi_k y
=
U_{1k}\mathcal C_{\bullet,\bullet,k}z$. Using the definition of $\Delta_L$ in~\eqref{eq:lemma16_centered_decomp_rewrite} we obtain
\begin{align*}
x^\top L P_{1k}\Phi_k y
&=
x^\top
(E_{\rm left}^{\rm p})_{\{N_{1k}+1,\ldots,N\},\bullet}
W_{\rm left}(W_{\rm left}^\top W_{\rm left})^{-1}
\mathcal C_{\bullet,\bullet,k}z
+
x^\top\Delta_L z                                                   \\
&=
x^\top
(E_{\rm left}^{\rm p})_{\{N_{1k}+1,\ldots,N\},\bullet}
W_{\rm left}(W_{\rm left}^\top W_{\rm left})^{-1}
(U_{1k}^\top U_{1k})^{-1}U_{1k}^\top\Phi_k y
+
x^\top\Delta_L z .
\end{align*}
We bound these two pieces separately. For the first one, Lemma~\ref{lemma:bilinearGaussian},  \eqref{assump:subblock-conditioning} and~\eqref{eq:prokection_1Phi_k} give 
\begin{align}\label{eq:fifth-term-left-projected-bound}
\bigg|
x^\top
&(E_{\rm left}^{\rm p})_{\{N_{1k}+1,\ldots,N\},\bullet}
W_{\rm left}(W_{\rm left}^\top W_{\rm left})^{-1}
(U_{1k}^\top U_{1k})^{-1}U_{1k}^\top\Phi_k y
\bigg|                                                  \nonumber \\
&\le
\left\|
x^\top
(E_{\rm left}^{\rm p})_{\{N_{1k}+1,\ldots,N\},\bullet}
W_{\rm left}(W_{\rm left}^\top W_{\rm left})^{-1}
(U_{1k}^\top U_{1k})^{-1/2}
\right\|_2                 \, 
\left\|
(U_{1k}^\top U_{1k})^{-1/2}U_{1k}^\top\Phi_k y
\right\|_2\nonumber \\
& \leq \frac{\sigma}{\gamma_{\min}}
\sqrt{\frac{N (r+\zeta_T)}{N_{1k} \, \rho_T}} \left\{ \frac{\sigma\sqrt{r+\zeta_N}}{\rho_N^{1/2}}
\sqrt{\frac{N_{1k}}{N}}
+
\frac{\sigma\sqrt T}{\rho_N^{1/2}}
\sqrt{\frac{N_{1k}}{N}}
\|V_{2k}^\top y\|_2 \right\} \nonumber\\
& \leq c_1
\frac{\sigma^2\sqrt{(r+\zeta_T)(r+\zeta_N)}}{
\gamma_{\min}\rho_N^{1/2} \rho_T^{1/2}}    
+
c_1
\frac{\sigma^2\sqrt{T(r+\zeta_T)}}{
\gamma_{\min}\rho_N^{1/2} \rho_T^{1/2}}
\|V_{2k}^\top y\|_2 \nonumber \\
& \leq c_1
\frac{\sigma^2\sqrt{(r+\zeta_T)(r+\zeta_N)}}{
\gamma_{\min}\rho_N^{1/2} \rho_T^{1/2}}    
+
c_1
\frac{\sigma \sqrt{r+\zeta_T}}{
\rho_T^{1/2}}
\|V_{2k}^\top y\|_2.
\end{align}
For the second projected piece, using
$\sigma_{\min}(\mathcal C_{\bullet,\bullet,k})\ge \gamma_{\min}>0$,  \eqref{assump:subblock-conditioning}, and~\eqref{eq:prokection_1Phi_k},
we have
\begin{align*}
\|z\|_2
&=
\left\|
\mathcal C_{\bullet,\bullet,k}^{-1}
(U_{1k}^\top U_{1k})^{-1}U_{1k}^\top\Phi_k y
\right\|_2                                                \\
&\le
\gamma_{\min}^{-1}
\left\|
(U_{1k}^\top U_{1k})^{-1/2}
\right\|_{\rm op}
\left\|
(U_{1k}^\top U_{1k})^{-1/2}U_{1k}^\top\Phi_k y
\right\|_2                                           \le
c_1
\frac{\sigma\sqrt{r+\zeta_N}}{\gamma_{\min}\rho_N^{1/2}}
+
c_1
\frac{\sigma\sqrt T}{\gamma_{\min}\rho_N^{1/2}}
\|V_{2k}^\top y\|_2 .
\end{align*}
Combining this with~\eqref{eq:lemma16_centered_bound_rewrite} gives
\begin{align}
|x^\top\Delta_L z|
&\le \|x^\top\Delta_L\|_2 \| z \|_2 \le
c_1
\frac{\sigma^2\sqrt{(N+T_{1,\mathrm{p}})(r+\zeta_T)}}{
\gamma_{\min}\rho_T}
\|z\|_2
+
c_1
\frac{\sigma\sqrt N}{\rho_T^{1/2}}
\|U_{2k}^\top x\|_2
\|z\|_2                                                   \notag\\
&\le
c_1
\frac{\sigma^3\sqrt{(N+T_{1,\mathrm{p}})(r+\zeta_T)(r+\zeta_N)}}{
\gamma_{\min}^2\rho_T\rho_N^{1/2}}                     
+
c_1
\frac{\sigma^3\sqrt{T(N+T_{1,\mathrm{p}})(r+\zeta_T)}}{
\gamma_{\min}^2\rho_T\rho_N^{1/2}}
\|V_{2k}^\top y\|_2                                       \notag\\
&\quad
+
c_1
\frac{\sigma^2\sqrt{N(r+\zeta_N)}}{
\gamma_{\min}\rho_T^{1/2}\rho_N^{1/2}}
\|U_{2k}^\top x\|_2         
+
c_1
\frac{\sigma^2\sqrt{NT}}{
\gamma_{\min}\rho_T^{1/2}\rho_N^{1/2}}
\|U_{2k}^\top x\|_2
\|V_{2k}^\top y\|_2 \nonumber \\
& \leq c_1
\frac{\sigma^2\sqrt{(r+\zeta_T)(r+\zeta_N)}}{
\gamma_{\min}\rho_N^{1/2} \rho_T^{1/2}}                     
+
c_1
\frac{\sigma \sqrt{r+\zeta_T}}{
\rho_T^{1/2}}
\|V_{2k}^\top y\|_2   +
c_1
\frac{\sigma \sqrt{r+\zeta_N}}{
\rho_N^{1/2}}
\|U_{2k}^\top x\|_2                                 
     \\
     & \qquad +
c_1
\frac{\sigma\sqrt{T}}{
\rho_N^{1/2}}
\|U_{2k}^\top x\|_2
\|V_{2k}^\top y\|_2,
\label{eq:fifth-term-delta-projected-bound}
\end{align}
and concludes the analysis of the first term in~\eqref{eq:fifthTerm}. It remains to control the component orthogonal to $U_{1k}$. In particular,  since $U_{1k}^\top(I_{N_{1k}}-P_{1k})\Phi_k y=0$, the population part in the decomposition of $L$ vanishes,  hence
\begin{align*}
L(I_{N_{1k}}-P_{1k})\Phi_k y
&=
\hat U_{2k}H_U
\hat H_k^{-1}
H_U^\top\hat U_{1k}^\top
(I_{N_{1k}}-P_{1k})\Phi_k y
-
U_{2k}H_k^{-1}U_{1k}^\top
(I_{N_{1k}}-P_{1k})\Phi_k y                              \\
&=
\hat U_{2k}H_U
\left(
\hat H_k^{-1}H_U^\top\hat U_{1k}^\top
-
H_k^{-1}U_{1k}^\top
\right)
(I_{N_{1k}}-P_{1k})\Phi_k y                              \\
&=
\hat U_{2k}H_U D_k(I_{N_{1k}}-P_{1k})\Phi_k y .
\end{align*}
Therefore, \eqref{eq:boundU_vectorX}, \eqref{eq:lemma15_centered_D_rewrite}, \eqref{eq:prokection_2Phi_k} and \eqref{assump:smallNoise} give
\begin{align}\label{eq:fifth-term-orthogonal-bound}
|
x^\top& L(I_{N_{1k}}-P_{1k})\Phi_k y
|
\le
\|H_U^\top\hat U_{2k}^\top x\|_2
\|D_k\|_{\rm op}
\|(I_{N_{1k}}-P_{1k})\Phi_k y\|_2 \nonumber \\
& \le
c_1
\left(
\frac{\sigma\sqrt{r+\zeta_T}}{\gamma_{\min}\rho_T^{1/2}}
+
\|U_{2k}^\top x\|_2
\right)
\frac{\sigma\sqrt N}{\gamma_{\min}\rho_T^{1/2}}
\sqrt{\frac{N}{N_{1k}}}                                      \notag\\
&\qquad\qquad\times
\left(
\frac{\sigma^3 T\sqrt{r+\zeta_N}}{\gamma_{\min}^2\rho_N}
+
\frac{\sigma^2\sqrt{N_{1k}(r+\zeta_N)}}{
\gamma_{\min}\rho_N^{1/2}}
+
\frac{\sigma^2\sqrt{T(N_{1k}+T)}}{
\gamma_{\min}\rho_N^{1/2}}
\|V_{2k}^\top y\|_2
\right) \nonumber \\
&=
c_1 \frac{
\sigma^5 N T
\sqrt{(r+\zeta_T)(r+\zeta_N)}
}{
\gamma_{\min}^4 \rho_T \rho_N \sqrt{N_{1k}}
}
+c_1
\frac{
\sigma^4 N
\sqrt{(r+\zeta_T)(r+\zeta_N)}
}{
\gamma_{\min}^3 \rho_T \rho_N^{1/2}
} 
\nonumber \\
     & \qquad \qquad + c_1
\frac{
\sigma^4 N
\sqrt{(r+\zeta_T)T(N_{1k}+T)}
}{
\gamma_{\min}^3 \rho_T \rho_N^{1/2}\sqrt{N_{1k}}
}
\|V_{2k}^\top y\|_2 \nonumber \\
&\qquad
+c_1
\frac{
\sigma^4 N T \sqrt{r+\zeta_N}
}{
\gamma_{\min}^3 \rho_T^{1/2}\rho_N \sqrt{N_{1k}}
}
\|U_{2k}^\top x\|_2 
+c_1
\frac{
\sigma^3 N \sqrt{r+\zeta_N}
}{
\gamma_{\min}^2 \rho_T^{1/2}\rho_N^{1/2}
}
\|U_{2k}^\top x\|_2
\nonumber \\
    & \qquad \qquad +c_1
\frac{
\sigma^3 N \sqrt{T(N_{1k}+T)}
}{
\gamma_{\min}^2 \rho_T^{1/2}\rho_N^{1/2}\sqrt{N_{1k}}
}
\|U_{2k}^\top x\|_2
\|V_{2k}^\top y\|_2 \nonumber \\
& \leq c_1
\frac{\sigma^2\sqrt{(r+\zeta_T)(r+\zeta_N)}}{
\gamma_{\min}\rho_N^{1/2} \rho_T^{1/2}}                     
+
c_1
\frac{\sigma \sqrt{r+\zeta_T}}{
\rho_T^{1/2}}
\|V_{2k}^\top y\|_2   +
c_1
\frac{\sigma \sqrt{r+\zeta_N}}{
\rho_N^{1/2}}
\|U_{2k}^\top x\|_2                                 
     \nonumber \\
    & \qquad \qquad +
c_1
\frac{\sigma\sqrt{N}}{
\rho_T^{1/2}}
\|U_{2k}^\top x\|_2
\|V_{2k}^\top y\|_2.
\end{align}

Combining~\eqref{eq:fifthTerm},
\eqref{eq:fifth-term-left-projected-bound},
\eqref{eq:fifth-term-delta-projected-bound}, and
\eqref{eq:fifth-term-orthogonal-bound}, we obtain
\begin{align*}
|x^\top L\Phi_k y|
&\le
c_1 
\frac{\sigma^2\sqrt{(r+\zeta_T)(r+\zeta_N)}}{
\gamma_{\min}\rho_N^{1/2} \rho_T^{1/2}}                     
+
c_1
\frac{\sigma \sqrt{r+\zeta_T}}{
\rho_T^{1/2}}
\|V_{2k}^\top y\|_2   +
c_1
\frac{\sigma \sqrt{r+\zeta_N}}{
\rho_N^{1/2}}
\|U_{2k}^\top x\|_2                                 
    \\
    & \qquad \qquad  +
c_1
\frac{\sigma\sqrt{N}}{
\rho_T^{1/2}}
\|U_{2k}^\top x\|_2
\|V_{2k}^\top y\|_2 +
c_1
\frac{\sigma\sqrt{T}}{
\rho_N^{1/2}}
\|U_{2k}^\top x\|_2
\|V_{2k}^\top y\|_2. 
\end{align*}
Combining all the previous inequalities gives a bound on $|\Delta_{xy}|$ and completes the proof.
\end{proof}

\medskip\medskip\medskip
\begin{lemma}\label{lemma:VarDominatesLemma1}
    Consider the setting of Lemma~\ref{lemma1}, and further assume \eqref{assump:Incoherence} with $\nu_x \neq 0, \nu_y \neq 0$. Also let $\mathcal G_1$ be the event such that $\mathbb P(\mathcal G_1)\geq 1 - \mathcal{O}(p_N^{-10} + p_T^{-10})$ under which~\eqref{eq:BoundOnDeltaXY} holds. Define 
    \[
    \Upsilon_{xy} := \frac{\sigma^2 (r + \zeta_N)}{\rho_N} \, \|U_{2k}^\top x\|_2^2 +\frac{\sigma^2 (r + \zeta_T)}{\rho_T} \, \|V_{2k}^\top y\|_2^2 + \frac{\sigma^2 N}{N_{1k}} \, \|U_{2k}^\top x\|_2^2 \, \|V_{2k}^\top y\|_2^2. 
    \]
    We have $\mathbb{E} [Z_{xy}^2]  \leq c_1 \, \Upsilon_{xy}$ for a sufficiently large constant $c_1 \equiv c_1(c_\ell, c_u, c_0,  c_{\mathrm{dim}}, c_{\mathrm{blk}}, \kappa, \nu_x, \nu_y) >~0$. Furthermore, under $\mathcal G_1$, the remainder satisfies $\Delta_{xy}^2 \leq c_1\Upsilon_{xy}$.
\end{lemma}
\begin{proof}
We will use $\mathbb E[Z_{xy}^2] \leq 4\sum_{i = 1}^4 \mathbb E [(Z_{xy}^{(i)})^2]$, and bound the second moment of each $Z_{xy}^{(i)}$ in~\eqref{eq:Z_xy} separately. First, from $Z_{xy}^{(2)} \sim \mathcal{N}(0, \sigma^2 \, \bigl\|x^\top U_{2k} \,  \mathcal{C}_{\bullet, \bullet, k} \, (W_{\rm up}^\top W_{\rm up})^{-1}W_{\rm up}^\top \bigr\|_2^2)$ we get
\begin{align*}
     \mathbb E [(Z_{xy}^{(2)})^2] = \operatorname{Var}(Z_{xy}^{(2)}) 
    & = \sigma^2 \, \bigl\|x^\top U_{2k} \,  \mathcal{C}_{\bullet, \bullet, k} \, (W_{\rm up}^\top W_{\rm up})^{-1}W_{\rm up}^\top \bigr\|_2^2 \\
    & \leq \sigma^2 \, \bigl\| U_{2k}^\top x\|_2^2 \,  \| \mathcal{C}_{\bullet, \bullet, k} \, (W_{\rm up}^\top W_{\rm up})^{-1}W_{\rm up}^\top\bigr\|_\mathrm{op}^2 \leq\frac{\kappa^2 \sigma^2 }{c_\ell \, \rho_N} \, \|U_{2k}^\top x\|_2^2,
\end{align*}
where the second inequality follows from Lemma~\ref{lemma:spectrumMp}. Furthermore, similar computations allow showing that $  \mathbb E [(Z_{xy}^{(4)})^2] = \operatorname{Var}(Z_{xy}^{(4)}) \leq \kappa^2 \sigma^2 \, c_\ell^{-1} \, \|V_{2k}^\top y\|_2^2 \, \rho_T^{-1}$ and $ \mathbb E [(Z_{xy}^{(3)})^2] = \operatorname{Var}(Z_{xy}^{(3)}) 
     \leq  \sigma^2 \, c_\ell^{-1} \, (N/N_{1k})\, \|U_{2k}^\top x \|_2^2 \, \|V_{2k}^\top y \|_2^2$. Finally, for the first term we have
\begin{align*} \mathbb E&[(Z_{xy}^{(1)})^2] \le \left\| (W_{\rm left}^\top W_{\rm left})^{-1/2} \mathcal C_{\bullet,\bullet,k} (W_{\rm up}^\top W_{\rm up})^{-1/2} \right\|_{\rm op}^2 \\ &\qquad \times \mathbb E\Bigg[ \left\| (W_{\rm left}^\top W_{\rm left})^{-1/2} W_{\rm left}^\top (E_{\rm left}^{\rm p})_{\mathcal I_k,\bullet}^{\top} x \right\|_2^2  \,  \left\| (W_{\rm up}^\top W_{\rm up})^{-1/2} W_{\rm up}^\top (E_{\rm up}^{\rm p})_{\bullet,\mathcal J_k} y \right\|_2^2 \Bigg] \\ &\le \left\| (W_{\rm left}^\top W_{\rm left})^{-1/2} \mathcal C_{\bullet,\bullet,k} (W_{\rm up}^\top W_{\rm up})^{-1/2} \right\|_{\rm op}^2 \\ 
&\qquad \times \Bigg\{ \mathbb E \left\| (W_{\rm left}^\top W_{\rm left})^{-1/2} W_{\rm left}^\top (E_{\rm left}^{\rm p})_{\mathcal I_k,\bullet}^{\top} x \right\|_2^4 \Bigg\}^{1/2} \, \Bigg\{ \mathbb E \left\| (W_{\rm up}^\top W_{\rm up})^{-1/2} W_{\rm up}^\top (E_{\rm up}^{\rm p})_{\bullet,\mathcal J_k} y \right\|_2^4 \Bigg\}^{1/2} \\  
& =  \sigma^4 r(r+2) \left\| (W_{\rm left}^\top W_{\rm left})^{-1/2} \mathcal C_{\bullet,\bullet,k} (W_{\rm up}^\top W_{\rm up})^{-1/2} \right\|_{\rm op}^2 \\ 
& \leq\frac{\gamma_{\max}^2}{c_\ell^2\gamma_{\min}^4} \frac{\sigma^4 r(r+2)}{\rho_T\rho_N} \le c_1 \frac{\sigma^4 r^2}{\gamma_{\min}^2\rho_T\rho_N}. \end{align*}
The second bound follows from the Cauchy--Schwarz inequality, the first and only equality uses the fact that
$\sigma^{-1} (W_{\rm left}^\top W_{\rm left})^{-1/2} W_{\rm left}^\top (E_{\rm left}^{\rm p})_{\mathcal I_k,\bullet}^{\top} x$
and
$\sigma^{-1} (W_{\rm up}^\top W_{\rm up})^{-1/2} W_{\rm up}^\top (E_{\rm up}^{\rm p})_{\bullet,\mathcal J_k} y$
are standard normal vectors in~$\mathbb R^r$, and hence have fourth moment $r(r+2)$ in squared Euclidean norm, while the penultimate inequality uses Lemma~\ref{lemma:spectrumMp} to obtain
$\lambda_{\min}(W_{\rm left}^\top W_{\rm left}) \ge c_\ell \gamma_{\min}^2 \rho_T$
and
$\lambda_{\min}(W_{\rm up}^\top W_{\rm up}) \ge c_\ell \gamma_{\min}^2 \rho_N$. Furthermore, we can show that this term is dominated by either one of the first two terms in $\Upsilon_{xy}$. Indeed, we have 
\begin{align*}
    \frac{\sigma^4 r^2}{\gamma_{\min}^2\rho_T\rho_N} & = \frac{\sigma^4 r^2}{\gamma_{\min}^2\rho_T\rho_N} \frac{N}{\nu_x^2 \, r} \|U_{2k}^\top x\|_2^2 = \frac{\sigma^2 N}{\gamma_\mathrm{min}^2 \, \rho_T}\, \frac{1}{\nu_x^2} \frac{\sigma^2 r}{\rho_N} \|U_{2k}^\top x\|_2^2 \leq \frac{c_0^2}{\nu_x^2} \frac{\sigma^2 (r+\zeta_N)}{\rho_N} \|U_{2k}^\top x\|_2^2.
\end{align*}
This, combined with the previous bounds and $\min(r + \zeta_N, r + \zeta_T) \geq 1$, gives $ \mathbb E [Z_{xy}^2] \leq c_1 \, \Upsilon_{xy}$.

Coming now to bounding the remainder, we observe that the second and third terms in~\eqref{eq:BoundOnDeltaXY} appear in the definition of $\Upsilon_{xy}$. It thus remains to control the other three. Using the definition of $\nu_x$ in \eqref{assump:Incoherence}, for the first one we get 
\begin{align*}
    \frac{\sigma^2\sqrt{(r+\zeta_T)(r+\zeta_N)}}{
\gamma_{\min}\rho_N^{1/2} \rho_T^{1/2}}   & = \frac{\sigma^2\sqrt{(r+\zeta_T)(r+\zeta_N)}}{
\gamma_{\min}\rho_N^{1/2} \rho_T^{1/2}}    \frac{\sqrt{N}}{\nu_x \sqrt{r}} \, \|U_{2k}^\top x\|_2 \\
& = \frac{\sigma \sqrt{N}}{\gamma_\mathrm{min} \rho_T^{1/2}} \frac{\sigma \sqrt{(r+\zeta_N)(1 + \zeta_T/r)}}{\nu_x \, \rho_N^{1/2}} \, \|U_{2k}^\top x\|_2 \\
& \leq \frac{c_0}{\nu_x}  \frac{\sigma \sqrt{(r+\zeta_N)(1 + \zeta_T/r)}}{\rho_N^{1/2}} \, \|U_{2k}^\top x\|_2.
\end{align*}
Arguing by symmetry and using $c_{\mathrm{blk}}  \min(\zeta_N, \zeta_T) \leq r$, we can thus conclude 
\[
\frac{\sigma^2\sqrt{(r+\zeta_T)(r+\zeta_N)}}{
\gamma_{\min}\rho_N^{1/2}\rho_T^{1/2}}
\le
c_0\sqrt{1+c_{\mathrm{blk}}^{-1}} \,
\max\left\{
\frac{1}{\nu_x}
\frac{\sigma\sqrt{r+\zeta_N}}{\rho_N^{1/2}}
\|U_{2k}^\top x\|_2,\,
\frac{1}{\nu_y}
\frac{\sigma\sqrt{r+\zeta_T}}{\rho_T^{1/2}}
\|V_{2k}^\top y\|_2
\right\}.
\]
Similarly, we can bound the fourth term  in~\eqref{eq:BoundOnDeltaXY} by
\begin{align*}
    \frac{\sigma \sqrt{N}}{ \rho_T^{1/2}} \, \|U_{2k}^\top x\|_2 \, \|V_{2k}^\top y\|_2 &= \frac{\sigma \sqrt{N}}{ \rho_T^{1/2}} \, \frac{\nu_x\, \sqrt{r}}{\sqrt{N}} \, \|V_{2k}^\top y\|_2  = \nu_x \, \frac{\sigma \sqrt{r}}{  \rho_T^{1/2}} \, \|V_{2k}^\top y\|_2.
\end{align*}
An analogous bound holds for the fifth term, thereby showing that $|\Delta_{xy}  | \leq c_1 \sqrt{\Upsilon}_{xy}$ on the event where~\eqref{eq:BoundOnDeltaXY} holds. This completes the proof.
 \end{proof}

\section{Additional results}\label{appendix:additionalResults}

\subsection{Sufficient conditions for  \eqref{assump:subblock-conditioning}}\label{sec:sufficientA1}
We comment further on   \eqref{assump:subblock-conditioning} by providing two sufficient conditions under which it holds, either deterministically or with high probability.
\begin{lemma}
\label{lem:A1-deterministic}
Let $U\in\mathbb{R}^{N\times r}$ have orthonormal columns and assume there exists $\mu\ge 2$ such that $\|U^\top \boldsymbol{e}_j\|_2^2 \le \mu r/N $ for all $j\in[N]$. Let~$U_1$ be any submatrix formed by selecting $N_1$ rows of $U$. If $N_1 \ge \frac{2\mu r}{2\mu r + 1}\,N$ then assumption \emph{  \eqref{assump:subblock-conditioning}} holds with $c_\ell=\frac12$ and $c_u=\frac54$.
\end{lemma}

\begin{proof}
Write the row vectors $u_j := U^\top \boldsymbol{e}_j\in\mathbb{R}^r$, so that $\sum_{j=1}^N u_j u_j^\top = U^\top U = I_r$.
Let $S\subset[N]$ index the $N_1$ selected rows, and let $S^\complement$ be the complement with $N_2:=|S^\complement|=N-N_1$.
Then
\[
U_1^\top U_1 = \sum_{j\in S} u_j u_j^\top
= I_r - \sum_{j\in S^\complement} u_j u_j^\top
= I_r - U_2^\top U_2,
\]
where $U_2$ is the submatrix of the remaining $N_2$ rows. Since $U_2^\top U_2\succeq 0$ we immediately get $U_1^\top U_1 \preceq I_r$.

For the lower bound, combining $\lambda_{\max}(U_2^\top U_2)\le \mathrm{tr}(U_2^\top U_2)=\|U_2\|_F^2$ with incoherence gives 
\[
\|U_2\|_\mathrm{op}^2 \leq \|U_2\|_F^2 = \sum_{j\in S^\complement} \|u_j\|_2^2 \le N_2 \,  \frac{\mu r}{N},
\]
which leads to  
\[
U_1^\top U_1 = I_r - U_2^\top U_2 \succeq \Bigl(1-\frac{\mu r\,N_2}{N}\Bigr) I_r \succeq \frac{1}{2} I_r \succeq \frac{1}{2} \frac{N_1}{N}I_r.
\]
In particular, the last step follows from $N_1/N \leq 1$, while in the penultimate inequality we used $N_1 \ge \frac{2\mu r}{2\mu r+1}N$, which implies $N_2/N \le \frac{1}{2\mu r+1}$, and thus $1-\frac{\mu r\,N_2}{N}  \ge 1-\frac{\mu r}{2\mu r+1}  = \frac{\mu r+1}{2\mu r+1}  \ge \frac12$.

For the upper bound, it is useful to notice that $\frac54\frac{N_1}{N} \ge \frac54 \frac{2\mu r}{2\mu r+1} \ge 1$, where the middle inequality follows from $\mu r\ge 2$. This, together with $U_1^\top U_1\preceq I_r$, shows that
$U_1^\top U_1 \preceq I_r \preceq \frac54 \frac{N_1}{N} I_r$, and concludes the proof.
\end{proof}

\medskip \medskip

\begin{lemma}
\label{lem:A1-random}
Let $U\in\mathbb{R}^{N\times r}$ have orthonormal columns and satisfy 
$\|U^\top \boldsymbol{e}_j\|_2^2 \le \mu r/N$ for all $j\in[N]$ for some $\mu\ge 1$. Let $U_1$ be formed by selecting $N_1$ rows from $U$ uniformly at random without replacement. For all $\varepsilon \in (0,1)$ we have
\[
\mathbb{P} \left\{(1-\varepsilon)\frac{N_1}{N}I_r \ \preceq\ U_1^\top U_1 \ \preceq\ (1+\varepsilon)\frac{N_1}{N}I_r \right\} \geq 1 - 2r\exp \left\{-\frac{N_1\varepsilon^2}{3\mu r}\right\}.
\]
\end{lemma}

\begin{proof}
For each $j\in[N]$, define $u_j:=U^\top \boldsymbol e_j\in\mathbb{R}^r$. Let $S$ denote the random set of $N_1$ row indices sampled uniformly at random without replacement from $[N]$, so that $U_1^\top U_1=\sum_{j\in S}u_j u_j^\top$. Equivalently, we may write $U_1^\top U_1=\sum_{i=1}^{N_1}X_i$, where $X_1,\dots,X_{N_1}$ are sampled uniformly without replacement from $\{u_j u_j^\top:j\in[N]\}$.

Each $u_j u_j^\top$ is positive semidefinite, and by the incoherence assumption we have $\lambda_{\max}(u_j u_j^\top)=\|u_j\|_2^2\le \mu r/N=:B$ for all $j\in[N]$. Moreover, each $X_i$ has marginal distribution uniform on $\{u_j u_j^\top:j\in[N]\}$, so
\[
\mathbb E X_i=\frac1N\sum_{j=1}^N u_j u_j^\top=\frac1N U^\top U=\frac1N I_r .
\]
Writing $\lambda_-:=\lambda_{\min}(\sum_{i=1}^{N_1}\mathbb E X_i)$ and $\lambda_+:=\lambda_{\max}(\sum_{i=1}^{N_1}\mathbb E X_i)$, we thus have $\lambda_-=\lambda_+=N_1/N$.

Although the matrices $X_i$ are dependent, the trace-moment argument in \citet{GrossNesme2010NoteSamplingCoRR} allows the usual matrix Chernoff bounds to be applied to sampling without replacement from a finite collection. Hence, using \citet[][Theorem~1.1]{tropp2012userfriendly} with dimension $r$, norm bound $B=\mu r/N$, and mean eigenvalues $\lambda_-=\lambda_+=N_1/N$, gives, for $\varepsilon\in(0,1)$,
\begin{align*}
\mathbb{P}\left\{\lambda_{\min}(U_1^\top U_1)\le (1-\varepsilon)\frac{N_1}{N}\right\}
&\le r\left\{\frac{e^{-\varepsilon}}{(1-\varepsilon)^{1-\varepsilon}}\right\}^{\lambda_-/B}
\\&= r\left\{\frac{e^{-\varepsilon}}{(1-\varepsilon)^{1-\varepsilon}}\right\}^{N_1/(\mu r)}
\le r\exp\left\{-\frac{N_1\varepsilon^2}{3\mu r}\right\},
\end{align*}
\begin{align*}
\mathbb{P}\left\{\lambda_{\max}(U_1^\top U_1)\ge (1+\varepsilon)\frac{N_1}{N}\right\}
&\le r\left\{\frac{e^\varepsilon}{(1+\varepsilon)^{1+\varepsilon}}\right\}^{\lambda_+/B}
\\ &= r\left\{\frac{e^\varepsilon}{(1+\varepsilon)^{1+\varepsilon}}\right\}^{N_1/(\mu r)}
\le r\exp\left\{-\frac{N_1\varepsilon^2}{3\mu r}\right\}.
\end{align*}
In the last inequalities we used the standard bounds
\[
\frac{e^{-\varepsilon}}{(1-\varepsilon)^{1-\varepsilon}}
\le
\exp\left\{-\frac{\varepsilon^2}{2}\right\}
\le
\exp\left\{-\frac{\varepsilon^2}{3}\right\},
\qquad
\frac{e^\varepsilon}{(1+\varepsilon)^{1+\varepsilon}}
\le
\exp\left\{-\frac{\varepsilon^2}{3}\right\},
\]
both of which hold for $\varepsilon\in(0,1)$. A union bound over the lower- and upper-tail events concludes the~proof.
\end{proof}

\medskip\medskip
Note that both lemmas become uninformative as soon as the rank $r$ is of the same order as the sampled dimension $N_1$. In Lemma~\ref{lem:A1-deterministic}, the sufficient condition $N_1 \ge \frac{2\mu r}{2\mu r+1}\,N$ forces $N_1/N \approx 1$ when $r$ is large. Interpreted in the four-block setting used in the main body, this means that the fraction of rows containing missing entries must be exceptionally small. In Lemma~\ref{lem:A1-random}, while $U_1^\top U_1$ concentrates around $(N_1/N)I_r$ when~$N_1$ is much larger than $\mu r \log r$, a Marchenko--Pastur-type heuristic suggests that $\lambda_{\min}(U_1^\top U_1) \approx \frac{N_1}{N}(1-\sqrt{\gamma})^2$ as $r/N_1 \to \gamma \in (0,1)$, so the lower constant $c_\ell$ in   \eqref{assump:subblock-conditioning} deteriorates and can be arbitrarily small when $r$ is too large relative to $N_1$. 

Returning to Lemma~\ref{lem:A1-deterministic}, the requirement that the missing block be small has close analogues in the MNAR causal-panel matrix-completion literature. For instance, horizontal regression \citep[see, e.g.,][for a discussion]{athey2021matrix} in the unconfoundedness literature is most appropriate when there are many control units relative to the number of periods, whereas vertical regression in the synthetic-control literature is most appropriate when there are many pre-treatment periods relative to the number of donor units. More precisely, horizontal regression is essentially feasible only when $N\gg T$, whereas vertical regression is viable when $T\gg N$. In either case, estimation is reliable only if the corresponding regression design matrices and the implied factor structure are sufficiently well-conditioned. Moreover, settings with limited missing values are also a central building block of \citet{choi2024matrix}. They first show that nuclear-norm regularization can accurately estimate the low-rank signal under MNAR when the total number of missing entries is sufficiently small (Assumption~(iii) in Theorem~2.1). They then extend to more general MNAR patterns by partitioning the missing set into small groups so that each subproblem has few missing entries.

\subsection{Hardness results for $c_\ell = 0$}\label{app:hardness}
We illustrate why restricting our attention to $c_\ell > 0$ in  \eqref{assump:subblock-conditioning} is essential for $\mu_{xy}^{(k)}$ to be identifiable. Indeed, when $c_\ell = 0$, the restricted Gram matrices $U_{1j}^\top U_{1j}$ and $V_{1j}^\top V_{1j}$ are allowed to be singular. This creates directions in which the slice-specific core $\mathcal C_{\bullet,\bullet,k}$ can be perturbed so that only the unobserved $d$-block of $\mathcal{M}_{\bullet, \bullet, k}$ changes, while all observed entries across all slices remain the same. Therefore, two elements of the class can induce the same distribution while having different values of $\mu_{xy}^{(k)}$, so consistent estimation is impossible. 

We recall
$\mu_{xy}^{(k)}(\mathcal M)=x^\top \mathcal M_{\bullet,\bullet,k}^{(d)}y$ and
$Z_\Omega=\{\,\mathcal M_{itj}+\mathcal E_{itj}:\Omega_{i,t,j}=~1,\ (i,t,j)\in[N]\times [T] \times [K]\,\}$, with the mask $\Omega$ fixed and known, and write $\mathbb P_\mathcal M$ and $\mathbb E_\mathcal M$ for probability and expectation under the law of $Z_\Omega$.

\begin{prop}\label{prop:hard1}
 Fix an index $k\in[K]$, constants
$\gamma_{\max}>\gamma_{\min}>0$, and unit vectors
$x\in \mathcal S^{N_{2k}-1}$, $y\in \mathcal S^{T_{2k}-1}$.
Let $\bar c_u := \max(
N/\min_{j\in[K]}N_{1j}, \, 
T /\min_{j\in[K]}T_{1j})$.
Then
\[
\inf_{\phi}
\sup_{\mathcal M\in\mathcal F(0,\bar c_u)}
\mathbb E_\mathcal M
\left[
\left\{\phi(Z_\Omega)-\mu_{xy}^{(k)}(\mathcal M)\right\}^2
\right]
\ge
\frac{(\gamma_{\max}-\gamma_{\min})^2}{4},
\]
where the infimum is over all Borel-measurable functions
$\phi$ of the observed entries $Z_\Omega$.
\end{prop}

\begin{proof}
Define $\bar x\in \mathcal{S}^{N-1}$ and $\bar y\in\mathcal{S}^{T-1}$ to be the vectors with entries $\bar x_i := x_{\,i-N_{1k}} \, \mathbbm{1}\{i> N_{1k}\}$ and $\bar y_t := y_{\,t-T_{1k}} \, \mathbbm{1}\{t> T_{1k}\}$, respectively. Choose vectors $\{u_2,\dots,u_r\}\subset~\mathbb{R}^N$ and $\{v_2,\dots,v_r\}\subset\mathbb{R}^T$ such that $\{\bar x,u_2,\dots,u_r\}$ and $\{\bar y,v_2,\dots,v_r\}$ are orthonormal sets. In particular, this is possible since $r\le \min(N,T)$. Define $U := \big(\,\bar x \;\vert\; u_2 \;\vert\; \cdots \;\vert\; u_r\,\big)\in\mathbb{R}^{N\times r},
\,\,
V := \big(\,\bar y \;\vert\; v_2 \;\vert\; \cdots \;\vert\; v_r\,\big) \in\mathbb{R}^{T\times r}$
so that $U^\top U=V^\top V=I_r$. For every slice $j\neq k$ we take $C_{\bullet, \bullet, j}=  \gamma_{\min} \, I_r$. For the specific slice $k$, writing $\boldsymbol{e}_1 = (1, 0, \ldots, 0)^\top \in \mathbb{R}^r$ for the first vector of the canonical basis in $\mathbb{R}^r$, we set $\mathcal C_{\bullet, \bullet, k}^- =  \gamma_{\min} \, I_r$ and $\mathcal C_{\bullet, \bullet, k}^+ =  \gamma_{\min} I_r + (\gamma_{\max} - \gamma_{\min})\,\boldsymbol{e}_1\boldsymbol{e}_1^\top$. Finally, we define $\mathcal M^\pm := \mathcal C^\pm \times_1 U \times_2 V \times_3 I_K$.

We first verify that $\mathcal M^+,\mathcal M^-\in\mathcal F(0,\bar c_u)$.
The orthonormality constraints on $U$ and $V$ hold by construction. Moreover,
for every $j\in[K]$,
\begin{align*}
0 \preceq U_{1j}^\top U_{1j} \preceq I_r
\preceq \bar c_u \frac{N_{1j}}{N} I_r,
\qquad
0 \preceq V_{1j}^\top V_{1j} \preceq I_r
\preceq \bar c_u \frac{T_{1j}}{T} I_r,
\end{align*}
by the definition of $\bar c_u$, hence Assumption  \eqref{assump:subblock-conditioning}
holds with $c_\ell=0$ and $c_u=\bar c_u$. Finally, for $j\neq k$,
all singular values of $\mathcal C_{\bullet,\bullet,j}$ are equal to
$\gamma_{\min}$; for $j=k$, the singular values of
$\mathcal C_{\bullet,\bullet,k}^-$ are all equal to $\gamma_{\min}$ and
those of $\mathcal C_{\bullet,\bullet,k}^+$ lie in
$[\gamma_{\min},\gamma_{\max}]$. Therefore both tensors belong to
$\mathcal F(0,\bar c_u)$.

Next, since $\bar x$ and $\bar y$ are supported on the row and column indices
of the missing $d$-block of slice $k$, the rank-one perturbation
$(\gamma_{\max}-\gamma_{\min})\bar x\bar y^\top$ is supported entirely on
that missing block. Hence $P_{\Omega_{\bullet, \bullet, j}}(\mathcal M_{\bullet,\bullet,j}^+)
=
P_{\Omega_{\bullet, \bullet, j}}(\mathcal M_{\bullet,\bullet,j}^-)$ for all $j\in[K]$. Since the noise distribution is the same under $\mathcal M^+$ and
$\mathcal M^-$, it follows that the induced laws of the observed data coincide, that is $\mathbb P_{\mathcal M^+} = \mathbb P_{\mathcal M^-}$. On the other hand, the corresponding target parameters are separated. Indeed,
\begin{align*}
\mu_{xy}^{(k)}(\mathcal M^+)-\mu_{xy}^{(k)}(\mathcal M^-)
&= x^\top
\left(\mathcal M_{\bullet,\bullet,k}^{+,(d)}
-\mathcal M_{\bullet,\bullet,k}^{-,(d)}\right)y = (\gamma_{\max}-\gamma_{\min})x^\top x\, y^\top y = \gamma_{\max}-\gamma_{\min}.
\end{align*}
Now, since
$\mathbb P_{\mathcal M^+}=\mathbb P_{\mathcal M^-}$, expectations under the
two laws are identical for every measurable function $\phi$ of~$Z_\Omega$. We can therefore bound
\begin{align*}
&\mathbb E_{\mathcal M^+}
\left[
\left\{\phi(Z_\Omega)-\mu_{xy}^{(k)}(\mathcal M^+)\right\}^2
\right]
+
\mathbb E_{\mathcal M^-}
\left[
\left\{\phi(Z_\Omega)-\mu_{xy}^{(k)}(\mathcal M^-)\right\}^2
\right] \\
&\qquad =
\mathbb E_{\mathcal M^+}
\left[
\left\{\phi(Z_\Omega)-\mu_{xy}^{(k)}(\mathcal M^+)\right\}^2
+
\left\{\phi(Z_\Omega)-\mu_{xy}^{(k)}(\mathcal M^-)\right\}^2
\right] \\
&\qquad \geq
\frac{1}{2}
\left\{
\mu_{xy}^{(k)}(\mathcal M^+)
-
\mu_{xy}^{(k)}(\mathcal M^-)
\right\}^2 =
\frac{1}{2}(\gamma_{\max}-\gamma_{\min})^2,
\end{align*}
where we used the elementary inequality $(a-b)^2+(a-c)^2\geq (b-c)^2/2$.
Consequently, we have
\begin{align*}
&\sup_{\mathcal M\in\mathcal F(0,\bar c_u)}
\mathbb E_\mathcal M
\left[
\left\{\phi(Z_\Omega)-\mu_{xy}^{(k)}(\mathcal M)\right\}^2
\right]
\\
&\geq
\frac{1}{2}
\left(
\mathbb E_{\mathcal M^+}
\left[
\left\{\phi(Z_\Omega)-\mu_{xy}^{(k)}(\mathcal M^+)\right\}^2
\right]
+
\mathbb E_{\mathcal M^-}
\left[
\left\{\phi(Z_\Omega)-\mu_{xy}^{(k)}(\mathcal M^-)\right\}^2
\right]
\right) \geq 
\frac{(\gamma_{\max}-\gamma_{\min})^2}{4}.
\end{align*}
Taking the infimum over all Borel-measurable functions $\phi$ gives the
claimed lower bound.
\end{proof}

The choice of $\bar c_u$ is to make the construction feasible over arbitrary missingness patterns. The same lower bound can be proved for each fixed $c_u>0$ when $N_{2k}\le N_{2j}$ and $T_{2k}\le T_{2j}$ for all $j\neq k$, using a similar construction.

\medskip
The second result shows that this issue is not merely an artifact of allowing slice-specific heterogeneity
in~$\mathcal C_{\bullet, \bullet, j}$. Even if we impose the strongest possible homogeneity assumption, i.e.~$\mathcal C_{\bullet, \bullet, 1}=\cdots=\mathcal C_{\bullet, \bullet, K}$, which is equivalent to $\mathcal{M}_{\bullet, \bullet, 1}=\cdots= \mathcal{M}_{\bullet, \bullet, K}$, setting $c_\ell = 0$ can still make the problem information-theoretically hard unless auxiliary layers provide enough
complementary information on the missing structure relevant to the target functional. We write $\mathcal{F}_{\mathrm{id}}(c_\ell, c_u) := \{\mathcal{M} \in \mathcal{F}(c_\ell, c_u) \, : \, \mathcal C_{\bullet, \bullet, 1}=\cdots=\mathcal C_{\bullet, \bullet, K}\}$.

\begin{prop}\label{prop:hard2}
     Fix an index $k\in[K]$, constants
$\gamma_{\max}\geq \gamma_{\min}>0$, and unit vectors
$x\in\mathcal S^{N_{2k}-1}$, $y\in\mathcal S^{T_{2k}-1}$.
Also let
$\bar x\in \mathcal{S}^{N-1}$ and $\bar y\in\mathcal{S}^{T-1}$ to be the vectors with entries $\bar x_i := x_{\,i-N_{1k}} \, \mathbbm{1}\{i> N_{1k}\}$ and $\bar y_t := y_{\,t-T_{1k}} \, \mathbbm{1}\{t> T_{1k}\}$, respectively. Define $
S_k(x,y)
:=
\sum_{j=1}^K
\left\|P_{\Omega_{\bullet, \bullet, j}}(\bar x\bar y^\top)\right\|_F^2$, and $\bar c_u := \max(
N/\min_{j\in[K]}N_{1j}, \, 
T /\min_{j\in[K]}T_{1j})$. Then
\[
\inf_{\phi}
\sup_{\mathcal M\in\mathcal F_{\mathrm{id}}(0,\bar c_u)}
\mathbb E_\mathcal M
\left[
\left\{
\phi(Z_\Omega)-\mu_{xy}^{(k)}(\mathcal M)
\right\}^2
\right]
\ge
\max_{\gamma\in[\gamma_{\min},\gamma_{\max}]}
\gamma^2
\left[
2-2\Phi\left(
\frac{\gamma\sqrt{S_k(x,y)}}{\sigma}
\right)
\right],
\]
where the infimum is over all Borel-measurable functions
$\phi$ of the observed entries $Z_\Omega$.
\end{prop}

\begin{proof}
Choose vectors $\{u_2,\dots,u_r\}\subset~\mathbb{R}^N$ and $\{v_2,\dots,v_r\}\subset\mathbb{R}^T$ such that $\{\bar x,u_2,\dots,u_r\}$ and $\{\bar y,v_2,\dots,v_r\}$ are orthonormal sets. In particular, this is possible since $r\le \min(N,T)$. Define $U^\pm := \big(\pm\,\bar x \;\vert\; u_2 \;\vert\; \cdots \;\vert\; u_r\,\big)\in\mathbb{R}^{N\times r},
\, \, \, 
V~:=\big(\,\bar y \;\vert\; v_2 \;\vert\; \cdots \;\vert\; v_r\,\big) \in\mathbb{R}^{T\times r}$
so that $({U^\pm})^\top U^\pm=V^\top V=I_r$. Also, for all $j \in [K]$ set $\mathcal C_{\bullet, \bullet, j} = \gamma I_r$, where $\gamma \in [\gamma_{\min}, \gamma_{\max}]$. 

These choices induce $\mathcal{M}^\pm \in \mathcal{F}_\mathrm{id}(0, \bar c_u)$ with $\mathcal{M}^{\pm}_{\bullet, \bullet, j} = \gamma \, U^\pm V^\top$. The precise computation follows an argument similar to the construction used in the proof of Proposition~\ref{prop:hard1}. Furthermore, we have
\[
\mu_{xy}^{(k)}(\mathcal M^\pm)
= x^\top \mathcal M_{\bullet,\bullet,k}^{\pm,(d)} y
= \bar x^\top \mathcal M_{\bullet,\bullet,k}^{\pm} \, \bar y
= \gamma \bar x^\top U^\pm V^\top \bar y
= \pm \gamma .
\]
hence $\{\mu_{xy}^{(k)}(\mathcal{M}^+) - \mu_{xy}^{(k)}(\mathcal{M}^-)\}^2 = 4 \gamma^2$. By Le Cam's two-point method~\citep[][Theorem~2.2]{tsybakov2009nonparametric},
for any measurable function $\phi$, the minimax risk is lower bounded by
$\gamma^2\{1-\operatorname{TV}(\mathbb P_{\mathcal M^+},
\mathbb P_{\mathcal M^-})\}$, where $\mathbb P_{\mathcal M^\pm}$ denotes
the law of the observed entries $Z_\Omega$ under the signal $\mathcal M^\pm$. In particular, under Gaussian noise with common variance~$\sigma^2$, independence of the errors across slices implies that the joint law of all observed entries across all slices is multivariate Gaussian with mean vector equal to the vectorization of $\{P_{\Omega_{\bullet, \bullet, j}}(\mathcal M^\pm_{\bullet, \bullet, j})\}_{j\in[K]}$ and covariance~$\sigma^2 I$. This, combined with $\mathcal{M}_{\bullet, \bullet, j}^+ - \mathcal M_{\bullet, \bullet, j}^- =\gamma \, (U^+-U^-)V^\top=2\gamma \, \bar x\bar y^\top$, gives 
\[
\operatorname{TV}(\mathbb P_{\mathcal M^+},\mathbb P_{\mathcal M^-}) = 2 \Phi\left( \frac{\sqrt{\sum_{j = 1}^K\| P_{\Omega_{\bullet, \bullet, j}}(2 \gamma \, \bar x \bar y^\top)\|_F^2}}{2\sigma}\right) - 1 = 2 \Phi\left( \frac{\gamma \, \sqrt{S_k(x,y)}}{\sigma}\right) - 1,
\]
and concludes the proof upon taking the maximum over $\gamma \in [\gamma_{\min}, \gamma_{\max}]$.
\end{proof}

Proposition~\ref{prop:hard2}  gives a lower bound on the minimax risk over $\mathcal F_{\mathrm{id}}(0,\bar c_u)$, where all
slices are identical. In particular, the lower bound depends on $S_k(x,y)
= \sum_{j=1}^K \left\|P_{\Omega_{\bullet, \bullet, j}}(\bar x\bar y^\top)\right\|_F^2
$, which quantifies how often the rank-one pattern $\bar x\bar y^\top$, supported on slice~$k$'s missing block, is observed across other layers. For example, when $x=N_{2k}^{-1/2}\mathbf 1_{N_{2k}}$ and $
y=T_{2k}^{-1/2}\mathbf 1_{T_{2k}}$
we have
\[
S_k(x,y)
=
\frac{1}{N_{2k}T_{2k}}
\sum_{j=1}^K
\left\|(1-\Omega_{\bullet, \bullet, k})\odot \Omega_{\bullet, \bullet, j}\right\|_0,
\]
so $S_k(x,y)$ is the fraction of slice $k$'s
missing block that is observed elsewhere.  When a slice $j \neq k$ has a much
smaller missing block, this overlap increases and drives the lower bound to zero. On the other hand, in general, if $S_k(x,y)\lesssim \sigma^2/\gamma^2$, the two alternatives in the
proof remain statistically close, and the minimax risk is of constant order. The extreme case corresponds to $S_k(x,y)=0$, which occurs when the entire $d$-block missing under
slice $k$ is also missing under every other slice; for example, this holds when
$N_{2k}\leq N_{2j}$ and $T_{2k}\leq T_{2j}$ for all $j\neq k$.

Taken together, Propositions~\ref{prop:hard1} and~\ref{prop:hard2} highlight two distinct failure modes when $c_\ell=0$. Proposition~\ref{prop:hard1} shows that with slice-specific cores, the target functional can be non-identifiable, as different parameter values can induce the same distribution of observed entries while yielding different $\mu_{xy}^{(k)}$. Proposition~\ref{prop:hard2} shows that even if the slices share a common core, the target may still be hard to estimate, as the minimax risk can remain bounded away from zero unless the missing entries of slice~$k$ are sufficiently observed in other slices.

\subsection{Background on tensors}\label{appendix:tensor}

We now provide a brief background on tensors to familiarize the reader with the notation used in our model $\mathcal{M}=\mathcal{C}\times_1 U\times_2 V\times_3 I_K$. Although the main body of the paper considers only order-$3$ tensors, we keep this section fairly general at first. We then state and prove a result connecting the Tucker2 model with standard low-rank matrix factorizations.

A tensor is a multidimensional array $\mathcal{X}\in\mathbb{R}^{n_1\times\cdots\times n_d}$ with entries $x_{i_1\ldots i_d}$, where $1\le i_j\le n_j$ for $j \in [d]$; order-$1$ tensors are vectors and order-$2$ tensors are matrices. Fixing all indices except $i_j$ yields a mode-$j$ fiber, which is the higher-order analogue of rows and columns. Fixing all but two indices yields a slice, which forms a two-dimensional subarray, e.g.\ for an order-$3$ tensor $\mathcal{X}$ the frontal slice $\mathcal{X}_{\bullet, \bullet, j}$ fixes the third index. 

To connect tensor algebra to matrix algebra, it is convenient to use matricization and define the mode-$j$ unfolding $X_{(j)}\in\mathbb{R}^{n_j\times \prod_{k\ne j} n_{k}}$ as a rearrangement of $\mathcal{X}$ so that the mode-$j$ fibers become the columns of a matrix. Under this representation, the $j$-mode product \citep[][Section~2.5]{Kolda2009Tensor} reduces to ordinary matrix multiplication. Specifically, for $A\in\mathbb{R}^{n^\prime \times n_j}$, the tensor $\mathcal{Y}=\mathcal{X}\times_j A$ has dimensions $n_1 \times \cdots \times n_{j-1} \times n^\prime \times n_{j+1} \times \cdots \times n_d$ and entries
\[
(\mathcal{X}\times_j A)_{i_1\ldots i_{j-1} k\, i_{j+1}\ldots i_d}
=\sum_{i_j=1}^{n_j} x_{i_1\ldots i_d}\,a_{k i_j}.
\]
Furthermore, its unfolding satisfies $(\mathcal{X}\times_j A)_{(j)}=A X_{(j)}$. 

A core model in multilinear algebra is the Tucker decomposition, which represents a tensor as a low-dimensional core transformed along each mode and is of the form $\mathcal{X}=\mathcal{G}\times_1 A^{(1)}\times_2\cdots\times_d A^{(d)}$, where the core $\mathcal{G}$ encodes interactions among latent components and the factor matrices map these components to the ambient spaces \citep[][Equations 4.1--4.2]{Kolda2009Tensor}. Setting one factor matrix to the identity yields the Tucker2 model, introduced in Section~\ref{sec:4block}; in our notation this gives $\mathcal{M}=\mathcal{C}\times_1 U\times_2 V\times_3 I_{K}$, so mode 3 is left unchanged, i.e.~slices are not mixed across $k$, while $U$ and $V$ act along modes 1 and 2, respectively.

\medskip
We next provide an equivalent characterization for this model. 
\begin{prop}\label{prop:Tucker2Equivalent}
    The following are equivalent: 
    \begin{enumerate}
        \item $M^{(j)}=U R_jV^\top$ for all $j \in [K]$, with common orthonormal $U \in \mathbb{R}^{N \times r}, V \in \mathbb{R}^{T \times r}$ and $ R_j \in \mathbb{R}^{r\times r}$;
        \item The column (resp.~row) spaces of all $M^{(j)}$ are contained in a common subspace $\mathcal{U}$ (resp.~$\mathcal{V}$) of dimension at most $r$;
        \item Stacking the matrices $M^{(j)}$ yields a tensor $\mathcal{M}\in\mathbb{R}^{N\times T\times K}$ that admits a Tucker2 decomposition $\mathcal{M}=\mathcal{C}\times_1 U\times_2 V\times_3 I_K$, with core $\mathcal{C}\in\mathbb{R}^{r\times r\times K}$ and shared mode-$1$/mode-$2$ orthonormal factor matrices $U$ and $V$.
    \end{enumerate} 
\end{prop}

\begin{proof}
(1) $\Rightarrow$ (2): If $M^{(j)} = U R_j V^\top$ with $U\in\mathbb R^{N \times r}$, $V\in\mathbb R^{T \times r}$ column-orthonormal, then 
$\operatorname{col}(M^{(j)}) \subseteq \operatorname{col}(U)=:\mathcal U$ and 
$\operatorname{row}(M^{(j)}) \subseteq \operatorname{col}(V)=:\mathcal V$ for all $j \in [K]$, so (2) holds. \medskip

\noindent (2) $\Rightarrow$ (1): Let $\mathcal U,\mathcal V$ be subspaces containing all column and row spaces, with $\dim(\mathcal U)\le r$ and $\dim(\mathcal V)\le r$, and let $U,V$ be orthonormal bases of $\mathcal U,\mathcal V$. Denote the orthogonal projections by $P_U := UU^\top$ and $P_V := VV^\top$. For each $j \in [K]$, the assumptions imply $P_U M^{(j)} = M^{(j)}$ and $M^{(j)} P_V = M^{(j)}$, hence
\[
M^{(j)} = P_U M^{(j)} P_V = U\bigl(U^\top M^{(j)} V\bigr)V^\top.
\]
Setting $R_j := U^\top M^{(j)} V$ gives (1). \medskip

\noindent (1) $\Rightarrow$ (3): Stack the matrices as a tensor ${\cal M} \in\mathbb R^{N \times T \times K}$ with frontal slices ${\cal M}_{\bullet, \bullet, j} = M^{(j)}$. Define a core tensor $\mathcal C \in\mathbb R^{r\times r\times K}$ by $\mathcal C_{\bullet, \bullet, j} := R_j$. Then $\mathcal M  = \mathcal C \times_1 U \times_2 V \times_3 I_{K}$, which is a Tucker2 decomposition with shared mode-1/mode-2 orthonormal factors $U,V$. \medskip

\noindent (3) $\Rightarrow$ (1): Conversely, suppose $\cal M$ has a Tucker decomposition
$\mathcal M = \mathcal C \times_1 U \times_2 V \times_3 I_K$
with $U\in\mathbb R^{N\times r}$, $V\in\mathbb R^{T\times r}$ column-orthonormal and $\mathcal C\in\mathbb R^{r\times r\times K}$. Writing $\mathcal C_{\bullet, \bullet, j}$ for the $j$-th frontal slice of $\mathcal C$, the $j$-th slice of $\cal M$ is
\[
\mathcal M_{\bullet, \bullet, j}
= \sum_{k=1}^{K} \delta_{jk}\, U \, \mathcal C_{\bullet, \bullet, k} \, V^\top
= U\Bigl(\sum_{k=1}^{K} \delta_{jk}\,\mathcal C_{\bullet, \bullet, k}\Bigr)V^\top = U \, \mathcal C_{\bullet, \bullet, j} V^\top.
\]
This concludes the proof.
\end{proof}

\medskip\medskip\medskip
\section{Auxiliary probabilistic results}\label{appendix:AuxilliaryRes}
In this appendix we collect some useful results that are used in the proofs of our main results. 

\medskip
We begin with two tail probability bounds. For $\sigma>0$, a random variable $X$ with mean $\mu=\mathbb{E}[X]$ is said to be $\sigma$-subgaussian if $\mathbb{E}[e^{\lambda(X-\mu)}] \leq e^{\sigma^2 \lambda^2 / 2}$ for all $\lambda \in \mathbb{R}$.

\begin{lemma}\label{lemma:bilinearGaussian}
Let $E \in \mathbb{R}^{n_1 \times n_2}$ be a random matrix with mean-zero independent $\sigma$-subgaussian entries. For any fixed
matrices $X \in \mathbb{R}^{n_1 \times p_1}$ and $Y \in \mathbb{R}^{n_2 \times p_2}$, for all $\delta \in (0,1)$ there exists an absolute constant $c_1>0$ such that
\[
\|X^\top E Y\|_\mathrm{op}
\;\le\;  c_1\,\sigma\,\|X\|_\mathrm{op}\,\|Y\|_\mathrm{op}\,
\sqrt{\operatorname{rank}(X)+\operatorname{rank}(Y)+\log (\delta^{-1})}
\]
with probability at least $1 - \delta$.
\end{lemma}

\begin{proof}
Write the compact singular value decompositions $X = U_X \Sigma_X V_X^\top$ and $Y = U_Y \Sigma_Y V_Y^\top$, where
$U_X \in \mathbb{R}^{n_1 \times r_X}, V_X \in \mathbb{R}^{p_1 \times r_X}, U_Y \in \mathbb{R}^{n_2 \times r_Y}, V_Y \in \mathbb{R}^{p_2 \times r_Y}$ have orthonormal columns,
$r_X:=\operatorname{rank}(X)$ and $r_Y:=\operatorname{rank}(Y)$, and $\|\Sigma_X\|_\mathrm{op}=\|X\|_\mathrm{op}$, $\|\Sigma_Y\|_\mathrm{op}=\|Y\|_\mathrm{op}$. Then, we can write $X^\top E Y
= V_X \Sigma_X \bigl(U_X^\top E U_Y\bigr)\Sigma_Y V_Y^\top$, and by the submultiplicativity and orthonormal invariance of the spectral norm, we have $\|X^\top E Y\|_\mathrm{op} \le \|X\|_\mathrm{op}\,\|Y\|_\mathrm{op}\,\|U_X^\top E U_Y\|_\mathrm{op}$. It remains to bound $\|U_X^\top E U_Y\|_\mathrm{op}$. For any $x \in \mathcal S^{r_X-1}$ and
$y \in \mathcal S^{r_Y-1}$, define 
\[
Z(x,y) := x^\top U_X^\top E U_Y y
= \sum_{i=1}^{n_1}\sum_{j=1}^{n_2} E_{ij}\,(U_X x)_i \, (U_Y y)_j, 
\]
so that $\|U_X^\top E U_Y\|_\mathrm{op} = \sup_{x \in \mathcal S^{r_X-1}, y \in \mathcal S^{r_Y-1}} |Z(x,y)|$.
For every fixed pair $(x,y)$ and any $\delta \in (0,1)$, Hoeffding's inequality for sums of subgaussian random variables and the fact that $\|U_X \, x\|_2 = \|U_Y \, y\|_2 = 1$ imply that there exists an absolute constant $c_1>0$ such that
\begin{equation}\label{eq:hoeffding_Z(x,y)}
    \mathbb{P}\!\left\{\,|Z(x,y)| \;>\; c_1 \sigma \sqrt{\log(\delta^{-1})}\right\} \le \delta.
\end{equation}
In order to deal with the supremum, we will combine the above display with a standard netting argument \citet[][Chapter 5]{wainwright2019high}. In particular, let $\mathcal{N}_X$ and $\mathcal{N}_Y$ be $\tfrac{1}{4}$-nets of $\mathcal S^{r_X-1}$ and
$\mathcal S^{r_Y-1}$, respectively. Their cardinalities satisfy
$|\mathcal{N}_X|\le 9^{r_X}$ and $|\mathcal{N}_Y|\le 9^{r_Y}$ \citep[Equation 4.20]{vershynin19hdp}. Now, for all $x_1 \in \mathcal S^{r_X-1}$ and $y_1 \in \mathcal S^{r_Y-1}$, we can find $x_2 \in \mathcal N_X$, $y_2 \in \mathcal N_Y$ such that $\|x_1 - x_2\|_2 \leq 1/4$ and $\|y_1 - y_2\|_2 \leq 1/4$. This, combined with $Z(x_1, y_1) = \{Z(x_1, y_1) - Z(x_2, y_1)\} + \{Z(x_2, y_1)-Z(x_2, y_2)\} + Z(x_2, y_2)$, allows showing that $\sup_{x \in \mathcal S^{r_X-1}, y \in \mathcal S^{r_Y-1}} |Z(x,y)| \leq 2 \, \max_{x \in \mathcal N_X, y \in \mathcal N_Y} |Z(x,y)|$ after taking the supremum on both sides. We thus get 
\begin{align*}
    \|U_X^\top E U_Y\|_\mathrm{op}
&= \sup_{x\in \mathcal S^{r_X-1},\,y\in \mathcal S^{r_Y-1}} |Z(x,y)|
\;\le\; 2 \max_{x\in\mathcal{N}_X,\;y\in\mathcal{N}_Y} |Z(x,y)| \\
& \leq c_1 \sigma \sqrt{\log (\,|\mathcal{N}_X|\,|\mathcal{N}_Y| \,/ \, \delta}) \leq c_1 \, \sigma \sqrt{r_X + r_Y + \log(\delta^{-1})}
\end{align*}
with probability at least $1-\delta$, where the second inequality follows from an application of~\eqref{eq:hoeffding_Z(x,y)} with $\delta/(|\mathcal{N}_X||\mathcal{N}_Y|)$ in place of $\delta$, and a union bound over
$\mathcal{N}_X \times \mathcal{N}_Y$. This concludes the proof.
\end{proof}

\medskip \medskip
\begin{lemma}\label{lemma:ExpOpNormE}
Let $E \in \mathbb{R}^{n_1 \times n_2}$ have independent entries distributed as $\mathcal N(0,\sigma^2)$. For every $p \ge 1$ we have
\[
\bigl(\mathbb{E}\|E\|_{\mathrm{op}}^p\bigr)^{1/p}
\le
\sigma\bigl(\sqrt {n_1}+\sqrt {n_2} + c_1\sqrt p\bigr),
\]
where $c_1>0$ is an absolute constant. 
\end{lemma}

\begin{proof}
By homogeneity, it suffices to consider the case $\sigma^2=1$. Indeed, writing $E=\sigma H$ with independent $H_{ij}\sim\mathcal N(0,1)$, we have $\|E\|_{\mathrm{op}}=\sigma\|H\|_{\mathrm{op}}$, and therefore $\bigl(\mathbb E \, \|E\|_{\mathrm{op}}^p\bigr)^{1/p}
=\sigma\bigl(\mathbb E \, \|H\|_{\mathrm{op}}^p\bigr)^{1/p}$. Thus it is enough to prove that $\bigl(\mathbb E\|H\|_{\mathrm{op}}^p\bigr)^{1/p}\le \sqrt{n_1}+\sqrt{n_2}+c_1\sqrt p$. In this regard, writing $c > 0$ for an absolute constant, we recall from \citet[][Theorem 7.3.1 and Corollary 7.3.2]{vershynin19hdp} that $\mathbb{E} \, \|H\|_{\mathrm{op}} \le \sqrt{n_1}+\sqrt{n_2}$, and $\mathbb{P}\bigl(\|H\|_{\mathrm{op}}\ge \sqrt{n_1}+\sqrt{n_2}+t\bigr)\le 2e^{-ct^2}$ for all $t \geq 0$. Furthermore, letting $Y:=(\|H\|_{\mathrm{op}}-\sqrt{n_1}-\sqrt{n_2})_+$, we also have $\|H\|_{\mathrm{op}}\le \sqrt{n_1}+\sqrt{n_2}+Y$ and $(\mathbb{E}\, \|H\|_{\mathrm{op}}^p)^{1/p}\le\sqrt{n_1}+\sqrt{n_2} + (\mathbb E Y^p)^{1/p}$, where the latter bound follows from Minkowski's inequality. 

It remains to bound $(\mathbb EY^p)^{1/p}$. Using the layer-cake formula, the change of variables $u=ct^2$, and Stirling's approximation, we obtain
\begin{align*} \mathbb EY^p &= \int_0^\infty p\,t^{p-1}\,\mathbb P(Y> t)\,dt = \int_0^\infty p\,t^{p-1}\,\mathbb P\bigl(\|H\|_{\mathrm{op}}> \sqrt{n_1}+\sqrt{n_2}+t\bigr)\,dt \\ &\le 2p\int_0^\infty t^{p-1}e^{-ct^2}\,dt = p\,c^{-p/2}\int_0^\infty u^{p/2-1}e^{-u}\,du \\ &= p\,c^{-p/2}\Gamma(p/2) = 2 \,c^{-p/2} \, \Gamma(p/2+1) \le (c_1\sqrt p)^p, 
\end{align*}
for some constant $c_1>0$ depending only on $c$. Combining the above bounds gives
\[
(\mathbb{E}\, \|H\|_{\mathrm{op}}^p)^{1/p}\le \sqrt{n_1}+\sqrt{n_2} + (\mathbb E Y^p)^{1/p} \leq \sqrt{n_1}+\sqrt{n_2} \, + c_1 \sqrt{p},
\]
thereby completing the proof.
\end{proof}

\medskip \medskip
We next recall Weyl's inequality for singular values and eigenvalues~\citep[][Lemmas 2.2--2.3]{spectralMethods}.
\begin{lemma}\label{lemma:Weyl} Let $A,E \in \mathbb{R}^{n \times m}$. Then, for every $1 \leq i \leq \min(n,m)$, the $i$-th largest singular values of $A$ and $A+E$ satisfy \[ \left|\sigma_i(A+E)-\sigma_i(A)\right| \leq \|E\|_{\mathrm{op}} . \] Moreover, if $n=m$ and $A,E \in \mathbb{R}^{n \times n}$ are symmetric, then, for every $1 \leq i \leq n$, the $i$-th largest eigenvalues of $A$ and $A+E$ satisfy \[ \left|\lambda_i(A+E)-\lambda_i(A)\right| \leq \|E\|_{\mathrm{op}} . \] \end{lemma}

\medskip\medskip\medskip
We recall that the Moore--Penrose pseudoinverse of $A=U\operatorname{diag}(\sigma_1,\ldots,\sigma_r)V^\top$, with column orthonormal $U \in \mathbb R^{n_1 \times r}, V \in \mathbb R^{n_2 \times r}$ and $\sigma_i>0$, is $A^\dagger=V\operatorname{diag}(\sigma_1^{-1},\ldots,\sigma_r^{-1}) \,U^\top$. The following lemma gives an exact identity for how the Moore--Penrose inverse changes when a full-column-rank matrix $A$ is perturbed to $B=A+\Delta$, as well as a simple operator-norm bound and a useful formula for the action of $B^\dagger-A^\dagger$ on $A$.

\begin{lemma}\label{lem:pinv_exact_identity}
	Let $A,B\in\mathbb{R}^{n_1\times n_2}$ have full column rank, and let $\Delta:=B-A$. Then
	$B^\dagger-A^\dagger
	=
	-\,A^\dagger\Delta B^\dagger
	+
	A^\dagger(A^\dagger)^\top\Delta^\top(I_{n_1}-BB^\dagger)$.
	Consequently,
	\[
		\|B^\dagger-A^\dagger\|_{\rm op}
		\le
		\|A^\dagger\|_{\rm op}\,\|\Delta\|_{\rm op}\,\|B^\dagger\|_{\rm op}
		+
		\|A^\dagger\|_{\rm op}^2\,\|\Delta\|_{\rm op}.
	\]
	Moreover, we have $(B^\dagger-A^\dagger)A=-\,B^\dagger\Delta$.
\end{lemma}

\begin{proof}
	Since $A$ and $B$ have full column rank, $A^\dagger A=I_{n_2}$, $B^\dagger B=I_{n_2}$, and $A^\dagger=(A^\top A)^{-1}A^\top$. Therefore,
	$B^\dagger-A^\dagger
	=
	(B^\dagger B-A^\dagger B)B^\dagger
	-
	A^\dagger(I_{n_1}-BB^\dagger)
	=
	-\,A^\dagger\Delta B^\dagger
	-
	A^\dagger(I_{n_1}-BB^\dagger)$.
	It remains to rewrite the last term. Since $BB^\dagger$ is the orthogonal projector onto $\operatorname{col}(B)$, we have $B^\top(I_{n_1}-BB^\dagger)=0$. As $B=A+\Delta$, this gives $A^\top(I_{n_1}-BB^\dagger)=-\Delta^\top(I_{n_1}-BB^\dagger)$, and hence
	$-A^\dagger(I_{n_1}-BB^\dagger)=(A^\top A)^{-1}\Delta^\top(I_{n_1}-BB^\dagger)$. Using $A^\dagger(A^\dagger)^\top=(A^\top A)^{-1}$ proves the first identity.

	The norm bound follows from this identity and $\|I_{n_1}-BB^\dagger\|_{\rm op}\le 1$. Finally, $(B^\dagger-A^\dagger)A=B^\dagger A-I_{n_2}=B^\dagger A-B^\dagger B=-B^\dagger\Delta$.
\end{proof}

\medskip \medskip
Finally, we present the auxiliary results on the Stiefel manifold used in the proofs. These include standard facts on the Haar measure and its generation via Gaussian QR decompositions; see, for example, \cite{Stewart1980,Mezzadri2007,Chikuse2003}. For integers $1\le q\le d$, the Stiefel manifold is
\[
\mathrm{St}(d,q)
:=
\{Q\in\mathbb R^{d\times q}:Q^\top Q=I_q\}.
\]
Thus, $\mathrm{St}(d,q)$ is the set of
$d\times q$ matrices whose columns are orthonormal. The special cases $q=1$ and $q=d$ reduce to
the unit sphere $\mathcal S^{d-1}$ and the orthogonal group $\mathbb O(d)$, respectively.

Although $\mathrm{St}(d,q)$ is not a group when $q<d$, it carries a natural probability measure that
is invariant under left multiplication by orthogonal matrices. A random matrix
$Q\in \mathrm{St}(d,q)$ is said to be Haar-distributed or uniformly distributed on the Stiefel manifold if $OQ \stackrel{d}{=} Q$ for every deterministic $O \in \mathbb O(d)$. Such a left-orthogonally invariant probability measure on $\mathrm{St}(d,q)$ is unique~\citep[e.g.][Theorem~1.2.2 and Section~1.3.1]{Chikuse2003}. In words, multiplying $Q$ by any deterministic rotation or reflection does not change its law, hence a
Haar-distributed element of $\mathrm{St}(d,q)$ may be viewed as a uniformly random matrix with orthonormal columns. A standard construction of such matrices is obtained from a Gaussian matrix. If
$G\in\mathbb R^{d\times q}$ has i.i.d. $\mathcal N(0,1)$ entries and
$G=QR$ is its thin QR decomposition with the diagonal entries of $R$ taken positive, then
$Q\in\mathrm{St}(d,q)$ is Haar-distributed. 

The following lemma formalizes a stability property of Haar-distributed
Stiefel matrices. If $Q$ is uniformly distributed on $\mathrm{St}(d,q)$,
then multiplying it on the right by any fixed
$H \in \mathrm{St}(q,\ell)$ produces an
$\ell$-dimensional orthonormal system which is itself uniformly distributed
on $\mathrm{St}(d,\ell)$. The lemma also gives a high-probability bound for
the size of this random orthonormal system after applying a fixed linear map
$A$. In particular, when $\ell$ is small relative to $d$, the operator
norm $\|AQH\|_{\mathrm{op}}$ is at most of the order
\[
\|A\|_{\mathrm{op}}
\sqrt{\frac{\operatorname{rank}(A)+\ell}{d}},
\]
with high probability, up to universal constants.

\begin{lemma}
\label{lem:haar-frame-compression}
Let $d,q,p\in\mathbb N$, and let $1\le \ell\le q\le d$. Let
$Q\in\mathbb R^{d\times q}$ be Haar-distributed on $\mathrm{St}(d,q)$,
and let~$\mathcal F$ be a sigma-field independent of $Q$. Let
$H \in \mathrm{St}(q,\ell)$ be $\mathcal F$-measurable, and let $A\in\mathbb R^{p\times d}$ be deterministic. Then, conditional on $\mathcal F$, the matrix $QH$ is Haar-distributed
on $\mathrm{St}(d,\ell)$. Moreover, there is a universal constant
$c_1>0$ such that, for every $t\ge 0$ with
$\sqrt{\ell}+t<\sqrt d$, we have
\begin{align}
\mathbb P\left\{
\|AQH\|_{\rm op}
\le
c_1\|A\|_{\rm op}
\frac{\sqrt{\operatorname{rank}(A)+\ell+t^2}}
{\sqrt d-\sqrt\ell-t}
\;\middle|\; \mathcal F
\right\}
\ge 1-2e^{-t^2/2}.
\label{eq:haar-frame-compression-raw}
\end{align}
In particular, if $\sqrt\ell+t\le \sqrt d/2$, then 
\begin{align}
\mathbb P\left\{
\|AQH\|_{\rm op}
\le
2 \,c_1\|A\|_{\rm op}
\sqrt{\frac{\operatorname{rank}(A)+\ell+t^2}{d}}
\;\middle|\; \mathcal F
\right\}
\ge 1-2e^{-t^2/2}.
\label{eq:haar-frame-compression-simple}
\end{align}
\end{lemma}

\begin{proof}
Since $Q^\top Q=I_q$ and $H^\top H=I_\ell$, we have
$(QH)^\top(QH)=H^\top Q^\top QH=I_\ell$. Hence
$QH\in\mathrm{St}(d,\ell)$.

We now identify the conditional law of $QH$. Fix $O\in\mathbb O(d)$. Since
$Q$ is Haar-distributed on $\mathrm{St}(d,q)$ and is independent of
$\mathcal F$, its conditional law given $\mathcal F$ is still Haar. Hence
$OQ\stackrel d=Q$ conditionally on $\mathcal F$. Since $H$ is
$\mathcal F$-measurable, it is fixed after conditioning on $\mathcal F$,
and therefore $OQH\stackrel d=QH$ conditionally on $\mathcal F$. Thus the conditional law of $QH$ is invariant under left multiplication by
every deterministic orthogonal matrix $O\in\mathbb O(d)$. By uniqueness of
the left-orthogonally invariant probability measure on $\mathrm{St}(d,\ell)$~\citep[e.g.][Theorem~1.2.2 and Section~1.3.1]{Chikuse2003},
this conditional law is the Haar measure on $\mathrm{St}(d,\ell)$.

Let $G\in\mathbb R^{d\times \ell}$ have independent
$\mathcal N(0,1)$ entries and be independent of $\mathcal F$. By the Gaussian representation of the Haar
measure on the Stiefel manifold \citep[Theorem~2.4.3]{Chikuse2003},
$QH\stackrel d=G(G^\top G)^{-1/2}$ conditionally on $\mathcal F$. Hence
\[
\|AQH\|_{\rm op}
\stackrel d=
\|AG(G^\top G)^{-1/2}\|_{\rm op}
\le
\frac{\|AG\|_{\rm op}}{\sigma_\ell(G)}
\]
conditionally on $\mathcal F$. By Lemma~\ref{lemma:bilinearGaussian}, applied with $X=A^\top$,
$Y=I_\ell$, $\sigma=1$, and $\delta=e^{-t^2/2}$, we have $\|AG\|_{\rm op}
\le
c_1\|A\|_{\rm op}
\sqrt{\operatorname{rank}(A)+\ell+t^2}$
with probability at least $1-e^{-t^2/2}$. Also, the standard lower-tail
bound for the smallest singular value of a Gaussian matrix \citep[Theorem~II.13]{DavidsonSzarek2001} gives
$\sigma_\ell(G)\ge \sqrt d-\sqrt\ell-t$ with probability at least
$1-e^{-t^2/2}$. On the intersection of these two events, which has
probability at least $1-2e^{-t^2/2}$, the bound
\eqref{eq:haar-frame-compression-raw} follows.

Finally, if $\sqrt\ell+t\le \sqrt d/2$, then
$\sqrt d-\sqrt\ell-t\ge \sqrt d/2$. Substituting this lower bound into
\eqref{eq:haar-frame-compression-raw} gives~\eqref{eq:haar-frame-compression-simple}.
\end{proof}
\end{document}